\documentclass{memo-l}
\usepackage[active]{srcltx}

\def\blu{} 

\usepackage{graphicx}
\usepackage{amsmath}
\usepackage{amssymb}
\usepackage{hyperref}
\usepackage{bbm}
\usepackage{longtable}
\usepackage{enumerate}

\newtheorem{thm}{Theorem}[chapter]
\newtheorem{lem}[thm]{Lemma}%
\newtheorem{prop}[thm]{Proposition}%
\newtheorem{cor}[thm]{Corollary}%

\theoremstyle{definition}
\newtheorem{definition}{Definition}[chapter]

\theoremstyle{remark}
\newtheorem{remark}{Remark}[chapter] %
\newtheorem{ex}{Example}[chapter]

\theoremstyle{plain}

\numberwithin{section}{chapter}
\numberwithin{equation}{chapter}

\def\EE{{\mathbb E}}

\def\NN{{\mathbb N}}
\def\PP{{\mathbb P}}

\def\RR{{\mathbb R}}
\def\TT{{\mathbb T}}

\def\ZZ{{\mathbb Z}}
\def\mm{{\mathsf m}}

\def\C{\operatorname{C}}
\def\undef{\operatorname{undef}}
\def\out{\operatorname{out}}

\def\pr{\operatorname{pr}}

\def\intl{{\operatorname{int}}}

\def\g{\operatorname{g}}
\def\UB{{\scrB_1^{d-1}}}

\def\US{{\S_1^{d-1}}}

\def\veca{{\text{\boldmath$a$}}}
\def\vecb{{\text{\boldmath$b$}}}
\def\vecB{{\text{\boldmath$B$}}}

\def\vece{{\text{\boldmath$e$}}}

\def\vech{{\text{\boldmath$h$}}}
\def\veck{{\text{\boldmath$k$}}}
\def\vecK{{\text{\boldmath$K$}}}

\def\vecell{{\text{\boldmath$\ell$}}}
\def\vecm{{\text{\boldmath$m$}}}

\def\vecq{{\text{\boldmath$q$}}}

\def\vecp{{\text{\boldmath$p$}}}
\def\vecr{{\text{\boldmath$r$}}}
\def\uvecs{\widehat{\text{\boldmath$s$}}}
\def\vecs{{\text{\boldmath$s$}}}

\def\vect{{\text{\boldmath$t$}}}
\def\vecu{{\text{\boldmath$u$}}}
\def\vecv{{\text{\boldmath$v$}}}
\def\vecV{{\text{\boldmath$V$}}}

\def\vecvp{{\text{\boldmath$v$}}_{+}}
\def\vecw{{\text{\boldmath$w$}}}

\def\vecx{{\text{\boldmath$x$}}}

\def\vecy{{\text{\boldmath$y$}}}
\def\vecz{{\text{\boldmath$z$}}}
\def\vecalf{{\text{\boldmath$\alpha$}}}
\def\vecbeta{{\text{\boldmath$\beta$}}}

\def\vecpsi{{\text{\boldmath$\psi$}}}

\def\vecomega{{\text{\boldmath$\omega$}}}

\def\vecxi{{\text{\boldmath$\xi$}}}

\def\vecnull{{\text{\boldmath$0$}}}

\def\scrA{{\mathcal A}}
\def\scrB{{\mathcal B}}

\def\scrD{{\mathcal D}}
\def\scrE{{\mathcal E}}
\def\scrF{{\mathcal F}}

\def\scrJ{{\mathcal J}}
\def\scrK{{\mathcal K}}
\def\scrL{{\mathcal L}}
\def\scrM{{\mathcal M}}
\def\scrN{{\mathcal N}}

\def\scrQ{{\mathcal Q}}
\def\scrP{{\mathcal P}}

\def\scrS{{\mathcal S}}

\def\scrU{{\mathcal U}}
\def\scrV{{\mathcal V}}
\def\scrW{{\mathcal W}}
\def\scrX{{\mathcal X}}

\def\fC{{\mathfrak C}}
\def\fD{{\mathfrak D}}

\def\fG{{\mathfrak G}}
\def\fg{{\mathfrak g}}

\def\fU{{\mathfrak U}}
\def\fw{{\mathfrak w}}
\def\fW{{\mathfrak W}}

\def\fZ{{\mathfrak Z}}

\def\e{\mathrm{e}}
\def\i{\mathrm{i}}

\def\diag{\operatorname{diag}}
\def\diam{\operatorname{diam}}
\def\dim{\operatorname{dim}}

\DeclareMathOperator*{\esssup}{ess\,sup}

\def\L{\operatorname{L{}}}
\def\Lloc{\operatorname{L_{\operatorname{loc}}^1}}

\def\GL{\operatorname{GL}}

\def\S{\operatorname{S{}}}

\def\SL{\operatorname{SL}}
\def\ASL{\operatorname{ASL}}

\def\SO{\operatorname{SO}}

\def\O{\operatorname{O{}}}
\def\T{\operatorname{T{}}}

\def\supp{\operatorname{supp}}

\def\vol{\operatorname{vol}}

\def\GamG{\Gamma\backslash G}

\def\Pac{P_{\mathrm{ac}}}

\def\Onder#1#2#3#4#5{#1 \setbox0=\hbox{$#1$}\setbox1=\hbox{$#2$}
       \dimen0=.5\wd0 \dimen1=\dimen0 \dimen2=\dp0 \dimen3=\dimen2
       \advance\dimen0 by .5\wd1 \advance\dimen0 by -#4
       \advance\dimen1 by -.5\wd1 \advance\dimen1 by -#4
       \advance\dimen2 by -#3 \advance\dimen2 by \ht1
       \advance\dimen2 by 0.3ex \advance\dimen3 by #5
        \kern-\dimen0\raisebox{-\dimen2}[0ex][\dimen3]{\box1}
       \kern\dimen1}

\newcommand{\of}{{\overline{f}}}

\newcommand{\vs}{\varsigma}
\newcommand{\tvs}{\widetilde\varsigma}
\newcommand{\tF}{\widetilde\scrF}

\newcommand{\TF}{\widetilde{F}}

\newcommand{\tfU}{\widetilde\fU}
\newcommand{\tfw}{\widetilde\fw}
\newcommand{\hK}{\widehat K}
\newcommand{\hs}{\widehat\vecs}

\newcommand{\oK}{\overline K}
\newcommand{\tB}{\widetilde B}

\newcommand{\tQ}{\widetilde\scrQ}
\newcommand{\oQ}{\overline\scrQ}

\newcommand{\tLambda}{\widetilde\Lambda}
\newcommand{\oXi}{\overline\Xi}
\newcommand{\tP}{\widetilde\scrP}
\newcommand{\tp}{\widetilde{p}}
\newcommand{\tPP}{\widetilde{\psi}}
\newcommand{\hXi}{\widehat\Xi}
\newcommand{\hR}{\widehat{R}}
\newcommand{\hlambda}{\widehat{\lambda}}

\newcommand{\tJ}{\widetilde J}

\newcommand{\tPhi}{\widetilde{\Phi}}

\newcommand{\hbeta}{\widehat{\vecbeta}}
\newcommand{\oxi}{\overline\xi}
\newcommand{\bs}{\backslash}

\newcommand{\hJ}{\widehat{J}}

\newcommand{\tOmega}{\widetilde\Omega}

\newcommand{\hZ}{\widehat{\Z^n}}

\newcommand{\tH}{\widetilde H}
\newcommand{\tlambda}{\widetilde\lambda}
\newcommand{\tmu}{\widetilde\mu}
\newcommand{\omu}{\overline\mu}
\newcommand{\hmu}{\widehat\mu}

\newcommand{\tX}{\widetilde X}
\newcommand{\hnu}{\hat\nu}

\newcommand{\tw}{\widetilde\vecw}

\newcommand{\tW}{\widetilde\scrW}
\newcommand{\oW}{\overline\scrW}

\newcommand{\Q}{\mathbb{Q}}
\newcommand{\R}{\mathbb{R}}
\newcommand{\Z}{\mathbb{Z}}

\renewcommand{\mod}{\text{ mod }}

\newcommand{\col}{\: : \:}

\newcommand{\bn}{\mathbf{0}}
\renewcommand{\AA}{\mathbf{A}}

\newcommand{\ttau}{{\widetilde{\tau}}}
\newcommand{\htau}{{\widehat{\tau}}}
\newcommand{\trho}{{\widetilde{\rho}}}

\newcommand{\hTheta}{{\widehat{\Theta}}}

\newcommand{\hrho}{{\widehat{\rho}}}
\newcommand{\tbe}{\widetilde{\vecbeta}}

\newcommand{\tv}{{\widetilde{\vecv}}}
\newcommand{\hq}{{\widehat{\vecq}}}

\newcommand{\tf}{\widetilde{f}}

\newcommand{\tq}{{\widetilde{\vecq}}}

\newcommand{\tD}{\widetilde{D}}

\newcommand{\ve}{\varepsilon}

\newcommand{\matr}[4]{\left( \begin{matrix} #1 & #2 \\ #3 & #4 \end{matrix} \right) }

\newcommand{\smatr}[4]{\left( \begin{smallmatrix} #1 & #2 \\ #3 & #4 \end{smallmatrix} \right) }

\makeindex

\begin{document}

\frontmatter

\title{Kinetic Theory for the Low-Density Lorentz Gas}

\author{Jens Marklof}
\address{School of Mathematics, University of Bristol, Bristol BS8 1TW, United Kingdom}
\email{j.marklof@bristol.ac.uk}

\author{Andreas Str\"ombergsson}
\address{Department of Mathematics, Box 480, Uppsala University,
SE-75106 Uppsala, Sweden}
\email{astrombe@math.uu.se}

\thanks{The research leading to these results has received funding from the European Research Council under the European Union's Seventh Framework Programme (FP/2007-2013) / ERC Grant Agreement n. 291147. 
A.S.\ is supported by the Swedish Research Council Grant 2016-03360.}

\date{14 June 2021}



\begin{abstract}
The Lorentz gas is one of the simplest and most widely-studied models for particle transport in matter. It describes a cloud of non-interacting gas particles in an infinitely extended array of identical spherical scatterers. The model was introduced by Lorentz in 1905 who, following the pioneering ideas of Maxwell and Boltzmann, postulated that
in the limit of low scatterer density,
the macroscopic transport properties of the model should be governed by a linear Boltzmann equation. The linear Boltzmann equation has since proved a useful tool in the description of various phenomena, including semiconductor physics and radiative transfer. A rigorous derivation of the linear Boltzmann equation from the underlying particle dynamics was given, for random scatterer configurations, in three seminal papers by Gallavotti, Spohn and Boldrighini-Bunimovich-Sinai. The objective of the present study is to develop an approach for a large class of deterministic scatterer configurations, including various types of quasicrystals. We prove the convergence of the particle dynamics to transport processes that are in general (depending on the scatterer configuration) not described by the linear Boltzmann equation. This was previously understood only in the case of the periodic Lorentz gas through work of Caglioti-Golse and Marklof-Str\"ombergsson. Our results extend beyond the classical Lorentz gas with hard sphere scatterers, and in particular hold for general classes of spherically symmetric finite-range potentials. We employ a rescaling technique that randomises the point configuration given by the scatterers' centers. The limiting transport process is then expressed in terms of a point process that arises as the limit of the randomised point configuration under a certain volume-preserving one-parameter linear group action. 
\end{abstract}

\maketitle

\tableofcontents

\mainmatter

\chapter{Introduction}
\label{INTROsec}

The Lorentz gas describes the dynamics of a cloud of point-particles in an array of spherical scatterers in $\RR^d$ ($d\geq2$), each of radius $\rho>0$, which are centered at a given infinite point set $\scrP$. Each particle moves with constant speed along straight lines in between scattering events. The point particles do not interact with each other, and their interaction with the scatterers is defined by specular reflection (as in Lorentz' orginal setting) or by scattering potentials. The main challenge posed by Lorentz' 1905 paper \cite{Lorentz} is, whether or not in the limit $\rho\to 0$ the dynamics of a macroscopic particle cloud is approximated by a solution of the linear Boltzmann equation.

Since the gas particles are assumed to be non-interacting, the problem is equivalent to the study of the one-particle dynamics. In this framework the initial particle density in phase space is interpreted as the probability density of the {\em random} initial condition of the single particle. Although the dynamics is governed by Hamilton's equations and therefore deterministic, the random initial condition means that the particle trajectory is now expressed as a random flight process, the {\em Lorentz process}. The question is, under which assumptions on the scatterer configuration the Lorentz process converges, as $\rho\to 0$, to a limiting process. This of course requires a suitable rescaling of time and space units in terms of the mean collision time and free path length, respectively. The Kolmogorov forward equation (Fokker-Planck-Kolmogorov equation) of the limiting process describes the macroscopic transport of the initial particle cloud, and the key question is whether this equation coincides with the linear Boltzmann equation, as postulated by Lorentz. 

There are two non-trivial instances where the problem is fully understood. The first is the case when $\scrP$ is a fixed realisation of a Poisson point process. Here Boldrighini, Bunimovich and Sinai \cite{Boldrighini83} proved that the Lorentz process converges to a limit that is consistent with the linear Boltzmann equation. Their work is preceded by two important papers by Gallavotti \cite{Gallavotti69}, who established convergence on average over the point configuration, and Spohn \cite{Spohn78}, who considered more general random point configurations (still on average) and scattering potentials. Although the paper \cite{Boldrighini83} is restricted to dimension $d=2$ and hard sphere scatterers, we will show here that its results generalise to general dimensions and ``soft'' potentials. 

The second instance is when the scatterer configuration $\scrP$ is given by a Euclidean lattice $\scrL$ of full rank, for example $\scrP=\ZZ^d$. Here Marklof and Str\"ombergsson \cite{partI,partII,partIII,partIV} proved convergence of the Lorentz process to a random flight process which, perhaps surprisingly, is independent of the choice of $\scrL$. The limit process is Markovian only on an extended phase space which, in addition to position and momentum, also includes the impact parameter and distance to the next collision. The corresponding transport equation is in particular not consistent with the linear Boltzmann equation. This new transport equation was obtained independently in dimension $d=2$ for $\scrP=\ZZ^2$ by Caglioti and Golse \cite{Caglioti08,Caglioti10}, subject to a heuristic assumption that was proved (in any dimension) in \cite{partII}. The fact that the linear Boltzmann equation must fail in the periodic setting already follows from the heavy, power-law tail of the distribution of free path lengths \cite{Dahlqvist97,Bourgain98,Golse00,Caglioti03,Boca07,partIV}, as pointed out by Golse \cite{Golse06,Golse08}. In the periodic setting, the limit transport process in fact satisfies a superdiffusive central limit theorem \cite{MarklofToth}, with a mean-square displacement proportional to $t\log t$ (where $t$ is time measured in units of the mean collision time), rather than the standard linear scaling which appears in the case of random scatterer configurations. 

In the present paper we develop a general framework which, under suitable hypotheses on the scatterer configuration $\scrP$ (see Section \ref{sec:outline}), allows the proof of convergence to a limiting transport process. The latter will in general depend on the choice of $\scrP$. Admissible choices of $\scrP$ include the Poisson and lattice setting discussed above, as well as new examples including more general periodic point sets and certain classes of quasicrystals. Our theory applies not only to the classical case of hard sphere scatterers (Section \ref{sec:classical}), but also radial potentials with compact support (Section \ref{sec:soft}). The assumption of compact support is crucial for our work, as is the assumption that there are no external force fields, which ensures that in-between collisions the particles move along straight lines. We refer the reader to Section \ref{sec:open} for a survey of open questions that naturally follow on from this study.

\section{Outline of assumptions on the scatterer configuration}\label{sec:outline}

The scatterers are centered at the points of a locally finite subset $\scrP$ of $\R^d$ \label{scrPdef}. We assume that $\scrP$ has constant asymptotic density $c_\scrP>0$. This means that for any bounded subset $B \subset\R^d$ with boundary of Lebesgue measure zero,
\begin{align}\label{ASYMPTDENSITY0}
\lim_{T\to\infty}\frac{\#(\scrP\cap TB)}{T^d}=c_\scrP\vol(B).
\end{align}
Rather than following the particle trajectory in a coordinate system in which the environment (i.e., the scatterer configuration) remains static, we will use the particle's coordinate frame in which the particle is at rest at the origin, with direction of travel along the first coordinate axis, and the environment is changing. As we will see, lengths in the direction of travel are naturally measured in units of $\rho^{1-d}$ (which is proportional to the mean free path length); the natural length scale perpendicular to the direction of travel is $\rho$, the radius of a scatterer. 
Let $\US$ \label{unitsphere} be the unit sphere in $\R^d$ centered at the origin. 
If $\vecq\in\RR^d$ and $\vecv\in\US$ are the particle position and direction of travel in physical space,
then in the particle frame the scattering configuration appears as the translated, rotated and rescaled point set\footnote{We identify points in $\RR^d$ with row vectors, and linear transformations are represented by matrix multiplication from the right.}
\begin{align}\label{XIRHOqv000}
(\scrP-\vecq)\,R(\vecv)\,D_\rho
\end{align}
with the diagonal matrix \label{Drhodef}
\begin{align*}
D_\rho= %
\diag(\rho^{d-1},\rho^{-1},\cdots,\rho^{-1})\in\SL(d,\R) 
\end{align*}
providing the required rescaling of length units, and a map $R:\US\to\SO(d)$ which rotates the direction of travel $\vecv$ to the unit vector $\vece_1=(1,0,\ldots,0)$ \label{unitvecdef}; that is $\vecv R(\vecv)=\vece_1$ for all $\vecv\in\US$. We furthermore assume that $R$ is continuous when restricted to $\US$ minus one point; the choice of $R$ is otherwise arbitrary but will remain the same in all subsequent statements. \label{Rdef} 
We write $\omega:=\vol_{\S_1^{d-1}}$
\label{omegadef}
for the Lebesgue measure on $\US$,
and $\omega_1:=\omega(\US)^{-1}\, \omega$ for
the uniform probability measure on $\US$.
We also let $\Pac(\US)$ \label{acBpm} be the space of Borel probability measures on $\US$ that are absolutely continuous
with respect to $\omega$.

Our proof will require the understanding of the point set \eqref{XIRHOqv000} for $\vecq\in\scrP$ (this corresponds to the case when the particle has just hit a scatterer at position $\vecq$), and $\vecv$ randomly distributed according to a Borel probability measure $\lambda$ on $\US$ which is absolutely continuous with respect to the uniform measure on $\US$. The point set \eqref{XIRHOqv000} is therefore a {\em random} point set, and a natural assumption is that it converges in distribution for $\rho\to 0$ to a limiting random point set.\footnote{We will provide a precise framework for the notion of random sets and their convergence via the theory of point processes in Section \ref{ASSUMPTsec}.} The limit will in general 
depend on $\vecq$, and we need to require some regularity in its dependence on $\vecq$, as well as some uniformity in the convergence. 

To this end we equip $\scrP$ with a {\em marking} as follows. Let ${\vs}$ be a map from $\scrP$ to a compact metric space $\Sigma$,
\label{Sigmadef}
and set
\begin{align*}
\scrX=\R^d\times\Sigma , \qquad \tP=\{(\vecp,{\vs}(\vecp))\col\vecp\in\scrP\}\subset\scrX.
\end{align*}
\label{scrXdef2}\label{tPdef}We refer to 
$\vs$ as a {\em marking}, and $\Sigma$, $\tP$ as the corresponding {\em space of marks} and {\em marked point set}, respectively; $\scrP$ will be called the {\em underlying point set}.

{\blu 
In order to provide a first intuition for the concept of marking, 
let us already here mention some key examples:
If $\scrP$ is a realisation of a Poisson point process with constant intensity,
then the marking can be taken to be trivial, i.e.\ $\Sigma$ can be taken to be a singleton set
(cf.\ Section \ref{PoissonSEC}).
The same is true if $\scrP$ is a Euclidean lattice;
however, if $\scrP$ is a more general \textit{periodic} point set in $\R^d$,
i.e.\ a finite union of translates of a fixed lattice $\scrL$,
then it is typically necessary to use a nontrivial marking,
and the natural choice is to let the space of marks $\Sigma$ 
contain one element for each translate of $\scrL$,
with the marking of each point $\vecq\in\scrP$ indicating which translate of $\scrL$ it belongs to
(cf.\ Section \ref{periodicpointsetsSEC}).
Finally, in the more general case when $\scrP$ is a \textit{quasicrystal} of cut-and-project type,
the natural choice is to let $\Sigma$ equal the closure of the \textit{window set}
used in the cut-and-project construction;
see Section \ref{QCexsec} below for details.}


 For $\vecx\in\RR^d$, $T\in\RR$ and $A\in\GL(d,\RR)$ we extend the natural action on $\RR^d$ to $\scrX$ by setting $(\vecw,\vs)+\vecx:=(\vecw+\vecx,\vs)$, $T(\vecw,\vs):=(T\vecw,\vs)$ and $(\vecw,\vs)A:=(\vecw A,\vs)$. 
Thus in particular
\begin{equation*}
\tP A +\vecx =\{(\vecp A+\vecx ,{\vs}(\vecp))\col\vecp\in\scrP\} .
\end{equation*}
For $\vecq\in\scrP$, we furthermore define $(\tP-\vecq)^*= \{(\vecp,{\vs}(\vecp))\col\vecp\in(\scrP-\vecq)\setminus\{\vecnull\} \}$.

The main assumption on the scatterer configuration $\scrP$ in the present work is that there is a marking $\vs$ and a Borel
probability measure $\mm$ on $\Sigma$ such that:
\begin{enumerate}[{\bf [P1]}]
\item {\em Uniform density:} The marks of the points in $\tP$ are asymptotically equidistributed in $(\Sigma,\mm)$.  That is,
for any bounded $B \subset\scrX$ with $\mu_\scrX(\partial B)=0$, we have
\begin{align}\label{ASYMPTDENSITY100}
\lim_{T\to\infty}\frac{\#(\tP\cap TB)}{T^d}=c_\scrP\mu_\scrX(B)
\end{align}
where $TB=\{(T\vecw,\vs)\col(\vecw,\vs)\in B\}$
and \label{mmdef} $\mu_\scrX=\vol\times\mm$.
Relation \eqref{ASYMPTDENSITY100} thus generalizes \eqref{ASYMPTDENSITY0}.
\item {\em Spherical equidistribution:} There exist
{\blu a subset $\scrE\subset\scrP$ of asymptotic density zero,
and a continuous family $\{ \Xi_\vs \col \vs\in\Sigma\}$ of random marked point sets,
such that for any $\vecq\in\scrP\setminus\scrE$} and $\lambda\in\Pac(\US)$,\label{Edef}
the sequence
$((\tP-\vecq)^*\,R(\vecv)\,D_\rho)_{\rho>0}$, with $\vecv$ random according to $\lambda$, converges in distribution to $\Xi_{\vs(\vecq)}$ as $\rho\to 0$; the convergence is uniform for $\vecq$ in a ball of radius $T\rho^{1-d}$ for any fixed $T\geq 1$.
\item {\em No escape of mass:} For every bounded Borel set $B\subset\RR^d$,
\begin{align*}%
\lim_{\xi\to\infty}\limsup_{\rho\to0}\hspace{7pt}
[\vol\times\omega]\bigl(\bigl\{(\vecq,\vecv)\in B\times\US\col
(\scrP-\rho^{1-d}\vecq)\,R(\vecv)\,D_\rho \cap \fZ_\xi=\emptyset\bigr\}\bigr)=0,
\end{align*}
with the open cylinder\label{fZxidef} 
$\fZ_\xi=\{\vecx=(x_1,\ldots,x_d)\col 0<x_1<\xi,\: x_2^2+\cdots+x_d^2<1\}$. 
\end{enumerate}

{\blu Among these three assumptions [P1-3], the \textit{key} assumption is [P2].
It postulates, in a more precise form,
the convergence discussed above; namely, 
the convergence in distribution for $\rho\to0$ of the random point set in \eqref{XIRHOqv000}.
This condition [P2] is also the one which, by far, is the most difficult to verify,
at least for every example of $\scrP$ satisfying [P1-3] which we are aware of.

Note that in [P2] we allow a subset $\scrE\subset\scrP$ of ``exceptional'' 
points for which the convergence may fail.
We remark here that if $\scrP$ is a periodic point set in $\R^d$,
then [P2] in fact holds with $\scrE=\emptyset$,
and the same is true if $\scrP$ is a quasicrystal of cut-and-project type
(cf.\ Sections \ref{periodicpointsetsSEC}--\ref{QCexsec}).
However, if $\scrP$ is a realisation of a Poisson point process with constant intensity,
then [P2] does \textit{not} hold with $\scrE=\emptyset$,
but we will prove in Section \ref{PoissonSEC} that
[P2] holds when $\scrE$ is chosen as the set of points $\vecq$ in $\scrP$
which, in an appropriate sense, lie exceptionally near some other point in $\scrP$
(cf.\ Proposition \ref{POISSONprop} and Remark \ref{POISSONENOTEMPTYrem}).

The condition [P3] will be used in one single, important,  
step of our development,
namely in the derivation 
of a \textit{macroscopic} analogue of the spherical equidistribution [P2],
i.e.\ a version of [P2] 
for initial conditions of the form
$(\vecq,\vecv)=(\rho^{1-d}\vecq',\vecv)$ 
with $(\vecq',\vecv)$ random according to a probability measure on $\T^1(\R^d)$
which is absolutely continuous with respect to Liouville measure, $\vol\times\omega$
(cf.\ Section \ref{GENLIMITsec}).}


\vspace{5pt}

We will furthermore require the following invariance and regularity assumptions on $\Xi_\vs$. Denote by $\Xi_\vs'$ the underlying random point set in $\RR^d$. $\scrB^d(\vecx,R)$ \label{scrBdef} denotes the open ball of radius $R$ centered at $\vecx$, and we use the notation $\SO(d-1)$ \label{SOd-1} for the subgroup of $K\in\SO(d)$ such that $\vece_1 K=\vece_1$.

\begin{enumerate}[{\bf [Q1]}]
\item {\em $\SO(d-1)$-invariance:} $\Xi_\vs$ and $\Xi_\vs K$ have the same distribution for every $K\in \SO(d-1)$.
\item {\em Coincidence-free first coordinates:} For every $\vs\in\Sigma$, the probability that [$\Xi_\vs'$ has two (or more) points with the same first coordinate] is zero.
\item {\em Small probability of large voids:}  For every $\ve>0$ there exists $R>0$ such that the probability that [$\Xi_\vs'$ has no point in $\scrB^d(\vecx,R)$] is less than $\ve$, uniformly for all $\vecx\in\RR^d$ and $\vs\in\Sigma$.
\end{enumerate}

The above hypotheses will be restated in a more concise measure-theoretic form in Section~\ref{ASSUMPTLISTsec}. In the following, a locally finite set $\scrP\subset\RR^d$ is called {\em admissible} if there exists a marking such that [P1-3] and [Q1-3] hold. Examples of admissible sets include realisations of Poisson point processes, locally finite periodic point sets, and Euclidean model sets. These examples are discussed in detail in Section \ref{EXAMPLESsec}.

\section{The Lorentz process for hard sphere scatterers} \label{sec:classical}

The setting of the classical Lorentz gas is that of a point particle moving with constant velocity $\vecv\in\RR^d\setminus\{\vecnull\}$ in an array of hard-sphere scatterers of radius $\rho$, where the scattering is given by specular reflection: The incoming and outgoing particle trajectories are contained in the same two-dimensional plane through the center of the scatterer, and the angle of incidence equals the angle of reflection. The scattering map is independent of the particle speed $\|\vecv\|$, and a simple scaling argument shows that we therefore may assume without loss of generality that $\|\vecv\|=1$. 
With this convention, the scattering process is defined by the map
\begin{equation}\label{scatmap}
\Psi:\scrS_-\to\scrS_+, \qquad (\vecv,\vecb)\mapsto (\vecv_+,\vecb)=(\vecv-2(\vecv\cdot\vecb)\vecb,\vecb),
\end{equation}
with the set of \textit{incoming data}\label{scrSmdef}
\begin{align*}
\scrS_-:=\{(\vecv,\vecb)\in\US\times\US\col\vecv\cdot\vecb<0\}
\end{align*}
describing the velocity and position (measured in units of the radius $\rho$) with which the particle enters the interaction region, 
and the corresponding set of \textit{outgoing data} 
\begin{align*}
\scrS_+:=\{(\vecv,\vecb)\in\US\times\US\col\vecv\cdot\vecb>0\}.
\end{align*}
Note that 
\begin{equation*}%
\vecb = \frac{\vecv_+-\vecv}{\|\vecv_+-\vecv\|}.
\end{equation*}

Let $\UB$ be the open unit ball in $\R^{d-1}$ centered at the origin.
The {\em impact parameter} $\vecw\in\UB$ is defined as the projection of $\vecb$ onto the hyperplane perpendicular to the incoming velocity $\vecv$. The {\em differential cross section} $\sigma(\vecv,\vecv_+)$ is defined as the Jacobian of the inverse of the map 
\begin{equation*}
\UB \to \US, \qquad \vecw \mapsto \vecv_+  \qquad \text{($\vecv$ fixed),}
\end{equation*}
so that 
$d\vecw=d\!\vol(\vecw)=\sigma(\vecv,\vecv_+)\,d\vecv_+$, where $d\vecv_+:=d\omega(\vecv_+)$.
The {\em total} scattering cross section is thus given by the volume $v_{d-1}=\vol(\scrB_1^{d-1})$ \label{DscatCS} of the unit ball in $\RR^{d-1}$. In the present setting we have the explicit formula
\begin{equation}\label{SIGMAforrefl}
\sigma(\vecv,\vecv_+) = \frac14\, \| \vecv-\vecv_+ \|^{3-d},
\end{equation}
see \eqref{dcsformula} below. %

The configuration space for the dynamics is given by
\begin{align}\label{Krhodef}
\scrK_\rho=\R^d\setminus\bigcup_{\vecp\in\scrP}\scrB^d(\vecp,\rho),
\end{align}
where $\scrB^d(\vecp,\rho)$ denotes the open ball with center $\vecp$ and radius $\rho$. 
For definiteness, at the time of any collision, we will consider the particle to be in 
\textit{outgoing} position, i.e.\ belong to the set \label{T1Krhooutdef}
\begin{align*}
\T^1(\partial\scrK_\rho)_{\out}:= \bigl\{(\vecq,\vecv)\in\partial\scrK_\rho\times\US\col
[\forall\vecp\in\scrP\col
\|\vecq-\vecp\|=\rho
\Rightarrow
(\vecq-\vecp)\cdot\vecv\geq0]\bigr\}.
\end{align*}
If two or more balls overlap then it is often unclear how to continue the particle path 
if it hits an intersection point.
In this situation we agree that the particle gets trapped and stays motionless for all future time.
Let $\Phi_t=\Phi_t^{(\rho)}$ \label{Phi101} be the billiard flow on $\T^1(\scrK_\rho):=\scrK_\rho\times\US$ thus defined.
\label{T1Krhodef}
It is standard to verify that the set 
of initial conditions in $\T^1(\scrK_\rho)$
for which the particle 
at some time point either in the past or in the future
collides with an intersection point of two or more scatterer boundaries,
has measure zero with respect to Liouville measure, $\vol\times\omega$.
Furthermore, for a dispersing billiard such as the Lorentz gas considered here, 
the number of collisions in any finite time interval is finite 
\cite[Thm.\ 1.1]{BFK}\footnote{This theorem applies to our situation since $\scrP$ is locally finite.}.

Set
\begin{align}\label{wini}
\fw(\rho)=\T^1(\scrK_\rho^\circ) \cup \T^1(\partial\scrK_\rho)_{\out};
\end{align}
this is the set of points that are not trapped.
Indeed, by our conventions,
if $(\vecq,\vecv)\in\fw(\rho)$ then
$\Phi_t(\vecq,\vecv)=(\vecq+t\vecv,\vecv)$ for all sufficiently small $t>0$,
whereas if $(\vecq,\vecv)\in\T^1(\scrK_\rho)\setminus\fw(\rho)$
then $\Phi_t(\vecq,\vecv)=(\vecq,\vecv)$ for all $t>0$.

For $(\vecq_0,\vecv_0)\in\fw(\rho)$,
let $\tau_1(\vecq_0,\vecv_0;\rho)$ be the first time at which the particle starting at $(\vecq_0,\vecv_0)$
hits a scatterer, i.e.
\begin{align}
\tau_1(\vecq_0,\vecv_0;\rho)=\inf\{t>0\col\vecq_0+t\vecv_0\notin\scrK_\rho\}.
\end{align}
Note that $\tau_1$ may equal $\infty$,
and we also define
$\tau_1(\vecq_0,\vecv_0;\rho)=\infty$ for all trapped points,
i.e.\ for all
$(\vecq_0,\vecv_0)\in\T^1(\scrK_\rho)\setminus\fw(\rho)$.
For $j\geq1$ we denote by 
\begin{equation*}
(\vecq_j,\vecv_j)=(\vecq_j(\vecq_0,\vecv_0;\rho),\vecv_j(\vecq_0,\vecv_0;\rho)) \in  \T^1(\partial\scrK_\rho)_{\out}
\end{equation*}
the position and outgoing velocity of the particle at the $j$th collision,
and let $\tau_{j+1}=\tau_{j+1}(\vecq_0,\vecv_0;\rho)$ be the time 
\label{taujdef}
which it travels between the $j$th and the $(j+1)$st collision.
Thus
\begin{align*}%
(\vecq_j,\vecv_j)=\Phi_{\tau_1(\vecq_{j-1},\vecv_{j-1};\rho)}(\vecq_{j-1},\vecv_{j-1})
\quad\text{and}\quad
\tau_{j+1}(\vecq_0,\vecv_0;\rho)=\tau_1(\vecq_j,\vecv_j;\rho).
\end{align*}
If $\tau_j(\vecq_0,\vecv_0;\rho)=\infty$ for some $j$
(meaning either that 
$(\vecq_{j-1},\vecv_{j-1})\in\T^1(\scrK_\rho)\setminus\fw(\rho)$ or
else $\vecq_{j-1}+t\vecv_{j-1}\in\scrK_\rho$ for all $t>0$)
then for definiteness we set
$\tau_{j+1}=\tau_{j+2}=\cdots=\infty$,
and also %
$\vecq_i=\vecq_{j-1}$ and $\vecv_i=\vecv_{j-1}$ for all $i\geq j$.

Let us denote by $n_t=n_t(\vecq_0,\vecv_0;\rho)$ the number of collisions within time $t$, i.e.,
\begin{equation}\label{eq:nT}
n_t = \max \big\{ n \in\ZZ_{\geq 0} : T_n \leq t \big\}  , \qquad T_n:=\sum_{j=1}^n \tau_j .
\end{equation}
Note that for all $t\geq0$ such that $\Phi_t(\vecq_0,\vecv_0)\in\fw(\rho)$, we have
\begin{equation*}
\Phi_t(\vecq_0,\vecv_0) = 
(\vecq_{n_t} + (t-T_{n_t}) \vecv_{n_t} ,\vecv_{n_t}).
\end{equation*}
It will be convenient to extend the billiard flow trivially to all of $\T^1(\RR^d)$ by 
also taking all points in $\T^1(\R^d)\setminus\T^1(\scrK_\rho)$ to be trapped,
i.e.\ by defining 
\begin{equation*}
\Phi_t(\vecq_0,\vecv_0) = 
(\vecq_0  ,\vecv_0)  \quad \text{if $(\vecq_0,\vecv_0)\in \T^1(\R^d)\setminus\fw(\rho)$.}
\end{equation*}
This is purely for notational reasons, since the relative measure of points not in $\fw(\rho)$ tends to zero as $\rho\to 0$.

Let us now describe the Boltzmann-Grad limit of the particle dynamics, where the point set $\scrP$ describing the scatterer configuration is fixed, and the radius $\rho$ of each scatterer tends to zero. As we will explain in more detail in Section \ref{TRANSKERRELsec}, the mean free path length, i.e.\ the mean time between consecutive collisions, is asymptotically given by $\oxi \rho^{1-d}$ with 
\begin{align}\label{OXIFORMULAG}
\oxi=\frac1{v_{d-1}c_\scrP}.
\end{align}
This implies that, in order to see non-trivial dynamical phenomena emerge as $\rho\to 0$, a rescaling of length and time units is necessary. The Boltzmann-Grad scaling, which we consider in this paper, considers length and time in units of the mean free path length. That is, we consider the rescaled flow 
\begin{equation}\label{tPhidef}
\widetilde \Phi_t^{(\rho)}=s_\rho\circ\Phi_{\rho^{1-d} t}^{(\rho)}\circ s_\rho^{-1}
\end{equation}
where
\begin{align}\label{Rrhodef}
s_\rho:\T^1(\R^d)\to\T^1(\R^d),\qquad
s_\rho(\vecq,\vecv)=(\rho^{d-1}\vecq,\vecv).
\end{align}

In this scaling, one expects a random flight process in the limit $\rho\to 0$. This is confirmed by our central result which is stated below as Theorem \ref{thm:M1}. On larger time and spacial scales one would expect diffusive or indeed superdiffusive limits.
Currently this is only understood in the case of random and lattice scatterer configurations,
when taking first the Boltzmann-Grad limit $\rho\to 0$,
and then the long time limit \cite{MarklofToth}.

For the purposes of the present study, we define a {\em random flight process} as a stochastic process of the form 
\begin{equation}\label{eq:RFP}
\Theta :  t \mapsto \Theta(t) = \bigg( \vecq_0 + \sum_{n=1}^{n_t} \xi_j \vecv_{j-1} + (t-T_{n_t}) \vecv_{n_t} ,\vecv_{n_t} \bigg) 
\end{equation}
with random $(\vecq_0,\vecv_0)\in\T^1(\RR^d)$, $\langle\xi_j\rangle_{j=1}^\infty\in (\RR_{\geq 0}\cup\{\infty\})^\NN$ and $\langle\vecv_j\rangle_{j=1}^\infty\in (\US)^\NN$. The quantities $n_t$, $T_n$ are defined as in \eqref{eq:nT} with $\tau_j$ replaced by $\xi_j$. We do not assume any independence in the above. With this, we may view
\begin{equation}\label{ThetarhoDEF}
\Theta^{(\rho)}: t \mapsto \Theta^{(\rho)}(t)= \widetilde \Phi_t^{(\rho)} (\vecq_0,\vecv_0) 
\end{equation}
as a random flight process, for random $(\vecq_0,\vecv_0)$ distributed according to a fixed Borel probability measure on $\T^1(\RR^d)$, and the random processes $\langle\xi_j^{(\rho)}\rangle_{j=1}^\infty$ and $\langle\vecv_j^{(\rho)}\rangle_{j=1}^\infty$ are defined through the deterministic functions $\xi_j^{(\rho)}=\rho^{d-1} \tau_j(\rho^{1-d}\vecq_0,\vecv_0;\rho)$ and $\vecv_j=\vecv_j(\rho^{1-d}\vecq_0,\vecv_0;\rho)$ of the random initial data $(\vecq_0,\vecv_0)$. This means they are highly correlated but nevertheless well-defined point processes.

We denote by $\Pac(\T^1(\RR^d))$ \label{acBpm2} the space of Borel probability measures on $\T^1(\RR^d)$ that are absolutely continuous with respect to Liouville measure, $\vol\times\omega$.

\begin{thm}\label{thm:M1}
Let $\scrP$ be admissible. Then, for any $\Lambda\in\Pac(\T^1(\RR^d))$, there is a random flight process $\Theta$ 
with $\PP(\xi_j=\infty)=0$ for all $j$,
such that $\Theta^{(\rho)}$ converges to $\Theta$ in distribution, as $\rho\to 0$.
\end{thm}

The key step in the proof of Theorem \ref{thm:M1} is to establish the joint limit distribution for 
\begin{equation}\label{eq:JTdist} 
\big\langle\tau_j(\rho^{1-d}\vecq_0,\vecv_0;\rho), \vecv_j(\rho^{1-d}\vecq_0,\vecv_0;\rho)\big\rangle_{j=1}^\infty,
\end{equation}
with $(\vecq_0,\vecv_0)$ random according to $\Lambda$. One of the central outcomes of our study is that we obtain the Markov property for the limit distribution, if we consider the joint distribution of \eqref{eq:JTdist} {\em and} the sequence of markings $\langle {\vs}_j(\rho^{1-d}\vecq_0,\vecv_0;\rho) \rangle_{j=1}^\infty$. Here ${\vs}_j(\rho^{1-d}\vecq_0,\vecv_0;\rho)=\vs(\vecp_j)\in\Sigma$ is the marking of the centre $\vecp_j\in\scrP$ of the scatterer involved in the $j$th collision; if this is not well-defined because scatterers overlap, choose any marking.

\begin{thm}\label{thm:M2}
Let $\scrP$ be admissible, and suppose $\Lambda\in\Pac(\T^1(\RR^d))$. If $(\vecq_0,\vecv_0)\in\T^1(\RR^d)$ is distributed according to $\Lambda$, then the random process
\begin{align*}
\NN & \to (\RR_{> 0}\cup\{+\infty\})\times\Sigma\times\US \notag \\ 
j & \mapsto \big( \rho^{d-1} \tau_j(\rho^{1-d}\vecq_0,\vecv_0;\rho), {\vs}_j(\rho^{1-d}\vecq_0,\vecv_0;\rho),\vecv_j(\rho^{1-d}\vecq_0,\vecv_0;\rho) \big) 
\end{align*}
converges in distribution to the second-order Markov process
\begin{equation}\label{eq:JTdist2} 
j\mapsto \big( \xi_j, {\vs}_j , \vecv_j \big) ,
\end{equation}
where for any Borel set $A\subset\R_{\geq0}\times\Sigma\times\US$,
\begin{equation*}
\PP\Big( (\xi_1, {\vs}_1 , \vecv_1) \in A \,\Big|\, (\vecq_0,\vecv_0) \Big) = \int_{A} p(\vecv_0;\xi,\vs,\vecv)\,  d\xi\,d\mm({\vs})\,d\vecv,
\end{equation*}
and for $j\geq 2$,
\begin{multline*}
\PP\Big( (\xi_j, {\vs}_j , \vecv_j) \in A \,\Big|\, (\vecq_0,\vecv_0),\; \big\langle (\xi_i,{\vs}_i,\vecv_i)\big\rangle_{i=1}^{j-1} \Big) \\
= \int_{A} p_\bn(\vecv_{j-2},\vs_{j-1},\vecv_{j-1};\xi,\vs,\vecv)\,  d\xi\,d\mm({\vs})\,d\vecv.
\end{multline*}
The functions $p,p_\bn$ are defined in Section \ref{COLLKERsec};
they depend on $\scrP$ but are independent of $\Lambda$,
and for any fixed $\vecv_0,\vs,\vecv$ both
$p(\vecv_0\:;\:\cdot\:)$ and $p_\bn(\vecv_0,\vs,\vecv\:;\:\cdot\:)$ are probability densities on $\R_{\geq0}\times\Sigma\times\US$.
In particular $\PP(\xi_j=\infty)=0$ for all $j$.
\end{thm}

This theorem is restated in non-probabilistic notation
and for general scattering maps
as Theorem \ref{MAINTECHNthm2G} below.
As we will see, the extension of the state space to include marking is in general\footnote{In the case of random (resp.\ periodic) scatterer configurations considered in previous studies, the process $j\mapsto (\xi_j,\vecv_j)$ is first-order (resp.\ second-order) Markovian and an extension of the state space as in \eqref{eq:JTdist2}  is not necessary.} necessary to obtain the Markov property. 

The collision kernels $p,p_\bn$ can be written as
\begin{align}\label{pgendef001}
p\bigl(\vecv;\xi,{\vs}_+,\vecvp\bigr)=\frac{\sigma(\vecv,\vecvp)}{v_{d-1}}
\,k^{\g}\bigl(\xi,(\vecw,{\vs}_+)\bigr),
\end{align}
\begin{align}\label{pbndefG001}
p_\bn\bigl(\vecv_0,{\vs},\vecv;\xi,{\vs}_+,\vecvp\bigr)=
\frac{\sigma(\vecv,\vecvp)}{v_{d-1}}
k\bigl((\vecw',{\vs}),
\xi,(\vecw,{\vs}_+)\bigr),
\end{align}
where $k^{\g}$ and $k$ are {\em transition} kernels that quantify the probability of hitting the next scatterer at distance $\xi$ with impact parameter $\vecw$ (which is a function of $\vecv$ and $\vecv_+$). The kernel $k^{\g}$ corresponds to the case of generic initial data, and $k((\vecw',\vs),\,\cdot\,)$ to the case of an initial condition relative to previous scattering event with marking $\vs$ and exit parameter $\vecw'$. The exit parameter can be viewed as the time-reversed impact parameter and thus is a function of $\vecv_0$ and $\vecv$. The transition kernels are central to our work and will be discussed in detail in Section \ref{FIRSTCOLLISIONsec}. 
The collision kernels have the following important properties: 
\begin{equation}\label{s-inv1}
p\bigl(\vecv K ;\xi,{\vs}_+,\vecvp K\bigr)=p\bigl(\vecv;\xi,{\vs}_+,\vecvp\bigr) \quad \forall K\in\SO(d),
\end{equation}
\begin{equation}\label{s-inv2}
p_\bn\bigl(\vecv_0 K,{\vs},\vecv K;\xi,{\vs}_+,\vecvp K\bigr) = p_\bn\bigl(\vecv_0,{\vs},\vecv;\xi,{\vs}_+,\vecvp\bigr) \quad \forall K\in\SO(d) ,
\end{equation}
\begin{equation}\label{intqe}
p\bigl(\vecv;\xi,{\vs}_+,\vecvp\bigr) = c_{\scrP} \int_{[\xi,\infty)\times\Sigma\times\US}  \sigma(\vecv_0,\vecv)\, p_\bn\bigl(\vecv_0,{\vs},\vecv;\xi',{\vs}_+,\vecvp\bigr) \, d\xi' \,d\mm({\vs}) \,d\vecv_0 ,
\footnote{This relation requires that angular momentum is preserved (or reversed) by the scattering map, which is the case for specular reflection and potential scattering, {\blu but which need not hold for the more general scattering maps
which we introduce in Section \ref{SCATTERINGMAPS} (cf.\ Remark \ref{COND3completecondrem}).}}
\end{equation}
\begin{equation}\label{kkg}
p\bigl(\vecv;\xi,{\vs}_+,\vecvp\bigr) \leq  c_{\scrP}\, \sigma(\vecv,\vecvp), \qquad 
p_\bn\bigl(\vecv_0,{\vs},\vecv;\xi,{\vs}_+,\vecvp\bigr) \leq   c_{\scrP}\, \sigma(\vecv,\vecvp) ,
\end{equation}
and
\begin{equation}\label{intqe00}
p\bigl(\vecv;0,{\vs}_+,\vecvp\bigr) = \lim_{\xi\to 0} p\bigl(\vecv;\xi,{\vs}_+,\vecvp\bigr) =  c_{\scrP} \, \sigma(\vecv,\vecvp) .
\end{equation}
Relations \eqref{s-inv1}--\eqref{intqe} follow from the corresponding properties of the transition kernels, and are established in Section \ref{COLLKERRELsec}. The bounds in \eqref{kkg} follow directly from definitions \eqref{kdef} and \eqref{kgdef}, and identity \eqref{intqe00}
from Corollary \ref{kgfromkCORRO}.

We can furthermore extend $\Theta$ to a Markov process $\hTheta$ by setting\footnote{The Markov property of $\Theta$ holds for Poisson scatterer configurations, but fails for all other examples discussed in this paper, including the periodic setting.}
\begin{equation}\label{eq:RFP2}
\hTheta :  t \mapsto \hTheta(t) = \big( \vecq(t) , \vecv(t) , \xi(t) ,\vs(t) ,\vecv_+(t) \big)  ,
\end{equation}
where
\begin{equation*}
\begin{split}
& \vecq(t)  = \vecq_0 + \sum_{n=1}^{n_t} \xi_j \vecv_{j-1} + (t-T_{n_t}) \vecv_{n_t} \qquad \text{(position at time $t$),}\\
& \vecv(t)  = \vecv_{n_t} \qquad \text{(velocity at time $t$),}\\
& \xi(t)  = T_{n_t+1} - t  \qquad \text{(distance at time $t$ to next scattering),}\\
& \vs(t)  =\vs_{n_t+1} \qquad \text{(marking of next scatterer),}\\
& \vecv_+(t)  = \vecv_{n_t+1}  \qquad \text{(velocity after next scattering).}
\end{split}
\end{equation*}
Recall that $T_n=\xi_1+\ldots +\xi_n$.
The Markov property of $\hTheta$ follows from the Markov property of \eqref{eq:JTdist2}, see Section \ref{KINETICEQsec} for details.

\section{The Lorentz process for potentials} \label{sec:soft}

In addition to the classic setting of hard spheres, our results will also apply to ``soft'' scatterers described by a Hamiltonian flow with a compactly supported potential. The Hamiltonian is
\begin{equation*}
H(\vecq,\vecxi)=\tfrac12 \|\vecxi\|^2 +  V_\rho(\vecq)
\end{equation*}
with position $\vecq\in\RR^d$, momentum $\vecxi\in\RR^d$, and potential
\begin{equation}\label{Vrhodef}
V_\rho(\vecq)=\sum_{\vecp\in\scrP^\circ} W\bigg(\frac{\vecq-\vecp}{\rho}\bigg) ,
\end{equation}
which is a superposition of translated and scaled copies of a single potential $W\in\C(\RR^d\setminus\{\vecnull\})$ 
which vanishes outside the unit ball.
Here $\scrP^\circ=\scrP^\circ(\rho)$ \label{Pcirc} is an arbitrary choice of a maximal subset of $\scrP$ 
subject to the property that $\|\vecq-\vecq'\|>2\rho$ for all $\vecq\neq\vecq'\in\scrP^\circ$. %

We use $\scrP^\circ$ 
in place of $\scrP$ in \eqref{Vrhodef} for simplicity of presentation, %
as this ensures %
that the
flow $\Phi_t$ introduced below is well-defined without having to exclude any singular trajectories,
for example %
trajectories for which the particle escapes to infinity in finite time.
We will see that the probability of the particle hitting a scatterer which is not separated from
all other scatterers tends to zero in the Boltzmann-Grad limit;
cf.\ Remarks \ref{unifmodThmREM}  and \ref{MAINTECHNthm2Grem} below.
Therefore our main results %
hold independently of
which convention is used to define
the flow $\Phi_t$ for such particle trajectories.

We assume that $W$ is spherically symmetric
and define (by a slight abuse of notation) $W\in\C(\RR_{>0})$ \label{SSP} %
by %
$W(\vecq)=W(r)$ with $r=\|\vecq\|$;
we will always assume $\liminf_{r\to0}r^2W(r)\geq0$ and that the restriction of $W(r)$ to $(0,1]$ is $\C^2$;
however we allow $W'(r)$ and $W''(r)$ to have discontinuities of the first kind at the point $r=1$.
The Hamiltonian flow is defined through Hamilton's equations
\begin{equation}\label{Hamilton-eq}
\dot\vecq = \nabla_\vecxi H, \qquad  \dot\vecxi = -\nabla_\vecq H .
\end{equation}
The total energy $H(\vecq,\vecxi)=E$ is a constant of motion, and by adjusting the potential by a scalar multiplier, we may assume without loss of generality that $E=\frac12$; this corresponds to a particle speed $\|\vecxi\|=1$ outside the support of $V_\rho$. Under this constraint, the accessible phase space is
\begin{equation}\label{phasespace}
\bigl\{(\vecq,\vecxi)\in {\blu \T(\RR^d)}  
\col \|\vecxi\|^2 +  2 V_\rho(\vecq) = 1\text{ and }[\vecxi=\bn\Rightarrow\nabla V_\rho(\vecq)\neq\bn]\bigr\},
\end{equation}
{\blu where $\T(\RR^d)=\RR^d\times\RR^d$ is the tangent bundle of $\RR^d$.}
Let us also assume $$\limsup_{r\to0}W(r)\neq\frac12.$$
Now for any initial data $(\vecq_0,\vecxi_0)$ in \eqref{phasespace}, 
the solution to \eqref{Hamilton-eq} is well defined for all times. 
For $\vecxi\neq\bn$, define the direction of travel by
\begin{equation*}
\vecv = \|\vecxi\|^{-1}\vecxi = (1- 2 V_\rho(\vecq))^{-1/2} \vecxi ,
\end{equation*}
and the accessible phase space of position and direction by 
\begin{multline}\label{tfwdef}
\tfw(\rho)=\bigl\{ (\vecq,\vecv)\in\T^1(\RR^d) : V_\rho(\vecq) < \tfrac12\text{ or }\\ [V_\rho(\vecq)=\tfrac12\text{ and } 
-\nabla V_\rho(\vecq)\in\R_{>0}\cdot\vecv]\bigr\}.
\end{multline}
The map $\vecxi\mapsto\vecv$ provides a bijection from \eqref{phasespace} onto $\tfw(\rho)$, and the Hamiltonian flow induces a flow on $\tfw(\rho)$ which we denote by $\Phi_t=\Phi_t^{(\rho)}$.\footnote{The last condition in \eqref{tfwdef} means that, 
by convention, we select the outgoing position when the speed is zero;
this means that the orbits of $\Phi_t$ are right continuous but not necessarily left continuous.
However a discontinuity can only occur in the case when a particle hits a scatterer with exactly vanishing impact parameter.}
As in the classical Lorentz gas, we extend its definition to $\T^1(\RR^d)$ by setting $\Phi_t(\vecq,\vecv) = 
(\vecq  ,\vecv)$ if $(\vecq,\vecv)\notin\tfw(\rho)$, and again define the rescaled flow by
\begin{equation}\label{tildePhidef2}
\widetilde \Phi_t^{(\rho)}(\vecq,\vecv) = s_\rho\circ\Phi_{\rho^{1-d} t}^{(\rho)}\circ s_\rho^{-1},
\end{equation} 
with $s_\rho$ as in \eqref{Rrhodef}.
For random initial data $(\vecq,\vecv)\in\T^1(\RR^d)$ distributed according to $\Lambda\in\Pac(\T^1(\RR^d))$, the quantity
$\Theta^{(\rho)}(t) = \widetilde \Phi_t^{(\rho)} (\vecq,\vecv)$ defines a continuous-time random process.\footnote{Note that
the typical orbits of $\Theta^{(\rho)}$ are continuous curves in $\T^1(\RR^d)$.
This stands in contrast to the random flight process in \eqref{ThetarhoDEF}.}

In our theorem %
we need to impose further conditions on the potential $W$,
which ensure that the scattering map is dispersing;
we discuss %
these in %
Section \ref{AppA1}.
The following counterpart of Theorem \ref{thm:M1} shows that in the Boltzmann-Grad limit,
$\Theta^{(\rho)}$ converges, as in the case of the classical Lorentz gas, to a random flight process.
\begin{thm}\label{thm:M1V}
Let $\scrP$ be admissible, and assume that the potential $W$ is dispersing in the sense of 
Definition \ref{DISPERSINGDEF} in Section \ref{AppA1}. 
Then, for any $\Lambda\in\Pac(\T^1(\RR^d))$, there is a random flight process $\Theta$ such that $\Theta^{(\rho)}$ converges to $\Theta$ in distribution, as $\rho\to 0$.
\end{thm}

The proof of Theorem \ref{thm:M1V} reduces to a statement analogous to Theorem \ref{thm:M2}, where the elastic reflection \eqref{scatmap} is replaced by a map $\scrS_-\to\scrS_+$ defined by the scattering at the given potential $W$. We will in fact establish Theorem  \ref{thm:M2} for a more general class of scattering maps, which include elastic reflections as well as potential scattering. 
Relations \eqref{pgendef001}--\eqref{intqe00} for the collision kernels $p,p_\bn$ remain valid in the present context. Note in particular that the transition kernels $k^{\g},k$ are independent of the choice of scattering process, and therefore the only dependence of the collision kernels on the choice of scattering potential (within the class considered here) is via the differential cross section in \eqref{pgendef001} and \eqref{pbndefG001}.

\section{The linear Boltzmann equation and generalisations} \label{sec:linear}

Let us now explain how the existence of the limiting random flight process $\Theta$ yields information on the macroscopic time evolution of an initial particle density $f_0\in\L^1(\T^1(\RR^d))$. 
We will use the shorthand notation $d\vecq=d\!\vol(\vecq)$ for $\vecq\in\R^d$ 
and, as before, $d\vecv=d\omega(\vecv)$ for $\vecv\in\US$.
For fixed $\rho>0$, the evolution of the microscopic density under the rescaled flow $\widetilde\Phi_t^{(\rho)}$ is given by the linear operator $L_t^{(\rho)}:\L^1(\T^1(\RR^d))\to\L^1(\T^1(\RR^d))$
defined by
\begin{align}\label{Ltrhodef}
\int_A \bigl[L_t^{(\rho)}\!f_0\bigr](\vecq,\vecv)\,d\vecq\,d\vecv
=\int_{\tPhi_{-t}^{(\rho)}(A)} f_0(\vecq,\vecv)\,d\vecq\,d\vecv
\end{align}
for every $f_0\in\L^1(\T^1(\R^d))$ and every Borel set $A\subset\T^1(\R^d)$.
To justify this definition, %
one should note that the flow $\tPhi_t^{(\rho)}$
preserves the measure $a\cdot\vol\times\omega$ %
on $\T^1(\R^d)$,
where $a\equiv1$ in the case of hard sphere scattering, while
\begin{align*}
a(\vecq,\vecv)=\begin{cases}\bigl(\tfrac12-V_\rho(\rho^{1-d}\vecq)\bigr)^{(d-2)/2}
&\text{if }\:(\rho^{1-d}\vecq,\vecv)\in\tfw(\rho),
\\
1&\text{otherwise,}
\end{cases}
\end{align*}
in the case of potential scattering\footnote{This is closely related to the fact that $\tPhi^{(\rho)}_t$
is a time change of the geodesic flow on the unit tangent bundle of the region 
$\{\vecq\in\R^d\col V_\rho(\rho^{1-d}\vecq)<\frac12\}$
equipped with the Riemannian metric $(\frac12-V_\rho(\rho^{1-d}\vecq))^{1/2}\,ds$;
cf.\ \cite[\mbox{p.\ 247}]{Arnold}.}.
Now since the two measures $a\cdot\vol\times\omega$ and $\vol\times\omega$
are equivalent, it follows that
push-forward by $\tPhi_{t}^{(\rho)}$ preserves 
the family of signed Borel measures on $\T^1(\R^d)$
which are absolutely continuous with respect to $\vol\times\omega$;
and the content of 
\eqref{Ltrhodef} is that
$L_t^{(\rho)}f_0$ equals the density (wrt.\ $\vol\times\omega$)
of the push-forward by $\tPhi_{t}^{(\rho)}$ of the measure $f_0\cdot\vol\times\omega$.
In fact $L_t^{(\rho)}f_0$ can be expressed by the following explicit, pointwise formula:
\begin{align*}
L_t^{(\rho)}f_0=a\cdot \Bigl(\Bigl(\frac{f_0}a\Bigr)\circ\tPhi_{-t}^{(\rho)}\Bigr).
\end{align*}
Note also that 
$\| L_t^{(\rho)} f_0\|_{\L^1}=\| f_0\|_{\L^1}$ for all $f_0\in\L^1(\T^1(\RR^d))$.

The following corollary of Theorem \ref{thm:M1} and Theorem \ref{thm:M1V} affirms the weak convergence of $L_t^{(\rho)}$ to a limit $L_t$.
\begin{cor}\label{weakconvcor}
Let $\scrP$ be admissible.
There is a family of linear operators $L_t:\L^1(\T^1(\RR^d))\to\L^1(\T^1(\RR^d))$ such that for any $f_0\in\L^1(\T^1(\RR^d))$, $A\subset\T^1(\RR^d)$ bounded with boundary of Lebesgue measure zero, and $t>0$,
\begin{equation*}
\lim_{\rho\to 0} \int_{A}  L_t^{(\rho)}f_0(\vecq,\vecv) \, d\vecq\,d\vecv =
 \int_{A}  L_t f_0(\vecq,\vecv) \, d\vecq\,d\vecv .
\end{equation*}
\end{cor}

To see why this holds, assume (without loss of generality) that $f_0\geq 0$ and that it is normalized as a probability density. Then, with the choice $\Lambda(d\vecq \, d\vecv)= f_0(\vecq,\vecv) \, d\vecq\, d\vecv$ we have 
\begin{equation*}
 \int_{A}  L_t^{(\rho)} f_0(\vecq,\vecv) \, d\vecq\,d\vecv = \PP(\Theta^{(\rho)}(t)\in A),\qquad 
 \int_{A}  L_t f_0(\vecq,\vecv) \, d\vecq\,d\vecv = \PP(\Theta(t)\in A),
\end{equation*}
and the statement follows from Theorem \ref{thm:M1} (resp.\ Therem \ref{thm:M1V}). Since $\Theta$ is in general not Markovian, we cannot expect the limiting operators $L_t$ to form a linear semi-group, and thus $L_t f_0$ cannot be written as the solution of a transport equation. This issue is resolved by considering the Markov process $\hTheta$ in \eqref{eq:RFP2}. There exists a corresponding evolution operator $K_t:\L^1(X)\to \L^1(X)$ on the extended phase space $X=\T^1(\RR^d)\times\R_{>0}\times\Sigma\times\US$, \label{XXdef} such that
\begin{equation*}
\int_{A}  K_t f_0(\vecq,\vecv,\xi,\vs,\vecv_+) \, d\vecq\,d\vecv\, d\xi\,d\mm(\vs)\,d\vecv_+ = \PP(\hTheta(t)\in A) ,
\end{equation*}
for functions of the form
\begin{equation*}
f_0(\vecq,\vecv,\xi,\vs,\vecv_+) \, d\vecq\,d\vecv=\Lambda(d\vecq \, d\vecv)\,p(\vecv;\xi,\vs,\vecv_+).
\end{equation*}
Since $\hTheta(t)$ is Markovian, the family $(K_t)_{\geq 0}$ forms a semigroup, and the function
\begin{equation*}
f(t,\vecq,\vecv,\xi,\vs,\vecv_+)=K_t f_0(\vecq,\vecv,\xi,\vs,\vecv_+)
\end{equation*}
is the solution of the Cauchy problem (see Section \ref{KINETICEQsec} for details) for the forward Kolmogorov equation (or Fokker-Planck-Kolmogorov equation) of $\hTheta$,
\begin{multline}\label{KINETICEQ00}
\bigl(\partial_t+\vecv\cdot\nabla_\vecq-\partial_{\xi}\bigr)f(t,\vecq,\vecv,\xi,\vs,\vecv_+) \\
= \int_{\Sigma\times\US} f\bigl(t,\vecq,\vecv_0,0,\vs',\vecv\bigr)
 p_\bn(\vecv_0,\vs',\vecv;\xi,\vs,\vecv_+)\,d\mm(\vs')\,d\vecv_0,
\end{multline}
subject to the initial condition 
\begin{equation*}
f(0,\vecq,\vecv,\xi,\vs,\vecvp) = f_0(\vecq,\vecv,\xi,\vs,\vecvp)  .
\end{equation*}
The particle density $f(t,\vecq,\vecv)=L_t f_0(\vecq,\vecv)$ in the original phase space is recovered by integrating over the auxiliary variables, i.e.,
\begin{equation*}
f(t,\vecq,\vecv) = \int_{\RR_{>0}\times\Sigma\times\US} f(t,\vecq,\vecv,\xi,\vs,\vecv_+) \,d\xi\,d\mm(\vs)\,d\vecv_+.
\end{equation*}
We note that, in view of \eqref{intqe} and \eqref{intqe00}, a stationary solution of \eqref{KINETICEQ00} is given by
\begin{equation*}
f(t,\vecq,\vecv,\xi,\vs,\vecv_+) = p(\vecv;\xi,\vs,\vecv_+) .
\end{equation*}

Let us suppose for a moment that the limiting process has exponentially distributed flight times $\xi$, i.e.\ the collision kernel is of the form
\begin{equation*}
p_\bn(\vecv_0,\vs',\vecv;\xi,\vs,\vecv_+) =  \oxi^{-1} p_\bn(\vecv_0,\vs',\vecv;\vs,\vecv_+)\, \e^{-\xi/\oxi} ,
\end{equation*}
with $\oxi$ the mean free path \eqref{OXIFORMULAG}.
Then the ansatz 
\begin{equation*}
f(t,\vecq,\vecv,\xi,\vs,\vecv_+) =  \oxi^{-1} f(t,\vecq,\vecv,\vs,\vecv_+) \, \e^{-\xi/\oxi}
\end{equation*}
reduces \eqref{KINETICEQ00} to
\begin{multline}\label{KINETICEQ007}
\bigl(\partial_t+\vecv\cdot\nabla_\vecq+\oxi^{-1} \bigr) f(t,\vecq,\vecv,\vs,\vecv_+) \\
=\oxi^{-1} \int_{\Sigma\times\US}  f\bigl(t,\vecq,\vecv_0,\vs',\vecv\bigr)
  p_\bn(\vecv_0,\vs',\vecv;\vs,\vecv_+)\,d\mm(\vs')\,d\vecv_0  .
\end{multline}
In case of a Poisson scatterer configuration, we have in fact that 
\begin{equation*}
p_\bn(\vecv_0,\vs',\vecv;\vs,\vecv_+)=\frac{\sigma(\vecv,\vecv_+)}{v_{d-1}}.
\end{equation*}
In this case \eqref{KINETICEQ007} reduces further, with the ansatz $$f(t,\vecq,\vecv,\vs,\vecv_+)=v_{d-1}^{-1}  f(t,\vecq,\vecv) \sigma(\vecv,\vecv_+),$$ to 
\begin{equation*}%
\bigl(\partial_t+\vecv\cdot\nabla_\vecq+\oxi^{-1}\bigr) f(t,\vecq,\vecv) \\
=c_\scrP  \int_{\US}  f\bigl(t,\vecq,\vecv_0) \, \sigma(\vecv_0,\vecv)\,d\vecv_0  ,
\end{equation*}
which can be written in the standard form of the linear Boltzmann equation,
\begin{equation}\label{KINETICEQ008}
\bigl(\partial_t+\vecv\cdot\nabla_\vecq\bigr) f(t,\vecq,\vecv) \\
=c_\scrP  \int_{\US}  \Bigl( f\bigl(t,\vecq,\vecv_0) -  f(t,\vecq,\vecv) \Bigr) \, \sigma(\vecv_0,\vecv)\,d\vecv_0 .
\end{equation}%
This illustrates that the transport equation \eqref{KINETICEQ00} may indeed be viewed as a generalisation of the linear Boltzmann equation \eqref{KINETICEQ008}. In contrast to random scatterer configurations, we will see that other examples discussed in this study lead to transport equations of the form \eqref{KINETICEQ00} that do not reduce to \eqref{KINETICEQ008} or even \eqref{KINETICEQ007}.

\section{Outline of the paper}

The assumptions on the scatterer configuration $\scrP$ are stated above in terms of convergence properties of random point sets. Section \ref{ASSUMPTsec} provides the measure-theoretic background for a rigorous formulation of these assumptions. In particular, we explain how to identify point sets in $\RR^d$ with counting measures, i.e., locally finite Borel measures that are superpositions of Dirac masses. The space $M(\scrX)$ of locally finite Borel measures on $\scrX$ (with $\scrX=\RR^d$ in this instance) is equipped with the vague topology, which in turn allows us to define Borel probability measures on $M(\scrX)$, and thus define the notion of a {\em random counting measure}, which is synonymous with {\em random point process}. This, as well as the extension to marked point sets and point processes (where $\scrX=\RR^d\times\Sigma$ in the above), is explained in Section \ref{sec:PP}, following a technical discussion of uniform convergence properties of families of general Borel probability measures in Section \ref{UNIFCONVsec}. 
Section \ref{ASSUMPTLISTsec} then proceeds to translate the assumptions of Section \ref{sec:outline} on the scattering configuration $\scrP$ into the language of random counting measures. 
Section \ref{FIRSTCONSsec} provides a number of immediate consequences of the assumptions made in Section \ref{ASSUMPTLISTsec} through a series of technical lemmas. The assumptions on $\scrP$ are stated in terms of point processes $\Xi_\vs$ that are constructed relative to points $\vecq\in\scrP$. Section \ref{GENLIMITsec} constructs a new point process $\Xi$ relative to almost all points $\vecq\in\RR^d$, which will be relevant for the particle dynamics in the case of  {\em macroscopic} initial conditions (in contrast to {\em microscopic} initial data on or near a scatterer). The properties of $\Xi$ are further analysed in Section \ref{XiPropSec}.

Section \ref{FIRSTCOLLISIONsec} provides the first milestone in understanding the Boltzmann-Grad limit of the Lorentz process. It establishes limit theorems for the time and location of the first collision for a particle with random initial velocity, and a given (deterministic) initial point either (a) on or near a scatterer (Theorem \ref{Thm2gen}), or (b) in generic position outside a scatterer (Theorem \ref{Thm2macr}). 
The preparatory Section \ref{TRANSKERsec} defines the {\em transition kernel}, which provides the joint limit distribution of the first hitting time and impact parameter. The limit theorems are stated and proved in Section \ref{FIRSTCOLLsec}. Invariance properties and relations between the transition kernels for on-scatterer vs.\ generic initial data are derived in Section \ref{TRANSKERRELsec}.
The discussion then turns to the velocity {\em after} the first collision, which of course depends on the choice of scattering map at an individual scatter. Our hypotheses on the scattering map include spherical symmetry and differentiability, and are listed in Section \ref{SCATTERINGMAPS}. They are sufficiently general to allow for elastic hard-sphere scattering (specular reflection) as in the orginal Lorentz gas, but also scattering by a general class of spherically symmetric potentials, which are discussed in detail in Section \ref{AppA1}. The limit distribution for the post-collision velocities are expressed in terms of {\em collision kernels}, which are defined in Section \ref{COLLKERsec} and further analysed in Section \ref{COLLKERRELsec}. The corresponding limit theorems are stated in Section \ref{NEXTVELOCITYsec} as Theorems \ref{unifmodThm2} and \ref{unifmodThm2macr}, for near-scatterer and macroscopic initial conditions, respectively. In preparation for the proof of the convergence of the full Lorentz process, we furthermore need to bound the probability of near-grazing collisions and other singular trajectories. This is carried out in Section \ref{bndgrazingprobSEC}.

The key results of this work, the convergence of the Lorentz process to a random flight process, are stated and proved in Section \ref{MAINRESsec}. We first establish the corresponding results in the discrete-time setting, where time is measured in terms of the number of collisions. This is captured in Theorem \ref{MAINTECHNTHM2A} in Section \ref{MAINTECHNthm2aSEC}, with 
Section \ref{scatmapsmore} and Section \ref{MAINTECHNthm2apfSEC} dedicated to its proof. Theorem \ref{MAINTECHNTHM2A} assumes initial data near a scatterer, and the analogous result for macroscopic initial conditions, stated as Theorem \ref{MAINTECHNthm2G}, is derived in Section \ref{MAINRESMACRO} as a consequence of Theorem \ref{MAINTECHNTHM2A}.
The extension of these results to the continuous-time setting follows from a number of technical  estimates, which are given in Section \ref{LIMITFLIGHTsec}. This completes the proof of the main results of this work, which are stated in the introduction as Theorems \ref{thm:M1} and \ref{thm:M1V}. 
Section \ref{KINETICEQsec} shows that the limiting random flight process has a Markovian extension, and that the transport equation \eqref{KINETICEQ00} is indeed the forward Kolmogorov equation of that Markov process. Theorem \ref{FPK} states the existence and uniqueness of the solution to the Cauchy problem, under the assumption that the collision kernel is continuous. 

The final part of this paper, Section \ref{EXAMPLESsec}, provides a detailed discussion of point sets for which the assumptions on the scatterer configuration $\scrP$ (as stated in Section \ref{ASSUMPTsec}) are satisfied. Section \ref{PoissonSEC} explores the case when $\scrP$ is the realisation of a Poisson process with constant intensity. Even in this classic setting, checking the validity of the required assumptions is not straightforward. Section \ref{periodicpointsetsSEC} confirms the required assumptions in the case of general locally finite periodic point sets $\scrP$. The convergence of the Lorentz process was, in the periodic setting, previously known only for Euclidean lattices. The most interesting new examples to which the results of the present study apply, are Euclidean model sets (also known as cut-and-project sets), which are discussed in Section \ref{QCexsec}. Model sets are point sets 
which are often aperiodic, and they serve as mathematical models for quasicrystals. 
Section \ref{AppA1} discusses %
the relationship
between the scattering potential %
and the scattering map and differential cross section.
In particular  
Lemma \ref{genokscatterercondLEM} 
describes a general class of repulsive potentials 
for which the assumptions in Section \ref{SCATTERINGMAPS}
are satisfied.
Section \ref{AppA2} gives an outline of %
how the methods of Sections \ref{FIRSTCOLLISIONsec} and \ref{MAINRESsec}
can be extended to deal with more general potentials,
for which the scattering map does not satisfy the assumptions in 
Section \ref{SCATTERINGMAPS}.
Finally, Section \ref{sec:open} comprises a selection of open questions and directions for future work.

\vspace{10pt}

\noindent
{\blu\textbf{Acknowledgment.}
We would like to thank the referee for carefully reading the paper 
and valuable suggestions on improvements of the exposition.}

\chapter{Point sets, point processes and key assumptions}
\label{ASSUMPTsec}

The aim of this section is to state and discuss the list of assumptions on the point set $\scrP$ in a more precise and general form, compared to the outline in Section \ref{sec:outline}. We will require a notion of uniform weak convergence of random point sets. The most natural framework for this is to identify point sets with counting measures, and define a random counting measure (point process) in a suitable probability space. We thus need to deal with probability measures on spaces of locally finite Borel measures and their convergence. Section \ref{UNIFCONVsec} explains the concept of uniform weak convergence on general topological spaces, which we then specialise to point processes and marked point processes in Section \ref{sec:PP}. The main assumptions of this paper are stated in full generality in Section \ref{ASSUMPTLISTsec}. They are explored in detail in Section \ref{FIRSTCONSsec}.

\section{Uniform convergence of families of probability measures}
\label{UNIFCONVsec}

For $S$ any topological space, we write $P(S)$ for the set of Borel probability measures on $S$,
\label{PSdef}
equipped with the weak topology.
From now on we will always assume that $S$ is separable and metrizable.
Then $P(S)$ is also metrizable
\cite[pp.\ 72-73]{billingsley99}.

An important notion for us will be a certain general version of \textit{uniform} convergence in $P(S)$.
The setting is as follows. %

Let $J$ be a fixed index set, and let $C$ be a compact subset of $P(S)$.
For each $0<\rho<1$,
let $J(\rho)$ be a subset of $J$,
and let $\{\mu_{j,\rho}\}_{j\in J(\rho)}$ be a family of probability measures in $P(S)$. Let $\{\nu_{j}\}_{j\in J}$  be a family of probability measures contained in $C$.
Then we say that 
\begin{align}\label{UNIFCONVdef1}
\textit{$\mu_{j,\rho}$ converges weakly to $\nu_{j}$ ($\mu_{j,\rho}\xrightarrow[]{\textup{ w }}\nu_{j}$) as $\rho\to0$, uniformly over $j\in J(\rho)$},
\end{align}
if, for some metric $d$ on $P(S)$ realizing %
the weak topology, we have
\begin{align}\label{UNIFCONVdef2}
\forall \ve>0\col
\exists\rho_0\in(0,1)\col
\forall\rho\in(0,\rho_0)\col
\forall j\in J(\rho)\col
d(\mu_{j,\rho},\nu_{j})<\ve.
\end{align}

Note that this definition is independent of the choice of $d$:
If \eqref{UNIFCONVdef2} holds for one metric $d$ realizing %
the weak topology of $P(S)$, 
then it holds for all such metrics.
This is a consequence of the following lemma.

\begin{lem}\label{METRSPACECPTlem}
If $d_1,d_2$ are two metrics on a set $M$ inducing the same topology,
and $C\subset M$ is a compact set with respect to that topology,
then for any $\ve>0$ there is some $\ve'>0$ such that
\begin{align*}
\forall x\in M,\: y\in C\col
d_1(x,y)<\ve'\Rightarrow d_2(x,y)<\ve.
\end{align*}
\end{lem}

\begin{proof}
Let $\ve>0$ be given,
and assume that there does not exist any corresponding $\ve'>0$.
Then there are sequences $x_1,x_2,\ldots$ in $M$ and $y_1,y_2,\ldots$ in $C$ such that
$d_1(x_n,y_n)\to0$ but $d_2(x_n,y_n)\geq\ve$ for all $n$.
By passing to a subsequence we may assume that there is $y\in C$ such that 
$y_n\to y$.
This notion is independent of the choice of metric,
i.e.\ we have both $d_1(y_n,y)\to0$ and $d_2(y_n,y)\to0$.
Now $d_1(x_n,y_n)\to0$ and  $d_1(y_n,y)\to0$ imply
$d_1(x_n,y)\to0$, i.e.\ $x_n\to y$, and thus
$d_2(x_n,y)\to0$.
Combined with $d_2(y_n,y)\to0$ this implies
$d_2(x_n,y_n)\to0$, contradicting the fact that $d_2(x_n,y_n)\geq\ve$ for all $n$.
\end{proof}

We next give some criteria for the uniform convergence in \eqref{UNIFCONVdef1} to hold.

\begin{lem}\label{UNIFPORTMlemC}
The uniform convergence in \eqref{UNIFCONVdef1} holds
if and only if the following condition is satisfied:
For any sequence $P=\{\rho_n\}\subset(0,1)$, $\rho_n\to0$,
and any choice of $j(\rho)\in J(\rho)$ for $\rho\in P$,
if there is some $\nu\in C$ such that
$\nu_{j(\rho)}\xrightarrow[]{\textup{ w }}\nu$ as $\rho\to0$ through $P$,
then also $\mu_{j(\rho),\rho}\xrightarrow[]{\textup{ w }}\nu$ as $\rho\to0$ through $P$.
\end{lem}

\begin{proof}
The uniform convergence in \eqref{UNIFCONVdef1} clearly implies the stated condition
(by using \eqref{UNIFCONVdef2} and the triangle inequality).
Now assume that the stated condition holds, but 
the uniform convergence in \eqref{UNIFCONVdef1} does not hold.
Then there exist $c>0$, a sequence $P=\{\rho_n\}\subset(0,1)$, $\rho_n\to0$,
and for each $\rho\in P$ some $j(\rho)\in J(\rho)$,
such that 
\begin{align}\label{UNIFPORTMlemCPF2}
d(\mu_{j(\rho),\rho},\nu_{j(\rho)})>c\qquad\text{for each }\:\rho\in P.
\end{align}
Since $\nu_{j(\rho)}\in C$ for all $\rho\in P$,
after replacing $P$ with an appropriate subsequence
we may also assume that 
there is some $\nu\in C$ such that
$\nu_{j(\rho)}\xrightarrow[]{\textup{ w }}\nu$ as $\rho\to0$ through $P$.
Hence by our assumption we must also have
$\mu_{j(\rho),\rho}\xrightarrow[]{\textup{ w }}\nu$ as $\rho\to0$ through $P$.
These together imply 
$d(\mu_{j(\rho),\rho},\nu_{j(\rho)})\to0$ as $\rho\to0$ through $P$,
contradicting \eqref{UNIFPORTMlemCPF2}.
\end{proof}

\begin{lem}\label{UNIFPORTMlemB}
The uniform convergence in \eqref{UNIFCONVdef1} holds
if and only if, for every fixed $f\in \C_b(S)$,
we have $\mu_{j,\rho}(f)\to\nu_{j}(f)$ as $\rho\to0$, uniformly over $j\in J(\rho)$.
\end{lem}

Here ``$\mu_{j,\rho}(f)\to\nu_{j}(f)$ as $\rho\to0$, uniformly over $j\in J(\rho)$''
is the standard notion of uniform convergence in $\R$:
 $\forall\ve>0$: $\exists \rho_0\in(0,1)$:
$\forall\rho\in(0,\rho_0)$: $\forall j\in J(\rho)$:
$|\mu_{j,\rho}(f)-\nu_{j}(f)|<\ve$.

\begin{proof}
Assume that for every fixed $f\in \C_b(S)$
we have $\mu_{j,\rho}(f)\to\nu_{j}(f)$ as $\rho\to0$, uniformly over $j\in J(\rho)$.
Consider any sequence $P=\{\rho_n\}\subset(0,1)$, $\rho_n\to0$ and any choice of
$j(\rho)\in J(\rho)$ for $\rho\in P$
such that $\nu_{j(\rho)}\xrightarrow[]{\textup{ w }}\nu\in C$ as $\rho\to0$ through $P$.
Then for every fixed $f\in \C_b(S)$ we have $\nu_{j(\rho)}(f)\to\nu(f)$ as $\rho\to0$ through $P$,
and combined with our assumption this implies
$\mu_{j(\rho),\rho}(f)\to\nu(f)$ as $\rho\to0$ through $P$.
Hence $\mu_{j(\rho),\rho}\xrightarrow[]{\textup{ w }}\nu$ as $\rho\to0$ through $P$,
and in view of Lemma \ref{UNIFPORTMlemC} it follows that \eqref{UNIFCONVdef1} holds.

Conversely, assume now that there is some $f\in \C_b(S)$ for which 
$\mu_{j,\rho}(f)\to\nu_{j}(f)$ does \textit{not} hold uniformly over $j\in J(\rho)$ as $\rho\to0$.
Then there exist $c>0$, a sequence $P=\{\rho_n\}\subset(0,1)$, $\rho_n\to0$,
and for each $\rho\in P$ some $j(\rho)\in J(\rho)$,
such that 
\begin{align}\label{UNIFPORTMlemBPF1}
\bigl|\mu_{j(\rho),\rho}(f)-\nu_{j(\rho)}(f)\bigr|>c\qquad\text{for each }\:\rho\in P.
\end{align}
Replacing $P$ with an appropriate subsequence
we may also assume that 
there is some $\nu\in C$ such that
$\nu_{j(\rho)}\xrightarrow[]{\textup{ w }}\nu$ as $\rho\to0$ through $P$.
Then $\nu_{j(\rho)}(f)\to\nu(f)$ as $\rho\to0$ through $P$,
and together with \eqref{UNIFPORTMlemBPF1} this implies that 
$\mu_{j(\rho),\rho}(f)\not\to\nu(f)$ as $\rho\to0$ through $P$.
Hence we do not have $\mu_{j(\rho),\rho}\xrightarrow[]{\textup{ w }}\nu$ as $\rho\to0$ through $P$,
and by Lemma \ref{UNIFPORTMlemC}, \eqref{UNIFCONVdef1} does not hold.
\end{proof}

\begin{remark}\label{UNIFPORTMlemBrem}
The proof of Lemma \ref{UNIFPORTMlemB} immediately extends to show that 
if %
\eqref{UNIFCONVdef1} holds,
and if $f:S\to\R$ is a bounded Borel measurable function 
whose set of discontinuities has measure zero with respect to each $\nu\in C$,
then $\mu_{j,\rho}(f)\to\nu_{j}(f)$ as $\rho\to0$, uniformly over $j\in J(\rho)$.
In particular, letting $f$ be a characteristic function,
it follows that if $B\subset S$ is any Borel set satisfying
$\nu(\partial B)=0$ for all $\nu\in C$,
then $\mu_{j,\rho}(B)\to\nu_{j}(B)$ as $\rho\to0$, uniformly over $j\in J(\rho)$.
The reverse implication is covered by the following lemma.
\end{remark}

\begin{lem}\label{UNIFPORTMlem}
In the above setting, let $\scrB$ be a family of Borel subsets of $S$ such that
$\nu(\partial B)=0$ for all $B\in\scrB$, $\nu\in C$,
and also, for any sequence $\nu_1,\nu_2,\ldots\in P(S)$ and any $\nu\in C$,
if $\nu_n(B)\to\nu(B)$ for every $B\in\scrB$ then $\nu_n\xrightarrow[]{\textup{ w }}\nu$.
Then a sufficient condition for the uniform convergence in \eqref{UNIFCONVdef1} to hold 
is that for every $B\in\scrB$,
we have $\mu_{j,\rho}(B)\to\nu_{j}(B)$ as $\rho\to0$, uniformly over $j\in J(\rho)$.
\end{lem}
\begin{proof}
Let $d$ be a metric on $P(S)$ which induces the weak topology.
Assume that for every $B\in\scrB$
we have $\mu_{j,\rho}(B)\to\nu_{j}(B)$ as $\rho\to0$ uniformly over $j\in J(\rho)$,
but that the uniform convergence in \eqref{UNIFCONVdef1} does \textit{not} hold.
Then there exist $c>0$, a sequence $P=\{\rho_n\}\subset(0,1)$, $\rho_n\to0$,
and for each $\rho\in P$ some $j(\rho)\in J(\rho)$,
such that 
\begin{align}\label{UNIFPORTMlemPF2}
d(\mu_{j(\rho),\rho},\nu_{j(\rho)})>c\qquad\text{for each }\:\rho\in P.
\end{align}
Since $\nu_{j(\rho)}\in C$ for all $\rho\in P$, there exist $\nu\in C$ and an infinite subsequence $P'\subset P$
such that $\nu_{j(\rho)}\xrightarrow[]{\textup{ w }}\nu$ as $\rho\to0$ through $P'$.
Hence for every $B\in\scrB$ we have
$\nu_{j(\rho)}(B)\to\nu(B)$ as $\rho\to0$ through $P'$, since $\nu(\partial B)=0$
\cite[Thm.\ 4.25]{kallenberg02};
and combining this with our assumption we conclude $\mu_{j(\rho),\rho}(B)\to\nu(B)$ as $\rho\to0$ through $P'$.
Because of the convergence determining property of $\scrB$ assumed in the statement of the lemma,
this implies $\mu_{j(\rho),\rho}\xrightarrow[]{\textup{ w }}\nu$  as $\rho\to0$ through $P'$.
Now $\nu_{j(\rho)}\xrightarrow[]{\textup{ w }}\nu$ and $\mu_{j(\rho),\rho}\xrightarrow[]{\textup{ w }}\nu$ 
together imply that $d(\mu_{j(\rho),\rho},\nu_{j(\rho)})\to0$ as $\rho\to0$ through $P'$,
contradicting \eqref{UNIFPORTMlemPF2}.
\end{proof}

\section{Point processes and marked point processes}\label{sec:PP}

Given a locally compact second countable Hausdorff (lcscH) space $\scrX$,
\label{scrXdef1}
we let $M(\scrX)$ be the set of locally finite Borel measures on $\scrX$.
\label{MXdef}
Recall that a Borel measure $\mu$ on $\scrX$ is said to be locally finite if
$\mu B<\infty$ for every relatively compact Borel set $B\subset\scrX$.
We equip $M(\scrX)$ with the vague topology.
Then $M(\scrX)$ is a Polish space
(i.e.\ separable and has a complete metrization).
We write $\scrM$ for the Borel $\sigma$-algebra of $M(\scrX)$.
\label{scrMdef}
We let $N(\scrX)$ be the set of \textit{counting measures} in $M(\scrX)$,
\label{NXdef}
and let $N_s(\scrX):=\{ \nu\in N(\scrX) \col \sup_{x\in\scrX} \nu\{x\}\leq 1\}$ be the subset of \textit{simple} counting measures.\label{NsXdef}
Then $N(\scrX)$ is a closed subset of $M(\scrX)$ (hence also Polish),
and $N_s(\scrX)$ is a Borel subset of $N(\scrX)$.
Define $\scrN=\{B\cap N(\scrX)\col B\in\scrM\}$, which yields the Borel $\sigma$-algebra of $N(\scrX)$.
\label{scrNdef}
The elements of $N_s(\scrX)$ may be identified with 
the family of locally finite subsets of $\scrX$
through $\nu\mapsto\supp(\nu)$. The inverse map is $\{x_i\}\mapsto\sum_i\delta_{x_i}$.
We will use this identification between point sets and simple counting measures throughout this work, often using the same notation for point set and counting measure.

A \textit{point process} is,
by definition,
a random element $\xi$ in $(N(\scrX),\scrN)$.
It is called \textit{simple} if $\xi\in N_s(\scrX)$ almost surely.
We identify $P(N_s(\scrX))$ with the set of probability measures $\nu\in P(N(\scrX))$ with 
$\nu(N_s(\scrX))=1$.
Then a point process $\xi$ is simple if and only if its law is in $P(N_s(\scrX))$.
The \textit{intensity measure} of $\xi$, $\EE\xi$, is the Borel measure on $\scrX$ given by
$(\EE\xi)B=\EE(\xi B)$ for any Borel set $B\subset\scrX$.
\label{intensitydef}
By abuse of notation, for $\nu\in P(N(\scrX))$,
we call ``the intensity of $\nu$'' the intensity measure of any point process whose distribution is $\nu$;
i.e.\ the Borel measure $B\mapsto\int_{N(\scrX)}\eta B\,d\nu(\eta)$ on $\scrX$.

From Section \ref{ASSUMPTLISTsec} onwards we will make the choice $\scrX=\R^d\times\Sigma$,
with $\Sigma$ a compact metric space.
A point process $\xi$ in $\scrX=\R^d\times\Sigma$ can be thought of as a 
\textit{marked} point processes 
with locations in $\R^d$ and marks in $\Sigma$.
Let $p_1$ be the projection map $\scrX\to\R^d$.
\label{p1def}
Note that since $\Sigma$ is compact, the ``ground process'',
$p_{1*}\xi$, automatically becomes a point process in $\R^d$.
We call $\xi$ \textit{simple as a marked point process}
if the ground process $p_{1*}\xi$ is simple.
We refer the reader to \cite[Ch.\ 6.4]{DaleyVereJones2003} for further background.

\vspace{5pt}

The following lemma gives a criterion for uniform convergence of sequences in
$P(S\times N(\scrX))$, for $S$ a lcscH space,
which will be useful for us. %
We will need it later in the case $S=\US$.
Let $J$ be a fixed index set,
and for each $0<\rho<1$,
let $J(\rho)$ be a subset of $J$,
\begin{lem}\label{PPUNICCONVCONCRETElemA}
Let $S$ be a lcscH space, let $S'=S\times N(\scrX)$, and let $C$ be a compact subset of $P(S')$.
Define
\begin{align}\label{Ffgdef}
F_{f,g}(\mu)=\int_{S'}g(p,\eta(f))\,d\mu(p,\eta)\qquad
(f\in\C_c(\scrX),\:
g\in\C_b(S\times\R),\:
\mu\in P(S')).
\end{align}
Let $\{\mu_{j,\rho}\}_{j\in J(\rho)}$ and $\{\nu_{j}\}_{j\in J}$ 
be families of probability measures in $P(S')$,
such that $\nu_{j}\in C$ for all $j\in J$.
Then the following two conditions are equivalent:
\\
(a) $\mu_{j,\rho}\xrightarrow[]{\textup{ w }}\nu_{j}$ as $\rho\to0$, uniformly over $j\in J(\rho)$;
\\[5pt]
(b) for any $f\in\C_c(\scrX)$ and $g\in \C_c(S\times\R)$,
$F_{f,g}(\mu_{j,\rho})\to F_{f,g}(\nu_{j})$
as $\rho\to0$, uniformly over $j\in J(\rho)$.
\end{lem}
\begin{proof}
For $f\in\C_c(\scrX)$ we define the map
$T_f:S'\to S\times\R$ by $T_f(p,\eta)=(p,\eta(f))$;
then $F_{f,g}(\mu)=\mu(g\circ T_{f})$ %
for any $g\in\C_b(S\times\R)$, $\mu\in P(S')$.
The map $T_f$ is %
continuous;
hence $g\circ T_f\in \C_b(S')$ for any $f,g$.
Now the implication (a)$\Rightarrow$(b) follows from Lemma \ref{UNIFPORTMlemB}.

Conversely, assume (b).
In order to prove (a),
by Lemma \ref{UNIFPORTMlemC} it suffices to prove that 
for a given sequence $P=\{\rho_n\}\subset(0,1)$ with $\rho_n\to0$,
and given $\nu\in C$ and $j(\rho)\in J(\rho)$ ($\rho\in P$)
subject to $\nu_{j(\rho)}\xrightarrow[]{\textup{ w }}\nu$ as $\rho\to0$ through $P$,
we have $\mu_{j(\rho),\rho}\xrightarrow[]{\textup{ w }}\nu$ as $\rho\to0$ through $P$.
For any $f\in\C_c(\scrX)$ and $g\in \C_c(S\times\R)$ we have
$\nu_{j(\rho)}(g\circ T_{f})\to\nu(g\circ T_{f})$ as $\rho\to0$ through $P$;
combined with (b),
this implies that %
$\mu_{j(\rho),\rho}(g\circ T_{f})\to\nu(g\circ T_{f})$, as $\rho\to0$ through $P$.
But $g$ is arbitrary in $\C_c(S\times\R)$;
hence we conclude
$T_{f*}(\mu_{j(\rho),\rho})\xrightarrow[]{\textup{ w }}T_{f*}(\nu)$ as $\rho\to0$ through $P$
(cf., e.g., \cite[Prop.\ 3.4.4]{sEtK86}).
The fact that this holds for all $f\in\C_c(\scrX)$ implies,
via a simple extension of 
\cite[Thm.\ 16.16 (ii)$\Rightarrow$(i)]{kallenberg02},
that $\mu_{j(\rho),\rho}\xrightarrow[]{\textup{ w }}\nu$ as $\rho\to0$ through $P$.
Hence (a) holds.
\end{proof}

\begin{lem}\label{PPUNICCONVCONCRETElem3}
Let $C$ be a compact subset of $P(N_s(\scrX))$
such that every $\nu\in C$ has the same intensity $\tmu$
(a fixed locally finite Borel measure on $\scrX$).
Let $\{\mu_{j,\rho}\}_{j\in J(\rho)}$ and $\{\nu_{j}\}_{j\in J}$ 
be families of probability measures in $P(N(\scrX))$,
such that $\nu_{j}\in C$ for all $j\in J$.
Assume that, for any relatively compact Borel set $B\subset\scrX$ with $\tmu(\partial B)=0$ 
and any $r\in\Z^+$,
\begin{align*}%
\mu_{j,\rho}(\{\eta\in N(\scrX)\col \eta B\geq r\})\to
\nu_{j}(\{\eta\in N(\scrX)\col \eta B\geq r\})
\end{align*}
as $\rho\to0$, uniformly over $j\in J(\rho)$.
Then $\mu_{j,\rho}\xrightarrow[]{\textup{ w }}\nu_{j}$ as $\rho\to0$, uniformly over $j\in J(\rho)$.
\end{lem}
\begin{proof}
Let $S=N(\scrX)$
and let $\scrB$ be the family of Borel subsets of $S$ of the form
$A=\{\eta\in S\col \eta B\geq r\}$,
where $r\in\Z^+$ and $B$ is a relatively compact Borel subset of $\scrX$ with $\tmu(\partial B)=0$.
Note that $\partial A\subset\{\eta\in S\col \eta(\partial B)\geq1\}$;
thus $\nu(\partial A)\leq\int_S\eta(\partial B)\,d\nu(\eta)=\tmu(\partial B)=0$
for each $\nu\in C$.
Furthermore, if $\nu\in C$ and $B$ is a relatively compact Borel set in $\scrX$ satisfying
$\eta(\partial B)=0$ for $\nu$-a.e.\ $\eta\in N(\scrX)$,
then also 
$\tmu(\partial B)=\int_{N(\scrX)}\eta(\partial B)\,d\nu(\eta)=0$.
Hence by \cite[\mbox{Thm 16.16 (iv)$\Rightarrow$(i)}]{kallenberg02},
for any sequence $\nu_1,\nu_2,\ldots\in P(S)$ and any $\nu\in C$,
if $\nu_n(A)\to\nu(A)$ for every $A\in\scrB$ then $\nu_n\xrightarrow[]{\textup{ w }}\nu$.
Hence Lemma \ref{UNIFPORTMlem} applies,
and shows the desired implication.
\end{proof}

\section{The list of assumptions}
\label{ASSUMPTLISTsec}

As in Section \ref{sec:outline}, let $\scrP$ be a fixed locally finite subset of $\R^d$ ($d\geq2$)
with constant asymptotic density $c_\scrP$. Recall also the definitions of $\scrX=\R^d\times\Sigma$, $\mu_\scrX=\vol\times\mm$
and
\begin{align*} %
\tP=\{(\vecp,{\vs}(\vecp))\col\vecp\in\scrP\}\subset\scrX
\end{align*}
from Section \ref{sec:outline}. Furthermore, for any $\vecq\in\R^d$, $\vecv\in\US$ and $0<\rho<1$, we set\footnote{In the notation of the introduction, $\tP_\vecq-\vecq=(\tP-\vecq)^*$.}
\begin{equation}\label{Pqdef}
\tP_\vecq=
\begin{cases}
\tP\setminus\{(\vecq,\vs(\vecq))\} & (\vecq\in\scrP)\\
\tP & (\vecq\notin\scrP) 
\end{cases}
\end{equation}
and
\begin{align}\label{XIRHOqv}
\scrQ_\rho(\vecq,\vecv)=(\tP_\vecq-\vecq)\,R(\vecv)\,D_\rho .
\end{align}

Given any $\lambda\in P(\US)$,
if we take $\vecv$ random in $(\US,\lambda)$ then $\scrQ_\rho(\vecq,\vecv)$ becomes a random point set.
We write $\mu_{\vecq,\rho}^{(\lambda)}\in P(N_s(\scrX))$ for the distribution\label{muqrholambdaDEF}
of the corresponding point process.
In other words, $\mu_{\vecq,\rho}^{(\lambda)}$ is the push-forward of $\lambda$ by 
the map
\begin{equation*}
\US\to N(\scrX), \qquad \vecv\mapsto \sum_{\vecp\in\scrQ_\rho(\vecq,\vecv)} \delta_\vecp.
\end{equation*}

The following are our hypotheses on $\scrP$. These will generalise and make precise the outline assumptions from Section \ref{sec:outline} (we will use the same labelling).

Our standing assumption is that there exists a continuous map ${\vs}\mapsto\mu_{\vs}$ from $\Sigma$ to $P(N(\scrX))$\label{muvsdef}, a Borel probability measure $\mm$ on $\Sigma$ 
with the following properties:
\begin{enumerate}[{\bf [P1]}]
\item {\em Uniform and zero density:} For any bounded $B \subset\scrX$ with $\mu_\scrX(\partial B)=0$, we have
\begin{align}\label{ASYMPTDENSITY1b}
\lim_{T\to\infty}\frac{\#(\tP\cap TB)}{T^d}=c_\scrP\mu_\scrX(B).
\end{align}
\item {\em Spherical equidistribution:} 
{\blu There exists a subset $\scrE\subset\scrP$ of asymptotic density zero such that}
for any fixed $T\geq1$ and $\lambda\in \Pac(\US)$, we have\footnote{For uniform convergence in $P(N(\scrX))$,
use \eqref{UNIFCONVdef1}--\eqref{UNIFCONVdef2}
with $\vecq$, $\scrP$, $\scrP_T(\rho)$ in place of $j,J,J(\rho)$.
Indeed, $\{\mu_{\vs}\col{\vs}\in\Sigma\}$ is a compact subset of $P(N(\scrX))$,
since it is the continuous image of the compact set $\Sigma$. \label{foot6}} 
\begin{align}\label{ASS:KEY}
\mu_{\vecq,\rho}^{(\lambda)}\xrightarrow[]{\textup{ w }}\mu_{{\vs}(\vecq)}
\quad\text{as }\:\rho\to0,\:\text{ uniformly for $\vecq\in\scrP_T(\rho):=\scrP\cap\scrB^d_{T\rho^{1-d}}\setminus\scrE$.}
\end{align}
\item {\em No escape of mass:} For every bounded Borel set $B\subset\RR^d$,
\begin{align*}%
\lim_{\xi\to\infty}\limsup_{\rho\to0}\hspace{7pt}
[\vol\times\omega]\bigl(\bigl\{(\vecq,\vecv)\in B\times\US\col
\scrQ_{\rho}(\rho^{1-d}\vecq,\vecv)\cap(\fZ_\xi\times\Sigma)=\emptyset\bigr\}\bigr)=0.
\end{align*}
\end{enumerate}

We will explain the assumption [P3] further in
Section \ref{GENLIMITsec},
where we prove the existence of a limit 
of $\scrQ_\rho(\rho^{1-d}\vecq,\vecv)$ 
for $(\vecq,\vecv)$ random in $\T^1(\R^d)$ with respect to an arbitrary
absolutely continuous probability measure.

Furthermore, we impose the following assumptions on the limiting distributions $\mu_{\vs}$:

\begin{enumerate}[{\bf [Q1]}]
\item {\em $\SO(d-1)$-invariance:} For every $\vs\in\Sigma$, 
\begin{equation*}
\text{$\mu_{\vs}$ is invariant under the action of $\SO(d-1):=\{k\in\SO(d)\col\vece_1k=\vece_1\}$.} \label{ASS:sodm1inv}
\end{equation*}
\item {\em Coincidence-free first coordinates:} For every $\vs\in\Sigma$, 
\begin{equation*}%
\mu_{\vs}(\{\nu\in N(\scrX)\col \exists x_1\in\R\text{ s.t.\ }\nu(\{x_1\}\times\R^{d-1}\times\Sigma)>1\})=0 . 
\end{equation*}
\item {\em Small probability of large voids:}  For every $\ve>0$ there exists $R>0$ such that for all $\vs\in\Sigma$ and $\vecx\in\RR^d$ we have 
\begin{equation*} %
\mu_{\vs}\bigl(\bigl\{\nu\in N(\scrX)\col \nu\bigl(\scrB^d(\vecx,R)\times\Sigma\bigr)=0\bigr\}\bigr)<\ve .
\end{equation*}
\end{enumerate}

Note that the assumption [Q1] is content-free for $d=2$.
For general $d$, %
[Q1] is equivalent to
requiring that the convergence in [P2] %
does not depend on our choice of the map 
$R:\US\to\SO(d)$;
cf.\ Remark \ref{ASS:sodm1invREM} below.
We will denote by $\Xi_{\vs}$ a point process in $\scrX$ with distribution $\mu_{\vs}$. 
\label{Xivsdef} $\Xi_{\vs}$ corresponds precisely to the family of random sets in the list of assumptions outlined in Section \ref{sec:outline}. 
Assumption [Q2] %
says that
almost surely, the points of $\Xi_{\vs}$ have pairwise distinct $\vece_1$-coordinates.
In particular $\Xi_\vs$ is simple as a marked point process,
viz., the ground process $p_{1*}(\Xi_\vs)$ in $\R^d$ is simple.
The assumption [Q3] %
says that the probability of $\Xi_{\vs}$ having empty intersection
with a large ball, at arbitrary position, is uniformly small.

\section{First consequences of the assumptions}
\label{FIRSTCONSsec}

We here derive some first %
consequences of the assumptions
imposed on the point set $\scrP$ in Section \ref{ASSUMPTLISTsec}.

The following is an immediate consequence of [P1].
\begin{lem}\label{ASYMPTDENSITY1CONSlem}
Let $f:\scrX\to\R$ be a bounded measurable function of compact support
whose set of discontinuities has measure zero with respect to $\mu_\scrX$.
Then
\begin{align}\label{ASYMPTDENSITY1CONSlemres}
\lim_{T\to\infty}T^{-d}\sum_{\vecy\in\tP}f(T^{-1}\vecy)=c_\scrP\int_{\scrX}f\,d\mu_\scrX.
\end{align}
\end{lem}
\begin{proof}
Take $R>0$ so that $\supp f\subset\scrX_R:=\scrB_R^d\times\Sigma$, where $\scrB_R^d$ \label{BRd} is the open ball of radius $R$ centered at the origin.
For $T>0$, let $N_T:=\#(\scrP\cap\scrB_{RT}^d)$;
then $T^{-d}N_T\to c_\scrP\mu_\scrX(\scrX_R)$ as $T\to\infty$.
For $T>0$ large (thus $N_T>0$) we let $X_T:=T^{-1}\vecy$ where $\vecy$ is chosen uniformly at random among the $N_T$
points in $\tP\cap\scrX_{RT}$; then $X_T$ is a random point in $\scrX$, and 
using $T^{-d}N_T\to c_\scrP\mu_\scrX(\scrX_R)$ and [P1] it follows that 
\begin{align*}
\lim_{T\to\infty}\PP(X_T\in B)=\frac{\mu_\scrX(B)}{\mu_\scrX(\scrX_R)}
\end{align*}
for any $B\subset\scrB_R^d\times\Sigma$ with $\mu_\scrX(\partial B)=0$.
Hence if we let $X_\infty$ be a random point in $\scrX_R$ with distribution
$\mu_\scrX(\scrX_R)^{-1}\mu_{\scrX\mid\scrX_R}$ then 
$X_T$ tends in distribution to $X_\infty$ as $T\to\infty$,
and so by the Portmanteau Theorem, $\lim_{T\to\infty}\EE f(X_T)=\EE f(X_\infty)$.
Again using $T^{-d}N_T\to c_\scrP\mu_\scrX(\scrX_R)$,
the last relation is seen to be equivalent with \eqref{ASYMPTDENSITY1CONSlemres}.
\end{proof}

{\blu From now on we \textit{fix, once and for all, a choice of a subset $\scrE\subset\scrP$
as in [P2].}
We also set}
\begin{align}\label{SIGMAp}
\Sigma'=\overline{\{{\vs}(\vecq)\col\vecq\in\scrP\setminus\scrE\}}.
\end{align}
This is clearly a closed subset of $\Sigma$, hence compact.
Once the basic results of the present section are established,
it will transpire that we may assume without loss of generality that $\Sigma'=\Sigma$;
cf.\ Remark \ref{SIGMAeqpREM} below.

\begin{lem}\label{SIGMApfullmeaslem}
$\mm(\Sigma')=1.$
\end{lem}
\begin{proof}
Assume the opposite, i.e.\ $\mm(U)>0$ where $U:=\Sigma\setminus\Sigma'$.
Let $d$ be the metric in $\Sigma$.
Then there is some open ball $B=B_\Sigma({\vs},r)=\{{\vs}'\in\Sigma\col d({\vs}',{\vs})<r\}$ 
satisfying $B\subset U$ and $\mm(B)>0$.
Take $0<r'<r$ so that also $\mm(B_\Sigma({\vs},r'))>0$.
Note that the boundaries $\partial B_\Sigma({\vs},r'')$ for $r''\in[r',r]$
are pairwise disjoint; hence there is some $r''\in(r',r)$
with $\mm(\partial B_\Sigma({\vs},r''))=0$.
Set $B':=B_\Sigma({\vs},r'')$;
thus $B'\subset U$, $\mm(B')>0$ and $\mm(\partial B')=0$.
Hence, by [P1], %
\begin{align*}
\lim_{R\to\infty}R^{-d}\#(\tP\cap(\scrB_R^d\times B'))
=c_\scrP\vol(\scrB_1^d)\mm(B')>0.
\end{align*}
Using also {\blu the fact that $\scrE$ has asymptotic density zero (cf.\ [P2])}, it follows that 
$\{\vecq\in\scrP\cap\scrB_R^d\setminus\scrE\col \vs(\vecq)\in B'\}$
must be nonempty for all sufficiently large $R$.
In particular $\Sigma'\cap B'\neq\emptyset$,
contradicting $B'\subset U=\Sigma\setminus\Sigma'$.
\end{proof}

\begin{lem}\label{DIAGINVlem}
For every ${\vs}\in\Sigma'$,
$\mu_{\vs}$ is invariant under the action of $D_r$, for all $r>0$.
\end{lem}
\begin{proof}
Fix $r>0$.
Take any sequence $(\vecq_n)\subset\scrP\setminus\scrE$
such that ${\vs}(\vecq_n)\to{\vs}$;
then take $(\rho_n)\subset(0,1)$ such that
$\rho_n\to0$ and $\vecq_n\in\scrP_1(\rho_n)$ for each $n$.
Fix any $\lambda\in \Pac(\US)$.
Then $\mu_{\vecq_n,\rho_n}^{(\lambda)}\xrightarrow[]{\textup{ w }}\mu_{{\vs}}$
as $n\to\infty$, by [P2]. %
We also have $\vecq_n\in\scrP_T(r\rho_n)$ for each $n$,
where $T=\max(1,r^{d-1})$, and hence
$\mu_{\vecq_n,r\rho_n}^{(\lambda)}\xrightarrow[]{\textup{ w }}\mu_{{\vs}}$
as $n\to\infty$.
But note that 
$\scrQ_{r\rho}(\vecq,\vecv)=\scrQ_\rho(\vecq,\vecv)D_r$, 
and hence
$\mu_{\vecq_n,r\rho_n}^{(\lambda)}=\mu_{\vecq_n,\rho_n}^{(\lambda)}\circ D_r^{-1}$,
where ``$D_r^{-1}$'' denotes the continuous map
$Y\mapsto YD_r^{-1}$ from $N_s(\scrX)$ to itself.
Hence $\mu_{\vecq_n,r\rho_n}^{(\lambda)}\xrightarrow[]{\textup{ w }}\mu_{{\vs}}\circ D_r^{-1}$ as $n\to\infty$,
and so $\mu_{{\vs}}\circ D_r^{-1}=\mu_{\vs}$. %
\end{proof}

Next we will show that since our key convergence assumption,
[P2],
is required to hold for \textit{all} $\lambda\in \Pac(\US)$,
it can in fact be upgraded to a convergence statement concerning the joint distribution of
$\vecv$ and $\scrQ_\rho(\vecq,\vecv)$.
Given $\lambda\in P(\US)$, we write $\tmu_{\vecq,\rho}^{(\lambda)}\in P(\US\times N(\scrX))$ for
the distribution of $(\vecv,\scrQ_\rho(\vecq,\vecv))$ %
for $\vecv$ random in $(\US,\lambda)$.
\label{tmuqrholambdaDEF}
\begin{lem}\label{PRODUNIFCONVlem}
Let $T\geq1$ and $\lambda\in \Pac(\US)$.
Then $\tmu_{\vecq,\rho}^{(\lambda)}\xrightarrow[]{\textup{ w }}\lambda\times\mu_{{\vs}(\vecq)}$
as $\rho\to0$, uniformly over all $\vecq\in\scrP_T(\rho)$.
\end{lem}
\begin{proof}
Let $f\in \C_c(N(\scrX))$ and $g\in\C_c(\US\times\R)$ be given.
By Lemma \ref{PPUNICCONVCONCRETElemA} it suffices to prove that 
$F_{f,g}(\tmu_{\vecq,\rho}^{(\lambda)})-F_{f,g}(\lambda\times\mu_{{\vs}(\vecq)})\to0$
as $\rho\to0$, uniformly over all $\vecq\in\scrP_T(\rho)$.
Let $\ve>0$ be given.
Since $g$ is continuous with compact support,
there is a partition of $\US$ into Borel subsets $S_1,\ldots,S_r$
such that for each $j\in\{1,\ldots,r\}$ we have
\begin{align}\label{PRODUNIFCONVlempf1}
|g(\vecv,y)-g(\vecv',y)|<\ve,
\qquad\forall y\in\R,\: \vecv,\vecv'\in S_j.
\end{align}
Let $J$ be the set of $j\in\{1,\ldots,r\}$ with $\lambda(S_j)>0$.
For each $j\in J$ we set $\lambda_j=\lambda(S_j)^{-1}\lambda_{|S_j}\in P(\US)$,
fix a point $\vecv_j\in S_j$,
and define $g_j\in\C_c(\R)$ through
$g_j(y)=g(\vecv_j,y)$.
Applying [P2] %
for $\lambda_j$,
together with Lemma \ref{PPUNICCONVCONCRETElemA} (with $S$ as a ``dummy'' singleton set),
we see that
$F_{f,g_j}(\mu_{\vecq,\rho}^{(\lambda_j)})-F_{f,g_j}(\mu_{{\vs}(\vecq)})\to0$
as $\rho\to0$, uniformly over all $\vecq\in\scrP_T(\rho)$.
(Here $F_{f,g_j}(\mu)=\int_{N(\scrX)}g_j(\eta(f))\,d\mu(\eta)$ for any $\mu$ in $P(N(\scrX))$.)
Hence there is some $\rho_0\in(0,1)$ such that
\begin{align}\label{PRODUNIFCONVlempf2}
\bigl|F_{f,g_j}(\mu_{\vecq,\rho}^{(\lambda_j)})-F_{f,g_j}(\mu_{{\vs}(\vecq)})\bigr|<\ve,
\qquad\forall\rho\leq\rho_0,\:\vecq\in\scrP_T(\rho), j\in J.
\end{align}
Now note that by definition,
\begin{align*}
F_{f,g_j}(\mu_{\vecq,\rho}^{(\lambda_j)})
=%
\int_{S_j}g\biggl(\vecv_j,\sum_{\vecy\in\scrQ_\rho(\vecq,\vecv)}f(\vecy)\biggr)\,d\lambda_j(\vecv)
\end{align*}
and
\begin{align*}
F_{f,g}(\tmu_{\vecq,\rho}^{(\lambda_j)})
=\int_{S_j}g\biggl(\vecv,\sum_{\vecy\in\scrQ_\rho(\vecq,\vecv)}f(\vecy)\biggr)\,d\lambda_j(\vecv);
\end{align*}
hence using \eqref{PRODUNIFCONVlempf1} we have
\begin{align*}
\bigl|F_{f,g_j}(\mu_{\vecq,\rho}^{(\lambda_j)})
-F_{f,g}(\tmu_{\vecq,\rho}^{(\lambda_j)})\bigr|\leq\ve.
\end{align*}
Multiplying this inequality by $\lambda(S_j)$ and adding over all $j$ we obtain
\begin{align}\label{PRODUNIFCONVlempf3}
\Bigl|F_{f,g}(\tmu_{\vecq,\rho}^{(\lambda)})
-\sum_{j\in J}\lambda(S_j)F_{f,g_j}(\mu_{\vecq,\rho}^{(\lambda_j)})\Bigr|\leq\ve,
\qquad\forall\rho\in(0,1),\:\vecq\in\scrP.
\end{align}
By a similar argument we also have
\begin{align}\label{PRODUNIFCONVlempf4}
\Bigl|F_{f,g}(\lambda\times\mu_{\vs})-\sum_{j\in J}\lambda(S_j)F_{f,g_j}(\mu_{\vs})\Bigr|\leq\ve,
\qquad\forall\rho\in(0,1),\:{\vs}\in\Sigma.
\end{align}
Using 
\eqref{PRODUNIFCONVlempf2},
\eqref{PRODUNIFCONVlempf3} and
\eqref{PRODUNIFCONVlempf4},
we conclude that
\begin{align*}
\bigl|F_{f,g}(\tmu_{\vecq,\rho}^{(\lambda)})-
F_{f,g}(\lambda\times\mu_{{\vs}(\vecq)})\bigr|\leq3\ve,
\qquad\forall\rho\leq\rho_0,\:\vecq\in\scrP_T(\rho).
\end{align*}
Since $\ve>0$ is arbitrary, this establishes the desired uniform convergence.
\end{proof}

\begin{remark}\label{PRODUNIFCONVspecrem}
Let us take $\lambda$ to be normalized Lebesgue measure,
i.e. $\lambda=\omega_1=\omega(\US)^{-1}\omega$.
In this case, Lemma \ref{PRODUNIFCONVlem} says
\begin{align}\label{PRODUNIFCONVspec}
\tmu_{\vecq,\rho}^{(\omega_1)}\xrightarrow[]{\textup{ w }}\omega_1\times\mu_{{\vs}(\vecq)}
\qquad
\text{as $\rho\to0$, uniformly over all $\vecq\in\scrP_T(\rho)$.}
\end{align}
Let us note that the convergence stated in Lemma \ref{PRODUNIFCONVlem}
for a \textit{general} $\lambda\in \Pac(\US)$,
is in fact a simple consequence of the special case \eqref{PRODUNIFCONVspec}: it follows from the fact that $\C(\US)$ is dense in $\L^1(\US,\omega)$. Of course also the convergence in Lemma \ref{PRODUNIFCONVlem}
implies the convergence assumed in [P2]. %
Hence \eqref{PRODUNIFCONVspec} is an \textit{equivalent reformulation} of the assumption [P2]. %
\end{remark}

{\blu For later use,} we will next prove that
the convergence in
Lemma \ref{PRODUNIFCONVlem}
can be upgraded
by including a ``$\vecbeta$-shift''.
For any open subset $U\subset\US$,
we let $\C_b(U,\R^d)$ be the space of all bounded continuous functions $\vecbeta:U\to\RR^{d}$,
provided with the supremum norm.
\label{CbURddef}
For any $\vecq\in\scrP$, $\vecbeta\in\C_b(U,\R^d)$, $\vecv\in U$ and $0<\rho<1$, we set
\begin{align}\label{XIrhoqvdef}
\scrQ_\rho(\vecq,\vecbeta,\vecv)=(\tP_\vecq-\vecq-\rho\vecbeta(\vecv))\,R(\vecv)\,D_\rho.
\end{align}
{\blu Thus $\scrQ_\rho(\vecq,\vecbeta,\vecv)$ gives the scattering configuration when viewed in the particle
frame of a particle at the point $\vecq+\rho\vecbeta(\vecv)$ in direction $\vecv$
(cf.\ the discussion at the beginning of Section~\ref{sec:outline}).
We will ultimately be interested in the case when 
$\vecbeta(\vecv)\in\US$ and
$\vecbeta(\vecv)\cdot\vecv>0$
for all $\vecv$.
In this case, considering a fixed $\vecq\in\scrP$
and random $\vecv$ means that we are considering a particle
which is just about to leave the boundary of the scatterer
$\vecq+\scrB_\rho^d$ in a random direction, 
with $\vecbeta$ specifying how the exact starting position 
depends on the random direction.}
%
%
%

Given $\lambda\in P(\S_1^{d-1})$ with $\lambda(U)=1$, let us write
$\tmu_{\vecq,\rho}^{(\vecbeta,\lambda)}\in P(\US\times N(\scrX))$ for the distribution of 
$(\vecv,\scrQ_\rho(\vecq,\vecbeta,\vecv))$ for $\vecv$ random in $(\US,\lambda)$
(equivalently, for $\vecv$ random in $(U,\lambda_{|U})$).
\label{tmuqrhobetalambdaDEF}
We define the projection map $\vecx\mapsto\vecx_\perp$ on $\R^d$ by
\begin{align}\label{Xperpdef1}
\vecx_\perp:=(0,x_2,\ldots,x_d)
\qquad\text{for }\: \vecx=(x_1,x_2,\ldots,x_d)\in\R^d.
\end{align}
The limit of $\scrQ_\rho(\vecq,\vecbeta,\vecv)$ as $\rho\to0$
will turn out to be the point process 
$\Xi_{\vs}-(\vecbeta(\vecv)R(\vecv))_\perp$, 
with $\vecv$ random in $(\US,\lambda)$ and independent from $\Xi_{\vs}$.
Let $\mu_{{\vs}}^{(\vecbeta,\lambda)}\in P(N(\scrX))$ be the distribution of this point process,
\label{muvsbetalambdaDEF}
and let $\tmu_{{\vs}}^{(\vecbeta,\lambda)}\in P(\US\times N(\scrX))$
be the distribution of $(\vecv,\Xi_{\vs}-(\vecbeta(\vecv)R(\vecv))_\perp)$.
Thus for any measurable function $f\geq0$ on $\US\times N(\scrX)$,
\begin{align}\label{MUsigmabetalambdadef}
\int_{\US\times N(\scrX)}f\,d\tmu_{{\vs}}^{(\vecbeta,\lambda)}
=\int_U\int_{N_s(\scrX)}f\bigl(\vecv,{Y}-(\vecbeta(\vecv)R(\vecv))_\perp\bigr)\,d\mu_{\vs}({Y})\,d\lambda(\vecv),
\end{align}
and in case $f(\vecv,Y)$ is independent of $\vecv$ this also equals $\int_{N(\scrX)}f\,d\mu_{{\vs}}^{(\vecbeta,\lambda)}$.
In particular $\Xi_{\vs}-(\vecbeta(\vecv)R(\vecv))_\perp$ is a simple (marked) point process,
just as $\Xi_{\vs}$.
It will be useful for us to prove a limit statement which is uniform both over $\vecbeta$ in compacta
and over $\vecq$ in $\scrP_T(\rho)$.
\begin{lem}\label{BETAUNIFCONVlem}
Let $U$ be an open subset of $\US$,
and let $\lambda\in\Pac(\US)$ be such that $\lambda(U)=1$.
Then for any $T\geq1$ and any compact subset $K\subset\C_b(U,\R^d)$,
we have $\tmu_{\vecq,\rho}^{(\vecbeta,\lambda)}\xrightarrow[]{\textup{ w }}\tmu_{{\vs}(\vecq)}^{(\vecbeta,\lambda)}$
as $\rho\to0$, uniformly over all $\vecq\in\scrP_T(\rho)$ and all $\vecbeta\in K$.
\end{lem}
\begin{remark}\label{BETAUNIFCONVlemrem1}
The uniform convergence in the lemma takes place in $P(\US\times N(\scrX))$, recall \eqref{UNIFCONVdef1}.
Note that $\{\tmu_{\vs}^{(\vecbeta,\lambda)}\col{\vs}\in\Sigma,\vecbeta\in K\}$ is a compact subset of $P(\US\times N(\scrX))$,
since it is the continuous image of the compact set $\Sigma\times K$; cf.\ footnote \ref{foot6}.
\end{remark}
\begin{proof}
Take $\vecv_0\in\US$ so that $R$ is continuous on $\US\setminus\{\vecv_0\}$.
Note that $\lambda(U\setminus\{\vecv_0\})=\lambda(U)=1$, since $\lambda$ is absolutely continuous
with respect to $\omega$;
thus we may replace $U$ by $U\setminus\{\vecv_0\}$ without affecting the content of the 
statement of the lemma.
Hence from now on we may assume that $R$ is continuous on $U$.
Let $\rho_n\in(0,1)$, $\vecq_n\in\scrP_T(\rho_n)$, $\vecbeta_n\in\C_b(U,\R^d)$
for $n=1,2,\ldots$,
and assume that $\rho_n\to0$,
${\vs}(\vecq_n)\to{\vs}$ and $\vecbeta_n\to\vecbeta$
as $n\to\infty$, with ${\vs}\in\Sigma$ and $\vecbeta\in K$.
We then claim that 
$\tmu_{\vecq_n,\rho_n}^{(\vecbeta_n,\lambda)}\xrightarrow[]{\textup{ w }}\tmu_{{\vs}}^{(\vecbeta,\lambda)}$
as $n\to\infty$.
By the same argument as in Lemma~\ref{UNIFPORTMlemC}, this will imply the lemma.

We extend $\vecbeta$ and each $\vecbeta_n$ to all $\US$ by setting
$\vecbeta(\vecv)=\vecbeta_n(\vecv)=\bn$ (say) for all $\vecv\in\US\setminus U$.
Consider the maps
\begin{align*}
F_n:\US\times N_s(\scrX)\to\US\times N_s(\scrX),
\qquad 
F_n(\vecv,Y)=(\vecv,Y-\rho_n\vecbeta_n(\vecv)R(\vecv)D_{\rho_n})
\end{align*}
and
\begin{align*}
F:\US\times N_s(\scrX)\to\US\times N_s(\scrX),
\qquad
F(\vecv,Y)=(\vecv,Y-(\vecbeta(\vecv)R(\vecv))_\perp).
\end{align*}
Using $\lambda(U)=1$ we have
$\tmu_{\vecq_n,\rho_n}^{(\vecbeta_n,\lambda)}=\tmu_{\vecq_n,\rho_n}^{(\lambda)}\circ F_{n}^{-1}$
and $\tmu_{\vs}^{(\vecbeta,\lambda)}=(\lambda\times\mu_{\vs})\circ F^{-1}$.
Now for any points $(\vecv,Y),(\vecv_1,Y_1),(\vecv_2,Y_2),\ldots\in\US\times N_s(\scrX)$
subject to $(\vecv_n,Y_n)\to(\vecv,Y)$ and $\vecv\in U$,
we have $F_n(\vecv_n,Y_n)\to F(\vecv,Y)$ as $n\to\infty$.
Furthermore
$\tmu_{\vecq_n,\rho_n}^{(\lambda)}\xrightarrow[]{\textup{ w }}\lambda\times\mu_{\vs}$, %
by Lemma \ref{PRODUNIFCONVlem}.
Hence %
$\tmu_{\vecq_n,\rho_n}^{(\lambda)}\circ F_n^{-1}
\xrightarrow[]{\textup{ w }}(\lambda\times\mu_{\vs})\circ F^{-1}$ as $n\to\infty$
(cf.\ \cite[Thm.\ 4.27]{kallenberg02}),
and the lemma is proved.
\end{proof}

{\blu The following lemma shows that, in the presence of [Q1], the assumption [P2] is independent of the choice of the map $R$.}

\begin{lem}\label{RINDEPlem}
Suppose that $\hR:\US\to\SO(d)$ is any map satisfying the same conditions as $R$,
i.e.\ $\vecv \hR(\vecv)=\vece_1$ for all $\vecv\in\US$,
and $\hR$ is continuous when restricted to $\US$ minus one point.
Define $\hXi_\rho(\vecq,\vecv)=(\tP_\vecq-\vecq)\,\hR(\vecv)\,D_\rho$,
and write $\hmu_{\vecq,\rho}^{(\lambda)}\in P(N(\scrX))$ for the distribution of 
$\hXi_\rho(\vecq,\vecv)$ for $\vecv$ random in $(\US,\lambda)$.
Then for any fixed $T\geq1$ and $\lambda\in\Pac(\US)$,
we have $\hmu_{\vecq,\rho}^{(\lambda)}\xrightarrow[]{\textup{ w }}\mu_{{\vs}(\vecq)}$
as $\rho\to0$, uniformly over all $\vecq\in\scrP_T(\rho)$.
\end{lem}
\begin{proof}
Define $K:\US\to\SO(d)$ by
$K(\vecv)=R(\vecv)^{-1}\hR(\vecv)$.
Then $\vece_1 K(\vecv)=\vece_1$, i.e.\ 
$K(\vecv)\in\SO(d-1)$ for each $\vecv\in\US$.
It follows that $K(\vecv)$ commutes with $D_\rho$, and so
\begin{align*}
\hXi_\rho(\vecq,\vecv)=\scrQ_\rho(\vecq,\vecv)\,K(\vecv).
\end{align*}
We introduce the map
\begin{align*}
F:\US\times N_s(\scrX)\to N_s(\scrX),
\qquad F(\vecv,Y)=YK(\vecv),
\end{align*}
and note that $\hmu_{\vecq,\rho}^{(\lambda)}=\tmu_{\vecq,\rho}^{(\lambda)}\circ F^{-1}$.

Let $\rho_n\in(0,1)$ and $\vecq_n\in\scrP_T(\rho_n)$ for $n=1,2,\ldots$,
and assume that $\rho_n\to0$ and ${\vs}(\vecq_n)\to{\vs}\in\Sigma$ as $n\to\infty$.
We then claim that 
$\hmu_{\vecq_n,\rho_n}^{(\lambda)}\xrightarrow[]{\textup{ w }}\mu_{{\vs}}$
as $n\to\infty$.
This will complete the proof of the lemma, by the same argument as in Lemma \ref{UNIFPORTMlemC}.
We have $\tmu_{\vecq_n,\rho_n}^{(\lambda)}\xrightarrow[]{\textup{ w }}\lambda\times\mu_{\vs}$, %
by Lemma \ref{PRODUNIFCONVlem}.
By the assumptions on $R$ and $\hR$,
there exist points $\vecv_0,\vecv_0'\in\US$ such that
$K$ is continuous on $\US\setminus\{\vecv_0,\vecv_0'\}$.
Now for any sequence of points $(\vecv,Y),(\vecv_1,Y_1),(\vecv_2,Y_2),\ldots\in\US\times N_s(\scrX)$
subject to $(\vecv_n,Y_n)\to(\vecv,Y)$ and $\vecv\notin\{\vecv_0,\vecv_0'\}$,
we have $F(\vecv_n,Y_n)\to F(\vecv,Y)$ as $n\to\infty$.
Hence by \cite[Thm.\ 4.27]{kallenberg02},
$\hmu_{\vecq,\rho}^{(\lambda)}=\tmu_{\vecq,\rho}^{(\lambda)}\circ F^{-1}
\xrightarrow[]{\textup{ w }}(\lambda\times\mu_{\vs})\circ F^{-1}$
as $n\to\infty$.
Finally $(\lambda\times\mu_{\vs})\circ F^{-1}=\mu_{\vs}$ by [Q1], %
and the proof is complete.
\end{proof}

\begin{remark}\label{ASS:sodm1invREM}
Note that Lemma \ref{RINDEPlem} was proved using only the key convergence assumption
[P2] %
together with the assumption that each $\Xi_{\vs}$ is $\SO(d-1)$-invariant,
[Q1]. %
Conversely it is easy to see that \textit{assuming} Lemma \ref{RINDEPlem},
the $\SO(d-1)$-invariance [Q1] %
follows as a consequence, for each ${\vs}\in\Sigma'$.
\end{remark}

The next lemma is a simple variant of the previous one.
\begin{lem}\label{RINDEPlem2}
Let $T\geq1$ and let $\nu$ be the (left and right) Haar measure on $\SO(d)$, normalized to be a probability measure.
Let $\mu_{\vecq,\rho}\in P(N_s(\scrX))$ be the distribution of $(\tP_\vecq-\vecq)KD_\rho$ 
for $K$ random in $(\SO(d),\nu)$.
Then  $\mu_{\vecq,\rho}\xrightarrow[]{\textup{ w }}\mu_{{\vs}(\vecq)}$
as $\rho\to0$, uniformly over all $\vecq\in\scrP_T(\rho)$.
\end{lem}
\begin{proof}
Let $\rho_n\in(0,1)$ and $\vecq_n\in\scrP_T(\rho_n)$ for $n=1,2,\ldots$,
and assume that $\rho_n\to0$ and ${\vs}(\vecq_n)\to{\vs}\in\Sigma$ as $n\to\infty$.
We then claim that 
$\mu_{\vecq_n,\rho_n}\xrightarrow[]{\textup{ w }}\mu_{{\vs}}$
as $n\to\infty$;
as usual this will complete the proof of the lemma.
Let $f\in\C_b(N_s(\scrX))$.
For any $K\in\SO(d-1)$ we have
\begin{align}\label{RINDEPlem2PF1}
\int_{\US}f((\tP_{\vecq_n}-\vecq_n)R(\vecv)KD_{\rho_n})\,d\omega_1(\vecv)\to\mu_{\vs}(f),
\qquad\text{as }\: n\to\infty,
\end{align}
by Lemma \ref{RINDEPlem}
(or directly from [P2] and [Q1], %
using $KD_\rho=D_\rho K$).
Let $\nu_1$ be the normalized Haar measure on $\SO(d-1)$.
By Lebesgue's Dominated Convergence Theorem,
it follows from \eqref{RINDEPlem2PF1} that
\begin{align*}
\int_{\SO(d-1)}\int_{\US}f((\tP_{\vecq_n}-\vecq_n)R(\vecv)KD_{\rho_n})\,d\omega_1(\vecv)\,d\nu_1(K)\to\mu_{\vs}(f),
\qquad\text{as }\: n\to\infty.
\end{align*}
However here the left hand side equals $\mu_{\vecq_n,\rho_n}(f)$,
since the push-forward of the measure $\omega_1\times\nu_1$ on $\US\times\SO(d-1)$ under the map
$\langle\vecv,K\rangle\mapsto R(\vecv)K$ equals $\nu$
(cf., e.g., \cite[Thm.\ 8.36]{knapp02}).
Hence the proof is complete.
\end{proof}

A symmetry related to the $\SO(d-1)$ invariance is the following.
\begin{lem}\label{ODM1NDlem}
Fix $K\in\SO(d)$ with $\vece_1 K=-\vece_1$.
Then for each ${\vs}\in\Sigma'$, $\mu_{\vs}$ is $K$-invariant.
\end{lem}
\begin{proof}
Let $\vs\in\Sigma'$ be given.
As in the proof of Lemma \ref{DIAGINVlem},
take sequences $(\vecq_n)\subset\scrP\setminus\scrE$ and $(\rho_n)\subset(0,1)$ so that
$\vecq_n\in\scrP_1(\rho_n)$ for each $n$,
and $\rho_n\to0$ and $\vs(\vecq_n)\to\vs$ as $n\to\infty$.
Fix any $\lambda\in\Pac(\US)$.
Define $\hlambda\in P(\US)$ 
by $\hlambda(B)=\lambda(-B)$ for any Borel set $B\subset\US$,
and define $\hR:\US\to\SO(d)$
through $\hR(\vecv)=R(-\vecv)K$.
Also let $\hmu_{\vecq,\rho}^{(\hlambda)}$ be the distribution of $(\tP_{\vecq}-\vecq)\hR(\vecv)D_\rho$
for $\vecv$ random in $(\US,\hlambda)$,
as in Lemma \ref{RINDEPlem}.
Now for each $n$, and for any Borel subset $A\subset N_s(\scrX)$,
we have, with $\vecq=\vecq_n$, $\rho=\rho_n$,
\begin{align*}
\hmu_{\vecq,\rho}^{(\hlambda)}(A)
&=\hlambda(\{\vecv\in\US\col(\tP_{\vecq}-\vecq)\hR(\vecv)D_\rho\in A\})
\\
&=\lambda(\{\vecv\in\US\col(\tP_{\vecq}-\vecq)R(\vecv)KD_\rho\in A\})
\\
&=\lambda(\{\vecv\in\US\col(\tP_{\vecq}-\vecq)R(\vecv)D_\rho\in AK^{-1}\})
=\mu_{\vecq,\rho}^{(\lambda)}(AK^{-1}).
\end{align*}
Hence $\hmu_{\vecq_n,\rho_n}^{(\hlambda)}=\mu_{\vecq_n,\rho_n}^{(\lambda)}\circ K^{-1}$ for each $n$,
where ``$K^{-1}$'' denotes the map $N_s(\scrX)\to N_s(\scrX)$, $A\mapsto AK^{-1}$.
We have $\mu_{\vecq_n,\rho_n}^{(\lambda)}\xrightarrow[]{\textup{ w }}\mu_{{\vs}}$, by [P2], %
and thus $\hmu_{\vecq_n,\rho_n}^{(\hlambda)}\xrightarrow[]{\textup{ w }}\mu_{{\vs}}\circ K^{-1}$.
On the other hand,
$\hmu_{\vecq_n,\rho_n}^{(\hlambda)}\xrightarrow[]{\textup{ w }}\mu_{{\vs}}$, by Lemma \ref{RINDEPlem}.
Hence $\mu_{{\vs}}\circ K^{-1}=\mu_{\vs}$.
\end{proof}

\vspace{10pt}

Next we prove that the intensity measure of $\Xi_{\vs}$ is bounded above by $c_\scrP\mu_\scrX$.
\begin{lem}\label{GROUNDINTENSITYlem2}
For any ${\vs}\in\Sigma'$ and any Borel set $B\subset\scrX$,
$\int_{N_s(\scrX)}\#(Y\cap B)\,d\mu_{\vs}(Y)\leq c_\scrP\mu_\scrX(B)$.
\end{lem}
\begin{remark}
The proof of Lemma \ref{GROUNDINTENSITYlem2} should be compared with the first half of the proof of
the Siegel-Veech formula in
\cite[Thm.\ 5.1]{qc}.
Cf.\ also Veech, \cite{veech}.
\end{remark}

\begin{proof}
Let ${\vs}\in\Sigma'$ be given.
It suffices to prove that for any given $f\in C_c(\scrX)$, $f\geq0$, 
we have
\begin{align}\label{GROUNDINTENSITYlem2pf1a}
\int_{N_s(\scrX)}\sum_{\vecq\in Y}f(\vecq)\,d\mu_{\vs}(Y)\leq c_\scrP\int_{\scrX} f\,d\mu_\scrX.
\end{align}
Take sequences $(\vecq_j)\subset\scrP\setminus\scrE$ and $(\rho_j)\subset(0,1)$ so that
$\vecq_j\in\scrP_1(\rho_j)$ for each $j$,
and $\rho_j\to0$ and $\vs(\vecq_j)\to\vs$ as $j\to\infty$.
By further shrinking each $\rho_j$ if necessary, we may also assume that
\begin{align}\label{GROUNDINTENSITYlem2PF3}
\rho_j^{\frac 32-d}>
\|\vecq_j\|+\sum_{\substack{\vecq\in\scrP\setminus\{\vecq_j\}\\\|\vecq\|\leq2\|\vecq_j\|}}\|\vecq-\vecq_j\|^{1-d}.
\end{align}
By Lemma \ref{RINDEPlem2} we have 
$\mu_{\vecq_j,\rho_j}\xrightarrow[]{\textup{ w }}\mu_{{\vs}}$ as $j\to\infty$,
and hence since 
$F:Y\mapsto\sum_{\vecq\in Y}f(\vecq)$ is a
nonnegative continuous %
function on $N_s(\scrX)$
(unbounded if $f\not\equiv0$),
$\mu_{{\vs}}(F)\leq\liminf_{j\to\infty}\mu_{\vecq_j,\rho_j}(F)$.
In other words, writing $\nu$ for the normalized Haar measure on $\SO(d)$,
\begin{align}\notag
\int_{N_s(\scrX)}\sum_{\vecy\in Y}f(\vecy)\,d\mu_{\vs}(Y)
&\leq
\liminf_{j\to\infty}\int_{\SO(d)}\sum_{\vecy\in(\tP_{\vecq_j}-\vecq_j)KD_{\rho_j}}f(\vecy)\,d\nu(K)
\\\notag
&=\liminf_{j\to\infty}\sum_{\vecy\in\tP_{\vecq_j}}\int_{\SO(d)}
f((\vecy-\vecq_j) KD_{\rho_j})\,d\nu(K)
\\\label{GROUNDINTENSITYlem2PF1}
&=\liminf_{j\to\infty}\rho_j^{d(d-1)}\sum_{\vecq\in\scrP\setminus\{\vecq_j\}}h_{\rho_j}(\rho_j^{d-1}\|\vecq-\vecq_j\|,{\vs}(\vecq)),
\end{align}
where $h_\rho\in\C_c(\R_{\geq0}\times\scrX)$ is given by
\begin{align}\label{GROUNDINTENSITYlem2PF4a}
h_\rho(r,{\vs}):=\rho^{-d(d-1)}\int_{\SO(d)}f(\rho^{1-d}r\vece_1 KD_\rho,{\vs})\,d\nu(K).
\end{align}

Set $B=\sup_{\scrX}f$ and take $R>0$ so that $\supp f\subset\overline{\scrB_R^d}\times\Sigma$.
Then for each $\rho\in(0,1)$, the support of $h_\rho$ is contained in $[0,R]\times\Sigma$.
Writing $\vecx=\vece_1K$ in \eqref{GROUNDINTENSITYlem2PF4a} we have
\begin{align}\notag
h_\rho(r,{\vs})
&=\rho^{-d(d-1)}\int_{\US}f(r(x_1,\rho^{-d}x_2,\ldots,\rho^{-d}x_d),{\vs})\,d\omega_1(\vecx).
\end{align}
It follows %
that %
\begin{align}\label{hrhounifconv}
h_\rho(r,{\vs})\to h_0(r,{\vs}):=\frac1{\omega(\US)}\sum_{s\in\{1,-1\}}\int_{\R^{d-1}}f(sr(1,\vecy),{\vs})\,d\vecy,
\qquad\text{as }\:\rho\to0,
\end{align}
with uniform convergence over all
$(r,{\vs})\in[\eta,R]\times\Sigma$,
for any fixed $\eta\in(0,R)$.
Note that $h_0$ by definition is a continuous function on $\R_{>0}\times\Sigma$ with support contained in
$(0,R]\times\Sigma$.
Note also that $\rho_j^{d-1}\vecq_j\to\bn$ as $j\to\infty$; cf.\ \eqref{GROUNDINTENSITYlem2PF3}.
Using [P1], Lemma \ref{ASYMPTDENSITY1CONSlem}, and the uniform convergence just pointed out,
it now follows that for any $\eta\in(0,R)$,
\begin{align*}
\lim_{j\to\infty}&\rho_j^{d(d-1)}\sum_{\vecq\in\scrP\setminus\{\vecq_j\}}
I(\|\rho_j^{d-1}\vecq\|\geq\eta)\, h_{\rho_j}(\rho_j^{d-1}\|\vecq-\vecq_j\|,{\vs}(\vecq))
\\
&=\lim_{j\to\infty}\rho_j^{d(d-1)}\sum_{\vecq\in\scrP\setminus\{\vecq_j\}}
I(\|\rho_j^{d-1}\vecq\|\geq\eta)\, h_0(\rho_j^{d-1}\|\vecq-\vecq_j\|,{\vs}(\vecq))
\\
&=\lim_{j\to\infty}\rho_j^{d(d-1)}\sum_{\vecq\in\scrP}
I(\|\rho_j^{d-1}\vecq\|\geq\eta)\, h_0(\rho_j^{d-1}\|\vecq\|,{\vs}(\vecq))
\\
&=c_\scrP\int_\scrX I(\|\vecx\|\geq\eta)\,h_0(\|\vecx\|,{\vs})\,d\mu_\scrX(\vecx,{\vs})
\\
&=c_\scrP\,\omega(\US)\int_\Sigma\int_\eta^\infty h_0(r,{\vs})\,r^{d-1}\,dr\,d\mm({\vs})
=c_\scrP\int_{\scrX\setminus((-\eta,\eta)\times\R^{d-1}\times\Sigma)} f\,d\mu_\scrX.
\end{align*}
To handle the contribution from $r\in[0,\eta]$, note that
\begin{align*}
h_\rho(r,{\vs}) & \leq B\rho^{-d(d-1)}\int_{\US}I\bigl(r\|(x_1,\rho^{-d}x_2,\ldots,\rho^{-d}x_d)\|\leq R\bigr)\,d\omega_1(\vecx) \\
& \leq CB\Bigl(\frac Rr\Bigr)^{d-1},
\end{align*}
where $C>0$ only depends on $d$.
(This bound is accurate for $(R/r)\rho^d$ small, otherwise wasteful.)
It follows that
\begin{align*}
&\limsup_{j\to\infty}\rho_j^{d(d-1)}\sum_{\vecq\in\scrP\setminus\{\vecq_j\}}
I(\|\rho_j^{d-1}\vecq\|<\eta)\, h_{\rho_j}(\rho_j^{d-1}\|\vecq-\vecq_j\|,{\vs}(\vecq))
\\
&\leq\limsup_{j\to\infty}
CBR^{d-1}\rho_j^{d-1}\sum_{\substack{\vecq\in\scrP\setminus\{\vecq_j\}\\\|\vecq\|<\eta\rho_j^{1-d}}}\|\vecq-\vecq_j\|^{1-d}
\\
&\leq\limsup_{j\to\infty}
2^{d-1}CBR^{d-1}\rho_j^{d-1}\hspace{-10pt}\sum_{\substack{\vecq\in\scrP\\2\|\vecq_j\|<\|\vecq\|<\eta\rho_j^{1-d}}}\|\vecq\|^{1-d}
\leq C'\eta,
\end{align*}
where $C'$ is a constant which only depends on $\scrP,B,R$.
Here we used \eqref{GROUNDINTENSITYlem2PF3} in the second inequality,
and for the third inequality we used [P1] and a dyadic decomposition of the relevant annulus.
Adding the two bounds it follows that the right hand side of 
\eqref{GROUNDINTENSITYlem2PF1} is bounded above by
$c_\scrP\int_\scrX f\,d\mu_\scrX+C'\eta$.
This is true for each $\eta\in(0,R)$;
hence we conclude that \eqref{GROUNDINTENSITYlem2pf1a} holds, and the lemma is proved.
\end{proof}

{\blu Next we will prove a lemma saying that for $\rho$ small,
each point in $\scrP_T(\rho)$ is well separated from all other points in $\scrP$,
when measured in the length scale of $\rho$.} 
For $\vecq\in\R^d$ we set
\begin{align}\label{dscrPdef}
d_\scrP(\vecq)=\inf\{\|\vecp-\vecq\|\col\vecp\in\scrP\setminus\{\vecq\}\}.
\end{align}
\begin{lem}\label{GOODDISTANCElem}
For any $T\geq1$,
the quantity $\inf_{\vecq\in\scrP_T(\rho)}d_\scrP(\vecq)/\rho$ tends to $\infty$ as $\rho\to0$.
\end{lem}
\begin{proof}
Assume the opposite; then there exist a constant $C>0$ and 
sequences $(\rho_n)$ and $(\vecq_n)$
such that $\rho_n\to0$, $\vecq_n\in\scrP_T(\rho_n)$, and $d_\scrP(\vecq_n)<C\rho_n$ for all $n$.
Passing to a subsequence we may also assume ${\vs}(\vecq_n)\to{\vs}\in\Sigma'$.
Fix any $\lambda\in\Pac(\US)$. 
Then $\mu^{(\lambda)}_{\rho_n,\vecq_n}\xrightarrow[]{\textup{ w }}\mu_{{\vs}}$ as $n\to\infty$,
by [P2]. %
For $\ve>0$ set $B_\ve=(-\ve,\ve)\times\scrB_C^{d-1}\times\Sigma$.
Then $\mu_\scrX(\partial B_\ve)=0$, and thus
$\int_{N_s(\scrX)}\#(Y\cap \partial B_\ve)\,d\mu_{\vs}(Y)=0$,
by Lemma \ref{GROUNDINTENSITYlem2}.
Hence 
\begin{align}\label{GOODDISTANCElemPF1}
\mu^{(\lambda)}_{\rho_n,\vecq_n}\bigl(\bigl\{Y\in N_s(\scrX)\col Y\cap B_\ve=\emptyset\bigr\}\bigr)
\to\mu_{\vs}\bigl(\bigl\{Y\in N_s(\scrX)\col Y\cap B_\ve=\emptyset\bigr\}\bigr).
\end{align}
But for each $n$ there is some $\vecp_n\in\scrP\setminus\{\vecq_n\}$ with
$\|\vecp_n-\vecq_n\|<C\rho_n$;
hence if $C\rho_n^d<\ve$ then $(\vecp_n-\vecq_n)R(\vecv)D_{\rho_n}\in (-\ve,\ve)\times\scrB_C^{d-1}$ for all $\vecv\in\US$.
It follows that the left hand side of \eqref{GOODDISTANCElemPF1} is zero for all large $n$;
hence also the right hand side must be zero,
and so $\int_{N_s(\scrX)}\#(Y\cap B_\ve)\,d\mu_{\vs}(Y)\geq1$ for all $\ve>0$.
For $\ve$ small this yields a contradiction against Lemma \ref{GROUNDINTENSITYlem2}.
\end{proof}
{\blu The following lemma will be used later to show that the scatterers centered at the points in the 
exceptional set $\scrE\subset\scrP$ play a negligible role in the 
the Lorentz process for the
scatterer configuration $\scrP$, as $\rho\to0$.}
\begin{lem}\label{EXCHITUNLIKELYlem}
Let $T\geq1$, $\lambda\in\Pac(\US)$, 
and let $B\subset\R^d$ be a bounded Borel set.
Then
\begin{align}\label{EXCHITUNLIKELYlemres}
\lambda(\{\vecv\in\US\col \scrE\cap (\vecq+B D_\rho^{-1} R(\vecv)^{-1})\neq\emptyset\})\to0
\end{align}
as $\rho\to0$, uniformly over all $\vecq\in\scrP_T(\rho)$.
\end{lem}
\begin{proof}
Enlarging $B$, we may assume $B=\scrB_R^d$ for some $R>1$.
Also by a standard approximation argument, it suffices to prove \eqref{EXCHITUNLIKELYlemres} for those $\lambda$
which have a continuous density with respect to $\omega$;
and thus in fact it suffices to prove \eqref{EXCHITUNLIKELYlemres} for the single case
$\lambda=\omega_1$,
normalized Lebesgue measure on $\US$.

Let $T\geq1$ and $\ve>0$ be given.
Take $0<r<1$ so small that $c_\scrP\vol(\scrB_r^d)<\ve$.
Using then Lemma \ref{GROUNDINTENSITYlem2},
[P2], %
and the same line of reasoning as in the proof of Lemma \ref{GOODDISTANCElem},
it follows that for each $T'\geq1$ there is some $\rho_0\in(0,1)$ such that
\begin{align}\label{EXCHITUNLIKELYlempf1}
\lambda(\{\vecv\in\US\col\scrQ_\rho(\vecq,\vecv)\cap(\scrB_r^d\times\Sigma)\neq\emptyset\})
<2\ve,
\qquad\forall \rho\in(0,\rho_0),\:\vecq\in\scrP_{T'}(\rho).
\end{align}
Set $k=2R/r>2$, $T'=k^{d-1}T$, and $\tB:=\scrB_r^dD_k^{-1}$.
Replacing $\rho$ by $k\rho$ in \eqref{EXCHITUNLIKELYlempf1}
and using the definition of $\scrQ_\rho(\vecq,\vecv)$ it follows that
for all $\rho\in(0,\rho_0/k)$ and $\vecq\in\scrP_{T'}(k\rho)=\scrP_T(\rho)$ we have
\begin{align}\label{EXCHITUNLIKELYlempf5}
\lambda(\{\vecv\in\US\col 
\scrP\cap(\vecq+\tB{D}_\rho^{-1}R(\vecv)^{-1})\setminus\{\vecq\}\neq\emptyset\})<2\ve.
\end{align}
Recall $B=\scrB_R^d$; one verifies that $|x_1|\geq k_1:=(r/2)^dR^{1-d}$
for all $\vecx\in B\setminus\tB$,
and hence 
\begin{align}\label{EXCHITUNLIKELYlempf5a}
({B}\setminus\tB){D}_\rho^{-1}\subset A(\rho):=\scrB_{R\rho^{1-d}}^d\setminus\scrB_{k_1\rho^{1-d}}^d,
\qquad\forall \rho>0.
\end{align}
Now for any $\rho\in(0,\rho_0/k)$ and $\vecq\in\scrP_{T}(\rho)$ we have,
using \eqref{EXCHITUNLIKELYlempf5}, \eqref{EXCHITUNLIKELYlempf5a} and $\scrE\subset\scrP\setminus\{\vecq\}$,
\begin{align}\label{EXCHITUNLIKELYlempf3}
\lambda(\{\vecv\in\US\col \scrE\cap (\vecq+{B}{D}_\rho^{-1} R(\vecv)^{-1})\neq\emptyset\})
\hspace{150pt}
\\\notag
<2\ve+\sum_{\vecp\in\scrE\cap(\vecq+A(\rho))}\lambda\bigl(\bigl\{
\vecv\in\US\col \vecp\in\vecq+B{D}_\rho^{-1} R(\vecv)^{-1}\bigr\}\bigr).
\end{align}
But if $\vecp\in\vecq+B{D}_\rho^{-1} R(\vecv)^{-1}$
then $\vecp$ has distance $<R\rho$ from the line $\vecq+\R\vecv$;
and if also $\vecp\in\vecq+A(\rho)$ then
the angle $\varphi(\vecv,\vecp-\vecq)$ between the vectors $\vecv$ and $\vecp-\vecq$
satisfies $\sin\varphi(\vecv,\vecp-\vecq)<(R/k_1)\rho^d$.
The measure of the set of such points $\vecv\in\US$ with respect to $\lambda=\omega_1$ is
bounded above by $C_1\rho^{d(d-1)}$,
where $C_1$ depends on $d,R,r$ but not on $\rho$ or $\vecp$.
Hence \eqref{EXCHITUNLIKELYlempf3} is
\begin{align*}
\leq2\ve+\#(\scrE\cap(\vecq+\scrB_{R\rho^{1-d}}^d))\cdot C_1\rho^{d(d-1)}
\leq2\ve+\#(\scrE\cap\scrB_{(T+R)\rho^{1-d}}^d)\cdot C_1\rho^{d(d-1)},
\end{align*}
and using {\blu the fact that $\scrE$ has asymptotic density zero (cf.\ [P2])}, 
this number is seen to be $<3\ve$
for all sufficiently small $\rho$.
\end{proof}

\begin{remark}\label{SIGMAeqpREM}
It follows from Lemma \ref{EXCHITUNLIKELYlem} that our key assumption, [P2], %
remains valid if
we change the marking of the $\scrE$-points in an arbitrary way
in the definition of $\scrQ_\rho(\vecq,\vecv)$.  %
In precise terms, %
if $\vs'$ is any map $\scrP\to\Sigma$ which has the same restriction as $\vs$ to $\scrP\setminus\scrE$,
and if
\begin{align*}
\tP'=\{(\vecp,\vs'(\vecp))\col\vecp\in\scrP\}, \quad \tP'_\vecq=\tP'\setminus(\{\vecq\}\times\Sigma), \quad
\scrQ'_\rho(\vecq,\vecv)=(\tP'_\vecq-\vecq)\,R(\vecv)\,D_\rho,
\end{align*}
and if $\mu_{\vecq,\rho}'^{(\lambda)}\in P(N_s(\scrX))$
is the distribution of $\scrQ'_\rho(\vecq,\vecv)$ for $\vecv$ random in $(\US,\lambda)$,
then 
\begin{align}\label{EREMARKOK}
\mu_{\vecq,\rho}'^{(\lambda)}\xrightarrow[]{\textup{ w }}\mu_{{\vs}(\vecq)}
\quad\text{as }\:\rho\to0,\:\text{ uniformly over all $\vecq\in\scrP_T(\rho)$.}
\end{align}
The same statement also holds if we \textit{remove} some or all $\scrE$-points in the definition of
$\tP'_\vecq$.

In particular, we may choose $\vs'$ so that $\vs'(\vecp)\in\Sigma'$ for all $\vecp\in\scrE$.
Then in fact we have $\vs'(\vecp)\in\Sigma'$ for all $\vecp\in\scrP$,
i.e.\ $\vs'$ can be viewed as a map from $\scrP$ to $\Sigma'$.
In this case, using 
Lemma~\ref{SIGMApfullmeaslem}
and \eqref{EREMARKOK} we see by direct inspection that
\textit{all the assumptions in Section \ref{ASSUMPTLISTsec}
remain true if we replace $\langle\scrP,\Sigma,\vs,\mm,\scrE\rangle$
by $\langle\scrP,\Sigma',\vs',\mm_{|\Sigma'},\scrE\rangle$.}
Note also that
$\Sigma'=\overline{\{{\vs}'(\vecq)\col\vecq\in\scrP\setminus\scrE\}}$,
since $\vs'(\vecq)=\vs(\vecq)$ for all $\vecq\in\scrP\setminus\scrE$.
In other words, after having replaced $\langle\scrP,\Sigma,\vs,\mm,\scrE\rangle$
by $\langle\scrP,\Sigma',\vs',\mm_{|\Sigma'},\scrE\rangle$, 
the following condition is satisfied:
\begin{align}\label{SIGMAeqp}
\Sigma=\Sigma'=\overline{\{{\vs}(\vecq)\col\vecq\in\scrP\setminus\scrE\}}.
\end{align}
\end{remark}

In view of Remark \ref{SIGMAeqpREM}, 
\textit{we may assume without loss of generality that \eqref{SIGMAeqp} holds.} We will make this assumption in Sections \ref{GENLIMITsec} -- \ref{KINETICEQsec}, in addition to the hypotheses [P1-3] and [Q1-3] as stated in Section \ref{ASSUMPTLISTsec}.

\section{A limiting process for macroscopic initial conditions}\label{GENLIMITsec}

As a complement to our key limit assumption [P2] %
we also need to understand the limit of
$\scrQ_\rho(\rho^{1-d}\vecq,\vecv)$ when $(\vecq,\vecv)$ is taken random in $\T^1(\R^d)$ with respect to
an arbitrary probability measure $\Lambda\in\Pac(\T^1(\R^d))$.
We will show in Theorem \ref{GENLIMITthm} that this limit exists and is independent of $\Lambda$.
To simplify notation, we will in the following use the shorthand $d\vecq=d\!\vol(\vecq)$ and $d\vecv=d\omega(\vecv)$.\label{dvdo}

We start by defining{\blu, in \eqref{GENMUdef} below,} the limit process, which is an explicit function of the point processes $\Xi_\vs$.
The construction depends on the choice of a real constant $c$ and a
nonempty bounded open set $\fD\subset\R^{d-1}$ with boundary of Lebesgue measure zero;
we consider these to be fixed once and for all in the following.
{\blu The fact that the limit process which we define in \eqref{GENMUdef} 
is independent of $c$ and $\fD$ is far from obvious,
but will follow from the proof of Theorem \ref{GENLIMITthm},
where $c$ and $\fD$ are used
in the construction of a certain decomposition of 
the distribution of $\scrQ_\rho(\rho^{1-d}\vecq,\vecv)$.}   

For $\tau>0$, let $\fC_\tau$ be the open cylinder\label{CtauDEF}
\begin{align*}
\fC_\tau=(c,c+\tau)\times\fD
\end{align*}
in $\R^d$.
For any fixed ${\vs}\in\Sigma$ and $\vecx\in\R^d$ 
we write $\oXi_{\vs}:=\Xi_{\vs}\cup\{(\bn,{\vs})\}$,
\label{oXivsDEF}
let $\omu_{\vs}$ be the distribution of $\oXi_{\vs}$,
and let $\omu_{\vs}^{(\vecx)}$ be the distribution of $\oXi_{\vs}+\vecx$.
Thus for any Borel subset $A\subset N_s(\scrX)$
we have $\omu_{\vs}^{(\vecx)}(A)=\mu_\vs(\{Y %
\col Y_{(\vecx,\vs)}\in A\})$,
where
\begin{align*}
Y_{(\vecx,\vs)}:=%
(Y+\vecx)\cup\{(\vecx,\vs)\}.
\end{align*}

From now on and throughout the rest of the paper,
it will be convenient to allow the following abuse of notation:
For any subsets $A\subset\R^d$ and $B\subset\scrX$, 
we write ``$A\cap B$'' or ``$B\cap A$'' for $B\cap(A\times\Sigma)$.
In particular in the following proposition,
``$Y\cap\fC_\tau$'' denotes %
$Y\cap(\fC_\tau\times\Sigma)$.
\begin{prop}\label{GENLIMITcor1}
\begin{align}\label{GENLIMITdisc12}
c_\scrP\int_{\Sigma}\int_0^{\infty}\int_{\fD}
\omu_{{\vs}}^{(c+\tau,\vecb)}\bigl(\bigl\{Y\in N_s(\scrX)\col Y\cap\fC_\tau=\emptyset\bigr\}\bigr)
\,d\vecb\,d\tau\,d\mm({\vs})=1.
\end{align}
\end{prop}
The proof of the proposition is given below.
We remark that the integrand in \eqref{GENLIMITdisc12} is a continuous function of
$\langle\vs,\tau,\vecb\rangle\in\Sigma\times\R_{>0}\times\fD$;
cf.\ Lemma \ref{MUSXFTCONTlem} below (applied with $f\equiv1$).

We now define $\mu\in P(N_s(\scrX))$ by
setting, for any Borel set $A\subset N_s(\scrX)$,
\begin{align}\label{GENMUdef}
\mu(A)=c_\scrP\int_{\Sigma}\int_0^{\infty}\int_{\fD}
\omu_{{\vs}}^{(c+\tau,\vecb)}\bigl(\bigl\{Y\in N_s(\scrX)\col Y\cap\fC_\tau=\emptyset\text{ and }Y\in A\bigr\}\bigr)
\,d\vecb\,d\tau\,d\mm({\vs}).
\end{align}
By Proposition \ref{GENLIMITcor1}, $\mu$ is indeed a Borel probability measure on $N_s(\scrX)$.
We denote by $\Xi$ a point process in $\scrX$ with distribution $\mu$;
{\blu we will call $\Xi$ the \textit{macroscopic limit process.}
In Proposition \ref{GENPALMprop} 
in the next section we will deduce from \eqref{GENMUdef} that 
$\Xi$ is compatible with the microscopic limit processes $\oXi_\vs$
in the sense that the
Palm distributions of $\Xi$ are given by the distributions of the appropriate translates of the $\oXi_\vs$'s.}

The following is the main result of this section.
For $\Lambda\in P(\T^1(\R^d))$, let $\mu_\rho^{(\Lambda)}$ be the distribution of
$\scrQ_\rho(\rho^{1-d}\vecq,\vecv)$ for $(\vecq,\vecv)$ random in $(\T^1(\R^d),\Lambda)$.
\label{murhoLambda}

\begin{thm}\label{GENLIMITthm}
Let $\Lambda\in \Pac(\T^1(\R^d))$.
Then $\mu_\rho^{(\Lambda)}\xrightarrow[]{\textup{ w }}\mu$ as $\rho\to0$.
\end{thm}
{\blu The rest of this section is devoted to the proof of 
Proposition \ref{GENLIMITcor1} and Theorem \ref{GENLIMITthm}.
We start by explaining the main ideas 
in the proof of Theorem \ref{GENLIMITthm}.
Let
$\Lambda\in\Pac(\T^1(\R^d))$ be given, and let $\Lambda'\in\L^1(\T^1(\R^d))$ be 
the density of $\Lambda$ with respect to $d\vecq\,d\vecv$.
Our task is to prove that for any fixed $f\in\C_b(N_s(\scrX))$,
the integral
\begin{multline}\label{GENLIMITdisc1}
\int_{\T^1(\R^d)}f(\scrQ_\rho(\rho^{1-d}\vecq,\vecv))\,\Lambda'(\vecq,\vecv)\,d\vecq\,d\vecv \\
=\rho^{d(d-1)}\int_{\US}\int_{\R^d}f(\scrQ_\rho(\vecq,\vecv))\,\Lambda'(\rho^{d-1}\vecq,\vecv)\,d\vecq\,d\vecv
\end{multline}
tends to $\int_{N_s(\scrX)}f\,d\mu$ as $\rho\to0$.
In \eqref{GENLIMITdisc1}, 
we wish to express $\scrQ_\rho(\vecq,\vecv)$ in terms of
$\scrQ_\rho(\vecq',\vecv)$ for some appropriate choice of a point $\vecq'$ \textit{in $\scrP$},
since we will then have a hope of applying [P2] to get hold of the limit as $\rho\to0$.
The point $\vecq'$ must of course depend on $\vecq$.

Our way to define $\vecq'$
is similar to 
the free path length problem for the Lorentz process,
but working with the flat scatterer $-\rho(\{0\}\times\fD)R(\vecv)^{-1}$
(this is a relatively open set contained in the orthogonal complement of $\vecv$,
of size proportional to $\rho$),
and starting at the point $\vecq+c\rho^{1-d}\vecv$ instead of $\vecq$:
For each $\vecq\in\R^d$ and $\vecv\in\US$,
we define 
\begin{align}\label{taurhoqvdef}
\tau_\rho(\vecq,\vecv):=\inf\bigl\{\xi>0\col
\vecq+(c+\xi)\rho^{1-d}\vecv\in\scrP-\rho(\{0\}\times\fD)R(\vecv)^{-1}\bigr\}.
\end{align}
This is a well-defined number in $\R_{>0}\cup\{+\infty\}$.
For generic $\vecq$ and $\vecv$ we have $\tau=\tau_\rho(\vecq,\vecv)<\infty$,
and there exist unique points $\vecq'\in\scrP$ and $\vecb\in\fD$ such that
$\vecq+(c+\tau)\rho^{1-d}\vecv=\vecq'-\rho(0,\vecb)R(\vecv)^{-1}$,
or equivalently,
\begin{align}\label{qintermsofqp}
\vecq=\vecq'-(c+\tau,\vecb)D_\rho^{-1}R(\vecv)^{-1}.
\end{align}

Let us write \label{oQrhoqvDEF}
\begin{align*}
\oQ_\rho(\vecq,\vecv)=(\tP-\vecq)\,R(\vecv)\,D_\rho.
\end{align*}
(Thus $\oQ_\rho(\vecq,\vecv)\supseteq\scrQ_\rho(\vecq,\vecv)$, with equality unless $\vecq\in\scrP$.)
Now for any $\vecq$ and $\vecq'$ related by \eqref{qintermsofqp},
if $\vecq\notin\scrP$ then
\begin{align*}
\scrQ_\rho(\vecq,\vecv)
=\oQ_\rho(\vecq,\vecv)
=(\tP-\vecq)\,R(\vecv)\,D_\rho
=\oQ_\rho(\vecq',\vecv)+(c+\tau,\vecb),
\end{align*}
and one verifies that in the above construction, $\vecq'\in\scrP$ arises from a given point
$\vecq\in\R^d$ if and only if
$\oQ_\rho(\vecq',\vecv)+(c+\tau,\vecb)$ is disjoint from $\fC_\tau$.
Hence, using \eqref{qintermsofqp} to replace the integration variable $\vecq$ by $\tau$ and $\vecb$,
and ignoring effects from overlapping scatterers and the possibility that $\tau_\rho(\vecq,\vecv)=\infty$,
the integral in \eqref{GENLIMITdisc1} can be rewritten as
\begin{align}\label{GENLIMITdisc1a}
\rho^{d(d-1)}\int_{\US}\sum_{\vecq'\in\scrP}\int_0^\infty\int_{\fD}
&I\Bigl(\bigl(\oQ_\rho(\vecq',\vecv)+(c+\tau,\vecb)\bigr)\cap\fC_\tau=\emptyset\Bigr)
\\\notag
&\times f\Bigl(\oQ_\rho(\vecq',\vecv)+(c+\tau,\vecb)\Bigr)
\,\Lambda'(\rho^{d-1}\vecq,\vecv)\,d\vecb\,d\tau\,d\vecv.
\end{align}
With this, we have 
achieved the goal of expressing the integral using only
point sets $\scrQ_\rho(\vecq',\vecv)$ with $\vecq'$ in $\scrP$.
Now, after moving the integration 
over $\vecv$ to the innermost position in 
\eqref{GENLIMITdisc1a},
and recalling the above definition of $\omu_\vs^{(\vecx)}\in P(N_s(\scrX))$,
[P2] heuristically suggests that for $\rho$ small,
\eqref{GENLIMITdisc1a} should be approximately equal to
(using also
$\rho^{d-1}\vecq\approx\rho^{d-1}\vecq'-(c+\tau)\vecv$):
\begin{multline*}
\rho^{d(d-1)}\sum_{\vecq'\in\scrP}\int_0^\infty\int_{\fD}
\biggl(\int_{N_s(\scrX)}
I\bigl(Y\cap\fC_\tau=\emptyset\bigr)\,f(Y)\,d\omu_{\vs(\vecq')}^{(c+\tau,\vecb)}(Y)\biggr)\\ \times \tLambda\bigl(\rho^{d-1}\vecq'-(c+\tau)\vecv\bigr)
\,d\vecb\,d\tau
\end{multline*}
with $\tLambda(\vecy):=\int_{\US}\Lambda'(\vecy,\vecv)\,d\vecv$.
Finally, moving the summation over $\vecq'$ inside the first two integrals,
Lemma \ref{ASYMPTDENSITY1CONSlem}
suggests that as $\rho\to0$, the above expression should tend to
\begin{align*}
\int_0^\infty\int_{\fD}c_{\scrP}\int_{\Sigma}
\int_{N_s(\scrX)}
&I\bigl(Y\cap\fC_\tau=\emptyset\bigr)\,f(Y)\,d\omu_{\vs}^{(c+\tau,\vecb)}(Y)\,
d\mm(\vs)
\,d\vecb\,d\tau.
\end{align*}
This equals $\int_{N_s(\scrX)}f\,d\mu$, the desired limit!

\vspace{5pt}

It remains to make the above discussion rigorous.
We will start by proving several auxiliary results.}

For $\vecx\in\R^d$ and $\lambda\in P(\US)$, let
$\omu_{\vecq,\rho}^{(\vecx,\lambda)}$ be the distribution of $\oQ_\rho(\vecq,\vecv)+\vecx$
for $\vecv$ random in $(\US,\lambda)$.
{\blu Given any $f\in\C_b(N_s(\scrX))$ and $\tau>0$,
define the function $f^{(\tau)}:N_s(\scrX)\to\R$ through
\begin{align}\label{ASS:KEYbblem2res}
f^{(\tau)}(Y):=I\bigl(Y\cap\fC_\tau=\emptyset\bigr)\, f(Y).
\end{align}
Note that with this definition, the integrand in 
\eqref{GENLIMITdisc1a}
equals $f^{(\tau)}\bigl(\oQ_\rho(\vecq',\vecv)+(c+\tau,\vecb)\bigr)$.}

\begin{lem}\label{ASS:KEYbblem2}
Let $T\geq1$ and $\lambda\in\Pac(\US)$; 
let $C$ be a compact subset of $\R^d$,
and let $f\in\C_b(N_s(\scrX))$ and $\xi_1>0$.
Then $\omu_{\vecq,\rho}^{(\vecx,\lambda)}(f^{(\tau)})-\omu_{{\vs}(\vecq)}^{(\vecx)}(f^{(\tau)})\to0$ as $\rho\to0$,
uniformly over all $\vecq\in\scrP_T(\rho)$, $\vecx\in C$, $\tau\in[0,\xi_1]$.
\end{lem}
\begin{remark}\label{ASS:KEYbblem2rem}
In particular the lemma (applied with $\tau=0$) implies that
\begin{align}\label{ASS:KEYbblem2remres}
\omu_{\vecq,\rho}^{(\vecx,\lambda)}\xrightarrow[]{\textup{ w }}\omu_{{\vs}(\vecq)}^{(\vecx)}
\qquad\text{as $\rho\to0$, uniformly over all $\vecq\in\scrP_T(\rho)$ and $\vecx\in C$.}
\end{align}
It would be easy to give a more direct proof of \eqref{ASS:KEYbblem2remres};
however in our proof of Theorem \ref{GENLIMITthm}
we need the more delicate convergence statement of Lemma \ref{ASS:KEYbblem2}.
\end{remark}
\begin{proof}
Note that $\omu_{\vecq,\rho}^{(\vecx,\lambda)}(f^{(\tau)})-\omu_{{\vs}(\vecq)}^{(\vecx)}(f^{(\tau)})=0-0=0$
whenever $\vecx\in\fC_\tau$;
hence from now on we assume $\vecx\notin \fC_\tau$.
By a standard subsequence argument, it suffices to prove that given any
$\rho_n\in(0,1)$, $\vecq_n\in\scrP_T(\rho_n)$, $\tau_n\in[0,\xi_1]$ and $\vecx_n\in C\setminus\fC_{\tau_n}$ 
($n=1,2,\ldots$)
such that $\lim_{n\to\infty}\rho_n=0$ and such that the limits
\begin{align*}
\vs:=\lim_{n\to\infty}\vs(\vecq_n)\in\Sigma,
\qquad\tau:=\lim_{n\to\infty}\tau_n\in\R,
\qquad\vecx:=\lim_{n\to\infty}\vecx_n\in C\subset\R^d
\end{align*}
all exist, %
then 
\begin{align}\label{ASS:KEYbblem2pf1}
\omu_{\vecq_n,\rho_n}^{(\vecx_n,\lambda)}(f^{(\tau_n)})-\omu_{{\vs}(\vecq_n)}^{(\vecx_n)}(f^{(\tau_n)})\to0
\qquad\text{as }\: n\to\infty.
\end{align}
Using $\vecx_n\in C\setminus\fC_{\tau_n}$
and the definitions of $\omu_{\vecq,\rho}^{(\vecx,\lambda)}$ and $\omu_\vs^{(\vecx)}$, 
\eqref{ASS:KEYbblem2pf1} is seen to be equivalent to 
\begin{align}\label{ASS:KEYbblem2pf2}
\mu_{\vecq_n,\rho_n}^{(\lambda)}(f_n)-\mu_{\vs(\vecq_n)}(f_n)\to0\qquad\text{as }\: n\to\infty,
\end{align}
where 
\begin{align*}
f_n(Y):=I\bigl((Y+\vecx_n)\cap\fC_{\tau_n}=\emptyset\bigr)\,f\bigl(Y_{(\vecx_n,\vs(\vecq_n))}\bigr).
\end{align*}
Define $F:N_s(\scrX)\to\R$ through
\begin{align*}
F(Y):=I\bigl((Y+\vecx)\cap\fC_\tau=\emptyset\bigr)\, f(Y_{(\vecx,\vs)}).
\end{align*}
Note that for any $Y,Y_1,Y_2,\ldots\in N_s(\scrX)$,
if $Y_n\to Y$,
$(\bn,\vs)\notin Y$ and 
$(Y+\vecx)\cap\partial\fC_\tau=\emptyset$, then $f_n(Y_n)\to F(Y)$ as $n\to\infty$.
Hence, using also
$\mu_{\vecq_n,\rho_n}^{(\lambda)}\xrightarrow[]{\textup{ w }}\mu_\vs$
(cf.\ [P2]) and
the fact that
\begin{align*}
\mu_\vs\bigl(\bigl\{Y\in N_s(\scrX)\col 
(\bn,\vs)\notin Y\text{ and }
(Y+\vecx)\cap\partial\fC_\tau=\emptyset\bigr\}\bigr)=1
\end{align*}
by Lemma \ref{GROUNDINTENSITYlem2},
it follows that
$\mu_{\vecq_n,\rho_n}^{(\lambda)}(f_n)\to\mu_\vs(F)$ as $n\to\infty$.
(Indeed, apply \cite[Thm.\ 4.27]{kallenberg02}
and then consider the expected value.)  %
Similarly, using $\mu_{\vs(\vecq_n)}\xrightarrow[]{\textup{ w }}\mu_\vs$,
we also have $\mu_{\vs(\vecq_n)}(f_n)\to\mu_\vs(F)$ as $n\to\infty$.
Hence \eqref{ASS:KEYbblem2pf2} holds, and the lemma is proved.
\end{proof}

\begin{lem}\label{MUSXFTCONTlem}
Fix $f\in\C_b(N_s(\scrX))$ and define $f^{(\tau)}$ as in {\blu \eqref{ASS:KEYbblem2res}}.
Then $\langle\vs,\vecx,\tau\rangle\mapsto\omu_\vs^{(\vecx)}(f^{(\tau)})$ 
is a continuous function on
$\{\langle\vs,\vecx,\tau\rangle\in\Sigma\times\R^d\times\R_{\geq0}\col\vecx\notin\fC_\tau\}$.
\end{lem}
\begin{proof}
This is an immediate modification of the proof of Lemma \ref{ASS:KEYbblem2}.
\end{proof}
{\blu In the following proposition, we will prove the desired limit statement for the
following truncated version of the expression 
in \eqref{GENLIMITdisc1}:
\begin{align}\label{GENLIMITdisc13}
J_\rho(f,T_1):=
\rho^{d(d-1)}\int_{\US}\int_{\R^d}I(\tau_\rho(\vecq,\vecv)<T_1)\,
f(\scrQ_\rho(\vecq,\vecv))\,\Lambda'(\rho^{d-1}\vecq,\vecv)\,d\vecq\,d\vecv,
\end{align}
for any fixed $T_1>0$.}
\begin{prop}\label{GENLIMITprop1}
For any $\Lambda\in\Pac(\T^1(\R^d))$,
$f\in\C_b(N_s(\scrX))$ and $T_1>0$,
\begin{align}\label{GENLIMITdisc11}
\lim_{\rho\to0}J_\rho(f,T_1)
=c_\scrP\int_{\Sigma}\int_0^{T_1}\int_{\fD}\int_{N_s(\scrX)}
I(Y\cap\fC_\tau=\emptyset)f(Y)\,d\omu_{{\vs}}^{(c+\tau,\vecb)}(Y)\,d\vecb\,d\tau\,d\mm({\vs}).
\end{align}
\end{prop}

\begin{proof}
Without loss of generality we may assume $\Lambda'\in\C_c(\T^1(\R^d))$,
since $\C_c(\T^1(\R^d))$ is dense in $\L^1(\T^1(\R^d))$.

{\blu Let us note that the definition of $\tau_\rho(\vecq,\vecv)$,
\eqref{taurhoqvdef},
can be equivalently expressed as
\begin{align*}
\tau_\rho(\vecq,\vecv)=\inf\bigl\{\xi>0\col\oQ_\rho(\vecq,\vecv)\cap\fC_\xi\neq\emptyset\bigr\}.
\end{align*}
Thus} for any $(\vecq,\vecv)\in\T^1(\R^d)$ satisfying $\tau=\tau_\rho(\vecq,\vecv)<\infty$, there exists a point
$\vecp\in\scrP$ such that
$(\vecp-\vecq)R(\vecv)D_\rho\in\{c+\tau\}\times\fD$.
If this point is unique, we call it $\vecz_\rho(\vecq,\vecv)$.
In the remaining cases 
(viz., when $\tau=\infty$ or there are at least two points $\vecp\in\scrP$
with $(\vecp-\vecq)R(\vecv)D_\rho\in\{c+\tau\}\times\fD$)
we take $\vecz_\rho(\vecq,\vecv)$ to be undefined.
Thus for each $\rho\in(0,1)$
we have defined a function $\vecz_\rho$ from $\T^1(\R^d)$ to $\scrP\sqcup\{\undef\}$.

Let $S_\scrP$ be the set of $\vecv\in\US$ such that 
the inner products $\vecp\cdot\vecv$ for $\vecp\in\scrP$
are pairwise distinct.
Then $\omega(\US\setminus S_\scrP)=0$,
and for every $\vecv\in S_\scrP$ we have
$\vecz_\rho(\vecq,\vecv)\in\scrP$ 
for \textit{all} $\vecq\in\R^d$ with $\tau_\rho(\vecq,\vecv)<\infty$
{\blu (since $\vecv\in S_{\scrP}$ implies that the points $(\vecp-\vecq) R(\vecv)D_\rho$ for $\vecp\in\scrP$
have pairwise distinct $\vece_1$-coordinates).}
Also $\scrQ_\rho(\vecq,\vecv)=\oQ_\rho(\vecq,\vecv)$ for almost all $(\vecq,\vecv)$.
Therefore,
\begin{align}\label{GENLIMITdisc3}
J_\rho(f,T_1)=
\rho^{d(d-1)}\sum_{\vecq'\in\scrP}
\int_{\US}\int_{\R^d}
I\bigl(\tau_\rho(\vecq,\vecv)<T_1\text{ and }\vecz_\rho(\vecq,\vecv)=\vecq'\bigr)\,
\hspace{60pt}
\\\notag
\times 
f(\oQ_\rho(\vecq,\vecv))\,\Lambda'(\rho^{d-1}\vecq,\vecv)\,d\vecq\,d\vecv.
\end{align}
Recall here that $\vecz_\rho(\vecq,\vecv)=\vecq'$ implies that there is some $\vecb\in\fD$
such that $(\vecq'-\vecq)R(\vecv)D_\rho=(c+\tau,\vecb)$ with $\tau=\tau_\rho(\vecq,\vecv)$,
or equivalently:
\begin{align}\label{GENLIMITdisc4}
\vecq=\vecq'-(c+\tau,\vecb)D_\rho^{-1}R(\vecv)^{-1}
=\vecq'-\bigl(\rho^{1-d}(c+\tau),\rho\vecb\bigr)R(\vecv)^{-1}.
\end{align}
Conversely for any given $\vecq'\in\scrP$, $\vecv\in S_\scrP$, $\vecb\in\fD$, $\tau>0$ and $\rho\in(0,1)$,
if $\vecq$ is given by \eqref{GENLIMITdisc4}
then the two relations $\tau_\rho(\vecq,\vecv)=\tau$
and $\vecz_\rho(\vecq,\vecv)=\vecq'$ hold if and only if
$\oQ_\rho(\vecq,\vecv)\cap\fC_\tau=\emptyset$.
Hence, using also $\det D_\rho^{-1}R(\vecv)^{-1}=1$,
it follows that
\begin{multline}\label{GENLIMITdisc5}
J_\rho(f,T_1)=
\rho^{d(d-1)}\sum_{\vecq'\in\scrP}
\int_{\US}\int_0^{T_1}\int_{\fD}
I\bigl(\oQ_\rho(\vecq,\vecv)\cap\fC_\tau=\emptyset\bigr) \\ \times
f(\oQ_\rho(\vecq,\vecv))\,\Lambda'(\rho^{d-1}\vecq,\vecv)\,d\vecb\,d\tau\,d\vecv,
\end{multline}
with $\vecq$ as in \eqref{GENLIMITdisc4}.
Note here that
\begin{align*}
\oQ_\rho(\vecq,\vecv)=\oQ_\rho(\vecq',\vecv)+(c+\tau,\vecb).
\end{align*}
Take $R>0$ such that $\supp\Lambda'\subset\scrB_R^d\times\Sigma$, and set
\begin{align*}
T:=R+|c|+T_1+\sup_{\vecb\in\fD}\|\vecb\|.
\end{align*}
Note that \eqref{GENLIMITdisc4} implies $\|\vecq'-\vecq\|<(|c+\tau|+\sup_{\vecb\in\fD}\|\vecb\|)\rho^{1-d}$;
hence every $\vecq'\in\scrP$ which gives a nonzero contribution to the sum
in \eqref{GENLIMITdisc5} satisfies
$\|\vecq'\|<T\rho^{1-d}$.
Using {\blu the fact that $\scrE$ has asymptotic density zero (cf.\ [P2])}, 
we then see that 
restricting the sum in \eqref{GENLIMITdisc5} to $\vecq'\in\scrP_T(\rho)$
gives an error which tends to $0$ as $\rho\to0$.
Furthermore, 
$\rho^{d-1}\vecq=\rho^{d-1}\vecq'-(c+\tau,\rho^d\vecb)R(\vecv)^{-1}$
has distance $\ll\rho^d$ from $\rho^{d-1}\vecq'-(c+\tau)\vecv$;
hence using $\#\scrP_T(\rho)\ll\rho^{-d(d-1)}$ (cf.\ [P1]) and the uniform continuity of $\Lambda'$,
and writing $\vecq$ in place of $\vecq'$,
we obtain %
\begin{multline}\label{GENLIMITdisc5a}
J_\rho(f,T_1)=
\rho^{d(d-1)}\sum_{\vecq\in\scrP_T(\rho)}
\int_0^{T_1}\int_{\fD}\int_{\US}
I\bigl((\oQ_\rho(\vecq,\vecv)+(c+\tau,\vecb))\cap\fC_\tau=\emptyset\bigr)
\hspace{60pt}
\\ 
\times f(\oQ_\rho(\vecq,\vecv)+(c+\tau,\vecb))\,\Lambda'(\rho^{d-1}\vecq-(c+\tau)\vecv,\vecv)\,d\vecv\,d\vecb\,d\tau + o(1).
\end{multline}
Here $o(1)$ denotes a term that tends to zero as $\rho\to0$.

Given any $\vecq\in\scrP_T(\rho)$, $\tau\in[0,T_1]$ and $\vecb\in\fD$, %
we set %
$\vecy=\rho^{d-1}\vecq$, $\vecx=(c+\tau,\vecb)$ and $\xi=c+\tau$;
then the innermost integral in \eqref{GENLIMITdisc5a} can be expressed as
\begin{align}\label{GENLIMITdiscn1}
\int_{\US} f^{(\tau)}\bigl(\oQ_\rho(\vecq,\vecv)+\vecx\bigr)\,\Lambda'(\vecy-\xi\vecv,\vecv)\,d\vecv,
\end{align}
with $f^{(\tau)}(Y):=I\bigl(Y\cap\fC_\tau=\emptyset\bigr)\, f(Y)$
as in Lemma \ref{ASS:KEYbblem2}.
For any $\vecy\in\R^d$ and $\xi\in\R$ we write
\begin{align*}
D_\Lambda(\vecy,\xi):=\int_{\US}\Lambda'\bigl(\vecy-\xi\vecv,\vecv\bigr)\,d\vecv.
\end{align*}
Then if $D_\Lambda(\vecy,\xi)>0$,
the integral in \eqref{GENLIMITdiscn1} equals
$D_\Lambda(\vecy,\xi)\,\omu_{\vecq,\rho}^{(\vecx,\lambda)}(f^{(\tau)})$, with
$\lambda\in P(\US)$ given by
\begin{align}
d\lambda(\vecv):=D_\Lambda(\vecy,\xi)^{-1}\Lambda'(\vecy-\xi\vecv,\vecv)\,d\vecv. %
\end{align}
Hence by Lemma \ref{ASS:KEYbblem2},
for any fixed $\vecy\in\R^d$ and $\xi\in\R$, we have
\begin{align}\label{GENLIMITdisc7}
\int_{\US}f^{(\tau)}(\oQ_\rho(\vecq,\vecv)+\vecx)\,\Lambda'(\vecy-\xi\vecv,\vecv)\,d\vecv
-D_\Lambda(\vecy,\xi)\,\omu_{{\vs}(\vecq)}^{(\vecx)}(f^{(\tau)})\to0
\end{align}
as $\rho\to0$, uniformly over all $\vecq\in\scrP_T(\rho)$, $\vecx\in [c,c+T_1]\times\overline\fD$ and $\tau\in[0,T_1]$.
Of course this convergence holds also when $D_\Lambda(\vecy,\xi)=0$, trivially.
By a standard compactness argument,
using $\Lambda'\in\C_c$,
the same convergence statement is upgraded to also hold uniformly over all
$\vecy\in\R^d$ and $\xi\in[c,c+T_1]$.
Using this fact in \eqref{GENLIMITdisc5a},
and again using $\#\scrP_T(\rho)\ll\rho^{-d(d-1)}$,
we conclude
\begin{align}\label{GENLIMITdisc10}
J_\rho(f,T_1)=o(1)+\rho^{d(d-1)}\sum_{\vecq\in\scrP_T(\rho)}F(\rho^{d-1}\vecq,{\vs}(\vecq)),
\end{align}
where $F:\scrX %
\to\R$ is given by
\begin{align*}
F(\vecy,{\vs})=
\int_0^{T_1}\int_{\fD}D_\Lambda(\vecy,c+\tau)\,\omu_{{\vs}}^{(c+\tau,\vecb)}(f^{(\tau)})\,d\vecb\,d\tau.
\end{align*}
It follows from Lemma \ref{MUSXFTCONTlem} that $F$ is a continuous function on $\scrX$.
Also $F$ has compact support since $\Lambda'$ has compact support;
in fact, using $\supp\Lambda'\subset\scrB_R^d\times\Sigma$
it follows that $\supp F\subset\scrB_{R+T_1}^d\times\Sigma$,
and hence the sum in \eqref{GENLIMITdisc10} remains unchanged if we
replace the summation range by $\scrP\setminus\scrE$;
and {\blu since $\scrE$ has asymptotic density zero,} 
we can change this further to $\scrP$ at the price of an $o(1)$ error.
Hence by Lemma~\ref{ASYMPTDENSITY1CONSlem}, 
\begin{align}\label{GENLIMITdisc10a}
\lim_{\rho\to0}J_\rho(f,T_1)=c_\scrP\int_{\scrX}F\,d\mu_\scrX.
\end{align}
Finally, using $\int_{\US}\int_{\R^d}\Lambda'\,d\vecy\,d\vecv=1$ it follows that
$\int_{\R^d}D_\Lambda(\vecy,\xi)\,d\vecy=1$ for every $\xi\in\R$.
Hence
\begin{align*}
c_\scrP\int_{\scrX}F\,d\mu_\scrX
=c_\scrP\int_{\Sigma}\int_0^{T_1}\int_{\fD}\omu_{{\vs}}^{(c+\tau,\vecb)}(f^{(\tau)})\,d\vecb\,d\tau\,d\mm({\vs}),
\end{align*}
and the proposition is proved.
\end{proof}

The following lemma is a simple consequence of the assumption [P3]. %
\begin{lem}\label{ASS:bdfreepathgenlem}
For any bounded Borel set $B\subset\R^d$,
\begin{align}\label{ASS:bdfreepathgenlemres}
\lim_{\xi\to\infty}\limsup_{\rho\to0}\hspace{7pt}
[\vol\times\omega]\bigl(\bigl\{(\vecq,\vecv)\in B\times\US\col
\oQ_{\rho}(\rho^{1-d}\vecq,\vecv)\cap\fC_\xi=\emptyset\bigr\}\bigr)=0.
\end{align}
\end{lem}
\begin{proof}
Fix $\vecy\in\R^{d-1}$ and $r>0$ so that $\fD$ contains $\vecy+\scrB_r^{d-1}$,
and set $\vecx=(c,\vecy)\in\R^d$.
Then $\fC_\xi$ contains $\vecx+(0,\xi)\times\scrB_r^{d-1}$.
Noticing also $\oQ_\rho(\vecq,\vecv)-\vecx=\oQ_\rho(\vecq+\vecx D_\rho^{-1}R(\vecv)^{-1},\vecv)$
it follows that the set considered in \eqref{ASS:bdfreepathgenlemres} is a subset of
\begin{align*}
\bigl\{(\vecq,\vecv)\in B\times\US\col
\oQ_\rho\bigl(\rho^{1-d}\vecq+\vecx D_\rho^{-1}R(\vecv)^{-1},\vecv\bigr)\cap((0,\xi)\times\scrB_r^{d-1})=\emptyset\bigr\}.
\end{align*}
Here $\|\vecx D_\rho^{-1}R(\vecv)^{-1}\|\leq\rho^{1-d}\|\vecx\|$;
hence by Fubini the measure considered in \eqref{ASS:bdfreepathgenlemres} is bounded above by
\begin{align*}%
M(\rho,\xi):=[\vol\times\omega]\bigl(\bigl\{(\vecq,\vecv)\in B'\times\US\col
\oQ_{\rho}(\rho^{1-d}\vecq,\vecv)\cap((0,\xi)\times\scrB_r^{d-1})=\emptyset\bigr\}\bigr),
\end{align*}
where $B'$ is the $\|\vecx\|$-neighbourhood of $B$;
this is still a bounded subset of $\R^d$.
Now note $(0,\xi)\times\scrB_r^{d-1}=\fZ_{r^{d-1}\xi} D_r^{-1}$,
and $\oQ_{\rho}(\vecq,\vecv)D_r=\oQ_{r\rho}(\vecq,\vecv)$.
Hence the assumption [P3] %
implies
$\lim_{\xi\to\infty}\limsup_{\rho\to0}M(\rho,\xi)=0$,
and the lemma is proved.
\end{proof}

\begin{proof}[Proof of Proposition \ref{GENLIMITcor1}]
Fix any $\Lambda\in P(\T^1(\R^d))$ having a density $\Lambda'\in\C_c(\T^1(\R^d))$
with respect to $d\vecq\,d\vecv$.
By Proposition \ref{GENLIMITprop1},
the left hand side of \eqref{GENLIMITdisc12} equals
$\lim_{T_1\to\infty}\lim_{\rho\to0}J_\rho(1,T_1)$.
Also by Lemma \ref{ASS:bdfreepathgenlem},
\begin{align}\label{GENLIMITdisc2}
\lim_{T_1\to\infty}\limsup_{\rho\to0}
\rho^{d(d-1)}\int_{\T^1(\R^d)}I(\tau_\rho(\vecq,\vecv)\geq{T_1})
\,\Lambda'(\rho^{d-1}\vecq,\vecv)\,d\vecq\,d\vecv=0,
\end{align}
and hence recalling \eqref{GENLIMITdisc13} we have
\begin{align*}
\lim_{T_1\to\infty}\lim_{\rho\to0}J_\rho(1,T_1)
=\lim_{\rho\to0}\rho^{d(d-1)}\int_{\T^1(\R^d)}\,\Lambda'(\rho^{d-1}\vecq,\vecv)\,d\vecq\,d\vecv
=1.
\end{align*}
\end{proof}

\begin{proof}[Proof of Theorem \ref{GENLIMITthm}]
It suffices to prove that for any given $f\in\C_b(N_s(\scrX))$,
$\mu_\rho^{(\Lambda)}(f)\to\mu(f)$,
or in other words,
\begin{align}\label{GENLIMITthmpf1}
\lim_{\rho\to0}\rho^{d(d-1)}\int_{\T^1(\R^d)}f(\scrQ_\rho(\vecq,\vecv))\,\Lambda'(\rho^{d-1}\vecq,\vecv)\,d\vecq\,d\vecv
=\mu(f).
\end{align}
Without loss of generality we may assume $\Lambda'\in\C_c(\T^1(\R^d))$.
Now \eqref{GENLIMITthmpf1} follows by taking $T_1\to\infty$ in Proposition \ref{GENLIMITprop1}
and changing the order of limits;
this is justified by \eqref{GENLIMITdisc2}.
\end{proof}

\begin{remark}\label{ASS:bdfreepathrem}
We used assumption [P3] %
for the derivation of Theorem \ref{GENLIMITthm};
cf.\ Lemma \ref{ASS:bdfreepathgenlem}.
On the other hand, let us note that \textit{if}
the statement of Theorem \ref{GENLIMITthm} holds for some fixed $\mu\in P(N_s(\scrX))$,
i.e.\ $\mu_\rho^{(\Lambda)}\xrightarrow[]{\textup{ w }}\mu$ as $\rho\to0$ for each fixed
$\Lambda\in \Pac(\T^1(\R^d))$,
and if furthermore 
$\mu(\{\emptyset\})=0$, %
then \textit{the condition [P3] %
must hold.}
Indeed, $\mu(\{\emptyset\})=0$ implies that for any $\ve>0$ there exists $R>0$ such that
\begin{align}\label{ASS:bdfreepathrempf1}
\mu(\{Y\col Y\cap((-R,R)\times\scrB_R^{d-1}\times\Sigma)=\emptyset\})<\ve.
\end{align}
Also Theorem \ref{GENLIMITthm} implies that $\mu$ is invariant under translations and under $\{D_r\}_{r>0}$
(cf.\ the proof of   
Proposition \ref{GENLIMITASLINVprop} below);
hence from \eqref{ASS:bdfreepathrempf1} it follows that
\begin{align}\label{ASS:bdfreepathrempf2}
\mu(A_\xi)<\ve
\qquad\text{with }\:
A_\xi=\{Y\col Y\cap(\fZ_\xi\times\Sigma)=\emptyset\},
\:
\xi=2R^d.
\end{align}
But $A_\xi$ is a closed subset of $N_s(\scrX)$;
hence $\mu_\rho^{(\Lambda)}\xrightarrow[]{\textup{ w }}\mu$
implies $\limsup_{\rho\to0}\mu_\rho^{(\Lambda)}(A_\xi)\leq\mu(A_\xi)<\ve$.
Applying this for $\Lambda=(\vol\times\omega)(B)^{-1}(\vol\times\omega)_{|B}$
and $\ve\to0$,
it follows that [P3] %
holds.
\end{remark}

\section{Properties of the point process $\Xi$}
\label{XiPropSec}

In this section we prove some important properties of the point process $\Xi$ with distribution $\mu$
defined by \eqref{GENMUdef}.
Our first result is that $\mu$ %
is invariant under the group generated by 
translations, $\{D_r\}$ and $\SO(d-1)$.
\begin{prop}\label{GENLIMITASLINVprop}
For any Borel subset $A\subset N_s(\scrX)$,
$\vecx\in\R^d$, $r>0$ and $k\in\SO(d-1)$,
\begin{align*}%
\mu(A)=\mu(A+\vecx)=\mu(AD_r)=\mu(Ak).
\end{align*}
\end{prop}
\begin{proof}
To prove translation invariance we prove that for any $f\in\C_b(N_s(\scrX))$ and $\vecx\in\R^d$,
if $f_\vecx(Y):=f(Y+\vecx)$ then $\mu(f_\vecx)=\mu(f)$.
Take any $\Lambda\in P(\T^1(\R^d))$ with $\Lambda'\in\C_c(\T^1(\R^d))$; then by Theorem \ref{GENLIMITthm},
\begin{align*}
\mu(f_\vecx)
&=\lim_{\rho\to0}\int_{\T^1(\R^d)}f\bigl(\scrQ_\rho(\rho^{1-d}\vecq,\vecv)+\vecx\bigr)\,\Lambda'(\vecq,\vecv)\,d\vecq\,d\vecv.
\end{align*}
Writing $\vecx_\rho:=(x_1,\rho^dx_2,\ldots,\rho^dx_d)$ we have
$\scrQ_\rho(\rho^{1-d}\vecq,\vecv)+\vecx=\scrQ_\rho(\rho^{1-d}(\vecq-\vecx_\rho R(\vecv)^{-1}),\vecv)$
for almost all $(\vecq,\vecv)$;
hence we get
\begin{align*}
\mu(f_\vecx)
&=\lim_{\rho\to0}\int_{\T^1(\R^d)}f\bigl(\scrQ_\rho(\rho^{1-d}\vecq,\vecv)\bigr)
\,\Lambda'(\vecq+\vecx_\rho R(\vecv)^{-1},\vecv)\,d\vecq\,d\vecv,
\\
&=\lim_{\rho\to0}\int_{\T^1(\R^d)}f\bigl(\scrQ_\rho(\rho^{1-d}\vecq,\vecv)\bigr)
\,\Lambda'(\vecq+x_1\vecv,\vecv)\,d\vecq\,d\vecv
=\mu(f),
\end{align*}
where the second equality follows using $\lim_{\rho\to0}\vecx_\rho=x_1\vece_1$ and $\Lambda'\in\C_c$,
and the third equality follows by again using Theorem \ref{GENLIMITthm}.

The invariance under $\{D_r\}$ is proved by a similar argument using Theorem \ref{GENLIMITthm}
and\linebreak
$\scrQ_\rho(\rho^{1-d}\vecq,\vecv)D_r=\scrQ_{r\rho}(\rho^{1-d}\vecq,\vecv)$.

Finally let $k\in\SO(d-1)$ and let $A$ be a Borel set in $N_s(\scrX)$.
It follows from Theorem \ref{GENLIMITthm}
that the measure $\mu$ does not depend on the choice of $\fD$ and $c$.
In particular, taking $c=0$ and replacing $\fD$ by $\fD k^{-1}$ in \eqref{GENMUdef}, we have
\begin{multline*}
\mu(A)=c_\scrP\int_{\Sigma}\int_0^{\infty}\int_{\fD k^{-1}}
\omu_{{\vs}}^{(\tau,\vecb)}\bigl(\bigl\{Y\col Y\cap((0,\tau)\times\fD k^{-1}) %
=\emptyset\text{ and }Y\in A\bigr\}\bigr)
\\ \times d\vecb\,d\tau\,d\mm({\vs}).
\end{multline*}
It follows from [Q1] %
and the definition of $\omu_{\vs}^{(\vecx)}$
that $\omu_{\vs}^{(\vecx)}(B)=\omu_{\vs}^{(\vecx k)}(Bk)$
for all Borel sets $B\subset N_s(\scrX)$ and all $\vs\in\Sigma$, $\vecx\in\R^d$.
Hence the integrand in the previous expression can be replaced by
\begin{align*}
\omu_{{\vs}}^{(\tau,\vecb k)}\bigl(\bigl\{Y\col Y\cap\fC_\tau=\emptyset\text{ and }Y\in Ak\bigr\}\bigr),
\end{align*}
and substituting now $\vecb=\vecb_{\text{new}}k^{-1}$ we obtain 
$\mu(A)=\mu(Ak)$.
\end{proof}
Also the property [Q2] %
extends to $\Xi$:
\begin{lem}\label{Xie1simplelem}
$\mu(\{Y\in N_s(\scrX)\col\exists x_1\in\R\textrm{ s.t.\ }\#(Y\cap(\{x_1\}\times\R^{d-1}))>1\})=0$.
\end{lem}
\begin{proof}
Let $A=\{Y\col\exists x_1\in\R\textrm{ s.t.\ }\#(Y\cap(\{x_1\}\times\R^{d-1}))>1\}$.
Then by \eqref{GENMUdef}, applied with $c=0$ for simplicity, it suffices to prove that
$\omu_\vs^{(\tau,\vecb)}(A)=0$ for all $\vs\in\Sigma$, $\tau>0$, $\vecb\in\fD$.
However it follows from the definitions of $A$ and $\omu_\vs^{(\tau,\vecb)}$ that
\begin{align*}
\omu_\vs^{(\tau,\vecb)}(A)
=\mu_\vs\bigl(\bigl\{Y\col Y\in A\:\text{ or }\:Y\cap(\{0\}\times\R^{d-1}\times\Sigma)\not\subset\{(\bn,\vs)\}\bigr\}\bigr).
\end{align*}
Hence $\omu_\vs^{(\tau,\vecb)}(A)=0$ follows as a consequence of [Q2] %
and Lemma \ref{GROUNDINTENSITYlem2}
(applied with $B=\{0\}\times\R^{d-1}\times\Sigma\setminus\{(\bn,\vs)\}$;
recall also $\Sigma'=\Sigma$; cf.\ \eqref{SIGMAeqp}).
\end{proof}

Next we prove that the probability of $\Xi$ having empty intersection with a large ball is small
(just as for $\Xi_\vs$; cf.\ [Q3]). %
\begin{lem}\label{GENnonemptylem}
For every $\ve>0$ there is some $R>0$ such that, for every $\vecx\in\R^d$,
\begin{align*}
\mu(\{Y\in N_s(\scrX)\col Y\cap\scrB^d(\vecx,R)=\emptyset\})<\ve.
\end{align*}
\end{lem}
\begin{proof}
By the translation invariance of $\mu$ (cf.\ Proposition \ref{GENLIMITASLINVprop}),
it suffices to prove the claim for $\vecx=\bn$.
Now by \eqref{GENMUdef}, our task is to prove
\begin{multline}\label{GENnonemptylempf1}
\lim_{R\to\infty}
\int_{\Sigma}\int_0^{\infty}\int_{\fD}
\omu_{{\vs}}^{(c+\tau,\vecb)}\bigl(\bigl\{Y\in N_s(\scrX)\col Y\cap\fC_\tau=\emptyset\text{ and }Y\cap\scrB_R^d=\emptyset
\bigr\}\bigr)\\ \times d\vecb\,d\tau\,d\mm({\vs})=0.
\end{multline}
However it follows from the definition of $\omu_{\vs}^{(\vecx)}$ and [Q3] %
that the integrand in the last expression tends pointwise to $0$ as $R\to\infty$.
Hence \eqref{GENnonemptylempf1} follows by Lebesgue's Dominated Convergence Theorem,
using the majorant coming from Proposition \ref{GENLIMITcor1}.
\end{proof}

Let $\scrN$ be the Borel $\sigma$-algebra of $N_s(\scrX)$.
The next proposition shows (when applied with $A=N_s(\scrX)$) that the intensity measure of $\Xi$ equals
$c_\scrP\cdot\mu_\scrX$,
and furthermore %
that the function $\scrX\times\scrN\to[0,1]$,
$((\vecx,\vs),A)\mapsto\omu_\vs^{(\vecx)}(A)$ gives 
the Palm distributions of $\Xi$ %
(cf.\ \cite[Ch.\ 10]{kallenberg86}).
\begin{prop}\label{GENPALMprop}
For any Borel sets $B\subset\scrX$ and $A\subset N_s(\scrX)$,
\begin{align*}%
\int_{A}\#(Y\cap B)\,d\mu(Y)=c_\scrP\int_{(\vecx,\vs)\in B} \omu_\vs^{(\vecx)}(A)\,d\vecx\,d\mm(\vs).
\end{align*}
\end{prop}
\begin{proof}
By \eqref{GENMUdef}, applied with $c=0$, we have
\begin{align}\notag
&\int_{A}\#(Y\cap B)\,d\mu(Y) \\
& =c_\scrP\int_{\Sigma}\int_0^{\infty}\int_{\fD}\int_{A} \notag
I(Y\cap\fC_\tau=\emptyset)\#(Y\cap B)\,d\omu_{{\vs}}^{(\tau,\vecb)}(Y)\,d\vecb\,d\tau\,d\mm({\vs})
\\\label{GENPALMpropPF1}
&=c_\scrP\int_{\Sigma}\int_{\R_{>0}\times\fD}\int_{N_s(\scrX)}
I\big(Y_{(\vecx,\vs)}\in A\big)\,
I\big(Y_{(\vecx,\vs)}\cap\fC_{x_1}=\emptyset\big)\\ 
& \hspace{120pt} \times \#(Y_{(\vecx,\vs)}\cap B)\,d\mu_{{\vs}}(Y)\,d\vecx\,d\mm({\vs}). \notag
\end{align}

Now assume $B\subset(0,\eta)\times\fD\times\Sigma$ for some $\eta>0$.
Set $\fD':=\fD-\fD$ (this is an open bounded subset of $\R^{d-1}$), and
$E_\eta:=\{Y\in N_s(\scrX)\col Y\cap((-\eta,0)\times\fD')=\emptyset\}$.
Then for every $(\vecx,\vs)\in B$ and $Y\in E_\eta$,
if $Y_{(\vecx,\vs)}\in A$ then the integrand in the last expression in \eqref{GENPALMpropPF1} 
is $\geq1$.
Also $\mu_\vs(E_\eta)\geq1-c_\scrP\eta\vol(\fD')$, by Lemma \ref{GROUNDINTENSITYlem2}
(and \eqref{SIGMAeqp}).
Hence
\begin{align*}
\int_{A} &\#(Y\cap B)\,d\mu(Y)
\geq c_\scrP\int_{(\vecx,\vs)\in B}\int_{E_\eta}
I\big(Y_{(\vecx,\vs)}\in A\big)\,d\mu_\vs(Y)\,d\vecx\,d\mm(\vs)
\\
&\geq c_\scrP\int_{(\vecx,\vs)\in B}\Bigl(\mu_\vs\bigl(\bigl\{Y\col
Y_{(\vecx,\vs)}\in A\bigr\}\bigr)-c_\scrP\eta\vol(\fD')\Bigr)
\,d\vecx\,d\mm(\vs)
\\
&\geq c_\scrP\int_{(\vecx,\vs)\in B}\omu_\vs^{(\vecx)}(A)\,d\vecx\,d\mm(\vs)-c_\scrP^2\eta^2\vol(\fD')\vol(\fD).
\end{align*}
On the other hand, for any $\vecx\in\R_{>0}\times\fD$, $\vs\in\Sigma$ and $Y\in N_s(\scrX)$, note that 
\begin{align}\label{BASICYxvsintBfact}
\#(Y_{(\vecx,\vs)}\cap B)\leq I((\vecx,{\vs})\in B)+\#((Y+\vecx)\cap B).
\end{align}
Using $B\subset(0,\eta)\times\fD\times\Sigma$
we see that every point in the intersection $(Y+\vecx)\cap B$
must come from a point $\vecy\in Y$ with
$0<y_1+x_1<\eta$ and $(y_2,\ldots,y_d)\in\fD'$.
If the integrand in \eqref{GENPALMpropPF1} is non-zero then for this point $\vecy$ we also have
$\vecy+\vecx\notin\fC_{x_1}$;
thus $\vecy+\vecx\in B\setminus\fC_{x_1}$,
which forces $0<x_1<\eta$.
Hence
\begin{align*}
\int_{A}\#(Y\cap B)\,d\mu(Y)
&\leq c_\scrP\int_{(\vecx,\vs)\in B}\int_{N_s(\scrX)}
I\big(Y_{(\vecx,\vs)}\in A\big)\,d\mu_{{\vs}}(Y)\,d\vecx\,d\mm({\vs})
\\
&
\qquad +c_\scrP\int_{(0,\eta)\times\fD\times\Sigma}\int_{N_s(\scrX)}
\#((Y+\vecx)\cap B)\,d\mu_{{\vs}}(Y)\,d\vecx\,d\mm({\vs})
\\
&\leq c_\scrP\int_{(\vecx,\vs)\in B}\omu_\vs^{(\vecx)}(A)\,d\vecx\,d\mm(\vs)
+c_\scrP^2\eta^2\vol(\fD)^2,
\end{align*}
where we again used Lemma \ref{GROUNDINTENSITYlem2}.

Let us write 
\begin{align*}
\delta(A,B):=\int_{A}\#(Y\cap B)\,d\mu(Y)-c_\scrP\int_{(\vecx,\vs)\in B}\omu_\vs^{(\vecx)}(A)\,d\vecx\,d\mm(\vs).
\end{align*}
We have proved above that if $B\subset(0,\eta)\times\fD\times\Sigma$ then
\begin{align}\label{GENPALMpropPF2}
\bigl|\delta(A,B)\bigr|\leq c_\scrP^2\eta^2\vol(\fD)\vol(\fD').
\end{align}
However from the definition of $\omu_\vs^{(\vecx)}$ and the fact that $\mu$ is translation invariant
(cf.\ Prop.\ \ref{GENLIMITASLINVprop}),
it follows that for any Borel sets $B\subset\scrX$ and $A\subset N_s(\scrX)$, and any $\vecx\in\R^d$,
\begin{align}\label{GENPALMpropPF3}
\delta(A+\vecx,B+\vecx)=\delta(A,B).
\end{align}
Using this relation it follows in particular that \eqref{GENPALMpropPF2} holds 
whenever $B\subset(c,c+\eta)\times\fD\times\Sigma$ for some $c\in\R$.
Furthermore $\delta(A,B)$ is additive in the second argument,
i.e.\ $\delta(A,\cup_{j=1}^kB_j)=\sum_{j=1}^k\delta(A,B_j)$ whenever $B_1,\ldots,B_k$ are pairwise disjoint.
Combining the last two facts one shows that for any $\eta>0$ and $k\in\Z^+$,
if $B\subset(0,k\eta)\times\fD\times\Sigma$ then
\begin{align}\label{GENPALMpropPF4}
\bigl|\delta(A,B)\bigr|\leq kc_\scrP^2\eta^2\vol(\fD)\vol(\fD').
\end{align}
Now given any Borel sets $B\subset(0,1)\times\fD\times\Sigma$ and $A\subset N_s(\scrX)$,
applying \eqref{GENPALMpropPF4} with $\eta=1/k$ and $k\to\infty$ we conclude that
$\delta(A,B)=0$.
Finally this relation is extended to hold for general $B$,
again using \eqref{GENPALMpropPF3} and the additivity of $\delta(A,B)$ with respect to $B$
(which also holds for any countable collection of pairwise disjoint sets $B_1,B_2,\ldots$).
\end{proof}

\chapter{First collisions}
\label{FIRSTCOLLISIONsec}

As a first step in our proof of a limiting Markov process and limiting evolution equation,
we will prove a result %
on the limiting joint distribution of the free path length,
impact parameter and the mark of the scatterer which is hit when
starting from random initial conditions.
The precise statement is given in Theorem \ref{Thm2gen} in Section \ref{FIRSTCOLLsec} below.

Throughout this section we assume the hypotheses [P1-3] and [Q1-3] stated in
Section \ref{ASSUMPTLISTsec};
{\blu we fix once and for all, a choice of a subset $\scrE\subset\scrP$ as in [P2],
and we furthermore assume,
without loss of generality,
that \eqref{SIGMAeqp} holds.}

\section{The transition kernel}
\label{TRANSKERsec}
Our first goal is to define the transition kernel, which is the limiting density function appearing in 
Theorem \ref{Thm2gen}.
We will identify $\R^{d-1}$ with the subspace $\{0\}\times\R^{d-1}$ of $\R^d$;
in particular for $\vecx=(x_1,\ldots,x_d)\in\R^d$ we view interchangeably
the point $\vecx_\perp$ (cf.\ \eqref{Xperpdef1})
as $(0,x_2,\ldots,x_d)$ or $(x_2,\ldots,x_d)$.
We extend $\vecx\mapsto\vecx_\perp$ to a map on $\scrX$ through $(\vecx,\vs)\mapsto(\vecx,\vs)_\perp:=(\vecx_\perp,\vs)$.
\label{perpdef2}
Thus $\scrX_\perp$ \label{Xperp} can be identified with $\R^{d-1}\times\Sigma$.
We set\label{OmegaDEF}
\begin{align*}
\Omega:=\scrB_1^{d-1}\times\Sigma\subset\scrX_\perp.
\end{align*}
We endow $\scrX_\perp$ with the measure
\begin{align}\label{ppdefG}
\mu_\Omega = \frac1{v_{d-1}}\vol_{\R^{d-1}}\times\mm,
\end{align}
where $v_{d-1}=\vol(\scrB_1^{d-1})$.
\label{vdm1def}
Note that $\mu_\Omega$ restricts to a probability measure on $\Omega$.
We introduce the reflection map
\begin{align}\label{iotaDEF}
\iota:\scrX\to\scrX,\qquad \iota(x_1,\vecx,{\vs})=(x_1,-\vecx,{\vs})
\hspace{40pt} %
(x_1\in\R,\:\vecx\in\R^{d-1},\:{\vs}\in\Sigma).
\end{align}
Note that $\iota$ preserves $\scrX_\perp$ and $\Omega$,
and using our identifications we have
$\iota(\vecx,{\vs})=(-\vecx,{\vs})$ for $(\vecx,{\vs})\in\scrX_\perp$.
Recall that $\fZ_\xi=(0,\xi)\times\scrB_1^{d-1}$, where in the following $\xi\in(0,\infty]$.\label{Zinfty}
Recall also the convention introduced in Section \ref{GENLIMITsec},
that for $A\subset\R^d$ and $B\subset\scrX$, 
we write ``$A\cap B$'' or ``$B\cap A$'' for $B\cap(A\times\Sigma)$.
In a similar vein, we may often speak of a point in $\scrX$ referring just to its $\R^d$-component;
for example, for $(\vecx,\vs)\in\scrX$ we may call the number $\vecx\cdot\vece_1$
``the $\vece_1$-coordinate of $(\vecx,\vs)$''.

We now define the map
\begin{align}\label{DeltaDEF}
\vecz:N_s(\scrX)\to\Delta:=(\R_{>0}\times\Omega)\sqcup\{\undef\}  %
\end{align}
as follows.
Given ${Y}\in N_s(\scrX)$,
let $\vecz=\vecz(Y)$ be that point in $Y\cap\fZ_\infty$ which has minimal $\vece_1$-coordinate;
if there does not exist a unique such point\footnote{i.e., if $Y\cap\fZ_\infty$ is empty or if there are two or more points in
$Y\cap\fZ_\infty$ with minimal $\vece_1$-coordinate.}
then let $\vecz({Y})=\undef$.
Here ``$\undef$'' is a dummy element not in $\R_{>0}\times\Omega$ and we provide
$\Delta$ with the disjoint union topology.
{\blu The motivation for the definition of the map $\vecz$
is that when working in the particle's coordinate frame
described at the beginning of Section \ref{sec:outline},
as $\rho\to0$, 
the scatterers are thin ellipsoids which 
approach codimension one unit discs orthogonal to $\vece_1$,
centered at the points of
$\scrQ_\rho(\vecq,\vecv)$.
Hence, formally, in the limit $\rho\to0$, 
the point $\vecz(\scrQ_\rho(\vecq,\vecv))$ corresponds to 
the center of the scatterer
which the particle will next collide with.}

\begin{lem}\label{BASICkidpreplemG}
The map $\vecz$ is Borel measurable,
and $\mu_{\vs}(\{Y\in N_s(\scrX)\col \vecz(Y-\vecx)=\undef\})=0$
for all $\vecx\in\R^d$ and ${\vs}\in\Sigma$.
\end{lem}
\begin{proof}
For any Borel subset $B\subset\R_{>0}\times\Omega$ and $0<r<s$, we set
$\fZ_{r,s}:=\fZ_s\setminus\fZ_r$ and 
\begin{align*}
A[r,s,B]:=\{Y\in N_s(\scrX)\col Y\cap\fZ_r=\emptyset\text{ and } & \#(Y\cap\fZ_{r,s})=\#((Y\cap B)\cap\fZ_{r,s})=1\}.
\end{align*}
Then $\vecz^{-1}(B)=\cup_{N=1}^\infty \cap_{n=N}^\infty\cup_{k=1}^\infty A\bigl[\tfrac kn,\tfrac{k+1}n,B\bigr]$,
which is a Borel subset of $N_s(\scrX)$.
Also $\vecz^{-1}(\{\undef\})=N_s(\scrX)\setminus \vecz^{-1}(\R_{>0}\times\Omega)$.
Hence the map $\vecz$ is Borel measurable.
Next, using 
$\fZ_\infty+\vecx\supset\scrB^d\bigl(R\vece_1+\vecx D_R^{-1},R\bigr)D_R$ ($\forall R>0$)
together with [Q3], %
\eqref{SIGMAeqp} and Lemma \ref{DIAGINVlem},
it follows that $\mu_{\vs}(\{Y\col (Y-\vecx)\cap\fZ_\infty=\emptyset\})=0$
for any $\vecx\in\R^d$ and ${\vs}\in\Sigma$.
The second statement of the lemma
follows from this fact and [Q2]. %
\end{proof}

\begin{lem}\label{zXisigmmxcontLEM}
The distribution of the random point $\vecz(\Xi_{\vs}-\vecx)$ in $\R_{>0}\times\Omega$
depends continuously on $(\vecx,{\vs})\in\scrX_\perp$.
\end{lem}
\begin{proof}
Let $C$ be the set of all $Y\in N_s(\scrX)$ which satisfy
$\vecz(Y)\neq\undef$ and $Y\cap\partial\fZ_{\infty}=\emptyset$.
Using Lemma \ref{BASICkidpreplemG}, Lemma \ref{GROUNDINTENSITYlem2} and \eqref{SIGMAeqp},
it follows that $\mu_\vs(\{Y\col Y-\vecx\in C\})=0$ for all $(\vecx,{\vs})\in\scrX_\perp$.
Furthermore one verifies that the map $\vecz$ is continuous at each point in $C$,
i.e.\ $\vecz(Y_n)\to\vecz(Y)$ holds whenever $Y_n\to Y$ in $N_s(\scrX)$ with $Y\in C$.
In view of these observations, the lemma follows from 
the generalized continuous mapping theorem, \cite[Thm.\ 4.27]{kallenberg02}.
\end{proof}

Given $\vecomega'=(\vecx,{\vs})\in\scrX_\perp$,
let us write $\kappa(\vecomega';\cdot)$ for the distribution of the 
random point $\iota(\vecz(\Xi_{{\vs}}-\vecx))$ in $\R_{>0}\times\Omega$;
thus for any Borel set $B\subset\R_{>0}\times\Omega$,
\begin{align}\label{kappadef}
\kappa((\vecx,{\vs});B):=\mu_{\vs}(\{Y\in N_s(\scrX)\col \iota(\vecz(Y-\vecx))\in B\}).
\end{align}
By Lemma \ref{GROUNDINTENSITYlem2} and \eqref{SIGMAeqp},
$\kappa(\vecomega';B)\leq c_\scrP\mu_\scrX(B)=c_\scrP v_{d-1}\int_B\,d\xi\,d\mu_{\Omega}(\vecomega)$
for every Borel set $B\subset\R_{>0}\times\Omega$.
We define $k(\vecomega',\cdot,\cdot)$ to be the corresponding probability density; that is,
we define the function
\begin{align}\label{kdef}
k:\scrX_\perp\times\R_{>0}\times\Omega\to[0,c_\scrP v_{d-1}]
\end{align}
so that for each $\vecomega'\in\scrX_\perp$,
$k(\vecomega',\cdot,\cdot)$ is uniquely defined as an element in
$\L^1(\R_{>0}\times\Omega,d\xi\,d\mu_\Omega)$,
and $\kappa(\vecomega',B)=\int_B k(\vecomega',\xi,\vecomega)\,d\xi\,d\mu_{\Omega}(\vecomega)$
for all Borel sets $B\subset\R_{>0}\times\Omega$.

\begin{lem}\label{UNIFKBOUNDGlem1}
We have $\kappa(\vecomega';[\xi,\infty)\times\Omega)\to0$
as $\xi\to\infty$, uniformly over all $\vecomega'\in\scrX_\perp$.
\end{lem}
\begin{proof}
Using $\scrB^d\bigl(R\vece_1+\vecx D_R^{-1},R\bigr)D_R 
\subset\fZ_{2R^d}+\vecx$ ($\forall R>0$)
together with [Q3], %
Lemma \ref{DIAGINVlem} and \eqref{SIGMAeqp},
we have $\mu_{\vs}(\{Y\col (Y-\vecx)\cap\fZ_\xi=\emptyset\})\to0$ as $\xi\to\infty$,
uniformly over all $\vecx\in\R^d$ and ${\vs}\in\Sigma'$.
\end{proof}

\begin{lem}\label{CONTINTEGRALlemG2}
Let $\C_b(\R_{>0}\times\Omega)$ be the space of bounded continuous functions on $\R_{>0}\times\Omega$,
equipped with the supremum norm.
The integral
\begin{align}\label{CONTINTEGRALlemG2res}
\int_{\R_{>0}\times\Omega}f(\vecy)\kappa(\vecomega'; d\vecy)
\end{align}
depends continuously on $\langle \vecomega',f\rangle$ in $\scrX_\perp\times \C_b(\R_{>0}\times\Omega)$.
\end{lem}
\begin{proof}
By Lemma \ref{zXisigmmxcontLEM}, the integral in \eqref{CONTINTEGRALlemG2res}
depends continuously on $\vecomega'$ for any fixed $f\in\C_b(\R_{>0}\times\Omega)$.
Now the desired conclusion follows by also noticing that
for fixed $\vecomega'$,
the expression in \eqref{CONTINTEGRALlemG2res}
is  a bounded linear functional of $f\in\C_b(\R_{>0}\times\Omega)$ of norm at most $1$.
\end{proof}

\begin{remark}\label{CONTINTEGRALlemG2REM}
In terms of $k$, Lemma \ref{CONTINTEGRALlemG2}
says that the integral
\begin{align*}%
\int_{\R_{>0}\times\Omega}f(\xi,\vecomega)k(\vecomega',\xi,\vecomega)\,d\xi\,d\mu_{\Omega}(\vecomega)
\end{align*}
depends continuously on $\langle \vecomega',f\rangle$ in $\scrX_\perp\times \C_b(\R_{>0}\times\Omega)$.
\end{remark}

We next introduce the corresponding notions for the macroscopic limit point process $\Xi$
introduced in Section \ref{GENLIMITsec}.
Recall that we write $\mu\in P(N_s(\scrX))$ for the distribution of $\Xi$.
The result of Lemma \ref{BASICkidpreplemG}
carries over to the present situation:
\begin{lem}\label{BASICkidpreplemgen}
We have 
$\mu(\{Y\in N_s(\scrX)\col \vecz(Y %
)=\undef\})=0$.
\end{lem}
\begin{proof}
Using $\fZ_\infty\supset\scrB^d(R\vece_1,R)D_R$ ($\forall R>0$)
together with Lemma \ref{GENnonemptylem} and Prop.\ \ref{GENLIMITASLINVprop} it follows that
$\mu(\{Y\col Y\cap\fZ_\infty=\emptyset\})=0$.
The lemma follows from this fact and Lemma~\ref{Xie1simplelem}.
\end{proof}

Let us write $\kappa^{\g}\in P(\R_{>0}\times\Omega)$ for the distribution of the random point
$\iota(\vecz(\Xi))$ in $\R_{>0}\times\Omega$.
\label{kappagDEF}
(The ``$\g$'' stands for ``generic initial condition''.)
Thus for any Borel set $B\subset\R_{>0}\times\Omega$,
\begin{align}\label{kappagdef}
\kappa^{\g}(B):=\mu(\{Y\in N_s(\scrX)\col \iota(\vecz(Y))\in B\}).
\end{align}

By Proposition \ref{GENPALMprop} (applied with $A=N_s(\scrX)$), 
$$\kappa(B)\leq c_\scrP\mu_\scrX(B)=c_\scrP v_{d-1}\int_B\,d\xi\,d\mu_{\Omega}(\vecomega)$$ 
for every Borel set $B\subset\R_{>0}\times\Omega$.
Hence as before, we can consider the corresponding probability density%
\begin{align}\label{kgdef}
k^{\g}:\R_{>0}\times\Omega\to[0,c_\scrP v_{d-1}],
\end{align}
so that $\kappa^{\g}(B)=\int_B k^{\g}(\xi,\vecomega)\,d\xi\,d\mu_{\Omega}(\vecomega)$
for all Borel sets $B\subset\R_{>0}\times\Omega$.
Note that $k^{\g}$ is uniquely defined as an element in
$\L^1(\R_{>0}\times\Omega,d\xi\,d\mu_\Omega)$.

\section{Limit theorem for the first collision}
\label{FIRSTCOLLsec}

From now on, we will say that a scatterer $\scrB^d(\vecq,\rho)$
($\vecq\in\scrP$)
is \textit{separated} if $\|\vecq'-\vecq\|>2\rho$ for all $\vecq'\in\scrP\setminus\{\vecq\}$.
Recall that
\begin{align*}
\fw(\rho)=\T^1(\scrK_\rho^\circ) \cup \T^1(\partial\scrK_\rho)_{\out}.
\end{align*}
Let $\fw(1;\rho)$ be the set of those initial conditions 
$(\vecq,\vecv)\in\fw(\rho)$ which lead to a collision 
with a separated scatterer in finite time,
viz., those $(\vecq,\vecv)\in\fw(\rho)$ which have 
$\tau_1(\vecq,\vecv;\rho)<\infty$
and for which $\vecq+\tau_1(\vecq,\vecv;\rho)\vecv$ 
lies on the boundary of a separated scatterer.
For $(\vecq,\vecv)\in\fw(1;\rho)$ we write $\vecq^{(1)}=\vecq^{(1)}(\vecq,\vecv;\rho)$ for the center of the
\label{q1scattererdef}
scatterer causing the first collision,
and let $\vecw_1=\vecw_1(\vecq,\vecv;\rho)\in\UB$ be the 
normalized impact parameter,
defined through $\vecw_1:=(\vecu_1 R(\vecv))_\perp$,
\label{w1def}
where $\vecu_1\in\US$ is the point given by
$\vecq+\tau_1(\vecq,\vecv;\rho)\vecv=\vecq^{(1)}+\rho\vecu_1$.
We then set:
\begin{align*}
\vecomega_1=\vecomega_1(\vecq,\vecv;\rho):=(\vecw_1,{\vs}(\vecq^{(1)}))\in\Omega.
\end{align*}
\label{vecomega1def}

Let $U$ be an open subset of $\US$ and $\vecbeta\in \C_b(U,\R^d)$. %
For $\vecq\in\scrP$ and $\vecv$ random in $U$, %
we will consider a point particle starting at the point
\begin{align}\label{qrhobetavDEF}
\vecq(\vecv)=\vecq_{\rho,\vecbeta}(\vecv):=\vecq+\rho\vecbeta(\vecv).
\end{align}
To avoid pathologies, we assume that
$U,\vecbeta$ are chosen such that
$(\vecbeta(\vecv)+\R_{>0}\vecv)\cap\scrB_1^d=\emptyset$ for all $\vecv\in U$.
We set 
\begin{align}\label{wqrbdef}
\fw_{\vecq,\rho}^{\vecbeta}=\{\vecv\in U\col(\vecq_{\rho,\vecbeta}(\vecv),\vecv)\in\fw(1;\rho)\}.
\end{align}
The following theorem gives the joint limit distribution of $\vecomega_1(\vecq(\vecv),\vecv;\rho)$
and the normalized free path length $\rho^{d-1}\tau_1(\vecq(\vecv),\vecv;\rho)$.
\begin{thm}\label{Thm2gen}
Let $U$ be an open subset of $\US$; let $K$ be a relatively compact subset of $\C_b(U,\R^d)$ such that 
$(\vecbeta(\vecv)+\R_{>0}\vecv)\cap\scrB_1^d=\emptyset$ for all $\vecbeta\in K$, $\vecv\in U$,
and let $\lambda\in\Pac(\US)$ be such that $\lambda(U)=1$.
Then for any $T\geq1$ and $f\in\C_b(U\times\R_{>0}\times\Omega)$,
\begin{align}\label{Thm2genres}
\int_{\fw_{\vecq,\rho}^{\vecbeta}} f\bigl(\vecv,\rho^{d-1} \tau_1(\vecq_{\rho,\vecbeta}(\vecv),\vecv;\rho),
\vecomega_1(\vecq_{\rho,\vecbeta}(\vecv),\vecv;\rho)\bigr)\,d\lambda(\vecv)
\hspace{120pt}
\\\notag
\to\int_{U}\int_0^\infty\int_\Omega
f(\vecv,\xi,\vecomega)
k\bigl(\bigl((\vecbeta(\vecv)R(\vecv))_\perp,{\vs}(\vecq)\bigr),\xi,\vecomega\bigr)
\,d\mu_{\Omega}(\vecomega)\,d\xi\,d\lambda(\vecv)
\end{align}
as $\rho\to0$, uniformly over all $\vecq\in\scrP_T(\rho)$ and $\vecbeta\in K$.
\end{thm}
\begin{remark}\label{Thm2genrem}
Taking $f\equiv1$ and using 
$\int_{\R_{>0}\times\Omega}k(\vecomega',\xi,\vecomega)\,d\xi\,d\mu_{\Omega}(\vecomega)=1$
($\forall\vecomega'$),
one sees that the %
theorem implies in particular that 
$\lambda(\fw_{\vecq,\rho}^{\vecbeta})\to1$ as $\rho\to0$, uniformly over all $\vecq\in\scrP_T(\rho)$
and $\vecbeta\in K$.
\end{remark}
\begin{remark}\label{Thm2genRELCPTrem}
Let $K$ be as in Theorem \ref{Thm2gen},
and let $\oK$ be the closure of $K$ in $\C_b(U,\R^d)$;
this is a compact subset of $\C_b(U,\R^d)$, and clearly every $\vecbeta\in\oK$ satisfies 
$(\vecbeta(\vecv)+\R_{>0}\vecv)\cap\scrB_1^d=\emptyset$ for all $\vecv\in U$.
Hence when proving Theorem \ref{Thm2gen} we may just as well replace $K$ by $\oK$ from the very beginning.
Thus, in the following we will assume that \textit{$K$ is compact.}
\end{remark}

To prepare for the proof of the theorem,
we introduce a %
slightly modified version of the map $\vecz$ from
Section \ref{TRANSKERsec}.
Let $0<\rho<1$.
For each $\vecx\in\fZ_\infty$ we set
\begin{align}\label{xirhoDEF}
\xi_\rho(\vecx)=\inf\{\xi\in\R_{>0}\col\vecx\in\xi\vece_1+\scrB_\rho^dD_\rho\}\in\R_{\geq0}.
\end{align}
Note that $\scrB_\rho^dD_{\rho}$ is the ellipsoid
$\{(x_1/\rho^d)^2+x_2^2+\cdots+x_d^2<1\}$;
hence we indeed have $\xi_\rho(\vecx)\in\R_{\geq0}$
for each $\vecx\in\fZ_\infty$,
with $\xi_\rho(\vecx)=0$ if and only if $\vecx\in\fZ_\infty\cap\overline{\scrB_\rho^d}D_{\rho}$.
\begin{definition}
The map
\begin{align}\label{zrhoDEF}
\vecz_\rho:N_s(\scrX)\to\Delta=(\R_{>0}\times\Omega)\sqcup\{\undef\}
\end{align}
is defined as follows.
For given $Y\in N_s(\scrX)$,     %
if $Y\cap\fZ_\infty\cap\overline{\scrB_\rho^d}D_{\rho}=\emptyset$
and if there exists a unique point $(\vecx,\vs)$ in $Y\cap(\fZ_\infty\times\Sigma)$ which minimizes $\xi_\rho(\vecx)$,
and if furthermore this point satisfies
\begin{align}\label{BALLSEPCOND}
Y\cap(\vecx+\overline{\scrB_{2\rho}^d}D_\rho)=\{\vecx\},
\end{align}
then set $\vecz_\rho(Y):=(\vecx,\vs)$;
otherwise set $\vecz_\rho(Y)=\undef$.
\end{definition}
We prove in Lemma~\ref{zrhomeasLEM} below that $\vecz_\rho$ is measurable.
The motivation for the definition of $\vecz_\rho$ is that
both $\tau_1(\vecq_{\rho,\vecbeta}(\vecv),\vecv;\rho)$
and $\vecomega_1(\vecq_{\rho,\vecbeta}(\vecv),\vecv;\rho)$
can be expressed as simple functions of
$\vecz_\rho(\scrQ_\rho(\vecq,\vecbeta,\vecv))$,
where we recall that the point set $\scrQ_\rho(\vecq,\vecbeta,\vecv)$ was defined in \eqref{XIrhoqvdef}.
The precise statement is as follows. %
\begin{lem}\label{ZTAU1CONNECTIONlem}
Let $\vecbeta\in K$, $\vecq\in\scrP$, $\vecv\in U$, $\rho\in(0,1)$,
and assume that 
$d_\scrP(\vecq)>(1+\|\vecbeta\|)\rho$ and 
$(\vecx,\vs)=\vecz_\rho(\scrQ_\rho(\vecq,\vecbeta,\vecv))\neq\undef$.
Then $\vecv\in\fw_{\vecq,\rho}^{\vecbeta}$,
$\rho^{d-1}\tau_1(\vecq_{\rho,\vecbeta}(\vecv),\vecv;\rho)=\xi_\rho(\vecx)\in\R_{>0}$,
and $\vecomega_1(\vecq_{\rho,\vecbeta}(\vecv),\vecv;\rho) %
=(-\vecx_\perp,\vs)$.
\end{lem}
Recall that we have provided $\C_b(U,\R^d)$ with the supremum norm; thus
$\|\vecbeta\|=\sup_{\vecu\in U}\|\vecbeta(\vecu)\|$.
\begin{proof}
The assumptions 
$(\vecbeta(\vecv)+\R_{>0}\vecv)\cap\scrB_1^d=\emptyset$ and 
$d_\scrP(\vecq)>(1+\|\vecbeta\|)\rho$ imply that
either $\vecq(\vecv)\in\scrK_\rho^\circ$ or else
$\vecq(\vecv)$ lies on the boundary of the scatterer $\scrB^d(\vecq,\rho)$, which is separated,
so that $\vecq(\vecv)\in\T^1(\partial\scrK_\rho)_{\out}$.
Hence $(\vecq(\vecv),\vecv)\in\fw(\rho)$.

Set $Y=\scrQ_\rho(\vecq,\vecbeta,\vecv)$,
so that $(\vecx,\vs)=\vecz_\rho(Y)\in\R_{>0}\times\Omega$.
Also set
\begin{align*}
\tq:=\vecx D_\rho^{-1}R(\vecv)^{-1}+\vecq(\vecv).
\end{align*}
Then $\tq\in\scrP\setminus\{\vecq\}$ and $\vs(\tq)=\vs$,
since $(\vecx,\vs)\in Y=\scrQ_\rho(\vecq,\vecbeta,\vecv)$; cf.\ \eqref{XIrhoqvdef}.
It follows from $\vecz_\rho(Y)\neq\undef$ that
the line segment $\{\xi\vece_1\col\xi\in[0,\xi_\rho(\vecx)]\}$
is disjoint from all the open ellipsoids $\vecx'+\scrB_\rho^dD_{\rho}$ for $(\vecx',\vs')\in Y$,
but $\xi\vece_1\in\vecx+\scrB_\rho^dD_{\rho}$ holds for each $\xi>\xi_\rho(\vecx)$
which lies sufficiently near $\xi_\rho(\vecx)$.
Applying the affine linear map $\vecy\mapsto\vecy D_\rho^{-1}R(\vecv)^{-1}+\vecq(\vecv)$,
using also $(\vecbeta(\vecv)+\R_{>0}\vecv)\cap\scrB_1^d=\emptyset$,
it follows that
\begin{align*}
\tau_1(\vecq(\vecv),\vecv;\rho)=\rho^{1-d}\xi_\rho(\vecx) %
\qquad\text{and}\qquad
\vecq^{(1)}(\vecq(\vecv),\vecv;\rho)=\tq.
\end{align*}
Similarly, \eqref{BALLSEPCOND} implies that the scatterer associated to $\tq$ is separated,
i.e.\ $\|\tq-\vecp\|>2\rho$ for all
$\vecp\in\scrP\setminus\{\tq\}$.
Hence $\vecv\in\fw_{\vecq,\rho}^{\vecbeta}$.
Also %
$\vecq(\vecv)+\tau_1\vecv={\vecq^{(1)}}+\rho\vecu_1$
implies $\vecu_1=\rho^{-1}(\tau_1\vecv-\vecx D_\rho^{-1}R(\vecv)^{-1})$,
and so %
$\vecw_1(\vecq(\vecv),\vecv;\rho)=(\vecu_1R(\vecv))_\perp=-\vecx_\perp$.
\end{proof}

\begin{lem}\label{zrhomeasLEM}
For each $0<\rho<1$, the map $\vecz_\rho:N_s(\scrX)\to\Delta$ is Borel measurable.
\end{lem}
\begin{proof}
For $n\in\Z^+$ and $\vecm\in\Z^d$ we let $C_{n,\vecm}\subset\R^d$ be the cube $C_{n,\vecm}=n^{-1}(\vecm+[0,1)^d)$,
and set $C^{(\rho)}_{n,\vecm}=C_{n,\vecm}+\overline{\scrB_{2\rho}^d}D_\rho$.
For any bounded Borel set $V\subset\R^d$, let
\begin{align*}
S^{(\rho)}(V)=\cup_{N=1}^\infty\cap_{n=N}^\infty\cap_{\vecm\in\Z^d}
\{Y\in N_s(\scrX)\col Y\cap V\cap C_{n,\vecm}=\emptyset\:\text{ or }
\hspace{60pt}
\\
\#(Y\cap V\cap C_{n,\vecm})=\#(Y\cap C^{(\rho)}_{n,\vecm})\}.
\end{align*}
This is clearly a Borel set in $N_s(\scrX)$, and one verifies that
$Y$ lies in $S^{(\rho)}(V)$ if and only if
$Y\cap(\vecy+\overline{\scrB_{2\rho}^d}D_\rho)=\{\vecy\}$ for every point $\vecy\in Y\cap V$.
Next for any Borel set $B\subset\R_{>0}\times\Omega$ 
and $0<r<s$, we set $\fZ_r^{(\rho)}=\{\vecx\in\fZ_\infty\col\xi_\rho(\vecx)<r\}$,
$\fZ_{r,s}^{(\rho)}=\fZ_s^{(\rho)}\setminus\fZ_r^{(\rho)}$,
and
\begin{multline*}
A^{(\rho)}[r,s,B]:=\bigl\{Y\in N_s(\scrX)\col\hspace{15pt} 
Y\cap\fZ^{(\rho)}_r=\emptyset, \\
\#(Y\cap\fZ^{(\rho)}_{r,s})=\#((Y\cap B)\cap\fZ^{(\rho)}_{r,s})=1,
\hspace{20pt}
\text{ and }\: Y\in S^{(\rho)}(\fZ^{(\rho)}_{r,s})\bigr\}.
\end{multline*}
Then
$\vecz_\rho^{-1}(B)=\cup_{N=1}^\infty \cap_{n=N}^\infty\cup_{k=1}^\infty A^{(\rho)}\bigl[\tfrac kn,\tfrac{k+1}n,B\bigr]$.
Hence $\vecz_\rho^{-1}(B)$ is a Borel set in $N_s(\scrX)$.
Also $\vecz_\rho^{-1}(\{\undef\})=N_s(\scrX)\setminus\vecz_\rho^{-1}(\R_{>0}\times\Omega)$ is a Borel set.
Hence the lemma is proved.
\end{proof}

For $0<\rho<1$, define the map
$F_\rho:\Delta\to\Delta$ through
\begin{align}\label{Frhodef}
F_\rho(\vecz)=\begin{cases}
\iota(\xi_\rho(\vecz),\vecz_\perp)&\text{if }\: \vecz\in\R_{>0}\times\Omega
\\
\undef&\text{if }\: \vecz=\undef.
\end{cases}
\end{align}
Let $U,\vecbeta,\lambda$ be given as in Theorem \ref{Thm2gen}.
For $\vecv$ random in $(U,\lambda_{|U})$,
we let $\eta_{\vecq,\rho}^{(\vecbeta,\lambda)}\in P(U\times\Delta)$ be the distribution
of $(\vecv,[F_\rho\circ\vecz_\rho](\scrQ_\rho(\vecq,\vecbeta,\vecv)))$
\label{etaqrhobetalambdaDEF}
and let $\eta_{\vs}^{(\vecbeta,\lambda)}\in P(U\times\Delta)$ be the distribution
of $(\vecv,[\iota\circ\vecz](\Xi_{\vs}-(\vecbeta(\vecv)R(\vecv))_\perp))$,
with $\Xi_{\vs}$ independent from $\vecv$.
The key step in the proof of Theorem \ref{Thm2gen}
is to show that
$\eta_{\vecq,\rho}^{(\vecbeta,\lambda)}$ converges weakly to $\eta_{{\vs}(\vecq)}^{(\vecbeta,\lambda)}$ as $\rho\to0$,
uniformly over $\vecq\in\scrP_T(\rho)$ and $\vecbeta\in K$.
We will establish this in Lemma \ref{KEYyCONVprop3}.
As a first step, we verify in the following lemma that $\eta_{{\vs}}^{(\vecbeta,\lambda)}$
depends continuously on $\vs$ and $\vecbeta$.
\begin{lem}\label{TETACONTlem}
The map $\Sigma\times\C_b(U,\R^d)\to P(U\times\Delta)$, $({\vs},\vecbeta)\mapsto\eta_{{\vs}}^{(\vecbeta,\lambda)}$,
is continuous.
\end{lem}
\begin{proof}
This is a consequence of Lemma \ref{zXisigmmxcontLEM}.
Indeed, given sequences $\vecbeta_n\to\vecbeta$ in $\C_b(U,\R^d)$ and ${\vs}_n\to{\vs}$ in $\Sigma$, 
and a function $f\in \C_b(U\times\Delta)$, our task is to prove that 
\begin{align}\label{TETACONTlempf1}
\int_{U}\int_{N_s(\scrX)}f\bigl(\vecv,[\iota\circ\vecz](Y-(\vecbeta_n(\vecv)R(\vecv))_\perp)\bigr)\,d\mu_{{\vs}_n}(Y)
\,d\lambda(\vecv)
\hspace{100pt}
\\\notag
\to\int_{U}\int_{N_s(\scrX)}f\bigl(\vecv,[\iota\circ\vecz](Y-(\vecbeta(\vecv)R(\vecv))_\perp)\bigr)\,d\mu_{{\vs}}(Y)
\,d\lambda(\vecv)
\qquad\text{as }\: n\to\infty.
\end{align}
Call the inner integral in the left hand side $g_n(\vecv)$, and the inner integral in the right hand side $g(\vecv)$;
then Lemma \ref{zXisigmmxcontLEM} implies that $g_n(\vecv)\to g(\vecv)$ for each fixed $\vecv\in U$.
Furthermore $|g_n(\vecv)|\leq\sup|f|$ and $|g(\vecv)|\leq\sup|f|$
for all $n$ and $\vecv$.
Hence \eqref{TETACONTlempf1}
follows by Lebesgue's Dominated Convergence Theorem.
\end{proof}

In the proof of Lemma \ref{KEYyCONVprop3}
we will apply the continuous mapping theorem for the maps
$F_\rho\circ\vecz_\rho:N_s(\scrX)\to\Delta$ with $\rho\to0$.
For this application we will need the following continuity fact.
\begin{lem}\label{Frhozrhocontlem}
Let $Y,Y_1,Y_2,\ldots\in N_s(\scrX)$ and $\rho_1,\rho_2,\ldots\in(0,1)$ be given
such that $Y_n\to Y$ and $\rho_n\to0$ as $n\to\infty$.
Assume furthermore that $\vecz(Y)\neq\undef$,
$Y\cap\partial\fZ_\infty=\emptyset$,
and $\vecy\cdot\vece_1\neq\vecz(Y)\cdot\vece_1$ for all $\vecy\in Y\setminus\{\vecz(Y)\}$.
Then $F_{\rho_n}(\vecz_{\rho_n}(Y_n))\to \iota(\vecz(Y))$ in $\Delta$.
\end{lem}
\begin{proof}
Let $\xi=\vecz(Y)\cdot\vece_1>0$.
Because of the assumptions, there is some $\ve>0$ such that
\begin{align*}
Y\cap(\overline{V}\times\Sigma)=\{\vecz(Y)\},
\end{align*}
with
\begin{align*}
V:=((-4\ve,\xi+4\ve)\times\scrB_{1+\ve}^{d-1})\cup((\xi-4\ve,\xi+4\ve)\times\scrB_4^{d-1}).
\end{align*}
Since $V$ is open and $\vecz(Y)\in V\times\Sigma$, it follows that $\#(Y_n\cap (V\times\Sigma))=1$ for all large $n$, 
and furthermore if $\vecz_n$ is the unique point in $Y_n\cap (V\times\Sigma)$
then $\vecz_n\to\vecz(Y)$ %
as $n\to\infty$.
It follows that for $n$ sufficiently large we have 
$\vecz_{\rho_n}(Y_n)=\vecz_n$
(here the fact that $V$ contains $(\xi-4\ve,\xi+4\ve)\times\scrB_4^{d-1}$ is used to guarantee that
$\vecz_n$ satisfies the condition \eqref{BALLSEPCOND}).
Hence $F_{\rho_n}(\vecz_{\rho_n}(Y_n))=F_{\rho_n}(\vecz_n)\to \iota(\vecz(Y))$ as $n\to\infty$.
\end{proof}
\begin{lem}\label{KEYyCONVprop3}
We have
$\eta_{\vecq,\rho}^{(\vecbeta,\lambda)}\xrightarrow[]{\textup{ w }}\eta_{{\vs}(\vecq)}^{(\vecbeta,\lambda)}$ as $\rho\to0$,
uniformly over all $\vecq\in\scrP_T(\rho)$ and $\vecbeta\in K$.
\end{lem}
\begin{remark}
Recall that we assume that $K$ is a compact subset of $\C_b(U,\R^d)$;
cf.\ Remark \ref{Thm2genRELCPTrem};
thus $\{\eta_{\vs}^{(\vecbeta,\lambda)}\col{\vs}\in\Sigma,\vecbeta\in K\}$
is a compact subset of $P(U\times\Delta)$,
being a continuous image of the compact set $\Sigma\times K$
(cf.\ Lemma \ref{TETACONTlem}).
Hence the general notion of uniform convergence from
\eqref{UNIFCONVdef1}--\eqref{UNIFCONVdef2} applies.
\end{remark}
\begin{proof}
Let $\rho_n\in(0,1)$, $\vecq_n\in\scrP_T(\rho_n)$, $\vecbeta_n\in\C_b(U,\R^d)$
for $n=1,2,\ldots$,
and assume that $\rho_n\to0$,
${\vs}(\vecq_n)\to{\vs}$ and $\vecbeta_n\to\vecbeta$
as $n\to\infty$, with ${\vs}\in\Sigma$ and $\vecbeta\in K$.
We then claim that 
$\eta_{\vecq_n,\rho_n}^{(\vecbeta_n,\lambda)}\xrightarrow[]{\textup{ w }}\eta_{{\vs}}^{(\vecbeta,\lambda)}$
as $n\to\infty$.
By the same argument as in Lemma \ref{UNIFPORTMlemC}
(using also Lemma \ref{TETACONTlem}),
this will imply the lemma.

Consider the maps
\begin{align*}
H_n:\US\times N_s(\scrX)\to\US\times\Delta,\qquad
H_n(\vecv,Y)=(\vecv,F_{\rho_n}(\vecz_{\rho_n}(Y)))
\end{align*}
and
\begin{align*}
H:\US\times N_s(\scrX)\to\US\times\Delta,\qquad
H(\vecv,Y)=(\vecv,\iota(\vecz(Y))),
\end{align*}
and note that $\eta_{\vecq_n,\rho_n}^{(\vecbeta_n,\lambda)}=\tmu_{\vecq_n,\rho_n}^{(\vecbeta_n,\lambda)}\circ H_n^{-1}$
and $\eta_{{\vs}}^{(\vecbeta,\lambda)}=\tmu_{\vs}^{(\vecbeta,\lambda)}\circ H^{-1}$,
after extending by zero from $P(U\times\Delta)$ to $P(\US\times\Delta)$.
We have $\tmu_{\vecq_n,\rho_n}^{(\vecbeta_n,\lambda)}\xrightarrow[]{\textup{ w }}\tmu_{\vs}^{(\vecbeta,\lambda)}$
by Lemma \ref{BETAUNIFCONVlem}.
Let $C$ be the set of all $Y\in N_s(\scrX)$ satisfying
$\vecz(Y)\neq\undef$, $Y\cap\partial\fZ_\infty=\emptyset$,
and $\vecy\cdot\vece_1\neq\vecz(Y)\cdot\vece_1$ for all $\vecy\in Y\setminus\{\vecz(Y)\}$.
Then by Lemma \ref{Frhozrhocontlem},
for any $\vecv,\vecv_1,\vecv_2,\ldots\in\US$
and $Y,Y_1,Y_2,\ldots\in N_s(\scrX)$ subject to $Y\in C$ and $(\vecv_n,Y_n)\to(\vecv,Y)$ as $n\to\infty$,
we have $H_n(\vecv_n,Y_n)\to H(\vecv,Y)$ as $n\to\infty$.
Furthermore, using the definition of $\mu_{\vs}^{(\vecbeta,\lambda)}$
together with [Q2], [Q3] and Lemma~\ref{GROUNDINTENSITYlem2},
one verifies that $\mu_{\vs}^{(\vecbeta,\lambda)}(C)=1$
(cf.\ also the proof of Lemma \ref{BASICkidpreplemG}).
Now the desired convergence follows by the continuous mapping theorem,
\cite[Thm.\ 4.27]{kallenberg02}.
\end{proof}

We noted in Remark \ref{Thm2genrem} that one consequence of Theorem \ref{Thm2gen}
is that $\lambda(\fw_{\vecq,\rho}^{\vecbeta})\to1$ as $\rho\to0$,
with uniformity in $\vecq$ and $\vecbeta$.
Still, it is convenient to prove this particular fact before completing the proof of the theorem:
\begin{lem}\label{WqrholargeLEM}
$\lambda(\fw_{\vecq,\rho}^{\vecbeta})\to1$ as $\rho\to0$, uniformly over all $\vecq\in\scrP_T(\rho)$
and $\vecbeta\in K$.
\end{lem}
\begin{proof}
Set $B=U\times\{\undef\}$.
Then $\eta_{{\vs}}^{(\vecbeta,\lambda)}(B)=\eta_{{\vs}}^{(\vecbeta,\lambda)}(\partial B)=0$ 
for all ${\vs}$, by Lemma \ref{BASICkidpreplemG};
hence $\eta_{\vecq,\rho}^{(\vecbeta,\lambda)}(B)\to0$
uniformly as $\rho\to0$, by Lemma \ref{KEYyCONVprop3} and 
Remark \ref{UNIFPORTMlemBrem}. 
Let $C:=\sup_{\vecbeta\in K}\|\vecbeta\|$;
this is a finite number since $K$ is compact.
By Lemma \ref{GOODDISTANCElem}, for all sufficiently small $\rho$ 
we have $d_\scrP(\vecq)>(1+C)\rho$ for all $\vecq\in\scrP_T(\rho)$.
Now the desired conclusion follows via Lemma~\ref{ZTAU1CONNECTIONlem}.
\end{proof}

\begin{proof}[Proof of Theorem \ref{Thm2gen}]
Let $f\in\C_b(U\times\R_{>0}\times\Omega)$ be given.
We extend $f$ to be zero on $U\times\{\undef\}$; then  $f\in\C_b(U\times\Delta)$.
By Lemma \ref{KEYyCONVprop3} and Lemma \ref{UNIFPORTMlemB} we have
$\eta_{\vecq,\rho}^{(\vecbeta,\lambda)}(f)-\eta_{{\vs}(\vecq)}^{(\vecbeta,\lambda)}(f)\to0$ as $\rho\to0$,
uniformly over all $\vecq\in\scrP_T(\rho)$ and $\vecbeta\in K$.
Here
\begin{align*}
\eta_{\vecq,\rho}^{(\vecbeta,\lambda)}(f)
=\int_U f\big(\vecv,[F_\rho\circ \vecz_\rho](\scrQ_\rho(\vecq,\vecbeta,\vecv))\big)\,d\lambda(\vecv).
\end{align*}
Now as in the proof of Lemma \ref{WqrholargeLEM},
$\lambda(\{\vecv\in U\col\vecz_\rho(\scrQ_\rho(\vecq,\vecbeta,\vecv))\neq\undef\})\to1$,
uniformly as $\rho\to0$;
and for $\rho$ sufficiently small,
$\vecz_\rho(\scrQ_\rho(\vecq,\vecbeta,\vecv))\neq\undef$ implies
$\vecv\in\fw_{\vecq,\rho}^{\vecbeta}$ and
\begin{align*}
[F_\rho\circ\vecz_\rho](\scrQ_\rho(\vecq,\vecbeta,\vecv))
=(\rho^{d-1}\tau_1(\vecq(\vecv),\vecv;\rho),\vecomega_1(\vecq(\vecv),\vecv;\rho)).
\end{align*}
Hence we conclude
\begin{align*}
\int_{\fw_{\vecq,\rho}^{\vecbeta}} f\bigl(\vecv,\rho^{d-1} \tau_1(\vecq(\vecv),\vecv;\rho),
\vecomega_1(\vecq(\vecv),\vecv;\rho)\bigr)\,d\lambda(\vecv)
-\eta_{{\vs}(\vecq)}^{(\vecbeta,\lambda)}(f)\to0,
\end{align*}
uniformly as $\rho\to0$.
Also,
\begin{align*}
\eta_{{\vs}}^{(\vecbeta,\lambda)}(f)
&=\int_{U}\int_{N_s(\scrX)}f(\vecv,[\iota\circ\vecz](Y-(\vecbeta(\vecv)R(\vecv))_\perp))\,d\mu_{\vs}(Y)\,d\lambda(\vecv)
\\
&=\int_{U}\int_\Omega\int_0^\infty
f(\vecv,\xi,\vecomega)k(((\vecbeta(\vecv)R(\vecv))_\perp,{\vs}),\xi,\vecomega)\,d\xi\,d\mu_{\Omega}(\vecomega)\,d\lambda(\vecv),
\end{align*}
by the definition of %
the transition kernel $k(\vecomega',\xi,\vecomega)$ in Section \ref{TRANSKERsec}.
Hence we obtain \eqref{Thm2genres}.
\end{proof}

We next give a corollary to Theorem \ref{Thm2gen} which will be useful
later when we prove 
that the transition kernel $k(\vecomega',\xi,\vecomega)$
possesses a time reversal symmetry;
cf.\ Proposition \ref{krelPROP1} below.
In order to extract information about $k(\vecomega',\xi,\vecomega)$ from 
Theorem \ref{Thm2gen} it is convenient to choose $\vecbeta$ to be the function
\begin{align}\label{BETAUdef}
\vecbeta_\vecu(\vecv)=\vecu R(\vecv)^{-1}+(1-\|\vecu\|^2)^{1/2}\vecv,
\end{align}
where $\vecu\in\overline{\UB}$ is fixed.
Note that $\vecbeta_\vecu(\vecv)\in\US$ and
$(\vecbeta_\vecu(\vecv)+\R_{>0}\vecv)\cap\scrB_1^d=\emptyset$ for all $\vecv\in\US$.
To simplify notation, we set
$\vecq_{\rho,\vecu}(\vecv)=\vecq_{\rho,\vecbeta_\vecu}(\vecv)$
and
$\fw_{\vecq,\rho}^{\vecu}=\fw_{\vecq,\rho}^{\vecbeta_\vecu}$.
\begin{cor}\label{Thm2genCOR}
For any fixed $f\in\C_c(\scrX\times\UB\times\US\times\R_{>0}\times\Omega)$ and
$\lambda\in\Pac(\US)$,
\begin{align}\notag
\rho^{d(d-1)}\sum_{\vecq\in\scrP}
\int_{\UB}\int_{\fw_{\vecq,\rho}^{\vecu}}\hspace{-1pt}
f\bigl((\rho^{d-1}\hspace{-1pt}\vecq,\vs(\vecq)),\vecu,\vecv,\rho^{d-1}\tau_1(\vecq_{\rho,\vecu}(\vecv),\vecv;\rho),
\vecomega_1(\vecq_{\rho,\vecu}(\vecv),\vecv;\rho)\bigr)
\\\notag
\times d\lambda(\vecv)\,d\vecu
\\ \label{Thm2genCORRES}
\to c_\scrP\int_\scrX\int_{\UB}\int_{\US}\int_0^\infty\int_\Omega
f(\vecp,\vecu,\vecv,\xi,\vecomega)\, k((\vecu,{\vs}(\vecp)),\xi,\vecomega)
\hspace{60pt}
\\\notag
\times d\mu_{\Omega}(\vecomega)\,d\xi\,d\lambda(\vecv)\,d\vecu\,d\mu_\scrX(\vecp)
\end{align}
as $\rho\to0$.
\end{cor}
\begin{proof}
Take $\vecv_0\in\US$ so that the function $R$ is continuous on $U:=\US\setminus\{\vecv_0\}$.
Then $K=\{\vecbeta_{\vecu|U}\col\vecu\in\overline{\UB}\}$ is a compact subset of $\C_b(U,\R^d)$
and $(\vecbeta(\vecv)+\R_{>0}\vecv)\cap\scrB_1^d=\emptyset$ for all $\vecbeta\in K$ and $\vecv\in U$.
Furthermore, using $f\in\C_c$, %
the family of functions $\{f(\vecp,\vecu,\cdot,\cdot,\cdot)\col\vecp\in\scrX,\:\vecu\in\UB\}$
is a compact subset of $\C_b(\US\times\R_{>0}\times\Omega)$,
and by restriction we obtain a compact subset of $\C_b(U\times\R_{>0}\times\Omega)$.
By a standard subsequence argument the convergence in Theorem \ref{Thm2gen} 
is seen to be uniform also over such a compact family of test functions.
Hence, using also $(\vecbeta_\vecu(\vecv)R(\vecv))_\perp=\vecu$, we obtain
\begin{align}\label{Thm2genCORPF1}
\int_{\fw_{\vecq,\rho}^{\vecu}}
f\bigl((\rho^{d-1}\vecq,\vs(\vecq)),\vecu,\vecv,\rho^{d-1}\tau_1(\vecq_{\rho,\vecu}(\vecv),\vecv;\rho),
\vecomega_1(\vecq_{\rho,\vecu}(\vecv),\vecv;\rho)\bigr)
\,d\lambda(\vecv)
\hspace{10pt}
\\\notag
-\int_{\US}\int_0^\infty\int_\Omega
f((\rho^{d-1}\vecq,\vs(\vecq)),\vecu,\vecv,\xi,\vecomega)\, k((\vecu,{\vs}(\vecq)),\xi,\vecomega)
\,d\mu_{\Omega}(\vecomega)\,d\xi\,d\lambda(\vecv)
\to0,
\end{align}
as $\rho\to0$, uniformly over all
$\vecq\in\scrP_T(\rho)$ and $\vecu\in\UB$.
This holds for any fixed $T$;
we apply it with $T$ so large that
the support of $f$ is contained in $\scrB_T^d\times\Sigma\times\UB\times\US\times\R_{>0}\times\Omega$;
then the left hand side of \eqref{Thm2genCORPF1} is identically zero when
$\|\vecq\|\geq T\rho^{1-d}$;
hence the convergence in \eqref{Thm2genCORPF1} 
is in fact uniform over all $\vecq\in\scrP\setminus\scrE$.
Using also 
{\blu $\sup|f|<\infty$, [P1], 
and the fact that $\scrE$ has asymptotic density zero (cf.\ [P2]),}
it follows that
up to an additive error which tends to zero as $\rho\to0$,
the left hand side of \eqref{Thm2genCORRES} equals
\begin{align*}%
\rho^{d(d-1)}\sum_{\vecq\in\scrP}
\int_{\UB}\int_{\US}\int_0^\infty\int_\Omega
f((\rho^{d-1}\vecq,\vs(\vecq)),\vecu,\vecv,\xi,\vecomega)\, k((\vecu,{\vs}(\vecq)),\xi,\vecomega)
\hspace{30pt}
\\
\times d\mu_{\Omega}(\vecomega)\,d\xi\,d\lambda(\vecv)\,d\vecu.
\end{align*}
Using here Remark \ref{CONTINTEGRALlemG2REM} and Lemma \ref{ASYMPTDENSITY1CONSlem},
we obtain the limit stated in \eqref{Thm2genCORRES}.
\end{proof}

\subsection{Macroscopic initial conditions}
\label{FIRSTCOLLMACRsec}

The following is the analogue of Theorem \ref{Thm2gen} for macroscopic initial conditions.
Let us write $\fW(1;\rho)$ for the set $\fw(1;\rho)$ in macroscopic coordinates,
i.e.\ $\fW(1;\rho)=\{(\vecq,\vecv)\in\T^1(\R^d)\col\langle\rho^{1-d}\vecq,\vecv\rangle\in\fw(1;\rho)\}$.
\label{fWjrhoDEF}
\begin{thm}\label{Thm2macr}
For any $\Lambda\in\Pac(\T^1(\R^d))$ and $f\in\C_b(\T^1(\R^d)\times\R_{>0}\times\Omega)$,
\begin{align}\notag
\lim_{\rho\to0}
\int_{\fW(1;\rho)} f\bigl(\vecq,\vecv,\rho^{d-1} \tau_1(\rho^{1-d}\vecq,\vecv;\rho),
\vecomega_1(\rho^{1-d}\vecq,\vecv;\rho)\bigr)\,d\Lambda(\vecq,\vecv)
\hspace{60pt}
\\\label{Thm2macrres}
=\int_{\T^1(\R^d)}\int_0^\infty\int_\Omega
f(\vecq,\vecv,\xi,\vecomega)
k^{\g}\bigl(\xi,\vecomega\bigr)
\,d\mu_{\Omega}(\vecomega)\,d\xi\,d\Lambda(\vecq,\vecv).
\end{align}
\end{thm}
\begin{remark}\label{Thm2macrrem}
In particular the theorem implies that
$\Lambda(\fW(1;\rho))\to1$ as $\rho\to0$.
(Indeed, take $f\equiv1$ in \eqref{Thm2macrres}.)
\end{remark}
The proof of the theorem follows the same lines as the proof of Theorem \ref{Thm2gen},
with the key input being the macroscopic limit result of
Theorem \ref{GENLIMITthm}.
In particular we use the same maps $\vecz$, $\vecz_\rho$ and $F_\rho$ as in the previous proof.
The following is the analogue of Lemma \ref{ZTAU1CONNECTIONlem}
for start from an arbitrary point in $\scrK_\rho^\circ$ (cf.\ \eqref{Krhodef}).
\begin{lem}\label{ZTAU1CONNECTIONmacrlem}
Let $\rho\in(0,1)$, $\vecq\in\scrK_\rho^\circ$, $\vecv\in\US$,
and assume that $(\vecx,\vs)=\vecz_\rho(\scrQ_\rho(\vecq,\vecv))\neq\undef$.
Then $(\vecq,\vecv)\in\fw(1;\rho)$,
$\rho^{d-1}\tau_1(\vecq,\vecv;\rho)=\xi_\rho(\vecx)>0$,
and $\vecomega_1(\vecq,\vecv;\rho) %
=(-\vecx_\perp,\vs)$.
\end{lem}
\begin{proof}
Very similar to Lemma \ref{ZTAU1CONNECTIONlem}.
\end{proof}
Let $\Lambda$ be as in Theorem \ref{Thm2macr}, and 
let $\Xi$ be the macroscopic 
limit point process defined in Section~\ref{GENLIMITsec}.
For $(\vecq,\vecv)$ random in $(\T^1(\R^d),\Lambda)$,
let $\eta_\rho^{(\Lambda)}\in P(\T^1(\R^d)\times\Delta)$ be the distribution of 
$(\vecq,\vecv,[F_\rho\circ\vecz_\rho](\scrQ_\rho(\rho^{1-d}\vecq,\vecv)))$,
\label{etarhoLambdaDEF}
and let $\eta^{(\Lambda)}\in P(\T^1(\R^d)\times\Delta)$ be the distribution of
$(\vecq,\vecv,[\iota\circ\vecz](\Xi))$, with $\Xi$ independent from $(\vecq,\vecv)$.
The following is the analogue of Lemma~\ref{KEYyCONVprop3}.
\begin{lem}\label{KEYyCONVprop3macr}
We have $\eta_{\rho}^{(\Lambda)}\xrightarrow[]{\textup{ w }}\eta^{(\Lambda)}$ as $\rho\to0$.
\end{lem}
\begin{proof}
This is similar to the proof of Lemma \ref{KEYyCONVprop3}.
For $(\vecq,\vecv)$ random in $(\T^1(\R^d),\Lambda)$,
let $\tmu_\rho^{(\Lambda)}\in P(\T^1(\R^d)\times\scrX)$ be the distribution of
$(\vecq,\vecv,\scrQ_\rho(\rho^{1-d}\vecq,\vecv))$.
Then Theorem \ref{GENLIMITthm} implies
(via a decomposition argument, similar in flavor to the proof of Lemma \ref{PRODUNIFCONVlem}) 
that 
\begin{align*}
\tmu_\rho^{(\Lambda)}\xrightarrow[]{\textup{ w }}\Lambda\times\mu
\quad\text{as }\:\rho\to0.
\end{align*}
Consider the maps
\begin{align*}
H_\rho:\T^1(\R^d)\times N_s(\scrX)\to\T^1(\R^d)\times\Delta,\qquad
H_n(\vecq,\vecv,Y)=(\vecq,\vecv,F_{\rho}(\vecz_{\rho}(Y)))
\end{align*}
and
\begin{align*}
H:\T^1(\R^d)\times N_s(\scrX)\to\T^1(\R^d)\times\Delta,\qquad
H(\vecq,\vecv,Y)=(\vecq,\vecv,\iota(\vecz(Y))),
\end{align*}
and note that $\eta_{\rho}^{(\Lambda)}=\tmu_{\rho}^{(\Lambda)}\circ H_\rho^{-1}$
and $\eta^{(\Lambda)}=(\Lambda\times\mu)\circ H^{-1}$.
Let $C$ be the set of all $Y\in N_s(\scrX)$ satisfying
$\vecz(Y)\neq\undef$, $Y\cap\partial\fZ_\infty=\emptyset$,
and $\vecy\cdot\vece_1\neq\vecz(Y)\cdot\vece_1$ for all $\vecy\in Y\setminus\{\vecz(Y)\}$.
Then by Lemma \ref{Frhozrhocontlem},
for any $\vect,\vect_1,\vect_2,\ldots\in\T^1(\R^d)$
and $Y,Y_1,Y_2,\ldots\in N_s(\scrX)$ subject to $Y\in C$ and $(\vect_n,Y_n)\to(\vect,Y)$ as $n\to\infty$,
we have $H_n(\vect_n,Y_n)\to H(\vect,Y)$ as $n\to\infty$.
Furthermore using Lemma \ref{Xie1simplelem}, Lemma \ref{GENnonemptylem}
and Proposition \ref{GENPALMprop} (with $A=N_s(\scrX)$, $B=\partial\fZ_\infty\times\Sigma$)
one verifies that $\mu(C)=1$.
Now the desired convergence follows by
\cite[Thm.\ 4.27]{kallenberg02}.
\end{proof}

\begin{lem}\label{Thm2macrremlem}
(Cf.\ Remark \ref{Thm2macrrem}.)
$\Lambda(\fW(1;\rho))\to1$ as $\rho\to0$.
\end{lem}
\begin{proof}
Set $B=\T^1(\R^d)\times\{\undef\}$.
Then $\eta^{(\Lambda)}(B)=0$ by Lemma \ref{BASICkidpreplemgen};
hence $\eta_{\rho}^{(\Lambda)}(B)\to0$ as $\rho\to0$, by Lemma \ref{KEYyCONVprop3macr},
and now the desired convergence follows using Lemma \ref{ZTAU1CONNECTIONmacrlem}.
\end{proof}

\begin{proof}[Proof of Theorem \ref{Thm2macr}]
Let $f\in\C_b(\T^1(\R^d)\times\R_{>0}\times\Omega)$ be given.
Extend $f$ to be zero on $\T^1(\R^d)\times\{\undef\}$;
then $f\in\C_b(\T^1(\R^d)\times\Delta)$.
Now by Lemma \ref{KEYyCONVprop3macr},
$\eta_{\rho}^{(\Lambda)}(f)\to\eta^{(\Lambda)}(f)$ as $\rho\to0$.
Also $\Lambda(\{(\vecq,\vecv)\col\vecz_\rho(\scrQ_\rho(\rho^{1-d}\vecq,\vecv))=\undef\})\to1$,
by the proof of Lemma \ref{Thm2macrremlem}.
Using this together with Lemma \ref{ZTAU1CONNECTIONmacrlem}, the theorem follows.
\end{proof}

\section{Relations for the transition kernels}
\label{TRANSKERRELsec}

\subsection{Symmetries}

\begin{lem}\label{SOdm1symLEM}
For any fixed $\vecomega'\in\scrX_\perp$ and $R\in\SO(d-1)$,
we have 
\begin{align*}
k(\vecomega' R,\xi,\vecomega R)=k(\vecomega',\xi,\vecomega),\qquad
k^{\g}(\xi,\vecomega R)=k^{\g}(\xi,\vecomega)
\end{align*}
for almost all $(\xi,\vecomega)\in\R_{>0}\times\Omega$
with respect to the measure
$d\xi\,d\mu_\Omega(\vecomega)$.
\end{lem}
\begin{proof}
The first statement follows from the definition of $k$ and the $\SO(d-1)$-invariance of $\mu_\vs$ (cf.\ [Q1]). %
The second statement follows from the definition of $k^{\g}$ and the $\SO(d-1)$-invariance of $\mu$ 
(cf.\ Proposition \ref{GENLIMITASLINVprop}).
\end{proof}
Next we prove a time reversal symmetry for $k$, using Corollary \ref{Thm2genCOR}.
\begin{prop}\label{krelPROP1}
Fix any $R\in \O(d-1)$ with $\det R=-1$.
Then 
\begin{align*}
k(\vecomega',\xi,\vecomega)=k(\vecomega R,\xi,\vecomega' R)
\end{align*}
for almost all $\langle\vecomega',\xi,\vecomega\rangle$
with respect to the measure
$d\mu_\Omega(\vecomega')\,d\xi\,d\mu_\Omega(\vecomega)$.
\end{prop}
\begin{proof}
Let $f\in\C_c(\scrX\times\UB\times\US\times\R_{>0}\times\Omega)$ and $\lambda\in P(\US)$.
Let $\scrP_\rho$ be the set of $\vecq\in\scrP$ corresponding to separated scatterers,
i.e.\ $\scrP_\rho=\{\vecq\in\scrP\col d_\scrP(\vecq)>2\rho\}$.
Inspecting the proof of Corollary \ref{Thm2genCOR} and recalling Lemma \ref{GOODDISTANCElem}
we see that the summation in the left hand side of \eqref{Thm2genCORRES} may be restricted to $\vecq\in\scrP_\rho$
without changing the limit.
Using also the fact that $\vecq^{(1)}(\vecq_{\rho,\vecu}(\vecv),\vecv;\rho)\in\scrP_\rho$
for all $\vecv\in\fw_{\vecq,\rho}^{\vecu}$, 
we obtain
\begin{align}\notag
\lim_{\rho\to0}\rho^{d(d-1)}\sum_{\vecq\in\scrP_\rho}\sum_{\vecq'\in\scrP_\rho}
\int_{\US}\int_{\UB}
I\bigl(\vecv\in\fw_{\vecq,\rho}^{\vecu}\text{ and }\vecq^{(1)}(\vecq_{\rho,\vecu}(\vecv),\vecv;\rho)=\vecq'\bigr)
\hspace{40pt}
\\\label{krelPROP1PF1}
\times 
f\bigl((\rho^{d-1}\vecq,\vs(\vecq)),\vecu,\vecv,\rho^{d-1}\tau_1(\vecq_{\rho,\vecu}(\vecv),\vecv;\rho),
\vecomega_1(\vecq_{\rho,\vecu}(\vecv),\vecv;\rho)\bigr)
\,d\vecu\,d\lambda(\vecv)
\hspace{10pt}
\\[4pt]\notag
=c_\scrP v_{d-1}\int_\Omega\int_{\US}\int_0^\infty\int_\Omega
f_0(\vecomega',\vecv,\xi,\vecomega)\, k(\vecomega',\xi,\vecomega)
\,d\mu_{\Omega}(\vecomega)\,d\xi\,d\lambda(\vecv)\,d\mu_{\Omega}(\vecomega'),
\end{align}
where
\begin{align*}
f_0((\vecu,{\vs}),\vecv,\xi,\vecomega)=\int_{\R^d}f((\vecx,{\vs}),\vecu,\vecv,\xi,\vecomega)\,d\vecx.
\end{align*}

Let $\rho\in(0,1)$ be fixed.
Given $\vecq,\vecq'\in\scrP_\rho$ and $\vecv\in\US$, we set
\begin{align*}
U[\vecq,\vecq',\vecv]=\{\vecu\in\UB\col\vecv\in\fw_{\vecq,\rho}^{\vecu}\:\text{ and }\:
\vecq^{(1)}(\vecq_{\rho,\vecu}(\vecv),\vecv;\rho)=\vecq'\}.
\end{align*}
Also set $K_\vecv:=R(\vecv)^{-1}R(-\vecv)$.
Given any $\vecu\in U[\vecq,\vecq',\vecv]$
we set $\vecu':=$\linebreak$\vecw_1(\vecq_{\rho,\vecu}(\vecv),\vecv;\rho)K_\vecv$;
also write $\tau_1=\tau_1(\vecq_{\rho,\vecu}(\vecv),\vecv;\rho)$ and
$\vecw_1=\vecw_1(\vecq_{\rho,\vecu}(\vecv),\vecv;\rho)$.
Then by the definition \eqref{BETAUdef},
$\vecbeta_{\vecu'}(-\vecv)$ is the unique point in $\US$ satisfying
$(\vecbeta_{\vecu'}(-\vecv)R(\vecv))_\perp=\vecu'K_{-\vecv}=\vecw_1$
and $\vecbeta_{\vecu'}(-\vecv)\cdot\vecv<0$;
hence by the definition of $\vecw_1$ we have 
$\vecq_{\rho,\vecu}(\vecv)+\tau_1\vecv=\vecq'+\rho\vecbeta_{\vecu'}(-\vecv)=\vecq'_{\rho,\vecu'}(-\vecv)$,
and thus also $\vecq_{\rho,\vecu}(\vecv)=\vecq'_{\rho,\vecu'}(-\vecv)-\tau_1\vecv$.
It follows that $\vecu'\in U[\vecq',\vecq,-\vecv]$,
$\tau_1(\vecq'_{\rho,\vecu'}(-\vecv),-\vecv;\rho)=\tau_1$,
and
\begin{align*}
\vecw_1(\vecq'_{\rho,\vecu'}(-\vecv),-\vecv;\rho)=
(\vecbeta_\vecu(\vecv)R(-\vecv))_\perp=(\vecu K_\vecv)_\perp=\vecu K_\vecv.
\end{align*}
The map $\vecu\mapsto\vecu':=\vecw_1(\vecq_{\rho,\vecu}(\vecv),\vecv;\rho)K_\vecv$
from $U[\vecq,\vecq',\vecv]$ to $U[\vecq',\vecq,-\vecv]$
is clearly injective,
and it follows from the above considerations that the composition of this map
with the corresponding injection
$\vecu'\mapsto\vecw_1(\vecq_{\rho,\vecu'}(-\vecv),-\vecv;\rho)K_{-\vecv}$ 
from $U[\vecq',\vecq,-\vecv]$ to $U[\vecq,\vecq',\vecv]$ is the identity map.
Hence both these maps are in fact bijections, and inverses of each other.
Note also that Lebesgue measure $d\vecu$ corresponds to $d\vecu'$ under the bijection.
Hence the left hand side of \eqref{krelPROP1PF1} equals
\begin{align}\notag
\lim_{\rho\to0}\rho^{d(d-1)}\sum_{\vecq\in\scrP_\rho}\sum_{\vecq'\in\scrP_\rho}
\int_{\US}\int_{\UB}
I\bigl(-\vecv\in\fw_{\vecq',\rho}^{\vecu'}\text{ and }
\vecq^{(1)}=\vecq\bigr)
\hspace{70pt}
\\\label{krelPROP1PF2}
\times f\Bigl(\bigl(\rho^{d-1}{\vecq}^{(1)},\vs({\vecq}^{(1)})\bigr),
\:\vecw_1 K_{-\vecv},\:\vecv,\:
\rho^{d-1}\tau_1,\:\bigl(\vecu'K_{-\vecv},{\vs}(\vecq')\bigr)\Bigr)
\,d\vecu'\,d\lambda(\vecv),
\end{align}
where now 
$\vecq^{(1)}=\vecq^{(1)}(\vecq'_{\rho,\vecu'}(-\vecv),-\vecv;\rho)$,
$\tau_1=\tau_1(\vecq'_{\rho,\vecu'}(-\vecv),-\vecv;\rho)$
and
$\vecw_1=\vecw_1(\vecq'_{\rho,\vecu'}(-\vecv),-\vecv;\rho)$.
Using the fact that ${\vecq}^{(1)}\in\scrP_\rho$ 
for all $-\vecv\in\fw_{\vecq',\rho}^{\vecu'}$ the summation over $\vecq$ in \eqref{krelPROP1PF2} may be removed,
keeping only the condition $-\vecv\in\fw_{\vecq',\rho}^{\vecu'}$ in the indicator function.
Furthermore in the first argument of $f$ in \eqref{krelPROP1PF2},
we may replace $\rho^{d-1}{\vecq}^{(1)}$ by
$\rho^{d-1}(\vecq'-\tau_1\vecv)$.
Indeed, using $\|{\vecq}^{(1)}-(\vecq'-\tau_1\vecv)\|=\rho$
and $f\in\C_c$ we see that the error in the integrand caused by the replacement is uniformly small;
and hence by [P1] the error caused in the total expression
tends to zero as $\rho\to0$.
Substituting also $-\vecv$ for $\vecv$, 
writing $\tlambda$ for the corresponding probability measure on $\US$,
we conclude that \eqref{krelPROP1PF2} equals
\begin{align}\notag
\lim_{\rho\to0}\rho^{d(d-1)}\sum_{\vecq\in\scrP}
\int_{\UB}\int_{\fw_{\vecq,\rho}^{\vecu}}
\tf\bigl((\rho^{d-1}\vecq,\vs(\vecq)),\:\vecu,\:\vecv,\:\rho^{d-1}\tau_1(\vecq_{\rho,\vecu}(\vecv),\vecv;\rho),\:
\hspace{40pt}
\\\label{krelPROP1PF3}
\vecomega_1(\vecq_{\rho,\vecu}(\vecv),\vecv;\rho)\bigr)
\,d\tlambda(\vecv)\,d\vecu,
\end{align}
where
\begin{align*}
\tf\Bigl((\vecx,{\vs}),\:\vecu,\:\vecv,\:\xi,\:(\vecw,{\vs}')\Bigr)
:=f\Bigl((\vecx+\xi\vecv,{\vs}'),\:\vecw{K}_\vecv,\:-\vecv,\:\xi,\:(\vecu{K}_\vecv,{\vs})\Bigr).
\end{align*}
Clearly $\tf\in\C_c(\scrX\times\UB\times\US\times\R_{>0}\times\Omega)$;
hence Corollary \ref{Thm2genCOR} applies,
and we conclude that \eqref{krelPROP1PF3} equals
\begin{align*}
c_\scrP\int_\scrX\int_{\UB}\int_{\US}\int_0^\infty\int_\Omega
\tf(\vecp,\vecu,\vecv,\xi,\vecomega)\, k((\vecu,{\vs}(\vecp)),\xi,\vecomega)
\hspace{70pt}
\\
\times d\mu_{\Omega}(\vecomega)\,d\xi\,d\tlambda(\vecv)\,d\vecu\,d\mu_\scrX(\vecp),
\end{align*}
and integrating out the $\R^d$-component of $\vecp$ and 
changing variables appropriately, we get
\begin{align}\label{krelPROP1PF4}
c_\scrP v_{d-1}\int_{\Omega}\int_{\US}\int_0^\infty\int_\Omega
f_0(\vecomega',\vecv,\xi,\vecomega)\, k(\vecomega K_\vecv,\xi,\vecomega'K_\vecv)
\,d\mu_{\Omega}(\vecomega)  %
\,d\xi\,d\lambda(\vecv)\,d\mu_\Omega(\vecomega').
\end{align}
Finally, note that for each $\vecv\in\US$
we have $K_\vecv=\smatr{-1}00{K_\vecv'}$ for some $K_\vecv'\in\O(d-1)$ with $\det K_\vecv'=-1$;
thus $R^{-1}K_\vecv'\in\SO(d-1)$, where $R$ is fixed as in the statement of the proposition.
Also by our identification of $\R^{d-1}\times\Sigma$ with $\{0\}\times\R^{d-1}\times\Sigma$ %
we have $\vecomega K_\vecv=\vecomega K_\vecv'$ for all $\vecomega\in\Omega$.
Hence by Lemma \ref{SOdm1symLEM},
for all $\vecv$ and $\vecomega'$ we have
$k(\vecomega K_\vecv,\xi,\vecomega'K_\vecv)=k(\vecomega R,\xi,\vecomega'R)$ for almost all
$(\xi,\vecomega')$.
Hence \eqref{krelPROP1PF4} equals
\begin{align*}
c_\scrP v_{d-1}\int_{\Omega}\int_{\US}\int_0^\infty\int_\Omega
f_0(\vecomega',\vecv,\xi,\vecomega)\, k(\vecomega R,\xi,\vecomega'R)
\,d\mu_{\Omega}(\vecomega)  %
\,d\xi\,d\lambda(\vecv)\,d\mu_\Omega(\vecomega').
\end{align*}

Summing up, using also an obvious surjectivity property of the map $f\mapsto f_0$,
we have proved that the last expression equals the expression in the right hand side of \eqref{krelPROP1PF1},
for every $f_0\in\C_c(\Omega\times\US\times\R_{>0}\times\Omega)$.
The proposition is an immediate consequence of this fact.
\end{proof}

\subsection{Expressions in terms of Palm distributions}
\label{TKPALMsec}

We will now show that if $\Xi_\vs$ has constant intensity measure $\EE\,\Xi_\vs=c_\scrP\mu_\scrX$
(as is true in all of the examples which we consider in Section~\ref{EXAMPLESsec}),
then the transition kernel $k$ can be given explicitly in terms of the Palm distributions of $\Xi_{\vs}$.

For any $\vs\in\Sigma$ such that $\EE\,\Xi_\vs=c_\scrP\mu_\scrX$,
we let $\nu_{\vs}$ be a version of the Palm distributions of $\Xi_{\vs}$.
Recall that this means that 
$\nu_{\vs}$ is a function $\scrX\times\scrN\to[0,1]$,
where $\scrN$ is the Borel $\sigma$-algebra of $N_s(\scrX)$,
such that 
$\nu_{\vs}(\vecx,{A})$ is Borel measurable in $\vecx\in\scrX$ for each ${A}\in\scrN$,
is a probability measure in ${A}\in\scrN$ for each $\vecx\in\scrX$,
and for any Borel sets $B\subset\scrX$ and ${A}\in\scrN$ one has
\begin{align}\label{PALMdefrel}
\int_{{A}}\#({Y}\cap B)\,d\mu_{\vs}({Y})=c_\scrP\int_B\nu_{\vs}(\vecy,{A})\,d\mu_\scrX(\vecy).
\end{align}
Cf.\ \cite[Ch.\ 10]{kallenberg86}.

\begin{prop}\label{BASICkidpropG}
Let $\vecomega'=(\vecx',{\vs}')\in\scrX_\perp$ be given,
and assume that 
$\Xi_{\vs'}$ has constant intensity measure $\EE\,\Xi_{\vs'}=c_\scrP\mu_\scrX$.
Then the relation 
\begin{align}\label{transkerdefG}
k(\vecomega',\xi,\vecomega)
=v_{d-1}c_\scrP %
\,\nu_{{\vs}'}\bigl((\xi,\vecx'-\vecx,{\vs}),\:
\bigl\{{Y}\in N_s(\scrX)\col
{Y}\cap(\fZ_\xi+\vecx')=\emptyset\bigr\}\bigr)
\end{align}
holds for almost every $(\xi,\vecomega)=(\xi,(\vecx,\vs))\in\R_{>0}\times\Omega$
with respect to the measure $d\xi\,d\mu_\Omega(\vecomega)$.
\end{prop}
\begin{proof}
By \cite[Lemma 10.1]{kallenberg86},
and using our assumption that $\EE\,\Xi_{{\vs}'}=c_\scrP\mu_\scrX$,
we have
for any Borel measurable function $f:\scrX\times N_s(\scrX)\to\R_{\geq0}$,
and any Borel subset $U\subset\scrX$,
\begin{align}\label{BASICkidpropGnewpf1}
\int_{N_s(\scrX)} \sum_{\vecy\in U\cap Y}f(\vecy,Y)\, d\mu_{{\vs}'}(Y)
=c_\scrP\int_U\int_{N_s(\scrX)}f(\vecy,Y)\,\nu_{{\vs}'}(\vecy,dY)\,d\mu_\scrX(\vecy).
\end{align}
We apply this relation with 
$U=\iota(B)+\vecx'$ for a given Borel set $B\subset\R_{>0}\times\Omega$,
and $f$ as the indicator function
\begin{align*}
f(\vecy,Y):=I((Y-\vecx')\cap\fZ_{y_1}=\emptyset)
\qquad\text{(where $y_1=\vecy\cdot\vece_1$)}.
\end{align*}
Then the integrand in the left hand side of \eqref{BASICkidpropGnewpf1}
equals $I(\vecz(Y-\vecx')\in \iota(B))$
for each $Y$ with $\vecz(Y-\vecx')\neq\undef$.
In the right hand side of \eqref{BASICkidpropGnewpf1}
we substitute $\vecy=(\xi,\vecx'-\vecx,{\vs})\in\R\times\R^{d-1}\times\Sigma$.
Using \eqref{kappadef} and Lemma \ref{BASICkidpreplemG} we then get
\begin{align*}
\kappa((\vecx',{\vs}');B)
=c_\scrP v_{d-1}\int_{B}\nu_{{\vs}'}((\xi,\vecx'-\vecx,{\vs}),
\{Y\in N_s(\scrX)\col (Y-\vecx')\cap\fZ_{y_1}=\emptyset\})
\hspace{10pt}
\\
\times d\xi\,d\mu_{\Omega}(\vecx,{\vs}).
\end{align*}
Using the fact that this holds for every Borel set $B\subset\R_{>0}\times\Omega$,
and comparing with the definition of $k(\vecomega',\xi,\vecomega)$, %
we obtain \eqref{transkerdefG}.
\end{proof}
The same technique also leads to the following formula for 
the ``generic'' transition kernel $k^{\g}$:
\begin{prop}\label{BASICkidpropgen}
For almost every $(\xi,(\vecx,\vs))\in\R_{>0}\times\Omega$,
\begin{align*}%
k^{\g}(\xi,(\vecx,\vs))
=v_{d-1}c_\scrP\mu_\vs\bigl(\{Y\col Y\cap(\fZ_\xi-(\xi,-\vecx))=\emptyset\}\bigr).
\end{align*}
\end{prop}
\begin{proof}
By Proposition \ref{GENPALMprop},
$\Xi$ has constant intensity $c_\scrP\mu_\scrX$,
and a version $\nu:\scrX\times\scrN\to[0,1]$ of the Palm distributions of $\Xi$ are given by
$\nu((\vecx,\vs),A)=\omu_\vs^{(\vecx)}(A)$.
Hence by the same argument as in the proof of Proposition \ref{BASICkidpropG} we get
\begin{align*}
k^{\g}(\xi,(\vecx,\vs))
&=v_{d-1}c_\scrP\,\nu\bigl((\xi,-\vecx,\vs),\{Y\col Y\cap\fZ_\xi=\emptyset\}\bigr)
\\
&=v_{d-1}c_\scrP\,\omu_\vs^{(\xi,-\vecx)}\bigl(\{Y\col Y\cap\fZ_\xi=\emptyset\}\bigr)
\\
&=v_{d-1}c_\scrP\,\mu_\vs\bigl(\{Y\col Y\cap(\fZ_\xi-(\xi,-\vecx))=\emptyset\}\bigr).
\end{align*}
\end{proof}
\begin{prop}\label{kgfromkprop}
For almost every $(\xi,\vecomega)\in\R_{>0}\times\Omega$,
\begin{align*}%
k^{\g}(\xi,\vecomega)=v_{d-1}c_\scrP\int_\xi^\infty\int_\Omega k(\vecomega',\xi',\vecomega)\,d\mu_\Omega(\vecomega')\,d\xi'.
\end{align*}
\end{prop}
\begin{proof}
Fix $R\in\O(d-1)$ with $\det R=-1$,
and set $K=\smatr{-1}00R\in\SO(d)$.
We then note that $(\fZ_\xi-(\xi,-\vecx))K=%
\fZ_\xi+\vecx R$,
and hence by Proposition \ref{BASICkidpropgen} and Lemma~\ref{ODM1NDlem},
for almost every $(\xi,(\vecx,\vs))\in\R_{>0}\times\Omega$ we have
\begin{equation}\label{lastline}
\begin{split}
k^{\g}(\xi,(\vecx,\vs)) 
&=v_{d-1}c_\scrP\,\mu_\vs\bigl(\{Y\col (Y-\vecx R)\cap\fZ_\xi=\emptyset\}\bigr)
\\
&=v_{d-1}c_\scrP\,\int_\xi^\infty\int_\Omega k\bigl((\vecx R,\vs),\xi',\vecomega'\bigr)\,d\mu_\Omega(\vecomega')\,d\xi'
\\
&=v_{d-1}c_\scrP\,\int_\xi^\infty\int_\Omega k\bigl((\vecx,\vs)R,\xi',\vecomega'R\bigr)\,d\mu_\Omega(\vecomega')\,d\xi',
\end{split}
\end{equation}
where we used the definition of $k$,
and then used the fact that $\vecomega'\mapsto\vecomega'R$ is a diffeomorphism of $\Omega$ onto itself preserving the
measure $\mu_\Omega$.
Now the desired formula follows by also using Proposition \ref{krelPROP1}.
\end{proof}

The following corollary shows that
$\oxi=(v_{d-1}c_\scrP)^{-1}$
(cf.\ \eqref{OXIFORMULAG})
equals the mean free path length
for the particle dynamics in the Boltzmann-Grad limit.
\begin{cor}
$$\int_\Omega\int_0^\infty\int_\Omega \xi \, k(\vecomega',\xi,\vecomega)\,d\mu_\Omega(\vecomega')\,d\xi\,d\mu_\Omega(\vecomega)
{\displaystyle =\frac1{v_{d-1}c_\scrP}}
=\oxi .$$
\end{cor}
\begin{proof}
We have
$\int_0^\infty\int_\Omega k^{\g}(\xi,\vecomega)\,d\mu_\Omega(\vecomega)\,d\xi=
\kappa^g(\R_{>0}\times\Omega)=1$.
Substituting the formula from Proposition \ref{kgfromkprop} into this relation,
the corollary follows.
\end{proof}

\begin{cor}\label{kgfromkCORRO}
There is a representative of $k^{\g}$ which is continuous on all $\R_{>0}\times\Omega$ and which satisfies
\begin{align*}
k^{\g}(0,\vecomega) := \lim_{\xi\to 0} k^{\g}(\xi,\vecomega)= v_{d-1}c_\scrP,
\qquad\forall\vecomega\in\Omega.
\end{align*}
\end{cor}
\begin{proof}
Fix $R\in\O(d-1)$ with $\det R=-1$.
As in the proof of Proposition \ref{kgfromkprop} we have
\begin{align}\label{kgfromkCORROpf1}
k^{\g}(\xi,\vecomega)
=v_{d-1}c_\scrP\,\int_\xi^\infty\int_\Omega k\bigl(\vecomega R,\xi',\vecomega'R\bigr)\,d\mu_\Omega(\vecomega')\,d\xi'
\end{align}
for almost all $(\xi,\vecomega)\in\R_{>0}\times\Omega$.
Now fix the representative of $k^{\g}$ for which \eqref{kgfromkCORROpf1} holds for \textit{all} $(\xi,\vecomega)$.
Then $k^{\g}$ is continuous,
as follows from Lemma \ref{zXisigmmxcontLEM} and the boundedness of $k$ (cf.\ \eqref{kdef}).
Using the boundedness of $k$ we also obtain
\begin{align*}
\lim_{\xi\to 0} k^{\g}(\xi,\vecomega)= 
v_{d-1}c_\scrP \,\int_{\R_{>0}}\int_\Omega k\bigl(\vecomega R,\xi',\vecomega'R\bigr)\,d\mu_\Omega(\vecomega')\,d\xi'
=v_{d-1}c_\scrP,
\end{align*}
where the last equality holds by the $\O(d-1)$-invariance of $\mu_{\Omega}$
and since $k(\vecomega R,\cdot,\cdot)$ is a probability density.
\end{proof}

\section{Scattering maps}
\label{SCATTERINGMAPS}

We now describe
the general scattering process which we will allow
in the statement of our main results on
the limit distribution of the
sequence $\tau_1,\vecv_1,\tau_2,\vecv_2,\ldots$
(cf.\ Theorems \ref{MAINTECHNTHM2A} and \ref{MAINTECHNthm2G} below).
As in Section \ref{sec:classical},
the scattering process is defined by a map $\Psi:\scrS_-\to\scrS_+$,
where
\begin{align*}
\scrS_-:=\{(\vecv,\vecb)\in\US\times\US\col\vecv\cdot\vecb<0\}
\end{align*}
and 
\begin{align*}
\scrS_+:=\{(\vecv,\vecb)\in\US\times\US\col\vecv\cdot\vecb>0\}
\end{align*}
are the sets of incoming and outgoing data, respectively.
\label{PSIdef}
We write $\Psi_1(\vecv,\vecb)\in\US$ and $\Psi_2(\vecv,\vecb)\in\US$
for the projection of $\Psi(\vecv,\vecb)$ onto the first and second component, respectively.
We assume that $\Psi$ satisfies the following conditions:\footnote{These conditions correspond to assumptions (i), (ii), (iv) in
\cite[Section 2.2]{partII};
the fourth assumption is not required in the present paper, cf.\ \eqref{Psi1vmv}.}
\begin{enumerate}[{\rm (i)}]
	\item $\Psi$ is \textit{spherically symmetric}, i.e., 
$\Psi_j(\vecv,\vecb)K=\Psi_j(\vecv K,\vecb K)$ for all $K\in \O(d)$, $j=1,2$;
	\item $\Psi_1(\vecv,\vecb)$ and $\Psi_2(\vecv,\vecb)$
are contained in the linear subspace spanned by $\vecv$ and $\vecb$;
	\item $\Psi:\scrS_-\to\scrS_+$ is $\C^1$ and 
for each fixed $\vecv\in\S_1^{d-1}$ the map $\Psi_1(\vecv,\cdot)$
is a $\C^1$ diffeomorphism from $\{\vecb\in\S_1^{d-1}\col\vecv\cdot\vecb<0\}$
onto some open subset of $\S_1^{d-1}$.
\end{enumerate}
As we will explain in Section~\ref{AppA1},
the above conditions are satisfied for many standard choices of
scatterers described by a Hamiltonian flow with a compactly supported potential.

We introduce some further notation associated to the scattering map.
We will write $\varphi(\vecu,\vecv)\in [0,\pi]$ for the angle between\label{angledef}
any two vectors $\vecu,\vecv\in\R^d\setminus\{\bn\}$.
Using (ii) it follows that 
\begin{align}\label{Psi1vmv}
\Psi_1(\vecv,-\vecv)={s_\Psi}\cdot\vecv \qquad\text{for all }\:\vecv\in\US,
\end{align}
where the constant ${s_\Psi}$ is either $1$ or $-1$.
It then follows from (i) and (iii) that there exists a constant $B_\Psi\in[0,\pi]$
such that for each $\vecv\in\US$, the image of the diffeomorphism $\Psi_1(\vecv,\cdot)$ equals
\begin{align} \label{VPDEF}
\scrV_{\vecv}:=\bigl\{\vecu\in\S_1^{d-1}\col {s_\Psi}\cdot(B_\Psi-\varphi(\vecu,\vecv))>0\bigr\}.
\end{align}
Let us write $\vecbeta_\vecv^-:\scrV_\vecv\to\{\vecb\in\S_1^{d-1}\col\vecv\cdot\vecb<0\}$ for the inverse map.
\label{betavm}
Then $\vecbeta_\vecv^-$ is spherically symmetric in the sense that
$\vecbeta_{\vecv K}^-(\vecu K)=\vecbeta_{\vecv}^-(\vecu)K$ 
for all $\vecv\in\S_1^{d-1}$, $\vecu\in\scrV_\vecv$, $K\in\O(d)$,
and in particular $\vecbeta_\vecv^-(\vecu)$ is jointly $\C^1$ in $\vecv,\vecu$.
We also define
\begin{align}
\vecbeta^+_{\vecv}(\vecu)=\Psi_2(\vecv,\vecbeta^-_{\vecv}(\vecu))
\qquad (\vecv\in\S_1^{d-1},\:\vecu\in\scrV_\vecv).
\end{align}
The map $\vecbeta^+$ is also spherically symmetric and jointly $\C^1$ in 
$\vecv,\vecu$.
Note that for any given $\vecv,\vecv_+\in\S_1^{d-1}$, there exist
$\vecb,\vecb_+\in\S_1^{d-1}$ such that %
$\Psi(\vecv,\vecb)=(\vecv_+,\vecb_+)$ if and only if
$\vecv_+\in\scrV_{\vecv}$, and in this case $\vecb$ and $\vecb_+$
are uniquely determined, as 
$\vecb=\vecbeta^-_{\vecv}(\vecv_+)$
and $\vecb_+=\vecbeta^+_{\vecv}(\vecv_+)$.

\begin{remark} \label{PRESERVELIOUVILLEREMARK}
Denote by $R_{\{\vecv_2\}^\perp}\in \O(d)$ the orthogonal reflection in the 
hyperplane $\{\vecv_2\}^\perp\subset\R^d$.
If the scattering map $\Psi$ is a %
diffeomorphism from $\scrS_-$ onto $\scrS_+$ which carries the volume measure
$|\vecv\cdot\vecb|\,d\!\vol_{\S_1^{d-1}}(\vecv)\,d\!\vol_{\S_1^{d-1}}(\vecb)$
on $\scrS_-$ to 
$(\vecv\cdot\vecb)\,d\!\vol_{\S_1^{d-1}}(\vecv)\,d\!\vol_{\S_1^{d-1}}(\vecb)$
on $\scrS_+$, then 
\begin{align} \label{PLBETACRITERION}
\vecbeta_{\vecv_1}^+(\vecv_2)\equiv -\vecbeta_{\vecv_2}^-(\vecv_1)
\qquad \text{or} \qquad
\vecbeta_{\vecv_1}^+(\vecv_2)\equiv \vecbeta_{\vecv_2}^-(\vecv_1)R_{\{\vecv_2\}^\perp}.
\end{align}
The reverse implication is also true;
see \cite[Remark 2.3]{partII} for a detailed discussion.
\footnote{Note that $s_{\Psi}=-1$ in \cite{partII}; however 
\cite[Remark 2.3]{partII} applies verbatim also when $s_{\Psi}=1$,
with the only modification that ``$\vartheta_j(0)=0$'' is replaced by ``$\vartheta_j(0)=\pi$''.}
In physical terms, \eqref{PLBETACRITERION}
reflects the preservation of the angular momentum $\vecb\wedge\vecv$, or 
its reversal, respectively. The first alternative in \eqref{PLBETACRITERION} holds for specular reflection as well as potential scattering.
\end{remark}

It will be useful later to have a reformulation of 
the condition (iii) in terms of the \textit{deflection angle}
of the scattering map,
i.e.\ the angle between the incoming and outgoing velocities expressed as a function of the length of the impact parameter.
In precise terms, for a scattering map $\Psi$ satisfying conditions (i) and (ii),
the deflection angle is a function $\theta:[0,1)\to\R$
satisfying the formula
\begin{align}\label{PSI1formula}
&\Psi_1(\vecv,\vecb)=(\cos\theta(w))\vecv+\frac{\sin\theta(w)}w \vecw
\qquad ((\vecv,\vecb)\in\scrS_-),
\end{align}
where
\begin{align}\label{PSI1formulaadd}
\vecw:=\vecb-(\vecv\cdot\vecb)\vecv\in\{\vecv\}^\perp
\quad\text{ and }\quad
w=\|\vecw\|.
\end{align}
In the special case $\vecw=\bn$ ($\Leftrightarrow$ $\vecb=-\vecv$),
we require $\theta(0)\equiv 0$ (mod $\pi$),  %
and the right hand side of \eqref{PSI1formula} should be interpreted as $(\cos\theta(0))\vecv$.
Note that we do not require $\theta$ to take values in $[-\pi,\pi]$; 
in fact in the case of potential scattering the natural definition of $\theta$ is a function which
can take any value in $[-\infty,\pi]$, 
cf.\ \eqref{deflangle} below.

The following lemma gives an equivalent formulation 
in terms of the deflection angle of the ``$\Psi_1$-part'' of condition (iii).
\begin{lem}\label{DEFLANGLElem}
Given any continuous %
function $\theta:[0,1)\to\R$ with $\theta(0)=k\pi$ ($k\in\Z$),
the following two statements are equivalent:
\begin{enumerate}[{\rm (1)}]
\item The map $\Psi_1$ given by \eqref{PSI1formula} is $\C^1$ and for each fixed $\vecv\in\S_1^{d-1}$ the map $\Psi_1(\vecv,\cdot)$
is a $\C^1$ diffeomorphism from $\{\vecb\in\S_1^{d-1}\col\vecv\cdot\vecb<0\}$
onto some open subset of $\S_1^{d-1}$.
\item %
$\theta:[0,1)\to\R$ is $\C^1$,
and for all $w\in[0,1)$ we have
$\theta'(w)\neq0$ and $|\theta(w)-k\pi|<\pi$.
\end{enumerate}
\end{lem}
\begin{proof}
The implication (1)$\Rightarrow$(2)
is straightforward,
e.g.\ using the fact that for all $w\in[0,1)$,
$\cos\theta(w)$ and $\sin\theta(w)$
can be expressed as the scalar product of
$\Psi_1\bigl(\vece_1,-(1-w^2)^{1/2}\vece_1+w\vece_2\bigr)$
with $\vece_1$ and $\vece_2$, respectively.
We turn to the proof of (2)$\Rightarrow$(1);
thus assume that (2) holds.
By immediate inspection, \eqref{PSI1formula} yields a spherically symmetric map $\Psi_1:\scrS_-\to\US$
which is continuous, and $\C^1$ on $\scrS_-\setminus\{(\vecv,-\vecv)\}$.
In order to verify that $\Psi_1$ is $\C^1$ on all $\scrS_-$,
note that we can write $\Psi_1(\vecv,\vecb)=g(\|\vecw\|^2)\vecv+f(\vecw)$,
where the maps
$f:\scrB_1^d\to\R^d$ and $g:[0,1)\to\R$ are given by 
\begin{align}\label{DEFLANGLElempf1}
&g(u)=\cos\theta(u^{1/2})\quad\text{and}
\\\notag
&f(\vecw)=f_1(\|\vecw\|)\,\vecw
\quad\text{with}\quad
f_1(w)=\begin{cases}w^{-1}\sin\theta(w)&\text{if }\: w>0
\\
\theta'(0)&\text{if }\:w=0.
\end{cases}
\end{align}
Therefore it suffices to verify $f\in\C^1(\scrB_1^d)$ and $g\in\C^1([0,1))$.
Both of these are straightforward.
(For $f$, an intermediate step is to note that
$f_1$ is continuous on $[0,1)$ and $\C^1$ on $(0,1)$,
and $\lim_{w\to0}wf_1'(w)=0$.)
It remains to verify the diffeomorphism statement.
By spherical symmetry it suffices to verify that 
$\Psi_1(\vece_1,\cdot)$ is a $\C^1$ diffeomorphism from $\{\vecb\in\US\col\vece_1\cdot\vecb<0\}$
onto an open subset of $\US$.
Let us first note that for every $0<w<1$,
the differential of the map $\vecb\mapsto\Psi_1(\vece_1,\vecb)$ 
at $\vecb=-(1-w^2)^{1/2}\vece_1+w\vece_2$
equals the linear map
from $T_{\vecb}(\US)=\{\vech\in\R^d\col\vech\cdot\vecb=0\}$ %
to $T_{\Psi_1(\vece_1,\vecb)}(\US)$ given by
\begin{align}\label{CONDiTOiiiLEMpf1}
\vech=(h_1,\cdots,h_d)\mapsto
\Bigl(-\theta'(w)\sin\theta(w)\,h_2,\;\theta'(w)\cos\theta(w)\,h_2,\;
\frac{\sin\theta(w)}w\,h_3,
\hspace{40pt}
\\\notag
\cdots,
\frac{\sin\theta(w)}w\,h_d\Bigr).
\end{align}
It follows from the assumption (2) that this map is non-singular for every $0<w<1$.
Furthermore the differential at $\vecb=-\vece_1$ is seen to be
scalar multiplication with $\theta'(0)$, which is again a non-singular map.
Hence by spherical symmetry, the differential of $\vecb\mapsto\Psi_1(\vece_1,\vecb)$
is non-singular throughout $\{\vecb\in\US\col\vece_1\cdot\vecb<0\}$.
It also follows from assumption (2) that this map
$\vecb\mapsto\Psi_1(\vece_1,\vecb)$ is injective.
Hence this map is indeed a $\C^1$ diffeomorphism from 
$\{\vecb\in\US\col\vece_1\cdot\vecb<0\}$ onto an open subset of $\US$.
\end{proof}

\begin{remark}\label{COND3completecondrem}
In the situation of Lemma \ref{DEFLANGLElem},
if the scattering map is also known to preserve angular momentum $\vecb\wedge\vecv$,
then one computes that
\begin{align}\label{PSI2formula}
&\Psi_2(\vecv,\vecb)=-w(\sin\theta)\vecv+(\cos\theta)\vecw-(\vecv\cdot\vecb)\Psi_1(\vecv,\vecb)
\qquad ((\vecv,\vecb)\in\scrS_-),
\end{align}
with $\vecw$ and $w$ as in \eqref{PSI1formulaadd}.
Hence in this case,
condition (2) in Lemma \ref{DEFLANGLElem}
implies that $\Psi$ satisfies \textit{all} of condition (iii).
Indeed, it only remains to verify that $\Psi_2$ is $\C^1$,
and this follows %
once we note that
$\Psi_2(\vecv,\vecb)=-(\vecw\cdot f(\vecw))\vecv+g(\|\vecw\|^2)\vecw-(\vecv\cdot\vecb)\Psi_1(\vecv,\vecb)$,
with $f,g$ as in \eqref{DEFLANGLElempf1}.
\end{remark}

Next, for a scattering map $\Psi$ satisfying conditions (i), (ii) and (iii),
the normalized impact parameter corresponding to incoming and outgoing velocities
$\vecv$ and $\vecv_+$, respectively, is
$\vecw:=(\vecbeta^-_{\vecv}(\vecv_+)R(\vecv))_\perp$,
and the {\em differential cross section} 
$\sigma(\vecv,\vecv_+)$ is defined as the Jacobian of the map\label{crosssecdef}
$\vecv_+\mapsto\vecw$
with respect to the measures $\omega=\vol_{\US}$ and $d\vecw$ ( = Lebesgue measure on $\R^{d-1}$).
Thus, for each fixed $\vecv\in\US$, 
\begin{align}\label{dcsdef}
d\vecw=\sigma(\vecv,\vecv_+)\,d\vecv_+\qquad(\vecv_+\in\scrV_{\vecv}).
\end{align}
Hence $\sigma$ is a continuous function on 
$\{(\vecv,\vecv_+)\in\US\times\US\col \vecv_+\in\scrV_\vecv\}$. %
In fact, %
\begin{align}\label{dcsformula}
\sigma(\vecv,\vecv_+)=
\begin{cases}
{\displaystyle\Bigl(\frac{w(\varphi)}{\sin\varphi}\Bigr)^{d-2} |w'(\varphi)|}&\text{if }\:\vecv_+\neq s_\Psi\cdot\vecv,
\\[8pt]
|w'(\varphi)|^{d-1}&\text{if }\:\vecv_+=s_\Psi\cdot\vecv,
\end{cases}
\end{align}
where $\varphi=\varphi(\vecv,\vecv_+)$
and $w(\varphi)=\|\vecw\|$.
Note that $\|\vecw\|$ is indeed a function of $\varphi$,
due to spherical symmetry.
The formula \eqref{dcsformula} is an immediate consequence of
\eqref{CONDiTOiiiLEMpf1} in the proof of Lemma \ref{DEFLANGLElem},
once we note that 
$\varphi\equiv\pm\theta(w)\mod2\pi$
(and $0\leq\varphi\leq\pi$).
In particular in the case of specular reflection, we have
$w(\varphi)=\sin(\varphi/2)$, and hence we recover the formula \eqref{SIGMAforrefl}.
We extend $\sigma$ to all of $\US\times\US$ through $\sigma(\vecv,\vecv_+)=0$ when $\vecv_+\notin\scrV_\vecv$;
this extension is generally not continuous.
It is clear from the above that $\sigma(\vecv K,\vecv_+K)=\sigma(\vecv,\vecv_+)$ for all $K\in\O(d)$, and furthermore $\sigma(\vecv_+,\vecv)=\sigma(\vecv,\vecv_+)$. 

\vspace{5pt}

Finally, with the scattering map in place,
we now extend some definitions from Section \ref{FIRSTCOLLsec}.
Recall that in Section \ref{FIRSTCOLLsec} we introduced 
the notion of a scatterer being 'separated';
we defined $\fw(1;\rho)$ to be the subset of initial conditions
$(\vecq,\vecv)\in\fw(\rho)$ which lead to a collision 
with a separated scatterer in finite time,
and for $(\vecq_0,\vecv_0)\in\fw(1;\rho)$
we defined $\vecq^{(1)}(\vecq_0,\vecv_0;\rho)$,
$\vecw_1(\vecq_0,\vecv_0;\rho)$
and $\vecomega_1(\vecq_0,\vecv_0;\rho)$.
More generally, we now define $\fw(j;\rho)$ \label{winij}
and $\vecq^{(j)},\vs_j,\vecw_j,\vecomega_j,\vecq_j,\vecv_j$, $\vecu_j$ 
\label{qjscattererdef}
for $j\geq1$
by the following recursive formulas:\footnote{All these are functions of $\vecq_0,\vecv_0,\rho$,
i.e.\ $\vecq^{(j)}=\vecq^{(j)}(\vecq_0,\vecv_0;\rho)$; 
$\vs_j=\vs_j(\vecq_0,\vecv_0;\rho)$, etc.}
Set $\fw(0;\rho):=\fw(\rho)$.
For $j\geq1$ and any %
$(\vecq_0,\vecv_0)\in\fw(j-1;\rho)$ 
we set:\footnote{It will be seen that
$(\vecq_0,\vecv_0)\in\fw(j-1;\rho)$ implies that $\vecq_{j-1}$ and $\vecv_{j-1}$ have been defined
and that $(\vecq_{j-1},\vecv_{j-1})\in\fw(\rho)$.}
\begin{align*}
\tau_j=\tau_j(\vecq_0,\vecv_0;\rho):=\tau_1(\vecq_{j-1},\vecv_{j-1};\rho).
\end{align*}
Then $\tau_j\in\R_{>0}\cup\{\infty\}$.
Let $\fw(j;\rho)$ be the subset of those $(\vecq_0,\vecv_0)\in\fw(j-1;\rho)$
for which $\tau_j<\infty$ and $\vecq_{j-1}+\tau_j\vecv_{j-1}$
lies on the boundary of a separated scatterer.
Next, for $(\vecq_0,\vecv_0)\in\fw(j;\rho)$, let
$\vecq^{(j)}\in\scrP$ be the center of the unique scatterer
with $\vecq_{j-1}+\tau_j\vecv_{j-1}\in\partial\scrB^d(\vecq^{(j)},\rho)$
and set $\vs_j=\vs(\vecq^{(j)})$;
let $\vecu_j\in\US$ be the point such that
$\vecq_{j-1}+\tau_j\vecv_{j-1}=\vecq^{(j)}+\rho\vecu_j$,
and then set
\begin{align*}
\vecv_j=\Psi_1(\vecv_{j-1},\vecu_j);
\qquad
\vecw_j=(\vecu_j R(\vecv_{j-1}))_\perp;
\qquad
\vecq_j=\vecq^{(j)}+\rho\,\Psi_2(\vecv_{j-1},\vecu_j);
\end{align*}
and finally set
\begin{align*}
\vecomega_j=(\vecw_j,\vs_j   %
)\in\Omega.
\end{align*}

The sequences
$\{\tau_j\}$, $\{\vs_j\}$ and $\{\vecv_j\}$
which we have thus associated to a given initial condition $(\vecq_0,\vecv_0)\in\fw(\rho)$
generalize %
the corresponding sequences defined in Section~\ref{sec:classical}
to the case of a general scattering process,
except %
that our present conventions regarding overlapping scatterers differ
from those in Section \ref{sec:classical}.
However we have seen that the probability of hitting a non-separated scatterer in the first collision
tends to zero,
and the same fact will turn out to hold for every later collision.
Therefore, the difference in conventions does not affect the limit result as $\rho\to0$;
cf.\ Remark \ref{MAINTECHNthm2Grem} below.

\section{Collision kernels}
\label{COLLKERsec}

We now define the \textit{collision kernels;}
these are simple transforms of the transition kernels defined in 
Section \ref{TRANSKERsec}.
Recall from Section \ref{SCATTERINGMAPS} the definition of the scattering map $\Psi$, the associated maps $\vecbeta_\vecv^{\pm}$,
and the differential cross section $\sigma(\vecv,\vecv_+)$.
For $\vecv_0\in\US$, $\vecv\in\scrV_{\vecv_0}$, $\vecvp\in\scrV_{\vecv}$,
and $\xi>0$, ${\vs}\in\Sigma$, ${\vs}_+\in\Sigma$,
we set
\begin{align}\label{pbndefG}
p_\bn\bigl(\vecv_0,{\vs},\vecv;\xi,{\vs}_+,\vecvp\bigr)=
\frac{\sigma(\vecv,\vecvp)}{v_{d-1}}
k\Bigl(\bigl(\vecbeta_{\vecv_0R(\vecv)}^+(\vece_1)_\perp,{\vs}\bigr),
\xi,\bigl(\vecbeta_{\vece_1}^-(\vecvp R(\vecv))_\perp,{\vs}_+\bigr)\Bigr).
\end{align}
We extend the function $p_\bn$ by setting
$p_\bn\bigl(\vecv_0,{\vs},\vecv;\xi,{\vs}_+,\vecvp\bigr)=0$ for any 
$\vecv,\vecvp\in\US$ with $\vecv\notin\scrV_{\vecv_0}$ or $\vecvp\notin\scrV_{\vecv}$.
More generally, given a function $\vecbeta\in \C_b(U,\R^d)$
where $U$ is an open subset of $\US$, we set
\begin{align}\label{p0bdef}
p_{\bn,\vecbeta}\bigl({\vs},\vecv;\xi,{\vs}_+,\vecvp\bigr)=
\frac{\sigma(\vecv,\vecvp)}{v_{d-1}}
k\Bigl(\bigl((\vecbeta(\vecv)R(\vecv))_\perp,{\vs}\bigr),
\xi,\bigl(\vecbeta_{\vece_1}^-(\vecvp R(\vecv))_\perp,{\vs}_+\bigr)\Bigr)
\end{align}
if $\vecv\in{U}$ and $\vecvp\in\scrV_\vecv$, and otherwise
$p_{\bn,\vecbeta}\bigl({\vs},\vecv;\xi,{\vs}_+,\vecvp\bigr)=0$.
We then have
\begin{align}\label{BASICp0idG}
p_\bn\bigl(\vecv_0,{\vs},\vecv;\xi,{\vs}_+,\vecvp\bigr)\equiv
p_{\bn,\vecbeta_{\vecv_0}^+}\bigl({\vs},\vecv;\xi,{\vs}_+,\vecvp\bigr).
\end{align}
Let us also define
\begin{align}\label{pgendef}
p\bigl(\vecv;\xi,{\vs}_+,\vecvp\bigr)=\frac{\sigma(\vecv,\vecvp)}{v_{d-1}}
\,k^{\g}\Bigl(\xi,\bigl(\vecbeta_{\vece_1}^-(\vecvp R(\vecv))_\perp,{\vs}_+\bigr)\Bigr)
\end{align}
if $\vecv_+\in\scrV_\vecv$,
and otherwise $p\bigl(\vecv;\xi,{\vs}_+,\vecvp\bigr)=0$.

The relation between the transition kernels and the collision kernels is captured by the formulas in the
following two lemmas.
Let $s_-$ be the diffeomorphism from $\UB$ onto the negative hemisphere $\{\vecx\in\US\col x_1<0\}$
which is inverse to the projection $\vecx\mapsto\vecx_\perp$, i.e.\
\begin{align}\label{smDEF}
s_-(\vecw)=\bigl(-(1-\|\vecw\|^2)^{1/2},\vecw\bigr),
\qquad\vecw\in\UB.
\end{align}
Then for any $\vecv\in\US$,
the inverse of the $\C^1$ diffeomorphism $\scrV_\vecv\to\UB$, 
$\vecv_+\mapsto\vecw:=(\vecbeta_\vecv^-(\vecv_+)R(\vecv))_\perp$,
is given by $\vecv_+=\Psi_1(\vece_1,s_-(\vecw))R(\vecv)^{-1}$.
\begin{lem}\label{BASICpktransfrem1lem}
For any bounded Borel measurable function $f:\R_{>0}\times\Sigma\times\US\to\R$
and any fixed $\vecbeta\in\C_b(U,\R^d)$, $\vecv\in U$, $\vs\in\Sigma$,
if $f_1:\R_{>0}\times\Omega\to\R_{\geq0}$ is defined through
\begin{align}\label{BASICpktransfrem1lemres1}
f_1(\xi,(\vecw,{\vs}))=f\bigl(\xi,{\vs},\Psi_1(\vece_1,s_-(\vecw))R(\vecv)^{-1}\bigr),
\end{align}
then 
\begin{align}\notag
\int_0^{\infty}\int_{\Omega}f_1(\xi,\vecomega)k\bigl(((\vecbeta(\vecv)R(\vecv))_\perp,\vs),\xi,\vecomega\bigr)
\,d\mu_{\Omega}(\vecomega)\,d\xi
\hspace{120pt}
\\\label{BASICpktransfrem1}
=\int_0^{\infty}\int_{\Sigma}\int_{\scrV_\vecv}
f(\xi,{\vs_+},\vecv_+)p_{\bn,\vecbeta}(\vs,\vecv;\xi,\vs_+,\vecv_+)\,d\vecv_+\,d\mm(\vs_+)\,d\xi.
\end{align}
\end{lem}
\begin{lem}\label{BASICpktransfrem2lem}
For any $\vecv\in\US$ and any bounded Borel measurable functions
$f:\R_{>0}\times\Sigma\times\US\to\R$
and $f_1:\R_{>0}\times\Omega\to\R_{\geq0}$
subject to \eqref{BASICpktransfrem1lemres1},
\begin{align*}%
\int_0^{\infty}\int_{\Omega} &f_1(\xi,\vecomega)k^{\g}\bigl(\xi,\vecomega\bigr)
\,d\mu_{\Omega}(\vecomega)\,d\xi
\\
&=\int_0^{\infty}\int_{\Sigma}\int_{\scrV_\vecv}
f(\xi,{\vs_+},\vecv_+)p(\vecv;\xi,\vs_+,\vecv_+)\,d\vecv_+\,d\mm(\vs_+)\,d\xi.
\end{align*}
\end{lem}
The proof of both lemmas is immediate from the definition of the differential cross section.
\begin{remark}\label{RINDEPANDCONTrem}
Using Lemma \ref{SOdm1symLEM} one sees that the formula \eqref{p0bdef} remains true if in the right hand side we
replace $R(\vecv)$ by any $\hR\in\SO(d)$ satisfying $\vecv \hR=\vece_1$,
so long as both the ``$R(\vecv)$'s'' are replaced by the \textit{same} $\hR$.
In other words, the function $p_{\bn,\vecbeta}$ does not depend on the choice of the function $R:\US\to\SO(d)$,
and the same is true for $p_\bn$.
Similarly the formula in Lemma~\ref{BASICpktransfrem1lem} %
remains valid if we replace $R$ by
any other (measurable) function $\hR:\US\to\SO(d)$ satisfying $\vecv\hR(\vecv)=\vece_1$ for all $\vecv\in\US$,
so long as we use the same function $\hR$ in both \eqref{BASICpktransfrem1lemres1} and \eqref{BASICpktransfrem1}.
In particular in this way we see, via Remark \ref{CONTINTEGRALlemG2REM},
that if $f\in\C_b(\R_{>0}\times\Omega)$ then the expression
in \eqref{BASICpktransfrem1} depends continuously on $\vecv\in U$, even at the (possible) discontinuity point of the
function $R$.
The analogous statement %
holds for Lemma~\ref{BASICpktransfrem2lem}.
\end{remark}

It will follow from Theorem \ref{unifmodThm2} below
that, for given $\vs,\vecv$,
the function\linebreak
$p_{\bn,\vecbeta}(\vs,\vecv;\xi,\vs_+,\vecv_+)$
is the limiting probability density (as $\rho\to0$)
of hitting
a scatterer with marking $\vs_+$ at time $\rho^{1-d}\xi$ and
in such a way that the exit velocity is $\vecv_+$,
when starting at the point $\vecq+\rho\vecbeta(\vecv)$ and with velocity $\vecv$,
where
$\vecq\in\scrP$ and $\vs=\vs(\vecq)$.

The first part of the following lemma shows that
$p_{\bn,\vecbeta}(\vs,\vecv;\xi,\vs_+,\vecv_+)$
is indeed a probability density in the variables $\xi,\vs_+,\vecv_+$.
For the second part of the lemma, we introduce the following notation,
for any $\vecv\in\US$ and $\eta>0$
(cf.\ \eqref{VPDEF}):
\begin{align} \label{WVPDEF}
\scrV^\eta_\vecv:=\bigl\{\vecu\in\S_1^{d-1}\col{s_\Psi}\cdot(B_\Psi-\varphi(\vecu,\vecv))>\eta\bigr\}\subset\scrV_\vecv.
\end{align}
\begin{lem}\label{REMOVINGCPTSUPPfactpfLEM1}
For any open set $U\subset\US$
and any $\vecbeta\in\C_b(U,\R^d)$, $\vecv\in U$ and $\vs\in\Sigma$,
\begin{align}\label{REMOVINGCPTSUPPfactpfLEM1res1}
&\int_0^\infty\int_\Sigma\int_{\scrV_\vecv}
p_{\bn,\vecbeta}(\vs,\vecv;\xi,\vs_+,\vecv_+)\,d\vecv_+\,d\mm(\vs_+)\,d\xi=1.
\end{align}
Also for any $\ve>0$ there exist $C>1$ and $\eta>0$ such that
\begin{align}\label{REMOVINGCPTSUPPfactpfLEM1res2}
&\int_{1/C}^C\int_\Sigma\int_{\scrV_\vecv^\eta}
p_{\bn,\vecbeta}(\vs,\vecv;\xi,\vs_+,\vecv_+)\,d\vecv_+\,d\mm(\vs_+)\,d\xi>1-\ve,
\end{align}
uniformly over all $U,\vecbeta,\vecv,\vs$ as above.
\end{lem}
\begin{proof}
The first statement follows from Lemma \ref{BASICpktransfrem1lem}
and the definition of the transition kernel $k$,
in particular the fact that $\kappa(\vecomega,\cdot)$ is a probability measure on $\R_{>0}\times\Omega$
for every $\vecomega\in\scrX_\perp$;
cf.\ Lemma \ref{BASICkidpreplemG}.
The second statement is an immediate consequence of
Lemma \ref{UNIFKBOUNDGlem1} and the fact that $k(\cdot,\cdot,\cdot)\leq C_\scrP v_{d-1}$.
\end{proof}

Similarly, it will follow from Theorem \ref{unifmodThm2macr}
that the function
$p(\vecv;\xi,\vs_+,\vecv_+)$
is the limiting probability density (as $\rho\to0$)
of hitting
a scatterer with marking $\vs_+$ at time $\rho^{1-d}\xi$ and
in such a way that the exit velocity is $\vecv_+$,
when starting with velocity $\vecv$ from a generic point in $\R^d$.
The following lemma shows that $p(\vecv;\xi,\vs_+,\vecv_+)$ is indeed a probability density in the variables
$\xi,\vs_+,\vecv_+$.
\begin{lem}\label{BASICpkproblem}
For any $\vecv\in\US$,
\begin{align*}
\int_0^{\infty}\int_{\Sigma}\int_{\scrV_\vecv}p(\vecv;\xi,\vs_+,\vecv_+)\,d\vecv_+\,d\mm(\vs_+)\,d\xi=1.
\end{align*}
\end{lem}
\begin{proof}
This follows in a similar manner using Lemma \ref{BASICpktransfrem2lem}
and the fact that $\kappa^{\g}$ is a probability measure on $\R_{>0}\times\Omega$.
\end{proof}

\section{Relations for the collision kernels}
\label{COLLKERRELsec}

\begin{lem}\label{SOdm1symLEMccc}
For any fixed $\vecv_0,\vecv\in\US$, ${\vs}\in\Sigma$ and $K\in\SO(d)$,
we have
\begin{align*}
p_\bn\bigl(\vecv_0 K,{\vs},\vecv K;\xi,{\vs}_+,\vecvp K\bigr) = p_\bn\bigl(\vecv_0,{\vs},\vecv;\xi,{\vs}_+,\vecvp\bigr) ,
\end{align*}
\begin{align*}
p\bigl(\vecv K ;\xi,{\vs}_+,\vecvp K\bigr)=p\bigl(\vecv;\xi,{\vs}_+,\vecvp\bigr)
\end{align*}
for almost all $(\xi,{\vs}_+,\vecvp)\in\R_{>0}\times\Sigma\times\US$
with respect to the measure
$d\xi\,d\mm(\vs_+)\,d\vecv_+$.
\end{lem}

\begin{proof}
We have
\begin{multline*}%
p_\bn\bigl(\vecv_0 K,{\vs},\vecv K;\xi,{\vs}_+,\vecvp K\bigr) \\
= \frac{\sigma(\vecv,\vecvp)}{v_{d-1}}
k\Bigl(\bigl(\vecbeta_{\vecv_0 K R(\vecv K)}^+(\vece_1)_\perp,{\vs}\bigr),
\xi,\bigl(\vecbeta_{\vece_1}^-(\vecvp K R(\vecv K))_\perp,{\vs}_+\bigr)\Bigr),
\end{multline*}
and the analogous relation for $p(\vecv K ;\xi,{\vs}_+,\vecvp K)$.
For fixed $K\in\SO(d)$, the function $\vecv\mapsto \tilde R(\vecv)= K R(\vecv K)\in\SO(d)$ has the property that $\vecv \tilde R(\vecv) = \vece_1$. The claim now follows from Remark \ref{RINDEPANDCONTrem}.
\end{proof}

\begin{prop}\label{kgfromkpropccc}
Assume \eqref{PLBETACRITERION} holds (i.e., the scattering map preserves or reverses angular momentum). Then, for any fixed $\vecv\in\US$,
\begin{align*}
p\bigl(\vecv;\xi,{\vs}_+,\vecvp\bigr) = c_{\scrP} \int_{[\xi,\infty)\times\Sigma\times\US}  \sigma(\vecv_0,\vecv)\, p_\bn\bigl(\vecv_0,{\vs},\vecv;\xi',{\vs}_+,\vecvp\bigr) \, d\xi' \,d\mm({\vs}) \,d\vecv_0 ,
\end{align*}
for almost all $(\xi,{\vs}_+,\vecvp)\in\R_{>0}\times\Sigma\times\US$
with respect to the measure
$d\xi\,d\mm(\vs_+)\,d\vecv_+$.
\end{prop}

\begin{proof}
By \eqref{pgendef} and Proposition \ref{kgfromkprop}, we have
\begin{align*}
p\bigl(\vecv;\xi,{\vs}_+,\vecvp\bigr) =
c_\scrP\, \sigma(\vecv,\vecvp)
\int_\xi^\infty\int_\Omega k\Bigl(\vecomega',
\xi',\bigl(\vecbeta_{\vece_1}^-(\vecvp R(\vecv))_\perp,{\vs}_+\bigr)\Bigr) \,d\mu_\Omega(\vecomega')\,d\xi'.
\end{align*}
We now use relation \eqref{pbndefG} in combination with \eqref{PLBETACRITERION} to obtain 
\begin{align}\label{pbndefG01}
p_\bn\bigl(\vecv_0,{\vs},\vecv;\xi,{\vs}_+,\vecvp\bigr)=
\frac{\sigma(\vecv,\vecvp)}{v_{d-1}}
k\Bigl(\bigl(\mp\vecbeta_{\vece_1}^-(\vecv_0 R(\vecv))_\perp,{\vs}\bigr),
\xi,\bigl(\vecbeta_{\vece_1}^-(\vecvp R(\vecv))_\perp,{\vs}_+\bigr)\Bigr).
\end{align}
Now $\vecbeta_{\vece_1}^-(\vecv_0 R(\vecv))=\vecbeta_{\vecv}^-(\vecv_0) R(\vecv)$, and hence by \eqref{dcsdef}, $d\mu_\Omega(\vecomega')= v_{d-1}^{-1} d\mm({\vs})$\linebreak$\times\sigma(\vecv,\vecv_0)\, d\vecv_0$. Finally, $\sigma(\vecv,\vecv_0)=\sigma(\vecv_0,\vecv)$.
\end{proof}

\section{Post-collision velocity}
\label{NEXTVELOCITYsec}

We now transform Theorem \ref{Thm2gen} to obtain the limit distribution of the velocity after the first collision.
For later use, we give a result which is uniform with respect to appropriate families of test functions $f$ and 
probability measures $\lambda$.

We recall some definitions from 
\cite[Section 2.4]{partII}.
\begin{definition}
Given any subset ${U}\subset\US$ we set
\begin{align}\label{PARTIALVEdef}
\partial_\ve{U}:=\bigl\{\vecv\in\S_1^{d-1} \col
\exists \vecw\in\partial{U}\col \varphi(\vecv,\vecw)<\ve\bigr\}.
\end{align}
A family $F$ of Borel subsets of $\S_1^{d-1}$ is called
\textit{equismooth} if for every $\delta>0$ there is some $\ve>0$ such that
$\omega(\partial_\ve{U})<\delta$ for all ${U}\in F$.
Finally, a family $F$ of measures on $\S_1^{d-1}$ is called
\textit{equismooth} if there exist an equicontinuous and uniformly bounded
family $F'$ of functions from $\S_1^{d-1}$ to $\R_{\geq 0}$ and
an equismooth family $F''$ of open subsets of $\S_1^{d-1}$, 
such that each $\mu\in F$ can be expressed as
$\mu=(g\cdot \omega)_{|{U}}$ for some $g\in F'$, ${U}\in F''$.
\end{definition}

Given any open set $U\subset\US$ we define
\begin{align}\label{XSPACEdef1}
{X_U}:=\Bigl\{\langle\vecv_0,\xi_1,{\vs}_1,\vecv_1\rangle\in
U\times\R_{>0}\times\Sigma\times\US
\col \vecv_1\in\scrV_{\vecv_0}\Bigr\}.
\end{align}

\begin{thm}\label{unifmodThm2}
Let $T\geq1$; 
let $U$ be an open subset of $\US$;
let $F_1$ be an equismooth family of probability measures on $\S_1^{d-1}$ 
such that $\lambda(U)=1$ for each $\lambda\in F_1$;
let $F_2$ be a uniformly bounded and pointwise equicontinuous family
of functions $f:{X_U}\to\R$; %
and let $F_3$ be a relatively compact subset of $\C_b(U,\R^d)$ such that 
$(\vecbeta(\vecv)+\R_{>0}\vecv)\cap\scrB_1^d=\emptyset$ 
for all $\vecbeta\in F_3$, $\vecv\in U$.
Then
\begin{align}\notag
\int_{\fw_{\vecq,\rho}^{\vecbeta}}f\Bigl(\vecv,\rho^{d-1}\tau_1(\vecq_{\rho,\vecbeta}(\vecv),\vecv;\rho)
,{\vs}_1(\vecq_{\rho,\vecbeta}(\vecv),\vecv;\rho),
\vecv_1(\vecq_{\rho,\vecbeta}(\vecv),\vecv;\rho)\Bigr)\,d\lambda(\vecv)
\hspace{50pt}
\\\label{unifmodThm2res}
-\int_{{X_U}} %
f\bigl(\vecv,\xi_1,{\vs}_1,\vecv_1\bigr)
p_{\bn,\vecbeta}\bigl(\vs(\vecq),\vecv;\xi_1,{\vs}_1,\vecv_1\bigr)
\, d\lambda(\vecv)\,d\xi_1\,d\mm({\vs}_1)\,d\vecv_1\to0
\end{align}
as $\rho\to0$,
uniformly with respect to all $\vecq\in\scrP_T(\rho)$, $\lambda\in F_1$, $f\in F_2$, $\vecbeta\in F_3$.
\end{thm}
\begin{remark}
In the left hand side of \eqref{unifmodThm2res}, note that by definition
\linebreak
$\vecv_1(\vecq_{\rho,\vecbeta}(\vecv),\vecv;\rho)=\Psi_1(\vece_1,s_-(\vecw_1))R(\vecv)^{-1}$
with $\vecw_1=\vecw_1(\vecq_{\rho,\vecbeta}(\vecv),\vecv;\rho)$.
In particular $\vecv_1\in\scrV_\vecv$.
\end{remark}
\begin{proof}
Without loss of generality we assume that $R$ is continuous on $U$
(otherwise replace $U$ with $U\setminus\{\vecv_0\}$ where $\vecv_0$ is the unique point where $R$
is discontinuous).
Now if $F_1$ and $F_2$ are singleton sets,
say $F_1=\{\lambda\}$ and $F_2=\{f\}$,
then \eqref{unifmodThm2res} is an immediate consequence of Theorem \ref{Thm2gen},
applied with $f_1\in\C_b(U\times\R_{>0}\times\Omega)$ defined by
\begin{align*}%
f_1(\vecv,\xi,(\vecw,{\vs}))=f(\vecv,\xi,{\vs},\Psi_1(\vece_1,s_-(\vecw))R(\vecv)^{-1}),
\end{align*}
combined with Lemma \ref{BASICpktransfrem1lem}.
The extension to uniformity over general sets $F_1$ and $F_2$ 
is carried out in the same way as in the proof of
\cite[Thm.\ 2.3; Steps 2--4]{partII}.
(One uses
Lemma~\ref{REMOVINGCPTSUPPfactpfLEM1} %
in place of
\cite[(2.42)]{partII}.
When proving uniformity over $F_1$, the key point is to note that 
the set of densities of the measures in $F_1$ with respect to $\omega$ form a relatively compact subset of $\L^1(U,\omega)$.)
\end{proof}

\begin{remark}
The proof of Theorem \ref{unifmodThm2} is significantly shorter than the proof of 
the corresponding result \cite[Thm.\ 2.3]{partII}.
The reason is that we have proved the auxiliary results about convergence of point processes,
Lemma \ref{BETAUNIFCONVlem}, with the appropriate uniformity with respect to $\vecbeta$,
which could then be carried over to Theorems \ref{Thm2gen} and \ref{unifmodThm2},
thereby avoiding the need of the discussion
\cite[pp.\ 241--244]{partII}.
\end{remark}

The following is the analogue of Theorem \ref{unifmodThm2} for macroscopic initial conditions.
Set
\begin{equation}\label{XXdef2}
X=\big\{ \big\langle\vecq,\vecv,\xi,{\vs},\vecv_+\big\rangle \in \T^1(\R^d)\times\R_{>0}\times\Sigma\times\US \col \vecv_+\in\scrV_{\vecv} \big\}.
\end{equation}
This is the extended phase space; cf.\ Section \ref{sec:linear}. %
\begin{thm}\label{unifmodThm2macr}
Let $\Lambda\in\Pac(\T^1(\R^d))$ 
and let $F$ be a uniformly bounded and pointwise equicontinuous family
of functions $f:X\to\R$.
Then
\begin{align*}%
\int_{\fW(1;\rho)}f\Bigl(\vecq,\vecv,\rho^{d-1}\tau_1(\rho^{1-d}\vecq,\vecv;\rho)
,{\vs}_1(\rho^{1-d}\vecq,\vecv;\rho),\vecv_1(\rho^{1-d}\vecq,\vecv;\rho)\Bigr)\,d\Lambda(\vecq,\vecv)
\hspace{10pt}
\\\notag
-\int_{X} %
f\bigl(\vecq,\vecv,\xi_1,{\vs}_1,\vecv_1\bigr)
p\bigl(\vecv;\xi_1,{\vs}_1,\vecv_1\bigr)
\, d\Lambda(\vecq,\vecv)\,d\xi_1\,d\mm({\vs}_1)\,d\vecv_1
\to0
\end{align*}
as $\rho\to0$,
uniformly with respect to all $f\in F$.
\end{thm}
\begin{proof}
This is similar to the proof of Theorem \ref{unifmodThm2}, but easier.
One uses Theorem \ref{Thm2macr} in place of Theorem \ref{Thm2gen}.
\end{proof}

\section{Bounding the probability of grazing a scatterer or hitting $\scrE$}
\label{bndgrazingprobSEC}
In the proof of our main result, Theorem \ref{MAINTECHNTHM2A},
we will also need the following two propositions,
which say that most initial velocities $\vecv$ give rise to
a ``good'' path and scatterer collision,  %
in the sense that
the particle never gets very near any other scatterer before the collision,
the scatterer involved in the collision does not 
belong to the exceptional set $\scrE$,
and the length of the impact parameter %
is not too close to $1$.
\begin{prop}\label{Thm2genaddlem}
For $U,K,T,\lambda$ as in Theorem \ref{Thm2gen} and $\vecbeta\in K$, set
\begin{align*}
\tfw_{\vecq,\rho}^{\vecbeta}:=\{\vecv\in\fw_{\vecq,\rho}^{\vecbeta}\col\vecq^{(1)}(\vecq_{\rho,\vecbeta}(\vecv),\vecv;\rho)
\in\scrP\setminus\scrE\}.
\end{align*}
Then %
$\lambda(\tfw_{\vecq,\rho}^{\vecbeta})\to1$ as $\rho\to0$,
uniformly over all $\vecq\in\scrP_T(\rho)$ and $\vecbeta\in K$.
\end{prop}

\begin{proof}
Let $\ve>0$ be given.
By Lemma \ref{UNIFKBOUNDGlem1} we can take $T_1>1$ so %
that $\kappa(\vecomega';[T_1-1,\infty)\times\Omega)<\ve$ for all $\vecomega'\in\scrX_\perp$,
and then by Theorem \ref{Thm2gen} there is some $\rho_1\in(0,1)$ such that
\begin{align}\label{Thm2genaddlempf1}
\lambda(\{\vecv\in \fw_{\vecq,\rho}^{\vecbeta}\col\rho^{d-1}\tau_1(\vecq_{\rho,\vecbeta}(\vecv),\vecv;\rho)\geq T_1\})<2\ve
\end{align}
for all $\rho\in(0,\rho_1)$, $\vecq\in\scrP_T(\rho)$, $\vecbeta\in K$.
Let $C:=\sup_{\vecbeta\in K}\|\vecbeta\|$
and $B=[-1,T_1+1]\times\scrB_{C+1}^{d-1}$.
Then by Lemma \ref{EXCHITUNLIKELYlem},
after shrinking $\rho_1$ appropriately we have
\begin{align}\label{Thm2genaddlempf2}
\lambda(\{\vecv\in\US\col\scrE\cap (\vecq+B D_\rho^{-1} R(\vecv)^{-1})\neq\emptyset\})<\ve
\end{align}
for all $\rho\in(0,\rho_1)$, $\vecq\in\scrP_T(\rho)$.
We may also assume $(C+1)\rho_1^d<1$.

Now let 
$\rho\in(0,\rho_1)$, $\vecq\in\scrP_T(\rho)$ and $\vecbeta\in K$ be given,
and consider any 
$\vecv\in \fw_{\vecq,\rho}^{\vecbeta}$
satisfying $\rho^{d-1}\tau_1<T_1$,
where $\tau_1:=\tau_1(\vecq_{\rho,\vecbeta}(\vecv),\vecv;\rho)$.
Then the scattering center $\vecq^{(1)}=\vecq^{(1)}(\vecq_{\rho,\vecbeta}(\vecv),\vecv;\rho)$
has distance $\rho$ from $\vecq+\rho\vecbeta(\vecv)+\tau_1\vecv$, and thus
\begin{align*}
\|(\vecq^{(1)}-\vecq)R(\vecv)-\tau_1\vece_1\|\leq(C+1)\rho,
\end{align*}
and using $(C+1)\rho^d<1$ this is seen to imply
$\vecq^{(1)}\in\vecq+B D_\rho^{-1} R(\vecv)^{-1}$.
Hence, by \eqref{Thm2genaddlempf1} and \eqref{Thm2genaddlempf2}, we have
\begin{align*}
\lambda(\{\vecv\in\fw_{\vecq,\rho}^{\vecbeta}\col\rho^{d-1}\tau_1(\vecq_{\rho,\vecbeta}(\vecv),\vecv;\rho)<T_1
\text{ and }\vecq^{(1)}(\vecq_{\rho,\vecbeta}(\vecv),\vecv;\rho)\notin\scrE\})>1-3\ve,
\end{align*}
and in particular $\lambda(\tfw_{\vecq,\rho}^{\vecbeta})>1-3\ve$.
\end{proof}
To prepare for the next proposition, 
recall \eqref{WVPDEF},
and define
\begin{align}\label{FUETA}
\fU_\eta:=\vecbeta^-_{\vece_1}\bigl(\scrV_{\vece_1}\setminus\scrV_{\vece_1}^{10\eta}\bigr)_\perp\subset\UB,
\qquad \text{for }\:\eta>0.%
\end{align}
(Note that \eqref{FUETA} differs from the notation in 
\cite[(2.33)]{partII}.)
\begin{definition}\label{fgDEF}
For $\vecv\in\fw_{\vecq,\rho}^\vecbeta$, we say that the particle path 
from $\vecq_{\rho,\vecbeta}(\vecv)$ to $\vecq_{\rho,\vecbeta}(\vecv)+\tau_1(\vecq_{\rho,\vecbeta}(\vecv),\vecv;\rho)\vecv$ 
is ``$\eta$-grazing''
if either $\vecw_1(\vecq_{\rho,\vecbeta}(\vecv),\vecv;\rho)\in\fU_\eta$ or
if there is some point $\vecq'\in\scrP\setminus\{\vecq,\vecq^{(1)}(\vecq_{\rho,\vecbeta}(\vecv),\vecv;\rho)\}$
which has distance $<(1+\eta)\rho$ from the line segment between
$\vecq_{\rho,\vecbeta}(\vecv)$ and $\vecq_{\rho,\vecbeta}(\vecv)+\tau_1(\vecq_{\rho,\vecbeta}(\vecv),\vecv;\rho)\vecv$.
Let $\fg_{\vecq,\rho,\eta}^\vecbeta$ be the set of those $\vecv\in\fw_{\vecq,\rho}^\vecbeta$
which give rise to $\eta$-grazing paths.
\end{definition}
\begin{prop}\label{ETAGRAZINGprop}
Let $U,K,T,\lambda$ be as in Theorem \ref{Thm2gen}
and let $\ve>0$.
Then there exist $\eta$ and $\rho_0$ in the interval $(0,1)$ so that
$\lambda(\fg_{\vecq,\rho,\eta}^\vecbeta)<\ve$ for all $\rho\in(0,\rho_0)$,
$\vecq\in\scrP_T(\rho)$, $\vecbeta\in K$.
\end{prop}
\begin{proof}
Without loss of generality we assume that $K$ is compact
(cf.\ Remark \ref{Thm2genRELCPTrem}).
Let us set \label{tfUdef}
$\tfU_\eta:=\fU_\eta\cup(\scrB_1^{d-1}\setminus\scrB_{1-\eta}^{d-1})$,
and note that this is a set of the form $\scrB_1^{d-1}\setminus\scrB_{r(\eta)}^{d-1}$ with
$r(\eta)\to1$ as $\eta\to0$.
Using Lemma \ref{UNIFKBOUNDGlem1}, Lemma \ref{GROUNDINTENSITYlem2} and \eqref{SIGMAeqp},
we have $\kappa(\vecomega';\R_{>0}\times\tfU_\eta\times\Sigma)\to0$ as $\eta\to0$,
uniformly over all $\vecomega'\in\scrX_\perp$.
Hence we may fix $\eta\in(0,1)$ so small that
\begin{align}\label{ETAGRAZINGproppf1}
\kappa(\vecomega';\R_{>0}\times\tfU_{2\eta}\times\Sigma)<\frac\ve4,
\qquad\forall\vecomega'\in\scrX_\perp.
\end{align}
Using Theorem \ref{Thm2gen} and Remark \ref{Thm2genrem}, it follows that there is some $\rho_1\in(0,1)$ such that
\begin{align}\label{ETAGRAZINGproppf2}
\lambda(\{\vecv\in\fw_{\vecq,\rho}^{\vecbeta}\col\vecw_1(\vecq_{\rho,\vecbeta}(\vecv),\vecv;\rho)\in\tfU_\eta\})
+\lambda(U\setminus\fw_{\vecq,\rho}^{\vecbeta})
<\frac\ve2
\end{align}
for all $\rho\in(0,\rho_1)$, $\vecq\in\scrP_T(\rho)$, $\vecbeta\in K$.
Let $C:=1+\sup_{\vecbeta\in K}\|\vecbeta\|$ and
\begin{align*}
\hK=\{\hbeta\col\vecbeta\in K\}\qquad
\text{with }\:
\hbeta(\vecv):=(1+\eta)^{-1}\vecbeta(\vecv)+2C\vecv.
\end{align*}
Then $\hK$ is a compact subset of $\C_b(U,\R^d)$ and
$(\vecbeta(\vecv)+\R_{>0}\vecv)\cap\scrB_1^d=\emptyset$ for all $\vecbeta\in\hK$, $\vecv\in U$,
and using Theorem \ref{Thm2gen} again we see that after possibly shrinking $\rho_1$,
\eqref{ETAGRAZINGproppf2} holds also for all
$\rho\in(0,\rho_1)$, $\vecq\in\scrP_T(\rho)$ and $\vecbeta\in\hK$.
Furthermore, by Lemma \ref{GOODDISTANCElem}, we may assume that 
$d_\scrP(\vecq)>5C\rho$ for all $\rho\in(0,\rho_1)$ and $\vecq\in\scrP_T(\rho)$.

Now take any $\rho\in(0,(1+\eta)^{-1}\rho_1)$, $\vecq\in\scrP_T(\rho)$ and $\vecbeta\in K$.
Set $\hrho:=(1+\eta)\rho$.
Then $\hrho\in(0,\rho_1)$, $\vecq\in\scrP_T(\hrho)$ and $\hbeta\in\hK$,
and so by the above we have
\begin{align}\notag
\lambda\Big(\Big\{\vecv\in\fw_{\vecq,\rho}^{\vecbeta}\col\vecw_1(\vecq_{\rho,\vecbeta}(\vecv),\vecv;\rho)\in\fU_\eta
\text{ or }\vecv\notin\fw_{\vecq,\hrho}^{\hbeta}
\text{ or }\vecw_1(\vecq_{\hrho,\hbeta}(\vecv),\vecv;\hrho)\notin\scrB_{1-\eta}^{d-1}\Big\}\Big)
\\\label{ETAGRAZINGproppf3}
<\frac\ve2+\frac\ve2=\ve.
\end{align}
Now assume that $\vecv\in\fw_{\vecq,\rho}^{\vecbeta}$ has the property that 
there exists some point $\vecq'\in\scrP\setminus\{\vecq,\vecq^{(1)}(\vecq_{\rho,\vecbeta}(\vecv),\vecv;\rho)\}$
which has distance $<\hrho$ from some point $\vecx$ on the line segment between
$\vecq_{\rho,\vecbeta}(\vecv)$ and $\vecq_{\rho,\vecbeta}(\vecv)+\tau_1\vecv$,
where $\tau_1=\tau_1(\vecq_{\rho,\vecbeta}(\vecv),\vecv;\rho)$.
Assume also $\vecv\in\fw_{\vecq,\hrho}^{\hbeta}$,
and set $\vecq^{(1)}=\vecq^{(1)}(\vecq_{\rho,\vecbeta}(\vecv),\vecv;\rho)$,
$\hq^{(1)}=\vecq^{(1)}(\vecq_{\hrho,\hbeta}(\vecv),\vecv;\hrho)$,
and $\htau_1=\tau_1(\vecq_{\hrho,\hbeta}(\vecv),\vecv;\hrho)$.
Note that $\vecq_{\hrho,\hbeta}(\vecv)=\vecq_{\rho,\vecbeta}(\vecv)+2C\hrho\vecv$.
Also %
$d_\scrP(\vecq)>5C\rho$, whence $\tau_1>2C\hrho$,
and it follows that the line segment from
$\vecq_{\hrho,\hbeta}(\vecv)$ to $\vecq_{\hrho,\hbeta}(\vecv)+\htau_1\vecv$ 
is a strict subset of the line segment from
$\vecq_{\rho,\vecbeta}(\vecv)$ and $\vecq_{\rho,\vecbeta}(\vecv)+\tau_1\vecv$.
If $\hq^{(1)}=\vecq^{(1)}$, then $\vecx$ must lie 
between $\vecq_{\hrho,\hbeta}(\vecv)+\htau_1\vecv$ and $\vecq_{\rho,\vecbeta}(\vecv)+\tau_1\vecv$;
this implies that 
$\vecx\in\scrB^d\bigl(\hq^{(1)},\hrho\bigr)\cap\scrB^d\bigl(\vecq',\hrho\bigr)$,
i.e.\ the scatterer 
$\scrB^d\bigl(\hq^{(1)},\hrho\bigr)$
is not separated, contradicting
$\vecv\in\fw_{\vecq,\hrho}^{\hbeta}$
{\blu (cf.\ the definitions at the beginning of Section \ref{FIRSTCOLLsec}).}
Hence $\hq^{(1)}\neq\vecq^{(1)}$,
and then from the definitions of these points it follows that $\hq^{(1)}$ has distance $\geq\rho$ from the 
ray $\vecq_{\hrho,\hbeta}(\vecv)+\R_{>0}\vecv$,
and so $\|\vecw_1(\vecq_{\hrho,\hbeta}(\vecv),\vecv;\hrho)\|\geq(1+\eta)^{-1}>1-\eta$,
i.e.\ $\vecv$ belongs to the set in \eqref{ETAGRAZINGproppf3}.

It follows from the above 
discussion that $\fg_{\vecq,\rho,\eta}^\vecbeta$ is a subset of the set in \eqref{ETAGRAZINGproppf3}.
Hence $\lambda(\fg_{\vecq,\rho,\eta}^\vecbeta)<\ve$ for all $\rho\in(0,(1+\eta)^{-1}\rho_1)$,
$\vecq\in\scrP_T(\rho)$, $\vecbeta\in K$,
and the proposition is proved.
\end{proof}

\subsection{Macroscopic initial conditions}
The macroscopic analogue of Proposition \ref{Thm2genaddlem} is as follows.
\begin{prop}\label{Thm2genaddgenlem}
Let $\Lambda\in \Pac(\T^1(\R^d))$. Then
\begin{align*}
\Lambda(\{(\vecq,\vecv)\in\fW(1;\rho)\col\vecq^{(1)}(\rho^{1-d}\vecq,\vecv;\rho)\in\scrE\})\to0
\qquad\text{as }\:\rho\to0.
\end{align*}
\end{prop}
\begin{proof}
The proof of Proposition \ref{Thm2genaddlem} carries over with simple modifications,
using Theorem~\ref{Thm2macr}  in place of Theorem \ref{Thm2gen}.
The only step which is not immediate is %
the following fact,
which is the required analogue of \eqref{Thm2genaddlempf2}:
For any relatively compact set $B\subset\R^d$,
\begin{align}\label{Thm2genaddgenlempf1}
\Lambda\bigl(\bigl\{\scrE\cap\bigl(\rho^{1-d}\vecq+BD_\rho^{-1}R(\vecv)^{-1}\bigr)\neq\emptyset\bigr\}\bigr)\to0
\quad\text{as }\:\rho\to0.
\end{align}
To prove \eqref{Thm2genaddgenlempf1}, we first note that,
as in the proof of Prop.\ \ref{GENLIMITprop1}, we may reduce to the case when
$\Lambda$ has a density
$\Lambda'\in\C_c(\T^1(\R^d))$ with respect to $\vol_{\R^d\times\US}$.  %
Then each point $\vecp\in\scrE$ gives a contribution to the expression in %
\eqref{Thm2genaddgenlempf1} which is bounded above by
\begin{align*}
\Bigl(\sup_{\T^1(\R^d)}|\Lambda'|\Bigr)\cdot
\int_{\R^d}\int_{\US}I\bigl(\vecp\in\bigl(\rho^{1-d}\vecq+BD_\rho^{-1}R(\vecv)^{-1}\bigr)\bigr)\,d\vecv\,d\vecq
\hspace{70pt}
\\
=\Bigl(\sup_{\T^1(\R^d)}|\Lambda'|\Bigr)\vol(B)\omega(\US)\cdot\rho^{d(d-1)}.
\end{align*}
Take $R,R'>0$ so that $\supp\Lambda'\subset\scrB_R^d\times\US$ and $B\subset\scrB_{R'}^d$.
Then $BD_\rho^{-1}\subset\scrB_{\rho^{1-d}R'}^d$,
and therefore only points $\vecp\in\scrE$ with 
$\|\vecp\|<\rho^{1-d}(R+R')$ can give any contribution to the expression in \eqref{Thm2genaddgenlempf1}.
Hence that expression is bounded above by
\begin{align*}
\#\bigl(\scrE\cap\scrB_{\rho^{1-d}(R+R')}^d\bigr)\cdot\Bigl(\sup_{\T^1(\R^d)}|\Lambda'|\Bigr)
\vol(B)\omega(\US)\cdot\rho^{d(d-1)}.
\end{align*}
Now \eqref{Thm2genaddgenlempf1}
follows from %
{\blu the fact that $\scrE$ has asymptotic density zero (cf.\ [P2]).}
\end{proof}

Finally we give the macroscopic analogue of Proposition \ref{ETAGRAZINGprop}.
\begin{definition}
Let $\fG_{\rho,\eta}$ be the set of all $(\vecq,\vecv)\in\fW(1;\rho)$ which give rise to $\eta$-grazing paths,
\label{fGrhoetaDEF}
i.e.\ paths such that $\vecw_1(\rho^{1-d}\vecq,\vecv;\rho)\in\fU_\eta$
or such that there exists some $\vecq'\in\scrP\setminus\{\vecq^{(1)}(\rho^{1-d}\vecq,\vecv;\rho)\}$
which has distance $<(1+\eta)\rho$ from the line segment between $\rho^{1-d}\vecq$
and $\rho^{1-d}\vecq+\tau_1(\rho^{1-d}\vecq,\vecv;\rho)\vecv$.
\end{definition}
\begin{prop}\label{ETAGRAZINGgenprop}
Let $\Lambda\in\Pac(\T^1(\R^d))$
and $\ve>0$.
Then there exist $\eta$ and $\rho_0$ in the interval $(0,1)$ such that
$\Lambda(\fG_{\rho,\eta})<\ve$ for all $\rho\in(0,\rho_0)$.
\end{prop}
\begin{proof}
The proof of Proposition \ref{ETAGRAZINGprop} carries over with some modifications.
First, fix $\eta\in(0,1)$ so small that
\begin{align*}
\kappa^{\g}\bigl(\R_{>0}\times\tfU_{2\eta}\times\Sigma\bigr)<\frac{\ve}4,
\end{align*}
where $\tfU_\eta=\fU_\eta\cup(\scrB_1^{d-1}\setminus\scrB_{1-\eta}^{d-1})$ as before.
We introduce the scaling map
\begin{align*}
S:\T^1(\R^d)\to\T^1(\R^d),\qquad
S(\vecq,\vecv)=((1+\eta)^{d-1}\vecq,\vecv).
\end{align*}
By Theorem \ref{Thm2macr}  and Remark \ref{Thm2macrrem},
applied to both the measures $\Lambda$ and $\Lambda\circ S^{-1}$,
it follows that there exists some $\rho_1\in(0,1)$
such that for every $\rho\in(0,\rho_1)$ we have
\begin{align}\label{ETAGRAZINGgenproppf1}
\Lambda\bigl(\{(\vecq,\vecv)\in\T^1(\R^d)\col (\vecq,\vecv)\notin\fW(1;\rho)\:\text{ or }\:
\vecw_1(\rho^{1-d}\vecq,\vecv;\rho)\in\tfU_\eta\}\bigr)<\frac{\ve}2
\end{align}
as well as
\begin{align}\label{ETAGRAZINGgenproppf2}
\Lambda\circ S^{-1}\bigl(\{(\vecq,\vecv)\in\T^1(\R^d)\col (\vecq,\vecv)\notin\fW(1;\rho)\:\text{ or }\:
\vecw_1(\rho^{1-d}\vecq,\vecv;\rho)\in\tfU_\eta\}\bigr)<\frac{\ve}2.
\end{align}

Now take any $\rho\in(0,(1+\eta)^{-1}\rho_1)$.
Set $\hrho:=(1+\eta)\rho$.
Then \eqref{ETAGRAZINGgenproppf2} holds with $\hrho$ in place of $\rho$,
and %
this statement %
can be equivalently expressed as:
\begin{align*}
\Lambda\bigl(\bigl\{(\vecq,\vecv)\in\T^1(\R^d)\col
(\rho^{1-d}\vecq,\vecv)\notin\fw(1;\hrho)\:\text{ or }\:
\vecw_1(\rho^{1-d}\vecq,\vecv;\hrho)\in\tfU_\eta\bigr\}\bigr)<\frac{\ve}2.
\end{align*}
The last bound together with \eqref{ETAGRAZINGgenproppf2} imply:
\begin{align}\notag
\Lambda\Bigl(\Bigl\{(\vecq,\vecv)\in\fW(1;\rho)\col
\vecw_1(\rho^{1-d}\vecq,\vecv;\rho)\in\fU_\eta
\:\text{ or }\:
(\rho^{1-d}\vecq,\vecv)\notin\fw(1;\hrho)
\hspace{50pt}
\\\label{ETAGRAZINGgenproppf3}
\:\text{ or }\:
\vecw_1(\rho^{1-d}\vecq,\vecv;\hrho)\notin\scrB_{1-\eta}^{d-1}\Bigr\}\Bigr)<\ve.
\end{align}
By the same argument as in the proof of Proposition \ref{ETAGRAZINGprop},
$\fG_{\rho,\eta}$ is verified to be a subset of the set in 
the left hand side of \eqref{ETAGRAZINGgenproppf3}.
Hence the proposition is proved.
\end{proof}

\chapter{Convergence to a random flight process}
\label{MAINRESsec}

\section{Joint distribution of path segments}
\label{MAINTECHNthm2aSEC}

Theorem \ref{MAINTECHNTHM2A} below is our first main result;
it gives the limit of the joint distribution of the first $n$ flight segments
and the marks of the corresponding scatterers.
It generalizes \cite[Thm.\ 4.1]{partII}
(specialized to start from a scatterer)
from the case of a lattice to the case of an arbitrary point set $\scrP$ satisfying the assumptions in 
Section \ref{ASSUMPTLISTsec}.

Recall the definitions of $\fw(j;\rho)$,
$\vecq^{(j)}(\vecq,\vecv;\rho)$,
$\vs_j(\vecq,\vecv;\rho)$, 
$\vecw_j(\vecq,\vecv;\rho)$, 
$\vecq_j(\vecq,\vecv;\rho)$
and $\vecv_j(\vecq,\vecv;\rho)$
given in Section \ref{SCATTERINGMAPS}.
Given an open subset $U\subset\US$ and a function $\vecbeta\in \C_b(U,\R^d)$, we set
\begin{align}\label{fwqrhonDEF}
\fw_{\vecq,\rho,n}^{\vecbeta}:=\{\vecv\in U\col(\vecq_{\rho,\vecbeta}(\vecv),\vecv)\in\fw(n;\rho)\}.
\end{align}
This notation generalizes that of \eqref{wqrbdef}, in that
$\fw_{\vecq,\rho}^{\vecbeta}=\fw_{\vecq,\rho,1}^\vecbeta$.
We also introduce the following notation generalizing \eqref{XSPACEdef1}:
\begin{align}\label{XSPACEdef}
X_U^{(n)}\hspace{-2pt}:=\hspace{-2pt}\Bigl\{\langle\vecv_0,\langle\xi_j,{\vs}_j,\vecv_j\big\rangle_{j=1}^{n}\rangle\in
U\hspace{-1pt}\times\hspace{-1pt}(\R_{>0}\times\hspace{-1pt}\Sigma\hspace{-1pt}\times\hspace{-1pt}\US)^n
:\: \vecv_j\in\scrV_{\vecv_{j-1}}\: %
(j=1,\ldots,n)\Bigr\}.
\end{align}

For $\vecv\in\US$, the tangent space $\T_\vecv(\US)$ is naturally identified with 
the orthogonal complement of $\vecv$ in $\R^d$.
For $\vech\in\T_\vecv(\US)$, we write $D_\vech$ for the corresponding derivative.
We use the standard Riemannian metric for $\S_1^{d-1}$, and denote by $\T^1_\vecv(\S_1^{d-1})$ the set of 
unit vectors in $\T_\vecv(\S_1^{d-1})$. 
For any %
open subset $U\subset\S_1^{d-1}$ we write 
\begin{align*}
\T^1(U)=\bigsqcup_{\vecv\in U} \T^1_\vecv(\S_1^{d-1}).
\end{align*}
for the unit tangent bundle of $U$.\label{UNITTBDLEofU}

We will formulate the limit result of Theorem \ref{MAINTECHNTHM2A}
in a way that is uniform with respect to certain families of functions $\vecbeta:U\to\R^d$.
This will be crucial for making it possible to prove the theorem by induction over $n$.
\begin{definition}
For $U$ an open subset of $\US$,
let $\C^1_b(U,\R^d)$ be the space of $\C^1$ functions $\vecbeta:U\to\RR^{d}$ which are bounded and satisfy
$\sup_{\vech\in \T^1(U)} \|D_\vech\vecbeta\|<\infty$.
\label{Cb1URddef}
We call a subset $F$ of $\C^1_b(U,\R^d)$ 
\textit{admissible}
if it is relatively compact as a subset of $\C_b(U,\R^d)$
and satisfies $\sup_{\vecbeta\in F}\sup_{\vech\in \T^1(U)} \|D_\vech\vecbeta\|<\infty$
and $(\vecbeta(\vecv)+\R_{>0}\vecv)\cap\scrB_1^d=\emptyset$ for all $\vecbeta\in F$ and $\vecv\in U$.
\end{definition}
\begin{thm}\label{MAINTECHNTHM2A}
Let $\scrP$ {\blu and $\scrE$}
satisfy all the conditions in Section \ref{ASSUMPTLISTsec} and \eqref{SIGMAeqp},
and let $\Psi$ be a scattering process satisfying the conditions in Section \ref{SCATTERINGMAPS}.
Let $n\in\Z_{\geq1}$ and $T\in\R_{\geq1}$;
let $U$ be an open subset of $\US$; %
let $F_1$ be an equismooth family of probability measures on $\S_1^{d-1}$ 
such that $\lambda(U)=1$ for each $\lambda\in F_1$;
let $F_2$ be a uniformly bounded and pointwise equicontinuous family
of functions $f:{X_U^{(n)}}\to\R$; %
and let $F_3$ be an admissible subset of $\C^1_b(U,\R^d)$.
Then
\begin{align}\notag
\int_{\fw_{\vecq,\rho,n}^{\vecbeta}}f\Bigl(\vecv,\Big\langle\rho^{d-1}\tau_j(\vecq_{\rho,\vecbeta}(\vecv),\vecv;\rho)
,{\vs}_j(\vecq_{\rho,\vecbeta}(\vecv),\vecv;\rho),
\vecv_j(\vecq_{\rho,\vecbeta}(\vecv),\vecv;\rho)\Big\rangle_{j=1}^{n}\Bigr)\,d\lambda(\vecv)
\\\label{unifmodThm2ares}
-\int_{{X_U^{(n)}}} %
f\bigl(\vecv_0,\big\langle\xi_j,{\vs}_j,\vecv_j\big\rangle_{j=1}^{n}\bigr)
p_{\bn,\vecbeta}\bigl(\vs(\vecq),\vecv_0;\xi_1,{\vs}_1,\vecv_1\bigr)
\hspace{85pt}
\\\notag
\times\prod_{j=2}^n p_\bn(\vecv_{j-2},{\vs}_{j-1},\vecv_{j-1};\xi_j,{\vs}_{j},\vecv_{j})
\, d\lambda(\vecv_0)\,\prod_{j=1}^n\bigl( d\xi_j\,d\mm({\vs}_j)\,d\vecv_j\bigr)\to0
\end{align}
as $\rho\to0$,
uniformly with respect to all $\vecq\in\scrP_T(\rho)$, $\lambda\in F_1$, $f\in F_2$, $\vecbeta\in F_3$.
\end{thm}
\begin{remark}\label{unifmodThmREM} 
Regarding the limit expression in \eqref{unifmodThm2ares}, 
one should note that
\begin{align*}\notag %
\int_{{X_U^{(n)}}} %
p_{\bn,\vecbeta}\bigl(\vs,\vecv_0;\xi_1,{\vs}_1,\vecv_1\bigr)
\prod_{j=2}^n p_\bn(\vecv_{j-2},{\vs}_{j-1},\vecv_{j-1};\xi_j,{\vs}_{j},\vecv_{j})
\,d\lambda(\vecv_0)\hspace{30pt}
\\
\times\prod_{j=1}^n\bigl( d\xi_j\,d\mm({\vs}_j)\,d\vecv_j\bigr)=1
\end{align*}
for all $\vs\in\Sigma$ and $\vecbeta\in\C_b(U,\R^d)$.
This follows by using \eqref{REMOVINGCPTSUPPfactpfLEM1res1} in Lemma \ref{REMOVINGCPTSUPPfactpfLEM1}
$n$ times.
In particular, taking $f\equiv1$ in \eqref{unifmodThm2ares}, 
the theorem implies that
$\lambda(\fw_{\vecq,\rho,n}^{\vecbeta})\to1$ as $\rho\to0$,
uniformly with respect to all
$\vecq\in\scrP_T(\rho)$, $\lambda\in F_1$, and $\vecbeta\in F_3$.
\end{remark}

\section{Auxiliary results}
\label{scatmapsmore}
We next review some results from 
\cite[Section 3]{partII}.

Recall the definition of the maps $\vecbeta_{\vecv}^{\pm}$
and of the differential cross section $\sigma(\vecv,\vecv_+)$
from Section \ref{SCATTERINGMAPS}.
Set
\begin{align} \label{C1THETAETADEF}
C_{\eta}:=1+\max\Bigl(
\sup_{\vech\in \T^1(\scrV_{\vecv}^\eta)} 
\bigl\|D_\vech\vecbeta^+_{\vecv}\bigr\|,
\sup_{\vech\in \T^1(\scrV_{\vecv}^\eta)} 
\bigl\|D_\vech\vecbeta^-_{\vecv}\bigr\|
\Bigr).
\end{align}
Then $C_\eta$ is independent of $\vecv$, depends continuously on 
$\eta>0$, and may approach infinity as $\eta\to 0$.

For any $\vecs\in\R^d\setminus\{\bn\}$ we let $\nu_\vecs$ be the 
probability measure on $\S_1^{d-1}$ which gives\label{nusDEF}
the direction of a ray after it has been scattered 
in the ball $\scrB_1^{d}$, given that the
incoming ray has direction \hypertarget{hsdeflink}{$\hs:=\|\vecs\|^{-1}\vecs$}
\label{hsdef}
and is part of the line 
$\vecx+\R\vecs$ with $\vecx$ picked at random in 
the $(d-1)$-dimensional unit ball $\{\vecs\}^\perp\cap \scrB_1^d$, with respect to 
the $(d-1)$-dimensional Lebesgue measure.
Thus 
\begin{align}\label{LAMBDA0DEFnew}
d\nu_\vecs(\vecv)=v_{d-1}^{-1}\sigma(\hs,\vecv)\,d\vecv. %
\end{align}
Let us write $\scrV_\vecs^\eta:=\scrV_{\hs}^\eta$
(cf.\ \eqref{WVPDEF}).
For $\eta>0$ so small that $\scrV_\vecs^\eta\neq\emptyset$,
we define $\nu_\vecs^\eta$ to be the probability measure
which is obtained by restricting $\nu_\vecs$ to $\scrV_\vecs^\eta$ and renormalizing,
i.e.\ 
\begin{align}\label{nuseta}
\nu_\vecs^\eta:=\nu_\vecs(\scrV_\vecs^\eta)^{-1}\cdot\nu_\vecs|_{\scrV_\vecs^\eta}.
\end{align}

Given $\vecs\in\R^d\setminus\{\bn\}$,
a number $\rho>0$, an open set $U\subset\US$ and a continuous function
$\vecbeta:U\to\R^d$
subject to the condition $\rho\vecbeta(\vecv)\notin\scrB^d(\vecs,\rho)$ $\forall\vecv\in U$,
we set 
\begin{align} \label{OMEGARHOMBDEF}
\scrU%
=\bigl\{\vecv\in U\col (\rho\vecbeta(\vecv)+\R_{>0}\vecv)\cap
\scrB^d(\vecs,\rho)
\neq\emptyset\bigr\}.
\end{align}
For $\vecv\in\scrU$ we set
\begin{equation}\label{TAUrhovecsvecbetaDEF}
\tau(\vecv)=\tau_{\rho,\vecs,\vecbeta}(\vecv)
:=\inf\bigl\{t>0\col\rho\vecbeta(\vecv)+t\vecv\in\scrB^d(\vecs,\rho)\bigr\},
\end{equation}
let $\vecB(\vecv)=\vecB_{\!\rho,\vecs,\vecbeta}(\vecv)$ be the impact location on $\S_1^{d-1}$, i.e., the point for 
which $\rho\vecbeta(\vecv)+\tau(\vecv)\vecv=\vecs+\rho \vecB(\vecv)$,
and let
\begin{align} \label{SHINELEM2VDEF}
\vecV(\vecv)=\vecV_{\!\!\!\rho,\vecs,\vecbeta}(\vecv):=\Psi_1(\vecv,\vecB(\vecv))\in\S_1^{d-1},
\end{align}
the outgoing direction after the ray
$\rho\vecbeta(\vecv)+\R_{>0}\vecv$ is scattered in the sphere 
$\vecs+\S_\rho^{d-1}$.

Let us write $\scrD_\vecs^\eta:=\{\vecv\in\US\col\varphi(\vecv,\vecs)<\eta\}$ for the 
\label{scrDsetaDEF}
ball of radius $\eta$ with center $\hs$ in $\US$.
We now have:
\begin{lem} \label{SHINELEM1}
Given any $0<\eta<\frac1{100}\bigl(\frac{\pi}2-{s_\Psi}(\frac{\pi}2-B_\Psi)\bigr)$,   %
$C\geq 10$ and $\ve>0$, there exists a constant
$\trho_0=\trho_0(\eta,C,\ve)>0$ such that 
for any $\rho\in (0,\trho_0)$,
any $\vecs\in\R^d$ with $\|\vecs\|\geq C^{-1}$,
any open subset $U\subset\US$ containing $\scrD_\vecs^\eta$,
and any $\C^1$-function $\vecbeta:U\to\R^d$
satisfying $\sup_{\vecv\in U}\|\vecbeta(\vecv)\|\leq C$ and
$\sup_{\vech\in \T^1(U)} \|D_\vech\vecbeta\|\leq C$,
all of the following statements hold:
\begin{enumerate}[{\rm (i)}]
\item Let $\overline{\vecV}=\overline{\vecV}_{\!\!\rho,\vecs,\vecbeta}$ be the 
restriction of $\vecV=\vecV_{\!\!\rho,\vecs,\vecbeta}$ to $\vecV^{-1}(\scrV_\vecs^\eta)$;
then $\overline{\vecV}$ is a $\C^1$ diffeomorphism onto $\scrV_\vecs^\eta$.
\item If $M\subset\scrV^{\eta}_\vecs$ 
is any Borel subset with $\nu_\vecs(M)>0$ and if $\mu$ denotes
the measure $\omega$ restricted to $\overline{\vecV}^{-1}(M)$ and
rescaled to be a probability measure, then
$\overline{\vecV}_*\mu=g\cdot \nu_{\vecs}(M)^{-1}{\nu_{\vecs}}_{|M}$
for some continuous function $g:M\to [1-\ve,1+\ve]$.
\item Define the $\C^1$ maps $\vecB^\pm=\vecB^\pm_{\rho,\vecs,\vecbeta}:\scrV^\eta_\vecs
\to \S_1^{d-1}$ through
$\vecB^\pm(\vecu)=\vecbeta^\pm_{\overline{V}^{-1}(\vecu)}(\vecu).$ Then 
$\bigl\| \vecB^\pm(\vecu)-\vecbeta^\pm_{\uvecs}(\vecu)\bigr\|<\ve$
for all $\vecu\in\scrV_\vecs^\eta$ and
$\|D_\vech \vecB^\pm\|<C_{\eta}$ for all 
$\vech\in \T^1(\scrV_\vecs^\eta)$.
\end{enumerate}
\end{lem}
\begin{proof}
This is 
\cite[Lemma 3.2]{partII},
mildly generalized by allowing a more general set $U$ in place of
``$\scrV_\vecr^\eta$'',
and allowing either $s_\Psi=1$ or $-1$ (whereas in \cite{partII} we assumed $s_\Psi=-1$).
The proof carries over immediately, 
using the assumption $\scrD_\vecs^\eta\subset U$;
cf.\ in particular \cite[(3.9)]{partII}.
\end{proof}

Next we prove a lemma which is useful for reducing to test functions $f$ of compact support
in the proof of Theorem \ref{MAINTECHNTHM2A}.
\begin{lem}\label{REMOVINGCPTSUPPfactLEM}
Let $U$ be an open subset of $\US$ and let $\lambda\in P(U)$ and $n\in\Z^+$.
Given any $\ve>0$ there is a compact subset $K\subset X_U^{(n)}$ %
such that 
\begin{align*}%
\int_{X_U^{(n)}\setminus K} p_{\bn,\vecbeta}\bigl(\vs,\vecv_0;\xi_1,{\vs}_1,\vecv_1\bigr)
\prod_{j=2}^n p_\bn(\vecv_{j-2},{\vs}_{j-1},\vecv_{j-1};\xi_j,{\vs}_{j},\vecv_{j})
\, d\lambda(\vecv_0)
\hspace{40pt}
\\
\times
\prod_{j=1}^n\bigl( d\xi_j\,d\mm({\vs}_j)\,d\vecv_j\bigr)
<\ve
\end{align*}
for all $\vs\in\Sigma$ and $\vecbeta\in \C_b(U,\R^d)$.
\end{lem}
\begin{proof}
Take $\ve'>0$ so small that $(1-\ve')^{n+1}>1-\ve$.
Let $K_U$ be a compact subset of $U$ such that $\lambda(K_U)>1-\ve'$.
By Lemma \ref{REMOVINGCPTSUPPfactpfLEM1} we can take $C>1$ and $\eta>0$ so that 
\begin{align}\label{REMOVINGCPTSUPPfactpfLEM1res2REP}
&\int_{1/C}^C\int_{\Sigma}\int_{\overline{\scrV^\eta_\vecv}}
p_{\bn,\vecbeta}(\vs,\vecv;\xi,\vs_+,\vecv_+)\,d\vecv_+\,d\mm(\vs_+)\,d\xi>1-\ve'
\end{align}
holds for any open set $U'\subset\US$, and any $\vecbeta\in\C^d(U',\R^d)$, $\vecv\in U'$, $\vs\in\Sigma$.
Now set
\begin{align*}
K:=%
\Bigl\{\langle\vecv_0;\langle\xi_j,{\vs}_j,\vecv_j\big\rangle_{j=1}^{n}\rangle\in{X_U^{(n)}}
\col \vecv_0\in K_U,\: \xi_j\in[C^{-1},C],\:
\hspace{70pt}
\\
\vecv_j\in\overline{\scrV^\eta_{\vecv_{j-1}}}\: %
(j=1,\ldots,n)\Bigr\}.
\end{align*}
Using \eqref{REMOVINGCPTSUPPfactpfLEM1res2REP} iteratively $n$ times it follows that
for any $\vs\in\Sigma$ and $\vecbeta\in \C_b(U,\R^d)$, 
\begin{align*}
\int_K p_{\bn,\vecbeta}\hspace{-1pt}\bigl(\vs,\vecv_0;\xi_1,{\vs}_1,\vecv_1\bigr)\hspace{-1pt}
\prod_{j=2}^n p_\bn(\vecv_{j-2},{\vs}_{j-1},\vecv_{j-1};\xi_j,{\vs}_{j},\vecv_{j})
\, d\lambda(\vecv_0)\prod_{j=1}^n\bigl( d\xi_j\,d\mm({\vs}_j)\,d\vecv_j\bigr)
\\
>(1-\ve')^n\int_{K_U}\,d\lambda(\vecv_0)>(1-\ve')^{n+1}>1-\ve.
\end{align*}
Recalling also Remark \ref{unifmodThmREM}, the lemma follows.
\end{proof}

Next we give a lemma about varying $\vecbeta$ in
$p_{\bn,\vecbeta}({\vs},\vecv;\xi,{\vs}_+,\vecvp)$.
\begin{lem}\label{TRKERCONTlem}
Let $U$ be an open subset of $\US$, and let $f\in \C_c(X_U^{(2)})$ and $\ve>0$.
Then there exists $\nu>0$ such that for any 
$\vecv_0\in U$, $\xi'>0$, $\vs'\in\Sigma$,
$U'\subset\US$, $\vecv'\in U'\cap\scrV_{\vecv_0}$
and any continuous functions 
$\vecbeta_1,\vecbeta_2:U'\to\US$,
if $\|\vecbeta_1(\vecv')-\vecbeta_2(\vecv')\|<\nu$ then
\begin{align}\notag
\biggl|\int_{\R_{>0}\times\Sigma\times\scrV_{\vecv'}}
f(\vecv_0,\xi',\vs',\vecv',\xi,\vs,\vecv)
\Bigl(p_{\bn,\vecbeta_1}\bigl(\vs',\vecv';\xi,\vs,\vecv\bigr)
-p_{\bn,\vecbeta_2}\bigl(\vs',\vecv';\xi,\vs,\vecv\bigr)\Bigr)
\hspace{20pt}
\\\label{TRKERCONTlemres}
\times
 d\xi\,d\mm(\vs)\,d\vecv\biggr|<\ve.
\end{align}
\end{lem}
\begin{proof}
By Lemma \ref{BASICpktransfrem1lem},
the expression inside the absolute value in the left hand side of \eqref{TRKERCONTlemres} equals
\begin{align*}
\int_0^\infty\int_{\Omega}
f_1(\xi,\vecomega)\Bigl(k\bigl(((\vecbeta_1(\vecv')R(\vecv'))_\perp,\vs'),\xi,\vecomega\bigr)
-k\bigl(((\vecbeta_2(\vecv')R(\vecv'))_\perp,\vs'),\xi,\vecomega\bigr)\Bigr)
\hspace{20pt}
\\
\times d\mu_\Omega(\vecomega)\,d\xi,
\end{align*}
where $f_1\in\C_b(\R_{>0}\times\Omega)$ is given by
\begin{align*}
f_1(\xi,(\vecw,\vs))
=f(\vecv_0,\xi',\vs',\vecv',\xi,\vs,\Psi_1(\vece_1,s_-(\vecw))R(\vecv')^{-1})
\end{align*}
with $s_-$ as in \eqref{smDEF}.
Let $F_1$ be the subset of $\C_b(\R_{>0}\times\Omega)$ given by
\begin{align*}
F_1:=\bigl\{\langle\xi,(\vecw,\vs)\rangle\mapsto
f\bigl(\vecv_0,\xi',\vs',\vecv',\xi,\vs,\Psi_1(\vece_1,s_-(\vecw))R(\vecv')^{-1}\bigr)
\hspace{80pt}
\\
\col 
\vecv_0\in U,\:  \xi'>0,\:  \vs'\in\Sigma,\:  \vecv'\in\scrV_{\vecv_0}\bigr\}.
\end{align*}
Using $f\in \C_c(X_U^{(2)})$ one shows that $F_1$ is relatively compact.
It now suffices to prove that there exists $\nu>0$ such that for any
$\vecw_0,\vecw_0'\in\overline{\UB}$, $\vs'\in\Sigma$, $f_1\in F_1$,
if $\|\vecw_0-\vecw_0'\|<\nu$ then
\begin{align*}
&\biggl|\int_0^\infty\int_{\Omega}
f_1(\xi,\vecomega)\bigl\{k((\vecw_0,\vs'),\xi,\vecomega)-k((\vecw_0',\vs'),\xi,\vecomega)\bigr\}
\,d\mu_\Omega(\vecomega)\,d\xi\biggr|<\ve.
\end{align*}
However this is a consequence of Remark \ref{CONTINTEGRALlemG2REM},
since $F_1$ is relatively compact.
\end{proof}

\section{Proof of Theorem \ref*{MAINTECHNTHM2A}}
\label{MAINTECHNthm2apfSEC}
We now prove Theorem \ref{MAINTECHNTHM2A}.
The proof is by induction.
The case $n=1$ is already covered by Theorem \ref{unifmodThm2}.
Hence we now fix $n\geq2$;
we assume that the statement of Theorem \ref{MAINTECHNTHM2A} holds 
with $n-1$ in the place of $n$;
our goal is to prove that the statement also holds for $n$.

\vspace{5pt}

\subsection{Initial reductions}

By standard approximation arguments similar to those used in the proof of
Theorem \ref{unifmodThm2}
(again cf.\ \cite[Thm.\ 2.3; Steps 2--4]{partII})
and utilizing Lemma \ref{REMOVINGCPTSUPPfactLEM},
it suffices to prove the desired statement
in the special case when $F_1$ and $F_2$ are \textit{singleton} sets, 
with the unique function $f\in F_2$ having compact support,
i.e.\ $f\in\C_c({X_U^{(n)}})$.
Hence from now on we restrict to that situation.
Let $\lambda$ be the unique element in $F_1$.
Since $F_1=\{\lambda\}$ is equismooth,
we have
$\lambda=(g\cdot\omega)_{|U_1}$ for some continuous function $g:\US\to\R_{\geq0}$
and some open set $U_1\subset\US$ satisfying $\omega(\partial U_1)=0$.
Let $K$ be the image of the support of $f$ under the projection map from $X_U^{(n)}$ to $U$;
this is a compact subset of $U$.
Let us show that without loss of generality we may assume $U_1\subset U$.
To this end, choose an open neighborhood $U'$ of $K$ satisfying $\overline{U'}\subset U$ and $\omega(\partial U')=0$.
Note that the desired limit statement, \eqref{unifmodThm2ares}, remains the same
if we replace %
$\lambda$ by %
$\lambda(U')^{-1}\cdot\lambda_{|U'}\in P(\US)$
(in the special case $\lambda(U')=0$ the limit statement is of course trivial).
This corresponds to replacing $g$ by $\lambda(U')^{-1}\cdot g$ and $U_1$ by $U_1\cap U'$;
and we note that $\omega(\partial(U_1\cap U'))=0$ since $\partial(U_1\cap U')\subset\partial U_1\cup\partial U'$.
After having carried out these replacements, we have
\begin{align*}
\lambda=(g\cdot\omega)_{|U_1}\in P(\US);\qquad \overline{U_1}\subset U;\qquad \omega(\partial U_1)=0.
\end{align*}

\subsection{Introducing auxiliary parameters, functions and spaces}

Since $f$ has compact support, we can choose $C_1>1$ so that 
$f(\vecv,\langle\xi_j,\vs_j,\vecv_j\rangle_{j=1}^{n})=0$
unless $\xi_1,\ldots,\xi_n$ all lie in the interval $(C_1^{-1},C_1)$.
Set
\begin{align*}
T_1:=T+C_1+1
\end{align*}
and
\begin{align*}
C_2:=\sup_{\vecbeta\in F_3}\max\Bigl(10,\:\sup_{\vecv\in U}\|\vecbeta(\vecv)\|,
\sup_{\vech\in \T^1(U)} \|D_\vech\vecbeta\|\Bigr). %
\end{align*}

Let us write $f_0=f$;
we will now define functions $f_m\in\C_c(X_U^{(n-m)})$
recursively for $m=1,\ldots,n-1$.
Assuming that $f_{m-1}\in\C_c(X_U^{(n-m+1)})$ has been defined,
we define $f_m$ on $X_U^{(n-m)}$ by 
\begin{align}\notag
f_m\bigl(\vecv_0,\langle\xi_j,\vs_j,\vecv_j\rangle_{j=1}^{n-m}\bigr):=
\int_{\R_{>0}\times\Sigma\times\scrV_{\vecv_{n-m}}}f_{m-1}\bigl(\vecv_0,\langle\xi_j,\vs_j,\vecv_j\rangle_{j=1}^{n-m},
\:\xi,\vs,\vecv\bigr)
\hspace{40pt}
\\\label{MAINTECHNthm2aPF6}
\times p_{\bn}(\vecv_{n-m-1},\vs_{n-m},\vecv_{n-m};\xi,\vs,\vecv)\,
d\xi\,d\mm({\vs})\,d\vecv.
\end{align}
The fact that $f_m\in\C_c(X_U^{(n-m)})$ indeed holds follows 
by using Remark \ref{CONTINTEGRALlemG2REM}
together with Lemma \ref{BASICpktransfrem1lem} 
(cf.\ also Remark \ref{RINDEPANDCONTrem}).

Let $\ve>0$ be given.
We fix $0<\eta<\frac1{100}\bigl(\frac{\pi}2-{s_\Psi}(\frac{\pi}2-B_\Psi)\bigr)$ so small that
\begin{align}\label{lambdapartialetaU1small}
\partial_{2\eta}(U_1)\subset U
\qquad\text{and}\qquad
\lambda(\partial_{2\eta}(U_1))<\ve/\|f\|_{\infty}
\end{align}
(this is possible since %
$\omega(\partial U_1)=0$), %
and also so that there is some $\rho_0'\in(0,1)$ so that
\begin{align}\label{ETAGRZsmallass}
\lambda(\fg_{\vecq,\rho,\eta}^\vecbeta)<\ve/\|f\|_{\infty}
\qquad\forall \rho\in(0,\rho_0'),\:\vecq\in\scrP_T(\rho),\:\vecbeta\in F_3
\end{align}
(as is possible by Prop.\ \ref{ETAGRAZINGprop}).
Fix a family of pairwise disjoint open subsets $D_1,\ldots,D_N$ of $\US$\label{Djintro}
such that each $\overline{D_\ell}$ is a diffeomorphic image of a closed $(d-1)$-simplex in $\R^{d-1}$
and has diameter $<\eta/C_\eta$
(with respect to the metric $\varphi$),
and so that $\US=\cup_{\ell=1}^N \overline{D_\ell}$.
Recall the definition of $\nu_\vecs\in P(\US)$
for $\vecs\in\R^d\setminus\{\bn\}$,
cf.\ \eqref{LAMBDA0DEFnew}; %
it depends only on $\hs$.
Given $\vecs\in\R^d\setminus\{\bn\}$ and $\ell\in\{1,\ldots,N\}$ with
$\nu_\vecs(D_\ell)>0$,
we let $\nu_{\ell,\vecs}\in P(\US)$ be the normalized 
restriction of the measure $\nu_\vecs$ to $D_\ell$:
\begin{align}\label{LAMBDALSdef}
\nu_{\ell,\vecs}=\nu_\vecs(D_\ell)^{-1}\cdot\nu_{\vecs}\big|_{D_\ell}.
\end{align}
For each $\ell\in\{1,\ldots,N\}$ we let 
\begin{align}\label{Aelldef}
A_\ell=\{\vecs\in\US\col D_\ell\subset\scrV_\vecs^{5\eta}\},
\end{align}
and set
\begin{align}\label{F1elldef}
F_{1,\ell}:=\{\nu_{\ell,\vecs}\col\vecs\in A_\ell\};
\end{align}
this is an equismooth family of probability measures.
Also for any $\vecv_0\in U\cap A_\ell$, $\xi_0>0$, $\vs_0\in\Sigma$
we define the function 
$f_{[\vecv_0,\xi_0,\vs_0]}:X_{D_\ell}^{(n-1)}\to\R$ by
\begin{align}\label{TFdef}
f_{[\vecv_0,\xi_0,\vs_0]}(\vecv,\langle\xi_j,\vs_j,\vecv_j\rangle_{j=1}^{n-1})
=f(\vecv_0,\xi_0,\vs_0,\vecv,\langle\xi_j,\vs_j,\vecv_j\rangle_{j=1}^{n-1}),
\end{align}
and set
\begin{align}\label{F2elldef}
F_{2,\ell}:=\bigl\{f_{[\vecv_0,\xi_0,\vs_0]}\col\vecv_0\in U\cap A_\ell,\:\xi_0>0,\:\vs_0\in\Sigma\bigr\}.
\end{align}
This is a uniformly bounded and equicontinuous family of functions on $X_{D_\ell}^{(n-1)}$.
Define
\begin{align}\label{F3elldef}
F_{3,\ell}=\Bigl\{\vecbeta:D_\ell\to\US\col
\vecbeta\text{ is $\C^1$},
\:\sup_{\vech\in\T^1(D_\ell)}\|D_\vech\vecbeta\|\leq C_\eta,
\hspace{60pt}
\\\notag
\: (\vecbeta(\vecv)+\R_{>0}\vecv)\cap\scrB_1^d=\emptyset\:\forall\vecv\in D_\ell\Bigr\}.
\end{align}
Then $F_{3,\ell}$ is relatively compact as a subset of $\C_b(D_\ell,\R^d)$.

Let us also take $\eta'>0$ so small that 
for any $\vecv_0\in U$, $\xi'>0$, $\vs'\in\Sigma$, $\ell\in\{1,\ldots,N\}$,
$\vecv'\in D_\ell\cap\scrV_{\vecv_0}$
and any continuous functions $\vecbeta_1,\vecbeta_2:D_\ell\to\US$,
if $\|\vecbeta_1(\vecv')-\vecbeta_2(\vecv')\|<\eta'$ then
\begin{align}\label{NUCHOICE}
\biggl|\int_{\R_{>0}\times\Sigma\times\scrV_{\vecv'}}
f_{n-2}(\vecv_0,\xi',\vs',\vecv',\xi,\vs,\vecv)
\hspace{170pt}
\\\notag
\times\Bigl\{p_{\bn,\vecbeta_1}(\vs',\vecv';\xi,\vs,\vecv)-
p_{\bn,\vecbeta_2}(\vs',\vecv';\xi,\vs,\vecv)\Bigr\}\,
d\xi\,d\mm(\vs)\,d\vecv\biggr|<\ve.
\end{align}
This is possible by Lemma \ref{TRKERCONTlem}.

\subsection{The choice of $\rho_0$}

Now take $\rho_0\in(0,1)$ so small that 
\begin{align}\notag
\biggl|\int_{\fw_{\vecq,\rho,n-1}^{\vecbeta}}
\tf\bigl(\vecv,\langle\rho^{d-1}\tau_j(\vecq_{\rho,\vecbeta}(\vecv),\vecv;\rho),\vs_j(\vecq_{\rho,\vecbeta}(\vecv),\vecv;\rho),
\vecv_j(\vecq_{\rho,\vecbeta}(\vecv),\vecv;\rho)\rangle_{j=1}^{n-1}\bigr)
\,d\mu(\vecv)
\\\label{RHO0CONDmain}
-\int_{X_{D_\ell}^{(n-1)}}
\tf\bigl(\vecv_0,\big\langle\xi_j,{\vs}_j,\vecv_j\big\rangle_{j=1}^{n-1}\bigr)
p_{\bn,\vecbeta}\bigl(\vs(\vecq),\vecv_0;\xi_1,{\vs}_1,\vecv_1\bigr)
\hspace{75pt}
\\\notag
\times\prod_{j=2}^{n-1} p_\bn(\vecv_{j-2},{\vs}_{j-1},\vecv_{j-1};\xi_j,{\vs}_{j},\vecv_{j})
\, d\mu(\vecv_0)\,\prod_{j=1}^{n-1}\bigl( d\xi_j\,d\mm({\vs}_j)\,d\vecv_j\bigr)
\biggr|<\ve
\end{align}
for all 
$\rho\in(0,\rho_0)$,
$\vecq\in\scrP_{T_1}(\rho)$,
$\ell\in\{1,\ldots,N\}$,
$\mu\in F_{1,\ell}$,
$\tf\in F_{2,\ell}$
and $\vecbeta\in F_{3,\ell}$.
This is possible by our induction hypothesis, i.e.\ the assumption that %
the statement of Theorem \ref{MAINTECHNTHM2A} holds with $n-1$ in the place of $n$.
We shrink $\rho_0$ further if necessary, so as to also ensure that
\begin{align}\label{RHO0CONDmain2}
\biggl|\int_{\fw_{\vecq,\rho,1}^\vecbeta}f_{n-1}\bigl(\vecv,\rho^{d-1}\tau_1(\vecq_{\rho,\vecbeta}(\vecv),\vecv;\rho),
\vs_1(\vecq_{\rho,\vecbeta}(\vecv),\vecv;\rho),
\vecv_1(\vecq_{\rho,\vecbeta}(\vecv),\vecv;\rho)\bigr)\,d\lambda(\vecv)
\\\notag
-\int_{X_U^{(1)}}
f_{n-1}(\vecv,\xi_1,\vs_1,\vecv_1)p_{\bn,\vecbeta}(\vs(\vecq),\vecv;\xi_1,{\vs}_1,\vecv_1\bigr)
\, d\lambda(\vecv)\,d\xi_1\,d\mm({\vs}_1)\,d\vecv_1
\biggr|<\ve
\end{align}
for all $\rho\in(0,\rho_0)$, $\vecq\in\scrP_T(\rho)$, and $\vecbeta\in F_3$.
This is possible by Theorem \ref{unifmodThm2}.

We shrink $\rho_0$ yet further if necessary, so as to also ensure that
the following four conditions \eqref{RHO0COND3}--\eqref{RHO0COND2} are fulfilled:
\begin{align}\label{RHO0COND3}
  \rho_0<\min\biggl\{\trho_0\Bigl(\eta,C_2,\min\Bigl(\frac{\ve}{\|f\|_\infty},\eta'\Bigr)\Bigr),
  \:\rho_0'
  ,\:\Bigl(\frac{\eta}{8C_1C_2(1+C_2)}\Bigr)^{1/(d-1)}
  \biggr\}
\end{align}
(where $\trho_0(\cdots)$ is as in Lemma \ref{SHINELEM1} and $\rho_0'$ is the number in \eqref{ETAGRZsmallass});
\begin{align}\label{EHITSMALLPROB}
\lambda(\{\vecv\in\fw_{\vecq,\rho,1}^{\vecbeta}\col\vecq^{(1)}(\vecq_{\rho,\vecbeta}(\vecv),\vecv;\rho)\in\scrE\})<\frac{\ve}{\|f\|_\infty}
\quad\forall \rho\in(0,\rho_0),\,\vecq\in\scrP_T(\rho),\,\vecbeta\in F_3
\end{align}
(as is possible by Prop.\ \ref{Thm2genaddlem});
\begin{align}\label{RHO0COND1}
&\varphi(\vecv,\vecv')\leq 4C_1(1+C_2)\rho_0^d
\text{ and }
|\xi_1-\xi_1'|\leq(1+C_2)\rho_0^d
\\\notag
&\quad\Rightarrow\quad
\bigl|f(\vecv,\langle\xi_j,\vs_j,\vecv_j\rangle_{j=1}^{n})-
f(\vecv',\xi_1',\vs_1,\vecv_1,\langle\xi_j,\vs_j,\vecv_j\rangle_{j=2}^{n})\bigr|<\ve
\end{align}
(this can be obtained since $f$ is continuous and has compact support);
and
\begin{align}\label{RHO0COND2}
\varphi(\vecv,\vecv')\leq 8C_1(1+C_2)\rho_0^d
\quad\Rightarrow\quad
|g(\vecv)-g(\vecv')|<\frac{\ve}{\omega(\US)\|f\|_\infty}.
\end{align}

\subsection{Modification and partition of the domain of integration}

Let us now consider any choice of $\rho\in(0,\rho_0)$, $\vecq\in\scrP_T(\rho)$ and $\vecbeta\in F_3$.
In order to complete the proof of Theorem \ref{MAINTECHNTHM2A},
we wish to prove that the difference in \eqref{unifmodThm2ares} is $\ll\ve$.
Thus we need to study the integral
\begin{align}\label{MAINTECHNthm2aPF1}
\int_{\fw_{\vecq,\rho,n}^{\vecbeta}}f\Bigl(\vecv,\big\langle\rho^{d-1}\tau_j(\vecv),
{\vs}_j(\vecv),\vecv_j(\vecv)\big\rangle_{j=1}^{n}\Bigr)\,d\lambda(\vecv),
\end{align}
where
$\tau_j(\vecv):=\tau_j(\vecq_{\rho,\vecbeta}(\vecv),\vecv;\rho)$
and
$\vs_j(\vecv):={\vs}_j(\vecq_{\rho,\vecbeta}(\vecv),\vecv;\rho)$
and $\vecv_j(\vecv):=$\linebreak$\vecv_j(\vecq_{\rho,\vecbeta}(\vecv),\vecv;\rho)$.
Let us fix sets $\tD_1,\ldots,\tD_N$ so that $D_\ell\subset \tD_\ell\subset\overline{D_\ell}$
for all $\ell$ %
and $\tD_1,\ldots,\tD_N$ partition $\US$,
i.e.\ $\tD_i\cap\tD_j=\emptyset$ for all $i\neq j$ and $\US=\cup_{\ell=1}^N\tD_\ell$.
Given any $\veca\in\US$ we let $[\veca]$ be the unique set $\tD_j$ for which $\veca\in\tD_j$.
Let us also write $\vecq^{(j)}(\vecv):=\vecq^{(j)}(\vecq_{\rho,\vecbeta}(\vecv),\vecv;\rho)$
and
\begin{align*}
\vecs_1(\vecv):=\vecq^{(1)}(\vecv)-\vecq.
\end{align*}

{\em We now come to a crucial step of our treatment:} We will prove that with small error %
the domain of integration in \eqref{MAINTECHNthm2aPF1}, $\fw_{\vecq,\rho,n}^{\vecbeta}$, 
can be slightly modified
in such a way that the new domain can be perfectly partitioned into a large number of small pieces
which can each be dealt with using \eqref{RHO0CONDmain}. %
First of all, since $\lambda$ is concentrated on $U_1$,   %
we may trivially replace $\fw_{\vecq,\rho,n}^{\vecbeta}$ by $U_1\cap\fw_{\vecq,\rho,n}^{\vecbeta}$.
We then throw away any $\vecv\in U_1$ for which
$\rho^{d-1}\tau_1(\vecv)\notin(C_1^{-1},C_1)$ or $\vecq^{(1)}(\vecv)\in\scrE$,
and also any $\vecv\in U_1$ which does not satisfy that the whole set $[\vecv_1(\vecv)]$ is
``far from grazing position and is fully lit upon from $U_1$''.
In precise terms, we replace the domain of integration by $U_2\cap\fw_{\vecq,\rho,n}^{\vecbeta}$, where
\begin{align}\label{U2def}
U_2:=\Bigl\{\vecv\in U_1\cap\fw_{\vecq,\rho,1}^{\vecbeta}:\:
C_1^{-1}\hspace{-1pt}<\hspace{-1pt}\rho^{d-1}\tau_1(\vecv)\hspace{-1pt}<\hspace{-1pt}C_1,
\:
\vecq^{(1)}(\vecv)\in\scrP\setminus\scrE,
\:
[\vecv_1(\vecv)]\subset\scrV_{\vecs_1(\vecv)}^{5\eta},
\\\notag
\text{and }
[\forall\vecalf\in[\vecv_1(\vecv)]\col\exists\vecv'\in U_1\cap\fw_{\vecq,\rho,1}^{\vecbeta}\text{ s.t. }
\vecq^{(1)}(\vecv')=\vecq^{(1)}(\vecv)\text{ and }\vecv_1(\vecv')=\vecalf]\Bigr\}.
\end{align}
The following lemma will allow us to bound the error caused by this replacement.
\begin{lem}\label{GOODCOLLlem}
If $\vecv\in(U_1\cap\fw_{\vecq,\rho,1}^{\vecbeta})\setminus U_2$ then one of the following holds:
\begin{enumerate}[{\rm (i)}]
\item $\rho^{d-1}\tau_1(\vecv)\notin(C_1^{-1},C_1)$;
\item $\vecq^{(1)}(\vecv)\in\scrE$;
\item $\vecv\in\partial_{2\eta}(U_1)$;
\item $\vecv\in\fg_{\vecq,\rho,\eta}^\vecbeta$.
\end{enumerate}
\end{lem}

\begin{proof}
Assume $\vecv\in(U_1\cap\fw_{\vecq,\rho,1}^{\vecbeta})\setminus U_2$,
$\rho^{d-1}\tau_1(\vecv)\in(C_1^{-1},C_1)$
and $\vecq^{(1)}(\vecv)\in\scrP\setminus\scrE$,
i.e.\ neither (i) or (ii) hold.
Then our task is to prove that either (iii) or (iv) holds.
Now $\tau_1=\tau_1(\vecv)$, $\vecq^{(1)}=\vecq^{(1)}(\vecv)$,
$\vecv_1=\vecv_1(\vecv)$,
$\vecw_1(\vecv)=\vecw_1(\vecq_{\rho,\vecbeta}(\vecv),\vecv;\rho)$
are well-defined,
with $\tau_1(\vecv)<\infty$.
Take $\ell$ so that $[\vecv_1]=\tD_\ell$.
Let us first assume $\tD_\ell\not\subset\scrV_{\vecs_1}^{5\eta}$ (with $\vecs_1=\vecs_1(\vecv)$),
and take $\vecalf\in \tD_\ell$ with $\vecalf\notin\scrV_{\vecs_1}^{5\eta}$,
i.e.\
\begin{align*}
s_\Psi\cdot(B_\Psi-\varphi(\vecalf,\vecs_1))\leq5\eta.
\end{align*}
Then $\varphi(\vecalf,\vecv_1)<\eta/C_\eta<\eta$ since $\vecalf,\vecv_1\in \overline{D_\ell}$.
Furthermore, since the ray $\vecq+\rho\vecbeta(\vecv)+\R_{>0}\vecv$ hits 
$\scrB^d(\vecq^{(1)},\rho)$
and $\|\vecbeta(\vecv)\|\leq C_2$,
we have $\|\vecs_1\|\geq\tau_1-(1+C_2)\rho>(2C_1)^{-1}\rho^{1-d}$ (cf.\ \eqref{RHO0COND3}),
and 
\begin{align}\label{GOODCOLLlempf1}
\varphi(\vecs_1,\vecv)
<\arcsin\frac{(1+C_2)\rho}{\|\vecs_1\|}
\leq\arcsin\bigl(2C_1(1+C_2)\rho^d\bigr)<4 C_1(1+C_2)\rho^d<\eta  %
\end{align}
(again cf.\ \eqref{RHO0COND3} for the last inequality).
Hence 
\begin{align*}
s_\Psi\cdot(B_\Psi-\varphi(\vecv,\vecv_1))<7\eta.
\end{align*}
This implies $\vecw_1\in\fU_{\eta}$ (cf.\ \eqref{FUETA}),
and so (iv) holds.

It remains to treat the case when $\tD_\ell\subset\scrV_{\vecs_1}^{5\eta}$.
It then follows from $\vecv\notin U_2$ that there is some $\vecalf\in\tD_\ell$ such that there does not exist any
$\vecv'\in U_1\cap\fw_{\vecq,\rho,1}^{\vecbeta}$ satisfying 
$\vecq^{(1)}(\vecv')=\vecq^{(1)}(\vecv)$ and $\vecv_1(\vecv')=\vecalf$.
We noted $\|\vecs_1\|>(2C_1)^{-1}\rho^{1-d}$ above;
hence $\|\vecs_1\|\geq C_2^{-1}$
(cf.\ \eqref{RHO0COND3}).
If (iii) holds then we are done;
hence let us assume that (iii) does not hold,
i.e.\ $\scrD_\vecv^{2\eta}\subset U_1$.
Then $\scrD_{\vecs_1}^\eta\subset U_1$, because of \eqref{GOODCOLLlempf1},
and hence all the assumptions of Lemma \ref{SHINELEM1} are fulfilled,
with $C_2,\vecs_1,U_1$ in place of $C,\vecs,U$.
It follows that 
$\overline{\vecV}$,
the restriction of $\vecV=\vecV_{\!\!\rho,\vecs_1,\vecbeta}$
to $\vecV^{-1}(\scrV_{\vecs_1}^\eta)$,
is a $\C^1$ diffeomorphism onto $\scrV_{\vecs_1}^\eta$.
In particular, since $\vecalf\in\tD_\ell\subset\scrV_{\vecs_1}^\eta$,
there is a unique $\vecv'\in U_1$ satisfying $\vecV(\vecv')=\vecalf$.
For this $\vecv'$, set $\tau'=\tau_{\rho,\vecs_1,\vecbeta}(\vecv')$,
so that the ray $\{\vecq+\rho\vecbeta(\vecv')+t\vecv'\col t>0\}$
hits the ball 
$\scrB^d(\vecq^{(1)}(\vecv),\rho)$
for $t=\tau'$ (cf.\ \eqref{TAUrhovecsvecbetaDEF}).
Thus $\tau_1(\vecv')\leq\tau'$.
If $\tau_1(\vecv')=\tau'$ then it would follow that
$\vecv'\in \fw_{\vecq,\rho,1}^{\vecbeta}$,
$\vecq^{(1)}(\vecv')=\vecq^{(1)}(\vecv)$ and $\vecv_1(\vecv')=\vecV(\vecv')=\vecalf$,
contrary to our present assumptions.
Hence we must have $\tau(\vecv')<\tau'$,
and there is some $\vecq'\in\scrP\setminus\{\vecq,\vecq^{(1)}(\vecv)\}$
so that the line segment $L'$ between $\vecq+\rho\vecbeta(\vecv')$
and $\vecq+\rho\vecbeta(\vecv')+\tau'\vecv'$ intersects 
$\scrB^d(\vecq',\rho)$. %
Let us also denote by $L$ the line segment between 
$\vecq+\rho\vecbeta(\vecv)$ and $\vecq+\rho\vecbeta(\vecv)+\tau_1(\vecv)\vecv$,
and take points $\vecu_1',\vecu_1\in\US$ so that
\begin{align*}
  \vecq+\rho\vecbeta(\vecv')+\tau'\vecv'=\vecq^{(1)}(\vecv)+\rho\vecu_1'
  \qquad\text{and}\qquad
  \vecq+\rho\vecbeta(\vecv)+\tau_1(\vecv)\vecv=\vecq^{(1)}(\vecv)+\rho\vecu_1.
\end{align*}

We have $\varphi(\vecs_1,\vecv)<4C_1(1+C_2)\rho^d$ by \eqref{GOODCOLLlempf1},
and in the same way, $\varphi(\vecs_1,\vecv')<4C_1(1+C_2)\rho^d$.
Hence $\varphi(\vecv,\vecv')<8C_1(1+C_2)\rho^d<\eta$.
Hence using $\sup_{\vech\in \T^1(U)} \|D_\vech\vecbeta\|$\linebreak$\leq C_2$ 
and $\scrD_\vecv^{2\eta}\subset U_1\subset U$,
we get $\|\vecbeta(\vecv)-\vecbeta(\vecv')\|<8C_1C_2(1+C_2)\rho^d$.
Furthermore we have 
$\varphi(\vecalf,\vecv_1)<\eta/C_\eta$ since $\vecalf,\vecv_1\in \overline{D_\ell}$.
Hence by Lemma \ref{SHINELEM1}(iii),
noticing that $\vecu_1'=\vecB^-(\vecalf)$ and $\vecu_1=\vecB^-(\vecv_1)$ with $\vecB^-=\vecB^-_{\rho,\vecs_1,\vecbeta}$,
it follows that $\varphi(\vecu_1,\vecu_1')<C_\eta\varphi(\vecalf,\vecv_1)<\eta$.
Hence the end-points of $L$ and $L'$ satisfy
\begin{align*}
  \|(\vecq^{(1)}(\vecv)+\rho\vecu_1')-(\vecq^{(1)}(\vecv)+\rho\vecu_1)\|<\eta\rho
\end{align*}
and
\begin{align*}
\|(\vecq+\vecbeta(\vecv))-(\vecq+\vecbeta(\vecv'))\|<8C_1C_2(1+C_2)\rho^d<\eta\rho
\end{align*}
(cf.\ \eqref{RHO0COND3}).
It follows that each point on $L'$ has distance $<\eta\rho$ to $L$.
Hence $\vecq'$ has distance $<(1+\eta)\rho$ from $L$,
and hence $\vecv\in\fg_{\vecq,\rho,\eta}^\vecbeta$,
i.e.\ (iv) holds.
\end{proof}
By Lemma \ref{GOODCOLLlem},
and since %
$f(\vecv,\langle\xi_j,\vs_j,\vecv_j\rangle_{j=1}^{n})=0$ whenever $\xi_1\notin(C_1^{-1},C_1)$,
the error caused by replacing the domain of integration in 
\eqref{MAINTECHNthm2aPF1} by $U_2\cap\fw_{\vecq,\rho,n}^{\vecbeta}$ is
\begin{align}\label{GOODCOLLlemAPPL}
\leq\bigl(\lambda(\{\vecv\in\fw_{\vecq,\rho,1}^{\vecbeta}\col\vecq^{(1)}(\vecv)\in\scrE\})+
\lambda(\partial_{2\eta}(U_1))+
\lambda(\fg_{\vecq,\rho,\eta}^\vecbeta)\bigr)\cdot\|f\|_{\infty}
<3\ve.
\end{align}
Cf.\ \eqref{lambdapartialetaU1small}, \eqref{ETAGRZsmallass} and \eqref{EHITSMALLPROB} for the last inequality.

Now our task is to understand
\begin{align}\label{MAINTECHNthm2aPF3}
\int_{U_2\cap\fw_{\vecq,\rho,n}^{\vecbeta}}f\Bigl(\vecv,\big\langle\rho^{d-1}\tau_j(\vecv),
{\vs}_j(\vecv),\vecv_j(\vecv)\big\rangle_{j=1}^{n}\Bigr)\,d\lambda(\vecv).
\end{align}
With the new domain of integration,
the integral can be decomposed as a sum over those $\vecq'$ which can appear as $\vecq^{(1)}$.
By the definition of $U_2$, each such point $\vecq'$ satisfies $\vecq'\in\scrP\setminus\scrE$ and
\begin{align*}
\|\vecq'\|\leq\|\vecq\|+(C_2+1)\rho+C_1\rho^{1-d}<(T+C_1+1)\rho^{1-d}=T_1\rho^{1-d}
\end{align*}
(we used \eqref{RHO0COND3} in the second inequality);
thus $\vecq'\in\scrP_{T_1}(\rho)$.
Given any $\vecq'\in\scrP_{T_1}(\rho)$
we write $\vecs_1:=\vecq'-\vecq$, %
and let $M(\vecq')$ be the corresponding set of
$\ell\in\{1,\ldots,N\}$ such that $\tD_\ell$ is far from grazing position and is fully lit upon from $U_1$,
i.e.
\begin{align*}\notag
M(\vecq')=\bigl\{\ell\col
\tD_\ell\subset\scrV_{\vecs_1}^{5\eta}
\:\:\text{ and }\:\:
\hspace{220pt}
\\\notag %
[\forall\vecalf\in\tD_\ell\col\exists\vecv'\in U_1\cap\fw_{\vecq,\rho,1}^{\vecbeta}\text{ s.t. }
\vecq^{(1)}(\vecv')=\vecq'\text{ and }\vecv_1(\vecv')=\vecalf]\bigr\}.
\end{align*}
Then \eqref{MAINTECHNthm2aPF3} can be expressed as
\begin{align}\label{MAINTECHNthm2aPF3a}
\sum_{\vecq'\in\scrP_{T_1}(\rho)}\:\sum_{\ell\in M(\vecq')}\:\int_{U_{\vecq',\ell}\cap\fw_{\vecq,\rho,n}^{\vecbeta}}
f\bigl(\vecv,\big\langle\rho^{d-1}\tau_j(\vecv),{\vs}_j(\vecv),
\vecv_j(\vecv)\big\rangle_{j=1}^{n}\bigr)\,d\lambda(\vecv),
\end{align}
where
\begin{align*}
U_{\vecq',\ell}:=\{\vecv\in U_1\cap\fw_{\vecq,\rho,1}^{\vecbeta}\col\vecq^{(1)}(\vecv)=\vecq',\:\vecv_1(\vecv)\in \tD_\ell\}.
\end{align*}
Indeed, for every $\vecv\in U_2\cap\fw_{\vecq,\rho,n}^{\vecbeta}$ there is exactly one choice of
$\vecq'\in\scrP_{T_1}(\rho)$ and $\ell\in M(\vecq')$ such that
$\vecv\in U_{\vecq',\ell}\cap\fw_{\vecq,\rho,n}^{\vecbeta}$,
and conversely for any $\vecq'\in\scrP_{T_1}(\rho)$, $\ell\in M(\vecq')$ and
$\vecv\in U_{\vecq',\ell}\cap\fw_{\vecq,\rho,n}^{\vecbeta}$,
we have $\vecv\in U_2$ or $\rho^{d-1}\tau_1(\vecv)\notin(C_1^{-1},C_1)$,
and in the latter case the integrand vanishes. %

\subsection{Step by step modification and approximation of the decomposed integral}

For any $\vecq',\ell$ and $\vecv\in U_{\vecq',\ell}$
as in \eqref{MAINTECHNthm2aPF3a}, we have
$|\tau_1(\vecv)-\|\vecs_1\||\leq(1+C_2)\rho$,
$\varphi(\hs_1,\vecv)<4 C_1(1+C_2)\rho^d$ (cf.\ \eqref{GOODCOLLlempf1}),
and $\vs_1(\vecv)=\vs(\vecq')$.
Hence by \eqref{RHO0COND1}, %
up to an error of absolute size $<\ve$,
\eqref{MAINTECHNthm2aPF3a} equals
\begin{align}\label{MAINTECHNthm2aPF8}
\sum_{\vecq'\in\scrP_{T_1}(\rho)}\sum_{\ell\in M(\vecq')}\int_{U_{\vecq',\ell}\cap\fw_{\vecq,\rho,n}^{\vecbeta}}
f\Bigl(\hs_1,\rho^{d-1}\|\vecs_1\|,\vs(\vecq'),\vecv_1(\vecv),
\hspace{70pt}
\\\notag
\big\langle\rho^{d-1}\tau_j(\vecv),{\vs}_j(\vecv),
\vecv_j(\vecv)\big\rangle_{j=2}^n\Bigr)\,d\lambda(\vecv).
\end{align}
(Note that $\hs_1\in U$,
since $\vecv\in U_1$ and $U_1\cup\partial_{2\eta}(U_1)\subset U$.)
Next recall that
$d\lambda(\vecv)=g(\vecv)\,d\omega(\vecv)$ for $\vecv\in U_1$. %
By \eqref{RHO0COND2}, if in each region $U_{\vecq',\ell}$
we replace the function $g$ by any constant equal to a value taken by $g$ in $U_{\vecq',\ell}$,
this causes a total error of absolute size $<\ve$;
and by the intermediate value theorem, an admissible such constant is %
$\lambda(U_{\vecq',\ell})/\omega(U_{\vecq',\ell})$.
Hence we conclude that \eqref{MAINTECHNthm2aPF8} differs by less than $\ve$ from
\begin{align}\notag
\sum_{\vecq'\in\scrP_{T_1}(\rho)}\sum_{\ell\in M(\vecq')}
\lambda(U_{\vecq',\ell})
\int_{U_{\vecq',\ell}\cap\fw_{\vecq,\rho,n}^{\vecbeta}}
f\Bigl(\hs_1,\rho^{d-1}\|\vecs_1\|,\vs(\vecq'),\vecv_1(\vecv),
\hspace{70pt}
\\\label{MAINTECHNthm2aPF2}
\big\langle\rho^{d-1}\tau_j(\vecv),{\vs}_j(\vecv),
\vecv_j(\vecv)\big\rangle_{j=2}^n\Bigr)
\,\frac{d\omega(\vecv)}{\omega(U_{\vecq',\ell})}.
\end{align}

As in the proof of Lemma \ref{GOODCOLLlem},
for any fixed $\vecq'\in\scrP_{T_1}(\rho)$ and $\ell\in M(\vecq')$,
we have a $\C^1$ diffeomorphism $\overline{\vecV}$ of from $\vecV^{-1}(\scrV_{\vecs_1}^\eta)$ onto $\scrV_{\vecs_1}^\eta$;
also $U_{\vecq',\ell}\subset \vecV^{-1}(\scrV_{\vecs_1}^\eta)$,
$\overline{\vecV}(U_{\vecq',\ell})=\tD_\ell$,
and $\vecv_1(\vecv)=\overline{\vecV}(\vecv)$ for all $\vecv\in U_{\vecq',\ell}$.
We take $\vecv_1=\vecv_1(\vecv)$ as a new variable of integration in \eqref{MAINTECHNthm2aPF2}.
By Lemma \ref{SHINELEM1}(ii),
using also \eqref{RHO0COND3} and our notation from \eqref{LAMBDALSdef},
$\omega(U_{\vecq',\ell})^{-1}\,d\omega(\vecv)$ in \eqref{MAINTECHNthm2aPF2}
transforms into $h(\vecv_1)\,d\nu_{\ell,\vecs_1}(\vecv_1)$,
where $h=h_{\vecq',\ell}$ is a continuous function on $\tD_\ell$ satisfying
$\bigl|h(\vecv_1)-1\bigr|\leq\ve/\|f\|_\infty$ for all $\vecv_1\in\tD_\ell$.
It follows that replacing $h(\vecv_1)\,d\nu_{\ell,\vecs_1}(\vecv_1)$ by $d\nu_{\ell,\vecs_1}(\vecv_1)$
causes a total error $\leq\ve$ in our expression. %
Also, the point where the particle leaves the $\vecq'$-scatterer is
$\vecq'+\rho \vecB^+(\vecv_1)$ 
with $\vecB^+=\vecB^+_{\rho,\vecs_1,\vecbeta}$ as in Lemma \ref{SHINELEM1}(iii),
and we note that for all $\vecv\in U_{\vecq',\ell}$
the condition $\vecv\in\fw_{\vecq,\rho,n}^{\vecbeta}$
is equivalent with $\vecv_1(\vecv)\in\fw_{\vecq',\rho,n-1}^{\vecB^+}$.
Hence, up to an error of absolute size $<\ve$, \eqref{MAINTECHNthm2aPF2} equals
\begin{align}\notag
\sum_{\vecq'\in\scrP_{T_1}(\rho)}\:\sum_{\ell\in M(\vecq')}
\lambda(U_{\vecq',\ell})
\int_{\tD_\ell\cap\fw_{\vecq',\rho,n-1}^{\vecB^+}}
f\Bigl(\hs_1,\rho^{d-1}\|\vecs_1\|,\vs(\vecq'),\vecv,
\hspace{85pt}
\\\label{MAINTECHNthm2aPF5}
\big\langle\rho^{d-1}\ttau_j(\vecv),{\tvs}_j(\vecv),
\tv_j(\vecv)\big\rangle_{j=1}^{n-1}\Bigr)\, 
d\nu_{\ell,\vecs_1}(\vecv), %
\end{align}
where
$\ttau_j(\vecv)=\tau_j(\vecq'+\rho \vecB^+(\vecv),\vecv;\rho)$,
$\tvs_j(\vecv)=\vs_j(\vecq'+\rho \vecB^+(\vecv),\vecv;\rho)$,
$\tv_j(\vecv)=\vecv_j(\vecq'+\rho \vecB^+(\vecv),\vecv;\rho)$.

Clearly we may replace $\tD_\ell$ by $D_\ell$ in \eqref{MAINTECHNthm2aPF5},
since $\partial D_\ell$ has measure zero.
Let us temporarily fix $\vecq'\in\scrP_{T_1}(\rho)$ and $\ell\in M(\vecq')$,
and set $\tbe:=\vecB^+_{|D_\ell}$;
then $D_\ell\cap\fw_{\vecq',\rho,n-1}^{\vecB^+}=\fw_{\vecq',\rho,n-1}^{\tbe}$,
and using also \eqref{TFdef} the integral appearing in \eqref{MAINTECHNthm2aPF5} can be rewritten as
\begin{align}\label{MAINTECHNthm2aPF5intrewr}
\int_{\fw_{\vecq',\rho,n-1}^{\tbe}} f_{[\hs_1,\rho^{d-1}\|\vecs_1\|,\vs(\vecq')]}\Bigl(\vecv,
\big\langle\rho^{d-1}\ttau_j(\vecv),{\tvs}_j(\vecv),
\tv_j(\vecv)\big\rangle_{j=1}^{n-1}\Bigr)\,d\nu_{\ell,\vecs_1}(\vecv).
\end{align}
Note that for all $\vecq'$ and $\ell$ appearing in \eqref{MAINTECHNthm2aPF5},
we have $\nu_{\ell,\vecs_1}\in F_{1,\ell}$,
$f_{[\hs_1,\rho^{d-1}\|\vecs_1\|,\vs(\vecq')]}\in F_{2,\ell}$,
and also  $\tbe:=\vecB^+_{|D_\ell}\in F_{3,\ell}$, by Lemma \ref{SHINELEM1}(iii).
Hence by \eqref{RHO0CONDmain},
the integral in \eqref{MAINTECHNthm2aPF5intrewr} differs by less than $\ve$ from
\begin{align}\notag
\int_{X_{D_\ell}^{(n-1)}}
f\Bigl(\hs_1,\rho^{d-1}\|\vecs_1\|,\vs(\vecq'),\vecv_0,\big\langle\xi_j,{\vs}_j,\vecv_j\big\rangle_{j=1}^{n-1}\Bigr)
p_{\bn,\tbe}\bigl(\vs(\vecq'),\vecv_0;\xi_1,{\vs}_1,\vecv_1\bigr)
\hspace{40pt}
\\\label{MAINTECHNthm2aPF7}
\times\prod_{j=2}^{n-1} p_\bn(\vecv_{j-2},{\vs}_{j-1},\vecv_{j-1};\xi_j,{\vs}_{j},\vecv_{j})
\, d\nu_{\ell,\vecs_1}(\vecv_0)\,\prod_{j=1}^{n-1}\bigl( d\xi_j\,d\mm({\vs}_j)\,d\vecv_j\bigr)
\\\notag
=\int_{X_{D_\ell}^{(1)}}f_{n-2}(\hs_1,\rho^{d-1}\|\vecs_1\|,\vs(\vecq'),\vecv_1,\xi_2,\vs_2,\vecv_2)
\, p_{\bn,\tbe}\bigl(\vs(\vecq'),\vecv_1;\xi_2,{\vs}_2,\vecv_2\bigr)
\hspace{40pt}
\\\notag
\times d\nu_{\ell,\vecs_1}(\vecv_1)\,d\xi_2\,d\mm({\vs}_2)\,d\vecv_2,
\end{align}
where in the last equality we used \eqref{MAINTECHNthm2aPF6} for $m=1,2,\ldots,n-2$, %
and then renamed the variables $\vecv_0,\xi_1,\vs_1,\vecv_1$ as $\vecv_1,\xi_2,\vs_2,\vecv_2$.
Here $\tbe=\vecB^+_{\rho,\vecs_1,\vecbeta|D_\ell}$,
and by Lemma \ref{SHINELEM1}(iii),
using $\rho<\rho_0$ and \eqref{RHO0COND3},
we have $\|\tbe(\vecv_1)-\vecbeta_{\hs_1}^+(\vecv_1)\|<\eta'$ for all $\vecv_1\in D_\ell$,
and hence by \eqref{NUCHOICE} and \eqref{BASICp0idG},
the expression in \eqref{MAINTECHNthm2aPF7} differs by less than $\ve$ from
\begin{align}\notag
\int_{X_{D_\ell}^{(1)}}f_{n-2}(\hs_1,\rho^{d-1}\|\vecs_1\|,\vs(\vecq'),\vecv_1,\xi_2,\vs_2,\vecv_2)\,
p_{\bn}\bigl(\hs_1,\vs(\vecq'),\vecv_1;\xi_2,{\vs}_2,\vecv_2\bigr)
\hspace{40pt}
\\\label{MAINTECHNthm2aPF9}
\times d\nu_{\ell,\vecs_1}(\vecv_1)\,d\xi_2\,d\mm({\vs}_2)\,d\vecv_2
\\\notag
=\int_{D_\ell}f_{n-1}(\hs_1,\rho^{d-1}\|\vecs_1\|,\vs(\vecq'),\vecv_1)\,d\nu_{\ell,\vecs_1}(\vecv_1),
\hspace{110pt}
\end{align}
where the equality holds by \eqref{MAINTECHNthm2aPF6}.
In conclusion, for any $\vecq'\in\scrP_{T_1}(\rho)$ and $\ell\in M(\vecq')$,
the integral appearing in \eqref{MAINTECHNthm2aPF5} differs by less than $2\ve$ from
the integral in \eqref{MAINTECHNthm2aPF9}.
Adding over $\vecq'$ and $\ell$, it follows that the whole expression in \eqref{MAINTECHNthm2aPF5} differs by less than $2\ve$ from
\begin{align}\label{MAINTECHNthm2aPF4}
&\sum_{\vecq'\in\scrP_{T_1}(\rho)}\sum_{\ell\in M(\vecq')}\lambda(U_{\vecq',\ell})
\int_{D_\ell} f_{n-1}(\hs_1,\rho^{d-1}\|\vecs_1\|,\vs(\vecq'),\vecv_1)\,d\nu_{\ell,\vecs_1}(\vecv_1).
\end{align}

\subsection{Conclusion of the proof}

It follows from the recursion formula
\eqref{MAINTECHNthm2aPF6}
together with \eqref{REMOVINGCPTSUPPfactpfLEM1res1} in Lemma~ \ref{REMOVINGCPTSUPPfactpfLEM1} %
that $\|f_{n-1}\|_\infty\leq\|f_{n-2}\|_\infty\leq\cdots\leq \|f\|_\infty$.
Similarly the continuity property \eqref{RHO0COND1} immediately carries over to $f_{n-1}$, %
i.e.\ we have
\begin{align*}%
&\varphi(\vecv,\vecv')\leq 4C_1(1+C_2)\rho_0^d
\text{ and }
|\xi_1-\xi_1'|\leq(1+C_2)\rho_0^d
\\\notag
&\hspace{50pt}\quad\Rightarrow\quad
\bigl|f_{n-1}(\vecv,\xi_1,\vs_1,\vecv_1)-
f_{n-1}(\vecv',\xi_1',\vs_1,\vecv_1)\bigr|<\ve.
\end{align*}
Hence by %
repeating the argument between
\eqref{MAINTECHNthm2aPF3a} and \eqref{MAINTECHNthm2aPF5} (``backwards''),
with the function $f_{n-1}$ in place of $f$
and with a slight simplification due to the fact that this time we are not intersecting by $\fw_{\vecq,\rho,n}^\vecbeta$
in the domain of integration,
it follows that \eqref{MAINTECHNthm2aPF4} differs by less than $3\ve$ from
\begin{align}\label{MAINTECHNthm2aPF10}
&\sum_{\vecq'\in\scrP_{T_1}(\rho)}\sum_{\ell\in M(\vecq')}
\int_{U_{\vecq',\ell}}  %
f_{n-1}(\vecv,\rho^{d-1}\tau_1(\vecv),\vs(\vecq'),\vecv_1(\vecv))\,d\lambda(\vecv).
\end{align}
Also by the argument between 
\eqref{MAINTECHNthm2aPF3} and \eqref{MAINTECHNthm2aPF3a},
this double sum equals %
\begin{align}\label{MAINTECHNthm2aPF11}
\int_{U_2}f_{n-1}(\vecv,\rho^{d-1}\tau_1(\vecv),\vs_1(\vecv),\vecv_1(\vecv))\,d\lambda(\vecv).
\end{align}
Now by Lemma \ref{GOODCOLLlem} and the bound in \eqref{GOODCOLLlemAPPL}
(and the fact that $\lambda$ is concentrated on $U_1$),
replacing the domain of integration in \eqref{MAINTECHNthm2aPF11} by $\fw_{\vecq,\rho,1}^{\vecbeta}$ 
causes a total error less than $3\ve$.
Hence, using also \eqref{RHO0CONDmain2},
it follows that \eqref{MAINTECHNthm2aPF11} differs by less than $4\ve$ from 
\begin{align}\notag
\int_{X_U^{(1)}}
f_{n-1}(\vecv,\xi_1,\vs_1,\vecv_1)\, p_{\bn,\vecbeta}(\vs(\vecq),\vecv;\xi_1,{\vs}_1,\vecv_1\bigr)
\, d\lambda(\vecv)\,d\xi_1\,d\mm({\vs}_1)\,d\vecv_1
\hspace{60pt}
\\\label{MAINTECHNthm2aPF12}
=\int_{{X_U^{(n)}}} %
f\bigl(\vecv_0,\big\langle\xi_j,{\vs}_j,\vecv_j\big\rangle_{j=1}^{n}\bigr)
\, p_{\bn,\vecbeta}\bigl(\vs(\vecq),\vecv_0;\xi_1,{\vs}_1,\vecv_1\bigr)
\hspace{85pt}
\\\notag
\times\prod_{j=2}^n p_\bn(\vecv_{j-2},{\vs}_{j-1},\vecv_{j-1};\xi_j,{\vs}_{j},\vecv_{j})
\, d\lambda(\vecv_0)\,\prod_{j=1}^n\bigl( d\xi_j\,d\mm({\vs}_j)\,d\vecv_j\bigr),
\end{align}
where the last equality holds by repeated use of \eqref{MAINTECHNthm2aPF6}.

Summing up, we have proved that for any
$\rho\in(0,\rho_0)$, $\vecq\in\scrP_T(\rho)$ and $\vecbeta\in F_3$,
the two integrals \eqref{MAINTECHNthm2aPF1} and \eqref{MAINTECHNthm2aPF12} 
differ by less than $15\ve$.
This completes the proof of Theorem \ref{MAINTECHNTHM2A}.
\hfill$\square$ %

\section{Macroscopic initial conditions}
\label{MAINRESMACRO}

Generalizing the notation $\fW(1;\rho)$ from Section \ref{FIRSTCOLLMACRsec},
let us write $\fW(n;\rho)$ for the set $\fw(n;\rho)$ in macroscopic coordinates,
i.e.\
\begin{align}\label{fWndef}
\fW(n;\rho)=\{(\vecq,\vecv)\in\T^1(\R^d)\col\langle\rho^{1-d}\vecq,\vecv\rangle\in\fw(n;\rho)\}.
\end{align}
The following space is the macroscopic analogue of $X_U^{(n)}$, cf.\ \eqref{XSPACEdef}:
\begin{align}\notag
X^{(n)}:=\Bigl\{\langle\vecq,\vecv_0,\langle\xi_j,{\vs}_j,\vecv_j\big\rangle_{j=1}^{n}\rangle\in
\T^1(\R^d)\times(\R_{>0}\times\Sigma\times\US)^n
\col
\hspace{70pt}
\\\label{Xmacrdef}
\vecv_j\in\scrV_{\vecv_{j-1}}\: %
(j=1,\ldots,n)\Bigr\}.
\end{align}
In particular note that $X^{(1)}=X$, the extended phase space defined in \eqref{XXdef2}.

\begin{thm}\label{MAINTECHNthm2G}
Let $\scrP$ {\blu and $\scrE$} satisfy all the conditions in Section \ref{ASSUMPTLISTsec} and \eqref{SIGMAeqp},
and let $\Psi$ be a scattering process satisfying the conditions in Section \ref{SCATTERINGMAPS}.
Then for any $n\geq1$, $\Lambda\in\Pac(\T^1(\R^d))$ and $f\in\C_b(X^{(n)})$, we have
\begin{align}\notag
\lim_{\rho\to0}\int_{\fW(n;\rho)}f\Bigl(\vecq,\vecv,
\Big\langle\rho^{d-1}\tau_j(\rho^{1-d}\vecq,\vecv;\rho),{\vs}_j(\rho^{1-d}\vecq,\vecv;\rho),
\vecv_j(\rho^{1-d}\vecq,\vecv;\rho)
\Big\rangle_{j=1}^n\Bigr)
\hspace{10pt}
\\\notag
\times d\Lambda(\vecq,\vecv)
\\\label{MAINTECHNthm2Gres}
=\int_{X^{(n)}} %
f\Bigl(\vecq,\vecv_0,\big\langle \xi_j,{\vs}_j,\vecv_j\big\rangle_{j=1}^n\Bigr)
\,p\bigl(\vecv_0;\xi_1,{\vs}_1,\vecv_1\bigr)
\hspace{100pt}
\\\notag
\times \prod_{j=2}^n p_\bn(\vecv_{j-2},{\vs}_{j-1},\vecv_{j-1};\xi_j,{\vs}_{j},\vecv_{j})
\, d\Lambda(\vecq,\vecv_0)\prod_{j=1}^n\bigl( d\xi_j\,d\mm({\vs}_j)\,d\vecv_j\bigr).
\end{align}
\end{thm}
\begin{remark}\label{MAINTECHNthm2Gprobrem}
Regarding the limit expression in \eqref{MAINTECHNthm2Gres},
one should note that
\begin{align*}%
\int_{X^{(n)}} %
p\bigl(\vecv_0;\xi_1,{\vs}_1,\vecv_1\bigr)
\prod_{j=2}^n p_\bn(\vecv_{j-2},{\vs}_{j-1},\vecv_{j-1};\xi_j,{\vs}_{j},\vecv_{j})
\,d\Lambda(\vecq,\vecv_0)
\hspace{60pt}
\\
\times\prod_{j=1}^n\bigl( d\xi_j\,d\mm({\vs}_j)\,d\vecv_j\bigr)=1.
\end{align*}
This follows by iterated use of \eqref{REMOVINGCPTSUPPfactpfLEM1res1} in Lemma \ref{REMOVINGCPTSUPPfactpfLEM1},
and Lemma \ref{BASICpkproblem}.
In particular, taking $f\equiv1$ in \eqref{MAINTECHNthm2Gres},
the theorem implies that
$\Lambda(\fW(n;\rho))\to1$ as $\rho\to0$.
\end{remark}

\begin{remark}\label{MAINTECHNthm2Grem}
In the special case of $\Psi$ being specular reflection,
Theorem \ref{MAINTECHNthm2G} is equivalent with Theorem \ref{thm:M2}.
Here one should note that the definition of the sequences
$\{\tau_j\}$, $\{\vs_j\}$, $\{\vecv_j\}$ which was given in Sec.\ \ref{SCATTERINGMAPS}
and which is used in Theorem \ref{MAINTECHNthm2G},
differs slightly from the definition used in Section \ref{sec:classical} and 
Theorem \ref{thm:M2}.
However, as far as the values of $\langle(\tau_j,\vs_j,\vecv_j)\rangle_{j=1}^n$ are concerned,
this difference occurs only for initial conditions
$(\vecq,\vecv)$ which lie outside $\fW(n;\rho)$,
and since $\Lambda(\fW(n;\rho))\to1$ as $\rho\to0$
(cf.\ Remark \ref{MAINTECHNthm2Gprobrem}),
it follows that Theorem \ref{thm:M2} and Theorem \ref{MAINTECHNthm2G} are indeed equivalent.
\end{remark}
\noindent
\textit{Proof.}
We derive Theorem \ref{MAINTECHNthm2G} as a consequence of Theorem \ref{MAINTECHNTHM2A}
together with Theorem~\ref{unifmodThm2macr}
and Propositions \ref{Thm2genaddgenlem} and \ref{ETAGRAZINGgenprop}.
As we will see, after expressing the integral in the left hand side of \eqref{MAINTECHNthm2Gres}
as an iterated integral over $\vecq$ and $\vecv$,
the inner integral (that is, the integral over $\vecv$)
can be treated by more or less exactly the same %
arguments as in the proof of Theorem \ref{MAINTECHNTHM2A}.

Some initial reductions:
The right hand side of \eqref{MAINTECHNthm2Gres} can be expressed as
$\int_{X^{(n)}} f\,d\mu_\Lambda$,
where $\mu_\Lambda$ is a Borel probability measure on $X^{(n)}$; %
cf.\ Remark \ref{MAINTECHNthm2Gprobrem} regarding the fact that $\mu_\Lambda(X^{(n)})=1$.
Hence by a standard approximation argument,
it suffices to prove \eqref{MAINTECHNthm2Gres} under the extra assumption that $f$ has compact support in $X^{(n)}$.
Next let $g\in\L^1(\T^1(\R^d))$ be the density of $\Lambda$ with respect to $d\vecq\,d\vecv$.
Again by Lemma \ref{REMOVINGCPTSUPPfactpfLEM1} and Lemma \ref{BASICpkproblem} we have
\begin{align*}
\int_{(\R_{>0}\times\Sigma\times\US)^n}
\Bigl|f\bigl(\cdots\bigr)\Bigr|
\,p\bigl(\vecv_0;\xi_1,{\vs}_1,\vecv_1\bigr)
\prod_{j=2}^n p_\bn(\vecv_{j-2},{\vs}_{j-1},\vecv_{j-1};\xi_j,{\vs}_{j},\vecv_{j})
\hspace{40pt}
\\
\,\times \prod_{j=1}^n\bigl( d\xi_j\,d\mm({\vs}_j)\,d\vecv_j\bigr)
\leq\|f\|_\infty
\end{align*}
for all $(\vecq,\vecv_0)\in\T^1(\R^d)$
(where the integrand should be interpreted to vanish when \linebreak
$\langle\vecq,\vecv_0;\langle\xi_j,{\vs}_j,\vecv_j\big\rangle_{j=1}^{n}\rangle\notin X^{(n)}$).
Hence, since $\C_c(\T^1(\R^d))$ is dense in $\L^1(\T^1(\R^d))$,
we may without loss of generality assume $g\in\C_c(\T^1(\R^d))$.

Since $f$ has compact support, we can choose $C_1>1$ and $T>0$ so that 
$f\bigl(\vecq,\vecv_0;\big\langle \xi_j,{\vs}_j,\vecv_j\big\rangle_{j=1}^n\bigr)$
vanishes unless
$\|\vecq\|<T$ and $\xi_1,\ldots,\xi_n$ all lie in the interval $(C_1^{-1},C_1)$.
Set
\begin{align*}
T_1:=T+C_1+1.
\end{align*}

We write $f_0=f$, and define functions
$f_m\in \C_c(X^{(n-m)})$ 
recursively for $m=1,\ldots,n-1$ exactly as in
\eqref{MAINTECHNthm2aPF6}
but with the extra parameter $\vecq$ in each $f_m$.

Let $\ve>0$ be given.
We fix $0<\eta<\frac1{100}\bigl(\frac{\pi}2-{s_\Psi}(\frac{\pi}2-B_\Psi)\bigr)$ and $\rho_0'\in(0,1)$ so that
\begin{align}\label{MAINTECHNthm2Gpf2}
\Lambda(\fG_{\rho,\eta})<\ve/\|f\|_{\infty}
\qquad\forall \rho\in(0,\rho_0');
\end{align}
this is possible by Prop.\ \ref{ETAGRAZINGgenprop}.
Given %
$\eta$, we let
$D_\ell,$ $\tD_\ell,$ $A_\ell$, $F_{1,\ell}$ and $F_{3,\ell}$ 
for $\ell=1,\ldots,N$
be exactly as in the proof of Theorem \ref{MAINTECHNTHM2A}.
Also for $\ell=1,\ldots,N$ 
we set
\begin{align*}
F_{2,\ell}:=\bigl\{f_{[\vecq,\vecv_0,\xi_0,\vs_0]}\col(\vecq,\vecv_0)\in\T^1(\R^d),\:
\vecv_0\in A_\ell,\:\xi_0>0,\:\vs_0\in\Sigma\bigr\},
\end{align*}
where $f_{[\vecq,\vecv_0,\xi_0,\vs_0]}\in\C_c(X_{D_\ell}^{(n-1)})$
is defined exactly as in \eqref{TFdef}
but with the extra parameter $\vecq$ in $f$.
Then $F_{2,\ell}$ is a uniformly bounded and equicontinuous family of functions on $X_{D_\ell}^{(n-1)}$.

Let us also take $\eta'>0$ sufficiently small so that the
condition formulated around \eqref{NUCHOICE} holds,
but with the extra parameter $\vecq$ in $f_{n-2}$
and with $(\vecq,\vecv')$ arbitrary in $\T^1(\R^d)$.
This is possible by an obvious modification of Lemma \ref{TRKERCONTlem}. %

Next take $\rho_0\in(0,1)$ so small that 
the inequality in \eqref{RHO0CONDmain} %
holds
for all 
$\rho\in(0,\rho_0)$,
$\vecq\in\scrP_{T_1}(\rho)$,
$\ell\in\{1,\ldots,N\}$,
$\mu\in F_{1,\ell}$,
$\tf\in F_{2,\ell}$
and $\vecbeta\in F_{3,\ell}$.
This is possible by Theorem \ref{MAINTECHNTHM2A} applied with $n-1$ in the place of $n$.
We shrink $\rho_0$ further if necessary, so as to also ensure that
\begin{align}\notag
\biggl|\int_{\fW(1;\rho)}f_{n-1}\bigl(\vecq,\vecv,\rho^{d-1}\tau_1(\rho^{1-d}\vecq,\vecv;\rho),
\vs_1(\rho^{1-d}\vecq,\vecv;\rho),
\vecv_1(\rho^{1-d}\vecq,\vecv;\rho)\bigr)\,d\Lambda(\vecq,\vecv)
\hspace{10pt}
\\\label{MAINTECHNthm2Gpf5}
-\int_{X^{(1)}}
f_{n-1}(\vecq,\vecv,\xi_1,\vs_1,\vecv_1)p(\vecv;\xi_1,{\vs}_1,\vecv_1\bigr)
\, d\Lambda(\vecq,\vecv)\,d\xi_1\,d\mm({\vs}_1)\,d\vecv_1
\biggr|<\ve
\end{align}
for all $\rho\in(0,\rho_0)$.
This is possible by Theorem \ref{unifmodThm2macr}.
We shrink $\rho_0$ still further if necessary, so as to also ensure that
the following four conditions \eqref{RHO0COND3macr}--\eqref{RHO0COND2macr} are fulfilled:
\begin{align}\label{RHO0COND3macr}
  \rho_0<\min\biggl\{\trho_0\Bigl(\eta,10,\min\Bigl(\frac{\ve}{\|f\|_\infty},\eta'\Bigr)\Bigr),
  \:\rho_0'
  ,\:\Bigl(\frac{\eta}{4C_1}\Bigr)^{1/(d-1)}
  \biggr\}
\end{align}
(where $\trho_0(\cdots)$ is as in Lemma \ref{SHINELEM1} and $\rho_0'$ is the number in \eqref{MAINTECHNthm2Gpf2});
\begin{align}\label{EHITSMALLPROBmacr}
\Lambda(\{(\vecq,\vecv)\in\fW(1;\rho)\col\vecq^{(1)}(\rho^{1-d}\vecq,\vecv;\rho)\in\scrE\})<\ve/\|f\|_\infty
\qquad\forall \rho\in(0,\rho_0)
\end{align}
(as is possible by Prop.\ \ref{Thm2genaddgenlem});
\begin{align}\label{RHO0COND1macr}
&\varphi(\vecv,\vecv')\leq 4C_1\rho_0^d
\text{ and }
|\xi_1-\xi_1'|\leq 2\rho_0^d
\\\notag
&\quad\Rightarrow\quad
\bigl|f(\vecq,\vecv,\langle\xi_j,\vs_j,\vecv_j\rangle_{j=1}^{n})-
f(\vecq,\vecv',\xi_1',\vs_1,\vecv_1,\langle\xi_j,\vs_j,\vecv_j\rangle_{j=2}^{n})\bigr|<\ve
\end{align}
(this can be obtained since $f$ is continuous and has compact support);
and
\begin{align}\label{RHO0COND2macr}
\varphi(\vecv,\vecv')\leq 8C_1\rho_0^d
\quad\Rightarrow\quad
|g(\vecq,\vecv)-g(\vecq,\vecv')|<\frac{\ve}{\vol(\scrB_T^d)\omega(\US)\|f\|_\infty}.
\end{align}

Now fix any $\rho\in(0,\rho_0)$.
We will prove that the integral in the left hand side of 
\eqref{MAINTECHNthm2Gres} differs from the right hand side by $\ll\ve$.
We will use the following short-hand notation:
\begin{align*}
&\tau_j(\vecv):=\tau_j(\rho^{1-d}\vecq,\vecv;\rho);
\qquad
\vs_j(\vecv):=\vs_j(\rho^{1-d}\vecq,\vecv;\rho);
\qquad
\vecv_j(\vecv):=\vecv_j(\rho^{1-d}\vecq,\vecv;\rho);
\\
&\vecq^{(j)}(\vecv):=\vecq^{(j)}(\rho^{1-d}\vecq,\vecv;\rho);
\qquad
\vecs_1(\vecv):=\vecq^{(1)}(\vecv)-\rho^{1-d}\vecq.
\end{align*}
For each $\vecq\in\R^d$ we set
\begin{align*}
\fW_\vecq(n;\rho):=\{\vecv\in\US\col (\vecq,\vecv)\in \fW(n;\rho)\},
\end{align*}
and let $\lambda_\vecq$ be the Borel measure $d\lambda_\vecq(\vecv):=g(\vecq,\vecv)\,d\vecv$ on $\US$.
Then since\linebreak $f(\vecq,\vecv,\langle\xi_j,\vs_j,\vecv_j\rangle_{j=1}^{n})=0$ whenever 
$\|\vecq\|\geq T$,
the integral in the left hand side of 
\eqref{MAINTECHNthm2Gres} equals
\begin{align}\label{MAINTECHNthm2GaPF1a}
\int_{\scrB_T^d}
\int_{\fW_\vecq(n;\rho)}f\Bigl(\vecq,\vecv,\big\langle\rho^{d-1}\tau_j(\vecv),\vs_j(\vecv),\vecv_j(\vecv)\big\rangle_{j=1}^n\Bigr)
\,d\lambda_\vecq(\vecv)\,d\vecq.
\end{align}
For each $\vecq\in\scrB_T^d$, we 
define the set $U_{2,\vecq}$ as the exact counterpart of $U_2$ in \eqref{U2def}:
\begin{align*}
U_{2,\vecq}:=\Bigl\{\vecv\in\fW_\vecq(1;\rho):\: C_1^{-1}\hspace{-1pt}<\hspace{-1pt}\rho^{d-1}\tau_1(\vecv)\hspace{-1pt}<\hspace{-1pt}C_1,
\:
\vecq^{(1)}(\vecv)\in\scrP\setminus\scrE,
\:
[\vecv_1(\vecv)]\subset\scrV_{\vecs_1(\vecv)}^{5\eta},
\:
\\
\text{and }
\bigl[\forall\vecalf\in[\vecv_1(\vecv)]\col\exists\vecv'\in\fW_{\vecq}(1;\rho)\text{ s.t. }\:
\vecq^{(1)}(\vecv')=\vecq^{(1)}(\vecv)
\text{ and }\vecv_1(\vecv')=\vecalf\bigr]\Bigr\}.
\end{align*}

We now have the following analogue of Lemma \ref{GOODCOLLlem}:
\begin{lem}\label{GOODCOLLlemmacr}
For $\vecq\in\scrB_T^d$, if $\vecv\in\fW_\vecq(1;\rho)\setminus U_{2,\vecq}$ then one of the following holds:
\begin{enumerate}[{\rm (i)}]
\item $\rho^{d-1}\tau_1(\vecv)\notin(C_1^{-1},C_1)$;
\item $\vecq^{(1)}(\vecv)\in\scrE$;
\item $(\vecq,\vecv)\in\fG_{\rho,\eta}$.
\end{enumerate}
\end{lem}
\begin{proof}
The proof of Lemma \ref{GOODCOLLlem} carries over
with very small and obvious modifications.
There are some simplifications due to the fact that we now have
``$\vecbeta\equiv\bn$'', meaning that we can replace $C_2$ by $0$,
and take $C=10$ in the application of Lemma \ref{SHINELEM1}.
\end{proof}

By Lemma \ref{GOODCOLLlemmacr},
and since %
$f(\vecq,\vecv,\langle\xi_j,\vs_j,\vecv_j\rangle_{j=1}^{n})=0$ whenever 
$\xi_1\notin(C_1^{-1},C_1)$,
the error caused by replacing the domain of integration in 
the inner integral in \eqref{MAINTECHNthm2GaPF1a} by $U_{2,\vecq}\cap\fW_\vecq(n;\rho)$ is
\begin{align}\label{GOODCOLLlemmacrAPPL}
\leq\bigl(\Lambda(\{(\vecq,\vecv)\in\fW(1;\rho)\col\vecq^{(1)}(\vecv)\in\scrE\})+
\Lambda(\fG_{\rho,\eta})\bigr)\cdot\|f\|_{\infty}
<2\ve.
\end{align}
Cf.\ \eqref{MAINTECHNthm2Gpf2} and \eqref{EHITSMALLPROBmacr} for the last inequality.

Now our task is to understand
\begin{align}\label{MAINTECHNthm2GaPF3aa}
\int_{\scrB_T^d}
\int_{U_{2,\vecq}\cap\fW_\vecq(n;\rho)}f\Bigl(\vecq,\vecv,\big\langle\rho^{d-1}\tau_j(\vecv),\vs_j(\vecv),\vecv_j(\vecv)\big\rangle_{j=1}^n\Bigr)
\,d\lambda_\vecq(\vecv)\,d\vecq.
\end{align}
\textit{Fix} an arbitrary $\vecq\in\scrB_T^d$,
and consider the inner integral in \eqref{MAINTECHNthm2GaPF3aa}.
This integral can be decomposed as a sum over those $\vecq'$ which can appear as $\vecq^{(1)}$.
By the definition of $U_{2,\vecq}$, each such point $\vecq'$ satisfies $\vecq'\in\scrP\setminus\scrE$ and
\begin{align*}
\|\vecq'\|\leq\|\rho^{1-d}\vecq\|+2\rho+C_1\rho^{1-d}<(T+C_1+1)\rho^{1-d}=T_1\rho^{1-d};
\end{align*}
thus $\vecq'\in\scrP_{T_1}(\rho)$.
Given any $\vecq'\in\scrP_{T_1}(\rho)$
we write $\vecs_1:=\vecq'-\rho^{1-d}\vecq$, %
and let $M_\vecq(\vecq')$ be the corresponding set of
$\ell\in\{1,\ldots,N\}$ such that $\tD_\ell$ is far from grazing position and is fully lit upon,
i.e.,
\begin{align*}
M_\vecq(\vecq')=\bigl\{\ell\col
\tD_\ell\subset\scrV_{\vecs_1}^{5\eta}
\:\:\text{ and }\:\:
\hspace{210pt}
\\
[\forall\vecalf\in\tD_\ell\col\exists\vecv'\in\fW_\vecq(1;\rho),\text{ s.t. }
\vecq^{(1)}(\vecv')=\vecq'\text{ and }\vecv_1(\vecv')=\vecalf]\bigr\}.
\end{align*}
Then for our fixed $\vecq\in\scrB_T^d$,
by the same argument as for \eqref{MAINTECHNthm2aPF3a},
the inner integral in \eqref{MAINTECHNthm2GaPF3aa} can be expressed as
\begin{align}\label{MAINTECHNthm2GaPF3a}
\sum_{\vecq'\in\scrP_{T_1}(\rho)}\:\sum_{\ell\in M_\vecq(\vecq')}\:\int_{U_{\vecq',\ell}\cap\fW_\vecq(n;\rho)}
f\bigl(\vecq,\vecv,\big\langle\rho^{d-1}\tau_j(\vecv),{\vs}_j(\vecv),
\vecv_j(\vecv)\big\rangle_{j=1}^{n}\bigr)\,d\lambda_\vecq(\vecv)
\end{align}
where
\begin{align*}
U_{\vecq',\ell}:=\{\vecv\in \fW_\vecq(1;\rho)\col\vecq^{(1)}(\vecv)=\vecq',\:\vecv_1(\vecv)\in \tD_\ell\}.
\end{align*}

Now, by a direct mimic of the treatment of \eqref{MAINTECHNthm2aPF3a}
in the proof of Theorem \ref{MAINTECHNTHM2A},
all the way to \eqref{MAINTECHNthm2aPF10},
one shows that for every $\vecq\in\scrB_T^d$, the expression in \eqref{MAINTECHNthm2GaPF3a} 
differs by at most
\begin{align*}
\ve\cdot\Bigl(6\,\lambda_\vecq(\US)+\frac2{\vol(\scrB_T^d)}\Bigr)
\end{align*}
from
\begin{align}\label{MAINTECHNthm2GaPF11rep}
\int_{U_{2,\vecq}}f_{n-1}(\vecq,\vecv,\rho^{d-1}\tau_1(\vecv),\vs_1(\vecv),\vecv_1(\vecv))\,d\lambda_\vecq(\vecv).
\end{align}
(The factor $\vol(\scrB_T^d)^{-1}$ in the bound comes from the corresponding factor in 
\eqref{RHO0COND2macr}, which was not present %
in the analogous assumption in \eqref{RHO0COND2}.)
Integrating now over $\vecq\in\scrB_T^d$,
we conclude that up to an error of absolute size $\leq8\ve$,
the double integral in \eqref{MAINTECHNthm2GaPF3aa} equals
\begin{align}\label{MAINTECHNthm2GaPF11a}
\int_{\scrB_T^d}
\int_{U_{2,\vecq}}
f_{n-1}(\vecq,\vecv,\rho^{d-1}\tau_1(\vecv),\vs_1(\vecv),\vecv_1(\vecv))\,d\lambda_\vecq(\vecv)\,d\vecq.
\end{align}
By Lemma \ref{GOODCOLLlemmacr} and the bound in \eqref{GOODCOLLlemmacrAPPL},
replacing $U_{2,\vecq}$ by $\fW_\vecq(1;\rho)$ in \eqref{MAINTECHNthm2GaPF11a} 
causes a total error less than $2\ve$;
and hence, using also \eqref{MAINTECHNthm2Gpf5},
it follows that \eqref{MAINTECHNthm2GaPF11a} differs by less than $3\ve$ from 
\begin{align}\notag
&\int_{X^{(1)}}
f_{n-1}(\vecq,\vecv,\xi_1,\vs_1,\vecv_1)p(\vecv;\xi_1,{\vs}_1,\vecv_1\bigr)
\, d\Lambda(\vecq,\vecv)\,d\xi_1\,d\mm({\vs}_1)\,d\vecv_1,
\end{align}
and this integral is equal to the right hand side of \eqref{MAINTECHNthm2Gres}.

Summing up, we have proved that for any
$\rho\in(0,\rho_0)$,
the integral in \eqref{MAINTECHNthm2GaPF1a} 
differs by less than $13\ve$
from the right hand side of \eqref{MAINTECHNthm2Gres}.
This completes the proof of Theorem \ref{MAINTECHNthm2G}.
\hfill$\square$ %

\section{Random flight processes}
\label{LIMITFLIGHTsec}
We will here discuss the deduction of Theorem \ref{thm:M1} and Theorem \ref{thm:M1V}
from Theorem \ref{MAINTECHNthm2G}.

For any metric space $S$ and positive real number $T$,
we write $D_S[0,T]$ (resp., $D_S[0,\infty)$) for the space of c\`adl\`ag functions $[0,T]\to S$
(resp., $[0,\infty)\to S$),
equipped with the Skorohod topology
(cf., e.g., \cite[Ch.\ 3]{billingsley99} and \cite[Ch.\ 3.5]{sEtK86}).
Given $\Lambda\in P(\T^1(\R^d))$ and $\rho>0$,
if $(\vecq,\vecv)$ is a random point in $(\T^1(\R^d),\Lambda)$
then $\Theta^{(\rho)}$ 
defined by $\Theta^{(\rho)}(t) = \widetilde \Phi_t^{(\rho)} (\vecq,\vecv)$
as in \eqref{ThetarhoDEF} or \eqref{tildePhidef2}, %
is a random element in $D_{\T^1(\R^d)}[0,\infty)$.%

We will first give a precise definition of the limiting flight processes $\Theta$ appearing in 
Theorems~\ref{thm:M1} and \ref{thm:M1V}.
To this end, we extend \eqref{Xmacrdef} by letting
\begin{align}\notag
X^{(\infty)}:=\Bigl\{\langle\vecq_0,\vecv_0;\langle\xi_j,{\vs}_j,\vecv_j\big\rangle_{j=1}^{\infty}\rangle\in
\T^1(\R^d)\times\prod_{j=1}^\infty(\R_{>0}\times\Sigma\times\US)
\col\:\:\hspace{30pt}
\\\label{Xinftydef}
\vecv_j\in\scrV_{\vecv_{j-1}},\:\forall j\geq1\Bigr\},
\end{align}
with the topology induced
from the product topology
on $\T^1(\R^d)\times\prod_{j=1}^\infty(\R_{>0}\times\Sigma\times\US)$.
Also let $\pr_n:X^{(\infty)}\to X^{(n)}$
be the projection taking
$\langle\vecq_0,\vecv_0;\langle\xi_j,{\vs}_j,\vecv_j\big\rangle_{j=1}^{\infty}\rangle$
to $\langle\vecq_0,\vecv_0;\langle\xi_j,{\vs}_j,\vecv_j\big\rangle_{j=1}^{n}\rangle$.
Given any Borel probability measure $\Lambda$ on $\T^1(\R^d)$,
we let $\nu_\Lambda$ be the unique Borel probability measure on $X^{(\infty)}$
which for any $n\geq1$ and any Borel set $A\subset X^{(n)}$ satisfies
\begin{align}\label{nuLambdadef}
\nu_\Lambda(\pr_n^{-1}(A))
=\int_A
p\bigl(\vecv_0;\xi_1,{\vs}_1,\vecv_1\bigr)
\,\prod_{j=2}^n p_\bn(\vecv_{j-2},{\vs}_{j-1},\vecv_{j-1};\xi_j,{\vs}_{j},\vecv_{j})
\hspace{20pt}
\\\notag
\times d\Lambda(\vecq,\vecv_0)\prod_{j=1}^n\bigl( d\xi_j\,d\mm({\vs}_j)\,d\vecv_j\bigr).
\end{align}
(The probability measure in the right hand side is exactly the one that
appears in \eqref{MAINTECHNthm2Gres} in Theorem \ref{MAINTECHNthm2G}.)
The existence and uniqueness of the measure $\nu_\Lambda$ is a consequence of %
the Kolmogorov extension theorem.
Note that $\nu_\Lambda$ is the distribution of a Markov process with memory two 
on the space $\R_{>0}\times\Sigma\times\US$.

Set
\begin{align*}
\scrF=\bigl\{\bigl(\vecq_0,\vecv_0,\big\langle\xi_j,{\vs}_j,\vecv_j\big\rangle_{j=1}^{\infty}\bigr)\in
X^{(\infty)}
\col\:\:
{\textstyle\sum_{j=1}^\infty\xi_j<\infty}\bigr\}.
\end{align*}
\begin{lem}\label{SUMINFlem2}
$\nu_\Lambda(\scrF)=0$.
\end{lem}
\begin{proof}
For any $t>0$ and any positive integer $n$ we have
\begin{align}\label{SUMINFlem2pf1}
\nu_\Lambda\Bigl(\Bigl\{\big\langle\vecq_0,\vecv_0,\big\langle\xi_j,{\vs}_j,\vecv_j\big\rangle_{j=1}^{\infty}\big\rangle
\in X^{(\infty)}
\col \sum_{j=1}^n\xi_j\leq t\Bigr\}\Bigr)\leq (c_\scrP v_{d-1})^n \frac{t^n}{n!}.
\end{align}
Indeed, using \eqref{nuLambdadef} and Lemmas \ref{BASICpktransfrem1lem} and \ref{BASICpktransfrem2lem} 
to express the left hand side as an integral over
$\T^1(\R^d)\times (\R_{>0}\times\Omega)^n$,
and then using the fact that both $k$ and $k^g$ take values in $[0,c_\scrP v_{d-1}]$
(cf.\ \eqref{kdef} and \eqref{kgdef}),
the left hand side of \eqref{SUMINFlem2pf1} is seen to be
bounded above by $(c_\scrP v_{d-1})^n$ times the Lebesgue volume of the
simplex $\{(\xi_1,\ldots,\xi_n)\in(\R_{>0})^n\col\sum_{j=1}^n\xi_j\leq t\}$;
this is the bound in \eqref{SUMINFlem2pf1}.

It follows from \eqref{SUMINFlem2pf1} that $\nu_\Lambda\bigl(\sum_{j=1}^\infty\xi_j<t\bigr)=0$, for every $t>0$.
The lemma follows from this fact.
\end{proof}
We next define a map 
\begin{align}\label{Jx}
J:X^{(\infty)}\to D_{\T^1(\R^d)}[0,\infty),
\end{align}
as follows.
For $x=\bigl(\vecq_0,\vecv_0;\big\langle\xi_j,{\vs}_j,\vecv_j\big\rangle_{j=1}^{\infty}\bigr)$
in $X^{(\infty)}\setminus\scrF$,
\begin{align}\label{Jxt}
J(x)(t) %
:=\Bigl(\vecq_0+\sum_{j=1}^n\xi_j\vecv_{j-1}+\bigl(t-\xi_1-\cdots-\xi_n\bigr)\vecv_n,
\vecv_n\Bigr),
\end{align}
where $n=n(\langle\xi_1,\xi_2,\ldots\rangle,t)$ is the nonnegative integer defined through the relation
\begin{align}\label{ndef}
\xi_1+\cdots+\xi_n\leq t<\xi_1+\cdots+\xi_{n+1}.
\end{align}
To make $J$ defined on all $X^{(\infty)}$ we choose a fixed (dummy) value $y^{0}\in D_{\T^1(\R^d)}[0,\infty)$
and declare $J(x):=y^0$ for all $x\in\scrF$.

The map $J$ is Borel measurable;
in fact $J$ is even continuous on $X^{(\infty)}\setminus\scrF$,
as one easily verifies using \cite[Prop.\ 3.6.5]{sEtK86}.
\begin{definition}\label{Thetadef}
We let $\Theta$ be the random element 
$J(x)$ in $D_{\T^1(\R^d)}[0,\infty)$
for $x$ random in $(X^{(\infty)},\nu_\Lambda)$.
\end{definition}

We will now give the proof of Theorems \ref{thm:M1} and \ref{thm:M1V}.
In the case of Theorem \ref{thm:M1}, \label{thmsM1M1Vpfs}
we consider the hard sphere scattering process introduced in Section~\ref{sec:classical},
and let the scattering map $\Psi:\scrS_-\to\scrS_+$ be as in \eqref{scatmap};
in the case of Theorem \ref{thm:M1V} we consider the Lorentz process for potentials
and let $\Psi:\scrS_-\to\scrS_+$ be the scattering map
associated to the fixed potential $W$.
(We recall explicit formulas for the correspondence $W\mapsto\Psi$ in Section \ref{AppA1} below.
Note that in Theorem \ref{thm:M1V} we are assuming that $W$ is such that $\Psi$ satisfies the
conditions in Section~\ref{SCATTERINGMAPS}.)
The choice of scattering map $\Psi$
then leads to corresponding collision kernels
$p(\vecv;\xi,\vs_+,\vecv_+)$
and $p_{\bn}(\vecv_0,\vs,\vecv;\xi,\vs_+,\vecv_+)$
(cf.\ Sec.\ \ref{COLLKERsec})
and a corresponding probability measure $\nu_\Lambda$ on $X^{(\infty)}$
(cf.\ \eqref{nuLambdadef})
and finally a random flight process $\Theta$ (cf.\ Def.~\ref{Thetadef}).
We will prove that Theorem~\ref{thm:M1} (resp., Theorem \ref{thm:M1V}) holds
with \textit{this} limiting random flight process $\Theta$.

Let us note that it suffices to prove that,
for each fixed $T>0$,
the random element $\Theta^{(\rho)}|_{[0,T]}$ in 
$D_{\T^1(\R^d)}[0,T]$ converges in distribution to
$\Theta|_{[0,T]}$, as $\rho\to0$.
Thus from now on we keep $T$ fixed.
For each $n\in\Z^+$ we define $\Sigma_n:X^{(n)}\to\R_{>0}$ by
\begin{align*}
\Sigma_n(x)=\xi_1+\cdots+\xi_n\qquad\text{for }\: 
x=\big\langle\vecq_0,\vecv_0;\big\langle\xi_j,{\vs}_j,\vecv_j\big\rangle_{j=1}^n\big\rangle\in X^{(n)}.
\end{align*}
We also view $\Sigma_n$ as a function on $X^{(\infty)}$,
via composition with the projection $\pr_n:X^{(\infty)}\to X^{(n)}$.
Then define the random element $\Theta_{n,T}$ in $D_{\T^1(\R^d)}[0,T]$  through
\begin{align}\label{ThetanTdef}
\Theta_{n,T}=\begin{cases}
J(x)&\text{if }\: \Sigma_n(x)>T,
\\
y^0|_{[0,T]}&\text{if }\:\Sigma_n(x)\leq T,
\end{cases}
\end{align}
with $x$ being the same random element in $(X^{(\infty)},\nu_\Lambda)$
as in Definition \ref{Thetadef},
and $y^0$ being the dummy constant in $D_{\T^1(\R^d)}[0,\infty)$ fixed above.
Let us record that,
as an immediate consequence of \eqref{SUMINFlem2pf1} applied with $t=T$, we have
\begin{align}\label{thm:M1Vpf2}
\lim_{n\to\infty}\PP(\Theta_{n,T}=\Theta_T)=1.
\end{align}

For any $\rho>0$, $(\vecq,\vecv)\in\fw(n;\rho)$ and $j\in\{1,\ldots,n\}$ we let
$\tau_j(\vecq,\vecv;\rho)$,
$\vs_j(\vecq,\vecv;\rho)$,
and 
$\vecv_j(\vecq,\vecv;\rho)$
be as defined in Section \ref{SCATTERINGMAPS}.
We define the map \label{Crhodef}
\begin{align*}
C_\rho:\T^1(\R^d)\to X^{(n)}
\end{align*}
by 
\begin{align*}
C_\rho(\vecq,\vecv)=\begin{cases}
\big\langle\vecq,\vecv,
\big\langle\rho^{d-1}\tau_j(\rho^{1-d}\vecq,\vecv;\rho),{\vs}_j(\rho^{1-d}\vecq,\vecv;\rho),
\vecv_j(\rho^{1-d}\vecq,\vecv;\rho)
\big\rangle_{j=1}^n\big\rangle
\\
\rule{0pt}{0pt}\hspace{210pt}
\text{if }\: (\vecq,\vecv)\in\fW(n;\rho),
\\[3pt]
x^0\hspace{200pt}\text{if }\: (\vecq,\vecv)\notin\fW(n;\rho)
\end{cases}
\end{align*}
(recall \eqref{fWndef}),
where $x^0$ is a (dummy) point in $X^{(n)}$ fixed once and for all.
Let us also set
\begin{align}\notag
\fW_T(n;\rho):=\bigl\{(\vecq,\vecv)\in\fW(n;\rho)\col \Sigma_n(C_\rho(\vecq,\vecv))>T\:\text{ and }\:
\hspace{100pt}
\\\label{fWtnDEF}
\vecv_{j-1}(\vecq,\vecv;\rho)\neq s_\Psi\cdot\vecv_j(\vecq,\vecv;\rho)
\text{ for }j=1,\ldots,n-1\bigr\}.
\end{align}

\begin{remark}
Recall \eqref{Psi1vmv} regarding $s_\Psi$;
thus the last condition in \eqref{fWtnDEF} means that none of the first $n-1$ %
collisions occurs with exactly vanishing impact parameter.
We need to exclude the case of vanishing impact parameter
since the collision time may be infinite in this case.
We remark that in the case of the hard sphere scattering process,
the last condition in \eqref{fWtnDEF} could be removed in the following proof.
\end{remark}

We will next define a map $J_\rho:X^{(n)}\to D_{\T^1(\R^d)}[0,T]$
such that $[J_\rho\circ C_\rho(\vecq,\vecv)](t)=\widetilde \Phi_t^{(\rho)} (\vecq,\vecv)$
for all $(\vecq,\vecv)\in\fW_T(n;\rho)$ and $t\in[0,T]$.
This is slightly more complicated in the case of the Lorentz process for potentials, and we discuss that case first.
Here, we first need to introduce one more piece of notation
regarding the Hamiltonian flow with the potential $W$.
Recall that for a particle entering the unit sphere with velocity $\vecv_-$
and exiting with velocity $\vecv_+$,
the point of entrance is uniquely determined to be
$\vecbeta_{\vecv_-}^-(\vecv_+)$,
and the point of exit is $\vecbeta_{\vecv_-}^+(\vecv_+)$.
It is also easily verified %
that the total time which the particle spends inside the unit sphere,
$T_{\vecv_-,\vecv_+}$, is finite whenever $\vecbeta_{\vecv_-}^-(\vecv_+)\neq-\vecv_-$.\footnote{Indeed,
we have
$T_{\vecv_-,\vecv_+}=T\bigl(\|(\vecbeta_{\vecv_-}^-(\vecv_+)R(\vecv_-))_{\perp}\|\bigr)$ 
in the notation of \eqref{scatteringtime};
hence the claim follows from \eqref{Twbound} and Lemma \ref{DEFLANGLElem}(2).}
In this case, let the particle path inside the unit sphere  be
$t\mapsto\vecpsi_{\vecv_-,\vecv_+}(t)$, for $t\in[0,T_{\vecv_-,\vecv_+}]$;
in particular $\vecpsi_{\vecv_-,\vecv_+}(0)=\vecbeta_{\vecv_-}^-(\vecv_+)$
and $\vecpsi_{\vecv_-,\vecv_+}(T_{\vecv_-,\vecv_+})=\vecbeta_{\vecv_-}^+(\vecv_+)$.
It follows that the particle path in the sphere of radius $\rho^d$ centered at the origin is given by
$t\mapsto \rho^d\,\vecpsi_{\vecv_-,\vecv_+}(\rho^{-d}t)$ for $t\in[0,\rho^d\,T_{\vecv_-,\vecv_+}]$.
\label{PATHINBALLRADIUSRHOd}
Now we define the map 
\begin{align*}
J_\rho:X^{(n)}\to D_{\T^1(\R^d)}[0,T]
\end{align*}
as follows. Let 
\begin{align*}
x=\bigl(\vecq,\vecv_0;\big\langle\xi_j,{\vs}_j,\vecv_j\big\rangle_{j=1}^{n}\bigr)\in X^{(n)}
\end{align*}
be given. If $\sum_{j=1}^n\xi_j\leq T$
or if $\vecv_j=s_\Psi\cdot\vecv_{j-1}$ for some $j=1,\ldots,n-1$
(cf.\ \eqref{Psi1vmv})
then we set $J_\rho(x)=y^0|_{[0,T]}$.
From now on assume $\sum_{j=1}^n\xi_j>T$
and $\vecv_j\neq s_\Psi\cdot\vecv_{j-1}$ 
($\Leftrightarrow \vecbeta_{\vecv_{j-1}}^-(\vecv_j)\neq\vecv_{j-1}$)
for each $j=1,\ldots,n-1$.
Set $\xi_j':=\xi_j+\rho^d T_{\vecv_{j-1},\vecv_j}$.
Given $t\in[0,T]$,
let $m$ be the largest number in $\{0,1,\ldots,n\}$
satisfying
$\sum_{j=1}^m\xi_j'\leq t$;
then in fact $0\leq m\leq n-1$.
If $t\leq\xi_{m+1}+\sum_{j=1}^m\xi_j'$
then we set
\begin{align*}
J_\rho(x)(t):=\Bigl(\vecq+\sum_{j=1}^m\Bigl(\xi_j\vecv_{j-1}
+\rho^d\bigl(\vecbeta_{\vecv_{j-1}}^+(\vecv_j)-\vecbeta_{\vecv_{j-1}}^-(\vecv_j)\bigr)\Bigr)
\hspace{80pt}
\\
+(t-\xi_1'-\cdots-\xi_m')\vecv_m,\:\vecv_m\Bigr),
\end{align*}
whereas if $\xi_{m+1}+\sum_{j=1}^m\xi_j'<t<\sum_{j=1}^{m+1}\xi_j'$
then set
$s:=\rho^{-d}\bigl(t-\bigl(\xi_{m+1}+\sum_{j=1}^m\xi_j'\bigr)\bigr)$ and 
\begin{align}\notag
J_\rho(x)(t):=\Bigl(\vecq+\sum_{j=1}^m\Bigl(\xi_j\vecv_{j-1}
+\rho^d\bigl(\vecbeta_{\vecv_{j-1}}^+(\vecv_j)-\vecbeta_{\vecv_{j-1}}^-(\vecv_j)\bigr)\Bigr)
+\xi_{m+1}\vecv_m
\hspace{40pt}
\\\label{Jrhoxdef2}
+
\rho^d\bigl(\vecpsi_{\vecv_{m},\vecv_{m+1}}(s)-\vecbeta_{\vecv_m}^-(\vecv_{m+1})\bigr)
,\:\|\dot\vecpsi_{\vecv_{m},\vecv_{m+1}}(s)\|^{-1}\dot\vecpsi_{\vecv_{m},\vecv_{m+1}}(s)\Bigr).
\end{align}
This completes the definition of $J_\rho$, in the case of the Lorentz process for potentials.

In the case of the hard sphere scattering process,
we define $J_\rho:X^{(n)}\to D_{\T^1(\R^d)}[0,T]$ simply by applying the
above definition with $T_{\vecv_-,\vecv_+}\equiv0$;
this means that $\xi_j'=\xi_j$ for all $j$ and the case \eqref{Jrhoxdef2} never occurs;
thus there is no reference to ``$\vecpsi_{\vecv_-,\vecv_+}(t)$''.

By inspection one verifies that,
both for the Lorentz process for potentials and for the hard sphere scattering process:
\begin{align}\label{JrhoCrhoeqtPhi}
[J_\rho\circ C_\rho(\vecq,\vecv)](t)=\widetilde \Phi_t^{(\rho)} (\vecq,\vecv),
\qquad\forall (\vecq,\vecv)\in\fW_T(n;\rho),\: t\in[0,T].
\end{align}
Furthermore $J_\rho\circ C_\rho(\vecq,\vecv)=y^0|_{[0,T]}$ for all 
$(\vecq,\vecv)\in\fW(n;\rho)\setminus\fW_T(n;\rho)$.

Recall from Sections \ref{sec:classical} and \ref{sec:soft}
that the random element $\Theta^{(\rho)}$ in $D_{\T^1(\R^d)}[0,\infty)$ 
is defined by $\Theta^{(\rho)}(t) = \widetilde \Phi_t^{(\rho)} (\vecq,\vecv)$
where $(\vecq,\vecv)$ is a random point in $(\T^1(\R^d),\Lambda)$.
Using the same random point %
$(\vecq,\vecv)$,
we now also introduce, for each fixed $n\in\Z^+$,
the random element $\Theta^{(\rho)}_{n,T}$ in $D_{\T^1(\R^d)}[0,T]$ 
through
\begin{align*}
\Theta^{(\rho)}_{n,T}:=\begin{cases} J_\rho\circ C_\rho(\vecq,\vecv)&\text{if }\: (\vecq,\vecv)\in\fW(n;\rho),
\\
y^0|_{[0,T]}&\text{if }\: (\vecq,\vecv)\notin\fW(n;\rho).
\end{cases}
\end{align*}
Note that 
$\Theta^{(\rho)}_{n,T}=\Theta^{(\rho)}|_{[0,T]}$ whenever
$(\vecq,\vecv)\in\fW_T(n;\rho)$;
hence
\begin{align*}
\PP\bigl(\Theta^{(\rho)}_{n,T}=\Theta^{(\rho)}|_{[0,T]}\bigr)
\geq\Lambda(\fW_T(n;\rho)).
\end{align*}
Furthermore,
by Theorem \ref{MAINTECHNthm2G} and Remark \ref{MAINTECHNthm2Gprobrem}
(and \eqref{nuLambdadef}),
$C_\rho(\vecq,\vecv)$ converges in distribution to a random point in 
$(X^{(n)},\nu_\Lambda\circ \pr_n^{-1})$ as $\rho\to0$.
It is immediate from \eqref{nuLambdadef} that
\begin{align}\notag
\nu_\Lambda\circ\pr_n^{-1}\bigl(\bigl\{z\in X^{(n)}\col \Sigma_n(z)=T\text{ or }\vecv_{j-1}(z)=s_\Psi\cdot\vecv_j(z)\text{ for some }
\hspace{40pt}
\\\label{thm:M1Vpf5}
j=1,\ldots,n-1\bigr\}\bigr)=0,
\end{align}
and hence by the Portmanteau Theorem,
\begin{align*}
\lim_{\rho\to0}\Lambda(\fW_T(n;\rho))=\nu_\Lambda\circ \pr_n^{-1}\bigl(\bigl\{z\in X^{(n)}\col \Sigma_n(z)>T\bigr\}\bigr).
\end{align*}
The last expression tends to $1$ as $n\to\infty$,
by \eqref{SUMINFlem2pf1} applied with $t=T$.
Hence we conclude:
\begin{align}\label{thm:M1Vpf3}
\lim_{n\to\infty}\liminf_{\rho\to0}\PP\bigl(\Theta^{(\rho)}_{n,T}=\Theta^{(\rho)}|_{[0,T]}\bigr)=1.
\end{align}
In view of \eqref{thm:M1Vpf3} 
and \eqref{thm:M1Vpf2},
in order to prove that 
$\Theta^{(\rho)}|_{[0,T]}$ converges in distribution to
$\Theta|_{[0,T]}$,
it now suffices to prove that for each fixed $n\in\Z^+$,
$\Theta^{(\rho)}_{n,T}$ converges in distribution to $\Theta_{n,T}$.

Thus from now on we keep $n\in\Z^+$ (as well as $T>0$) fixed.
We will prove the desired convergence by
using Theorem \ref{MAINTECHNthm2G} and the continuous mapping theorem.
We first need to introduce one more map.
We define
\begin{align}\label{tJdef1}
\tJ:X^{(n)}\to D_{\T^1(\R^d)}[0,T]
\end{align}
by setting,
for $x=\bigl(\vecq,\vecv_0;\big\langle\xi_j,{\vs}_j,\vecv_j\big\rangle_{j=1}^{n}\bigr)\in X^{(n)}$:
\begin{align*}
\tJ(x)=y^0|_{[0,T]}\qquad\text{if }\:\Sigma_n(x)\leq T,
\end{align*}
while if $\Sigma_n(x)>T$ then
\begin{align*}
\tJ(x)(t):=\Bigl(\vecq_0+\sum_{j=1}^m\xi_j\vecv_{j-1}+\bigl(t-\Sigma_m(x)\bigr)\vecv_m,
\vecv_m\Bigr),
\end{align*}
where $m$ is the unique integer in $\{1,\ldots,n-1\}$ such that
$\Sigma_m(x)\leq t<\Sigma_{m+1}(x)$.
The point of this definition is that now 
the random element $\Theta_{n,T}$ %
in \eqref{ThetanTdef} can be expressed as
\begin{align*}
\Theta_{n,T}=\tJ(\pr_n(x)),
\end{align*}
where $x$ is a random point in $(X^{(\infty)},\nu_\Lambda)$ as before.

\begin{lem}\label{JrhotendtotJLEM}
If $\{\rho_k\}$ is any sequence in $(0,1)$ with $\rho_k\to0$,
and $\{z_k\}$ is any sequence in $X^{(n)}$
such that $z:=\lim_{k\to\infty}z_k$ exists in $X^{(n)}$, and $\Sigma_n(z)\neq T$
and $\vecv_{j-1}(z)\neq s_\Psi\cdot\vecv_j(z)$ for each $j=1,\ldots,n-1$,
then $\lim_{k\to\infty}J_{\rho_k}(z_k)=\tJ(z)$ in $D_{\T^1(\R^d)}[0,T]$.
\end{lem}
\begin{proof}
This is easily verified by comparing the definitions of $J_\rho$ and $\tJ$.
One also uses the basic fact that the collision time for any scatterer collision
is uniformly bounded so long as the impact parameter is bounded away from zero
(cf.\ Lemma \ref{DEFLANGLElem}(2) and \eqref{Twbound} in Section~\ref{AppA1} below).
\end{proof}

We continue to let $x$ be a random point in $(X^{(\infty)},\nu_\Lambda)$
and also let $(\vecq,\vecv)$ be a random point in $(\T^1(\R^d),\Lambda)$.
As we have noted above,
$C_\rho(\vecq,\vecv)$
tends in distribution to $\pr_n(x)$ as $\rho\to0$
(by Theorem \ref{MAINTECHNthm2G} and \eqref{nuLambdadef}).
Hence by the continuous mapping theorem
\cite[Thm.\ 4.27]{kallenberg02},
together with Lemma \ref{JrhotendtotJLEM} and \eqref{thm:M1Vpf5},
we conclude that $J_\rho\circ C_\rho(\vecq,\vecv)$
tends in distribution to $\tJ(\pr_n(x))$ as $\rho\to0$.
Equivalently, $\Theta_{n,T}^{(\rho)}$ tends in distribution to $\Theta_{n,T}$.

\vspace{2pt}

This completes the proof of both Theorems \ref{thm:M1} and \ref{thm:M1V}.
\hfill$\square$

\section{Semigroups and kinetic transport equations}
\label{KINETICEQsec}

This section provides more details on the forward Kolmogorov equation \eqref{KINETICEQ00} for the random flight process $\hTheta$ introduced in \eqref{eq:RFP2}. We follow closely \cite[Section 6]{partII}, and will only highlight key steps. We start by providing a precise definition for $\hTheta$ and showing the process is Markovian.

Define $\L^1(X)$ and $\Lloc(X)$ as the spaces of integrable/locally integrable functions $X\to\RR$ with respect to the measure $d\vecq\,d\vecv\, d\xi\,d\mm(\vs)\,d\vecv_+$, where $X$ is the extended phase space as defined in \eqref{XXdef2}. We generalise $\nu_\Lambda$ in \eqref{nuLambdadef} as follows. Given any non-negative function $f\in\Lloc(X)$ we define
$\hat\nu_f$ to be the (unique) Borel measure on $X^{(\infty)}$
which for any $n\geq1$ and any Borel set $A\subset X^{(n)}$ satisfies
\begin{multline}\label{nuLambdadef231}
\hat\nu_f(\pr_n^{-1}(A))
=\int_A
f\bigl(\vecq,\vecv_0,\xi_1,{\vs}_1,\vecv_1\bigr) \\ \times
\prod_{j=2}^n p_\bn(\vecv_{j-2},{\vs}_{j-1},\vecv_{j-1};\xi_j,{\vs}_{j},\vecv_{j})
\, d\vecq\, d\vecv_0 \prod_{j=1}^n\bigl( d\xi_j\,d\mm({\vs}_j)\,d\vecv_j\bigr).
\end{multline}
The same formula also associates to any 
$f\in\L^1(X)$
a signed %
Borel measure $\hnu_f$ on $X^{(\infty)}$.
Note that if $f$ is a probability density then $\hnu_f$ is a probability measure.
In analogy with \eqref{Jx} we define the map
\begin{align*}%
\widehat J:X^{(\infty)}\to D_{X}[0,\infty),
\end{align*}
by 
\begin{align*}%
\widehat J(x)(t) %
:=\Bigl(\vecq_0+\sum_{j=1}^n\xi_j\vecv_{j-1}+\bigl(t-\Sigma_n(x)\bigr)\vecv_n,
\vecv_n, \Sigma_{n+1}(x)-t, \vs_{n+1}, \vecv_{n+1} \Bigr),
\end{align*}
for $x=\bigl(\vecq_0,\vecv_0;\big\langle\xi_j,{\vs}_j,\vecv_j\big\rangle_{j=1}^{\infty}\bigr)\in X^{(\infty)}\setminus\scrF$
(with $n=n(\langle\xi_1,\xi_2,\ldots\rangle,t)$ as before); again declare the dummy variable $\widehat J(x):=y^0$ for all $x\in\scrF$.
This map $\hJ$ is Borel measurable;
in fact $\hJ$ is even continuous on $X^{(\infty)}\setminus\scrF$,
as one verifies using \cite[Prop.\ 3.6.5]{sEtK86}.

\begin{definition}\label{ThetaHATdef}
For $f\in\L^1(X)$ a probability density, we let $\hTheta$ be the random element 
$\widehat J(x)$ in $D_{X}[0,\infty)$
for $x$ random in $(X^{(\infty)},\hat\nu_f)$.
\end{definition}

That is, for all probability densities $f\in\L^1(X)$ and Borel sets $A\subset X$, 
\begin{equation*}
\PP(\hTheta(t)\in A) = \hat\nu_f\{x\in X^{(\infty)} \col \widehat J(x)(t)\in A\}.
\end{equation*}

\begin{definition}\label{KtDEFdef}
The evolution operator $K_t:\Lloc(X)\to \Lloc(X)$ for $\hTheta$ is defined by the relation
\begin{equation}\label{KtDEF}
\int_{A}  K_t f(\vecq,\vecv,\xi,\vs,\vecv_+) \, d\vecq\,d\vecv\, d\xi\,d\mm(\vs)\,d\vecv_+ = \hat\nu_f\{x\in X^{(\infty)} \col \widehat J(x)(t)\in A\} ,
\end{equation}
for all non-negative $f\in\Lloc(X)$ and Borel sets $A\subset X$,
and extended to all $\Lloc(X)$ by linearity.
\end{definition}

We note that $K_t$ preserves the subspace of non-negative functions in $\Lloc(X)$; $K_t$ also preserves $\L^1(X)$.
If $f\in\L^1(X)$ is non-negative, then $\| K_t f \|_{\L^1(X)}=\| f \|_{\L^1(X)}$. This follows from \eqref{KtDEF} for probability densities $f$, and for general non-negative $f$ by linearity. We thus have by the triangle inequality
\begin{equation} \label{contract001}
\| K_t f \|_{\L^1(X)}\leq \| f \|_{\L^1(X)}
\end{equation}
for all $f\in\L^1(X)$.
We have the following expansion in terms of number of collisions $n$ within time $t$,
\begin{equation} \label{KTEXPLICIT}
K_t = \sum_{n=0}^\infty K_t^{(n)}	,
\end{equation}
where
\begin{align}\label{KTndef}
	\int_{A}  K_t^{(n)} f(\vecq,\vecv,\xi,\vs,\vecv_+) \, d\vecq\,d\vecv\, d\xi\,d\mm(\vs)\,d\vecv_+ 
\hspace{150pt}
\\\notag
=\hat\nu_f\{x\in X^{(\infty)}\col\widehat J(x)(t)\in A,\, \Sigma_n(x)\leq t<\Sigma_{n+1}(x) \} .
\end{align}
More explicitly, in the case $n=0$,
\begin{equation} \label{KT0EXPLDEF}
K_t^{(0)}  f(\vecq,\vecv,\xi,\vs,\vecv_+) = f(\vecq-t\vecv,\vecv,\xi+t,\vs,\vecv_+) ,
\end{equation}
and for $n\geq 1$,
\begin{equation} \label{KTNEXPLDEF}
\begin{split}
& K_t^{(n)} f(\vecq,\vecv,\xi,\vs,\vecvp) \\
& =
\int_{\xi_1+\ldots+\xi_n\leq t}  f\bigg(\vecq-\bigg(\sum_{j=1}^n \xi_j\vecv_{j-1} + (t-\xi_1-\ldots-\xi_n) \vecv_n\bigg),\vecv_0,\xi_1,\vs_1,\vecv_1\bigg)\\ 
& \times \prod_{j=2}^{n+1} p_\bn(\vecv_{j-2},{\vs}_{j-1},\vecv_{j-1};\xi_j,{\vs}_{j},\vecv_{j})
\,  \prod_{j=1}^n\bigl( d\vecv_{j-1}\,d\xi_j\,d\mm({\vs}_j)\bigr) ,
\end{split}
\end{equation}
subject to 
\begin{equation}\label{convent007}
\vecv_n=\vecv, \quad \xi_{n+1}=\xi+t-(\xi_1+\ldots+\xi_n), \quad \vs_{n+1}=\vs, \quad \vecv_{n+1}=\vecv_+.
\end{equation}

We have for any $f\in \L^1(X)$ and $n\geq1$:
\begin{equation} \label{KTNEXPLDEF001}
\| K_t^{(n)} f \|_{\L^1(X)} \leq  \frac{(c_\scrP v_{d-1} t)^{n-1}}{(n-1)!} \, \| f \|_{\L^1(X)} .
\end{equation}
This bound is proved by first applying Lemma \ref{REMOVINGCPTSUPPfactpfLEM1}
to the integrals over $\xi,\vs_{n+1},\vecv_{n+1}$,
and then mimicking the proof of \eqref{SUMINFlem2pf1}.
It follows from \eqref{KTNEXPLDEF001} that the sum \eqref{KTEXPLICIT} is uniformly operator convergent
on $\L^1(X)$. 

The semigroup property established in the following proposition implies that $\hTheta$ is Markovian.

\begin{prop} \label{SEMIGROUPLEM}
The family $(K_t)_{t\geq 0}$ forms a linear semigroup on 
$\Lloc(X)$, and a (strongly continuous) linear contraction semigroup on $\L^1(X)$.
\end{prop}

\begin{proof}
(This is almost identical to the proof of \cite[Proposition 6.3]{partII}.) For $f\in \Lloc(X)$, $0\leq s\leq t$, $0\leq m\leq n$, and \eqref{convent007},
\begin{equation} \label{KNMMMFREL}
\begin{split}
& K_{t-s}^{(n-m)}K_{s}^{(m)} f (\vecq,\vecv,\xi,\vs,\vecvp)  \\
& =
\int_{\Box}  f\bigg(\vecq-\bigg(\sum_{j=1}^n \xi_j\vecv_{j-1} + (t-\xi_1-\ldots-\xi_n) \vecv_n\bigg),\vecv_0,\xi_1,\vs_1,\vecv_1\bigg)
\\ 
& \times \prod_{j=2}^{n+1} p_\bn(\vecv_{j-2},{\vs}_{j-1},\vecv_{j-1};\xi_j,{\vs}_{j},\vecv_{j})
\, \prod_{j=1}^n\bigl( d\vecv_{j-1}\,d\xi_j\,d\mm({\vs}_j)\bigr) ,
\end{split}
\end{equation}
with the range of integration $\Box$ restricted to 
\begin{equation}\label{KNMMMFRELrestr}
\begin{cases}
\displaystyle
\sum_{j=1}^n \xi_j \leq t \text{ and }  \sum_{j=1}^m \xi_j\leq s <\sum_{j=1}^{m+1}\xi_j & (m<n) \\[10pt]
\displaystyle
\sum_{j=1}^n \xi_j \leq s & (m=n).
\end{cases}
\end{equation}
Therefore
\begin{equation*}
	\sum_{m=0}^n K_{t-s}^{(n-m)} K_s^{(m)} = K_t^{(n)}
\end{equation*}
and thus
\begin{equation*}
	K_{t-s}K_s = \sum_{m,n=0}^\infty K_{t-s}^{(m)} K_s^{(n)}
	= \sum_{n=0}^\infty \sum_{m=0}^n K_{t-s}^{(n-m)} K_s^{(m)} =K_t. 
\end{equation*}
This proves the semigroup property. Strong continuity follows from a standard argument, see \cite[Proposition 6.3]{partII}. The contraction property is already established \eqref{contract001}.
\end{proof}

Set $f_t=K_t f$. 
For $f$ sufficiently nice (see below for details), $t\geq0$ and $h$ small,
we have in view of the semigroup property and the expansion \eqref{KTEXPLICIT},
\begin{multline*}%
f_{t+h}(\vecq,\vecv,\xi,\vs,\vecv_+) 
= f_t(\vecq-h\vecv, \vecv,\xi+h,\vs,\vecv_+) \\
+ \int_{0<\xi_1<h}  
f_t\big(\vecq-\big( \xi_1\vecv_0 + (h-\xi_1) \vecv\big),\vecv_0,\xi_1,\vs_1,\vecv\big)
\, p_\bn(\vecv_0,{\vs}_1,\vecv;\xi+h-\xi_1,{\vs},\vecvp)
\\
\times d\vecv_0\, d\xi_1\,d\mm({\vs}_1) 
+ O(h^2).
\end{multline*}
If we divide this expression by $h$ and formally take the limit $h\to 0$, we recover the transport equation \eqref{KINETICEQ00}. To make this rigorous, we need to assume suitable differentiability assumptions for $f$. To this end, we define the following spaces of continuous and continuously differentiable functions.

For functions $X\to\RR$ we define the norm
\begin{equation}\label{sigmasup}
\|f\|_\sigma:=\esssup_{(\vecq,\vecv,\xi,\vs,\vecvp)\in X} \frac{|f(\vecq,\vecv,\xi,\vs,\vecvp)|}{\sigma(\vecv,\vecvp)} ,
\end{equation}
and let $\L_\sigma^\infty(X)$ be the space of $f$ with $\|f\|_\sigma<\infty$. We denote by $\C_\sigma(X)\subset \L_\sigma^\infty(X)$ the subspace of continuous functions, and furthermore set
\begin{equation*}
\C_\sigma^1(X):=\bigl\{f\in\C_\sigma(X)\col
\partial_{q_1}f,\ldots,\partial_{q_d}f,
\partial_\xi f\in\C_\sigma(X)
\bigr\} .
\end{equation*}

Similarly, we consider function spaces with an additional time-dependence, where for any given $T>0$, $X$ in the above definitions is replaced by $[0,T]\times X$. In particular, we set
\begin{equation} \label{Csigma1TXdef}
\C_\sigma^1([0,T]\times X)\\
:=\bigl\{f\in\C_\sigma([0,T]\times X)\col
\partial_t f,\partial_{q_1}f,\ldots,\partial_{q_d}f,
\partial_\xi f\in\C_\sigma([0,T]\times X)
\bigr\}.
\end{equation}

In the following we will assume that the collision kernel $p_\bn$ is a continuous function in all variables.
This allows us to solve the Cauchy problem of the forward Kolmogorov equation for the Markov process $\hTheta$. Examples of case where $p_\bn$ is continuous include Poisson scatterer configurations and Euclidean lattices in dimension $d\geq 3$ 
\cite[Remark 4.1]{partII}\footnote{For Euclidean lattices in dimension $d=2$ one has a completely
explicit formula for $p_{\bn}$; cf.\ \cite{partIII},
and $p_{\bn}$ is continuous except at points with 
$\vecbeta_{\vecv_0R(\vecv)}^+(\vece_1)_\perp=\vecbeta_{\vece_1}^-(\vecvp R(\vecv))_\perp$
(there is a misprint in the statement of this condition in \cite[Remark 4.1]{partII},
however it appears in the correct form in \cite[Lemma 6.5(iii)]{partII}).
Using this precise control on the set of discontinuities of $p_{\bn}$
one can show that Theorem \ref{FPK} holds also for the case of Euclidean lattices in dimension $d=2$;
cf.\ \cite[Theorem 6.4]{partII}.}.

Set
\begin{equation*}
Y = \{ (\vecv_0,\vs',\vecv;\xi,\vs,\vecvp) \in \US\times\Sigma\times\US\times\RR_{>0}\times\Sigma\times\US : \vecv\in\scrV_{\vecv_0}\; \vecvp\in\scrV_{\vecv} \},
\end{equation*}
\begin{equation}\label{sigmasup22}
\|\varphi\|_\sigma:=\esssup_{(\vecv_0,\vs',\vecv;\xi,\vs,\vecvp)\in Y} \frac{|\varphi(\vecv_0,\vs',\vecv;\xi,\vs,\vecvp)|}{\sigma(\vecv,\vecvp)} .
\end{equation}
We let $\L_\sigma^\infty(Y)$ be the space of $\varphi$ with $\|\varphi\|_\sigma<\infty$, and $\C_\sigma(Y)\subset \L_\sigma^\infty(Y)$ the subspace of continuous functions. 

\begin{thm}\label{FPK}
For $T>0$, $f_0\in\C_\sigma^1(X)$ and $p_\bn\in\C_\sigma(Y)$, the function 
$f(t,\vecq,\vecv,\xi,\vs,\vecvp):=K_t f_0(\vecq,\vecv,\xi,\vs,\vecvp)$ is the unique solution 
in $\C_\sigma^1([0,T]\times X)$ of the integro-differential equation  \eqref{KINETICEQ00},
\begin{multline}\label{KINETICEQ001}
\bigl(\partial_t+\vecv\cdot\nabla_\vecq-\partial_\xi\bigr)f(t,\vecq,\vecv,\xi,\vs,\vecvp) \\
= \int_{\Sigma\times\US} f(t,\vecq,\vecv_0,0,\vs',\vecv) \, p_\bn(\vecv_0,\vs',\vecv;\xi,\vs,\vecv_+)\,d\mm(\vs')\,d\vecv_0 
\end{multline}
with $f(0,\vecq,\vecv,\xi,\vs,\vecvp)=f_0(\vecq,\vecv,\xi,\vs,\vecvp)$.
\end{thm}

\begin{proof}
The proof is virtually identical to that of \cite[Theorem 6.4]{partII}. We will therefore only sketch the main steps. Key are the following two lemmas.

\begin{lem}\label{Lim1}
For every $f_0\in\C_\sigma(X)$, the function $f(t,\vecq,\vecv,\xi,\vs,\vecvp):=$\linebreak$K_t f_0(\vecq,\vecv,\xi,\vs,\vecvp)$ belongs to $\C_\sigma([0,T]\times X)$ for all $T>0$.
\end{lem}

\begin{proof}
See \cite[Lemma 6.6]{partII}. We have by \eqref{KTNEXPLDEF},
\begin{multline} \label{KTNEXPLDEF00154}
 | K_t^{(n)} f_0 |  \leq \| f_0 \|_\sigma
\int_{\xi_1+\ldots+\xi_n\leq t}  \sigma(\vecv_0,\vecv_1) \\ 
\times \prod_{j=2}^{n+1} p_\bn(\vecv_{j-2},{\vs}_{j-1},\vecv_{j-1};\xi_j,{\vs}_{j},\vecv_{j})
\, \prod_{j=1}^n\bigl( d\vecv_{j-1}\,d\xi_j\,d\mm({\vs}_j)\bigr) ,
\end{multline}
subject to \eqref{convent007}. The same proof as for \eqref{SUMINFlem2pf1} then yields
\begin{equation}\label{almost11}
\| K_t^{(n)} f_0 \|_\sigma \leq \frac{(c_\scrP v_{d-1} t)^n}{n!} \| f_0 \|_\sigma,
\end{equation}
and hence
\begin{equation*}%
\| f(t,\,\cdot\,)\|_\sigma = \| K_t f_0\|_\sigma \leq \e^{c_\scrP v_{d-1} t} \| f_0 \|_\sigma,
\end{equation*}
which shows that $f$ is bounded. 

It now remains to establish continuity. In view of \eqref{almost11},  it suffices to prove continuity for each function $f^{(n)}(t,\vecq,\vecv,\xi,\vs,\vecvp):=K_t^{(n)} f_0(\vecq,\vecv,\xi,\vs,\vecvp)$, $n\geq 0$, which in turn follows from \eqref{KTNEXPLDEF}, using the assumed continuity of $f_0$ and $p_\bn$. 
\end{proof}

\begin{lem}\label{Lim2}
For $f_0\in\C_\sigma^1(X)$ and
$f(t,\vecq,\vecv,\xi,\vs,\vecvp):=K_t f_0(\vecq,\vecv,\xi,\vs,\vecvp)$, the derivatives $\partial_t f,\partial_{q_1}f,\ldots,\partial_{q_d}f,
\partial_\xi f$ exist and belong to $\C_\sigma([0,T]\times X)$ for all $T>0$, and $f$ is a solution of the transport equation \eqref{KINETICEQ001}.
\end{lem}

\begin{proof}
The proof follows the same strategy as Lemma \ref{Lim1}. See \cite[Lemmas 6.7, 6.8]{partII} for details.
\end{proof}

The remaining step in the proof of Theorem \ref{FPK} is thus the uniqueness of the solution, which follows again from a standard argument, cf.~\cite[Lemmas 6.9, 6.10]{partII}.
\end{proof}

The analysis of the above Cauchy problem can be extended in principle to cases when $p_\bn$ is not everywhere continuous. We will not pursue this here, but instead demonstrate that, given initial data $f_0$, and two collision kernels $p_\bn$ and $\widetilde p_\bn$ that are close in $\L^1$, the resulting time-evolved densities $K_t f_0$ and $\widetilde K_t f_0$ remain close in $\L^1$ for all $t\in[0,T]$ ($T$ fixed). This means in particular that the solutions of \eqref{KINETICEQ001} provide arbitrarily good approximations of processes with general collision kernels.

To make this precise, let $\L_\sigma^1(Y)$ the space of $\varphi$ that are integrable with respect to the measure $$\sigma(\vecv_0,\vecv)\, d\vecv_0\,d\mm(\vs')\,d\vecv\,d\xi\,d\mm(\vs)\,d\vecvp,$$  
and denote by $\|\varphi\|_{\L_\sigma^1(Y)}$ the corresponding norm.
Recall that, by Lemma \ref{REMOVINGCPTSUPPfactpfLEM1},
every collision kernel satisfies
\begin{align}\label{COLLKERNELPROBDENSITY}
\int_0^\infty\int_\Sigma\int_{\scrV_\vecv}
p_{\bn}(\vecv_0,\vs',\vecv;\xi,\vs,\vecv_+)\,d\vecv_+\,d\mm(\vs_+)\,d\xi
=1
\end{align}
for all $\vecv_0\in\US$, $\vs'\in\Sigma$ and $\vecv\in\scrV_{\vecv_0}$.
This implies that $p_{\bn}\in\L_\sigma^1(Y)$,
with $\|p_{\bn}\|_{\L_\sigma^1(Y)}=v_{d-1}\,\omega(\US)$.
Of course also $p_{\bn}\in\L_\sigma^\infty(Y)$,
with $\|p_{\bn}\|_\sigma\leq c_{\scrP}v_{d-1}$.
Finally we define
\begin{equation*}
\overline f(\vecv,\xi,\vs,\vecvp)= \int_{\RR^d} |f(\vecq,\vecv,\xi,\vs,\vecvp)| \, d\vecq,
\end{equation*}
which we view as a function on $X$ which is independent of $\vecq$.

\begin{prop}\label{L1closenessprop}
Let $p_\bn$ and $\tp_\bn$ be two nonnegative functions in $\L_\sigma^\infty(Y)$
both satisfying the relation \eqref{COLLKERNELPROBDENSITY}
for all $\vecv_0\in\US$, $\vs'\in\Sigma$ and $\vecv\in\scrV_{\vecv_0}$.
Then for any $f\in \L^1(X)$ and $t>0$, %
\begin{equation}\label{abs321}
\|K_t f - \widetilde K_t f \|_{\L^1(X)} \leq 2\| \overline f\|_\sigma\,  \| p_\bn-\widetilde p_\bn\|_{\L_\sigma^1(Y)} \, 
t\, \exp\big( v_{d-1} (\|p_\bn\|_\sigma+\|\widetilde p_\bn\|_\sigma)\, t\big).
\end{equation}
\end{prop}
\begin{proof}
Let us modify the definition of $K_t^{(n)}$ in \eqref{KTNEXPLDEF}, replacing the fixed collision kernel $p_\bn$ by a sequence of general functions $\varphi_1,\varphi_2,\ldots\in\L_\sigma^1(Y)\cap\L_\sigma^\infty(Y)$. That is, we set
\begin{equation} \label{KTNEXPLDEFtilde}
\begin{split}
& \vecK_t^{(n)} f(\vecq,\vecv,\xi,\vs,\vecvp) \\
& =
\int_{\xi_1+\ldots+\xi_n\leq t}  f\bigg(\vecq-\bigg(\sum_{j=1}^n \xi_j\vecv_{j-1} + (t-\xi_1-\ldots-\xi_n) \vecv_n\bigg),\vecv_0,\xi_1,\vs_1,\vecv_1\bigg)\\ 
& \times \prod_{j=2}^{n+1} \varphi_{j-1}(\vecv_{j-2},{\vs}_{j-1},\vecv_{j-1};\xi_j,{\vs}_{j},\vecv_{j})
\, \prod_{j=1}^n\bigl( d\vecv_{j-1}\,d\xi_j\,d\mm({\vs}_j)\bigr) ,
\end{split}
\end{equation}
subject to \eqref{convent007}.
In analogy with the proof of \eqref{SUMINFlem2pf1} (this time also using\linebreak 
$\sigma(\vecv_{j-1},\vecv_j)=\sigma(\vecv_j,\vecv_{j-1})$ to integrate out the terms with index $\leq n-2$), we then have for any $f\in \L^1(X)$, $n\geq 1$,
\begin{equation} \label{KTNEXPLDEF001tildeA}
\| \vecK_t^{(n)} f \|_{\L^1(X)} \leq  
\frac{v_{d-1}^{n-1}t^n}{n!}\|\of\|_\sigma\|\varphi_n\|_{\L_\sigma^1(Y)}\prod_{i=1}^{n-1}\|\varphi_i\|_\sigma.
\end{equation}
Alternatively, if $\varphi_n=p_{\bn}$ (or $\tp_{\bn}$),
then we may first apply the relation
\eqref{COLLKERNELPROBDENSITY}
to integrate out the variables $\xi,\vs_{n+1},\vecv_{n+1}$.
Bounding the remaining factors appropriately we conclude that,
for any $1\leq j<n$,
\begin{equation} \label{KTNEXPLDEF001tilde}
\| \vecK_t^{(n)} f \|_{\L^1(X)} \leq  
\frac{v_{d-1}^{n-2}t^{n-1}}{(n-1)!}\|\of\|_{\sigma}\|\varphi_j\|_{\L_\sigma^1(Y)}
\prod_{\substack{i=1\\i\neq j}}^{n-1}\|\varphi_i\|_{\sigma}.
\end{equation}
We use the formal relation
\begin{equation*}
A_1\cdots A_n-B_1\cdots B_n = \sum_{j=1}^ n A_1\cdots A_{j-1} (A_j-B_j) B_{j+1}\cdots B_n
\end{equation*}
to expand $K_t^{(n)}  f - \widetilde K_t^{(n)} f$ into a sum of $n$ terms, each of the form \eqref{KTNEXPLDEFtilde} with $\varphi_1=\ldots=\varphi_{j-1}=p_\bn$, $\varphi_j=p_\bn-\widetilde p_\bn$, $\varphi_{j+1}=\ldots=\varphi_n=\widetilde p_\bn$. 
Using the bounds \eqref{KTNEXPLDEF001tildeA} and \eqref{KTNEXPLDEF001tilde} we then obtain
\begin{equation}\label{onlyK}
\|K_t^{(1)}  f - \widetilde K_t^{(1)}  f \|_{\L^1(X)} \leq \| \overline f \|_\sigma
\| p_\bn-\widetilde p_\bn\|_{\L_\sigma^1(Y)} \, t
\end{equation}
and for $n\geq2$:
\begin{align}\notag
\|K_t^{(n)}  & f - \widetilde K_t^{(n)}  f \|_{\L^1(X)} \\  \notag
& \leq \|\of\|_{\sigma}\|p_\bn-\tp_\bn\|_{\L_\sigma^1(Y)}\biggl(
\frac{v_{d-1}^{n-2}t^{n-1}}{(n-2)!}
\bigl(\|p_\bn\|_{\sigma}+\|\tp_\bn\|_\sigma\bigr)^{n-2}
+\frac{v_{d-1}^{n-1}t^{n}}{n!}
\|p_{\bn}\|_\sigma^{n-1}\biggr)
 \\ \label{onlyKA}
& \leq \|\of\|_{\sigma}\|p_\bn-\tp_\bn\|_{\L_\sigma^1(Y)}\sum_{k=n-1}^n
\frac{v_{d-1}^{k-1}t^k}{(k-1)!}
\bigl(\|p_\bn\|_{\sigma}+\|\tp_\bn\|_\sigma\bigr)^{k-1}.
\end{align}
Since $K_t^{(0)}  f - \widetilde K_t^{(0)}  f = 0$, the bound \eqref{abs321} follows from summing \eqref{onlyK}
and \eqref{onlyKA} over $n\geq 2$.
\end{proof}

\chapter{Examples, extensions, and open questions}
\label{EXAMPLESsec}

\section{The Poisson case}
\label{PoissonSEC}

Fix a constant $c>0$.
In the present section we will prove that
all the assumptions in Section \ref{ASSUMPTLISTsec}
are (almost surely) satisfied in the case when $\scrP$ is a fixed realization of a Poisson process 
in $\R^d$ with constant intensity $c$.
In fact we will prove that the key limit statement, [P2], %
holds with the limit measure being \textit{independent} of $\vecq\in\scrP$.
Hence in the present section, 
the space of marks
$\Sigma$ can be taken to be a singleton set, 
and we may %
remove it entirely from our notation,
writing $\scrX=\R^d$ and $\tP=\scrP$.
However, we will still write ``$\mu_\vs$'' for the unique limit measure
appearing in [P2], so as to avoid a clash of notation with the macroscopic limit measure
$\mu$ defined in Section \ref{GENLIMITsec}.
(In the end it turns out that $\mu_\vs=\mu$.)

\begin{prop}\label{POISSONprop}
Fix constants $c>0$ and $0<\alpha<1$.
Let $\psi\in P(N_s(\R^d))$ be the distribution of a Poisson process in $\R^d$ ($d\geq2$) with constant intensity $c$.
For $\psi$-almost every $\scrP\in N_s(\R^d)$,
{\blu all the assumptions in Section \ref{ASSUMPTLISTsec} are satisfied,
with $c_{\scrP}=c$,
the unique limit distribution $\mu_{\vs}$ being equal to $\psi$,
and with an admissible choice of $\scrE$ in [P2] being
$\scrE=\{\vecq\in\scrP\col d_\scrP(\vecq)\leq\|\vecq\|^{-\alpha/(d-1)}\}$.}
\end{prop}
The proof of the proposition builds on ideas from Boldrighini, Bunimovich and Sinai, \cite{Boldrighini83}.
It is well-known that $\psi$-almost every $\scrP\in N_s(\R^d)$
has constant asymptotic density $c$,
i.e.\ [P1] %
holds with $c_\scrP=c$.
Also the properties [Q1]--[Q3] %
are well-known to hold for $\mu_\vs=\psi$.
Hence our task is to prove 
{\blu [P2] and [P3].
The following lemma shows that the set $\scrE$ 
defined in Proposition \ref{POISSONprop}
has asymptotic density zero.}
\begin{lem}\label{POISSONlem0}
Let $c,\alpha,\psi,\scrP,\scrE$ be as in Proposition \ref{POISSONprop}.
Then for $\psi$-almost all $\scrP\in N_s(\R^d)$, 
{\blu the set $\scrE$ has asymptotic density zero,}
i.e.\ $T^{-d}\#(\scrE\cap\scrB_T^d)\to0$ as $T\to\infty$.
\end{lem}
\begin{proof}
(Cf.\ \cite[Prop.\ 3.4]{Boldrighini83}.)
Set $r(\vecy)=\|\vecy\|^{-\alpha/(d-1)}$.
By basic properties of the Poisson process we have, for any $T\geq 1$,
\begin{align*}
\EE_\psi &\#(\scrE\cap\scrB_T^d)
=c\int_{\scrB_T^d}
\psi\big(\big\{Y\in N_s(\R^d)\col Y\cap\big(\vecy+\overline{\scrB^d_{r(\vecy)}}\big)\neq\emptyset\big\}\big)\,d\vecy
\\
&=c\int_{\scrB_T^d}\Bigl(1-e^{-c\vol(\scrB_{r(\vecy)}^d)}\Bigr)\,d\vecy
\leq c^2\int_{\scrB_T^d}\vol\bigl(\scrB_{r(\vecy)}^d\bigr)\,d\vecy
\ll T^{d(d-1-\alpha)/(d-1)}.
\end{align*}
Hence for any fixed $\beta>0$,   
we have by Markov's inequality
\begin{align*}
\psi\bigl(\bigl\{\scrP\in N_s(\R^d)\col\#(\scrE\cap\scrB_T^d)>T^{\beta+d(d-1-\alpha)/(d-1)}\bigr\}\bigr)
\ll T^{-\beta}
\end{align*}
as $T\to\infty$.
Applying this for $T=2^n$, $n=1,2,\ldots$, and using $\sum_{n=1}^\infty 2^{-n\beta}<\infty$,
it follows by the Borel-Cantelli Lemma that,
for $\psi$-almost all $\scrP$,
$\#(\scrE\cap\scrB_{2^n}^d)\leq 2^{n(\beta+d(d-1-\alpha)/(d-1))}$ for all sufficiently large $n$.
Taking here $\beta\in(0,\frac{d\alpha}{d-1})$, so that $\beta+d(d-1-\alpha)/(d-1)<d$,
and using the fact that $\#(\scrE\cap\scrB_T^d)$ is an increasing function of $T$,
the lemma follows.
\end{proof}

\begin{remark}\label{POISSONENOTEMPTYrem}
Of course there is some flexibility regarding the choice of $\scrE$ in Proposition \ref{POISSONprop};
however let us note that $\scrE$ certainly cannot be taken to be empty.
Indeed, it is easily seen that, $\psi$-almost surely,
for all small $\rho$ there exist points $\vecq\in\scrP\cap\scrB_{\rho^{1-d}}^d$ for which
$d_\scrP(\vecq)<\rho$ (and in fact even $d_\scrP(\vecq)\ll\rho^{d-1}$).
As seen in Lemma \ref{GOODDISTANCElem},
for $\rho$ sufficiently small those 
points necessarily have to be in $\scrE$ in order for the uniform convergence in [P2] to hold.
\end{remark}

The remainder of this section will be spent on the proof of the key convergence property [P2];
the property [P3] will also follow as a consequence of the proof.
The method of proof is basically the same as in \cite[Props.\ 2.3-2.5]{Boldrighini83}.

First we discretize the choices of test sets $A\subset N_s(\R^d)$,
measures $\lambda\in P(\US)$ and center points $\vecq$.

Let $\scrF$ be the family of all boxes $B\subset\R^d$ of the form
$B=\prod_{j=1}^d[\alpha_j,\beta_j)$ where $\alpha_j,\beta_j\in\Q$ and $\alpha_j<\beta_j$ for all $j$.
Let $\tF$ be the set of finite unions of boxes $B\in\scrF$,
and let $\scrA$ be the family of all sets $A\subset N_s(\R^d)$ of the form
$A=\{Y\in N_s(\R^d)\col \#(Y\cap B)\geq r\}$
for $B\in\tF$ and $r\in\Z^+$.
Note that $\scrF$ is countable, and hence so are $\tF$ and $\scrA$.

Let $\scrS$ be a countable family of subsets $S\subset\US$,
chosen so that each $S\in\scrS$ 
has diameter $<\pi/2$ with respect to the metric $\varphi$ on $\US$,
each $S\in\scrS$ is a diffeomorphic image of the closed unit cube $[0,1]^{d-1}$,
and furthermore so that for each $\ve>0$ there is a finite subfamily $F\subset\scrS$ such that the sets
$S\in F$ form a partition of $\US$ (up to sets of measure zero) and each $S\in F$ has diameter 
$<\ve$%
.\footnote{To see that this is possible, for each $k\in\Z^+$ consider the decomposition
of each $(d-1)$-face of the cube $[-1,1]^d$ into $k^{d-1}$ congruent $(d-1)$-cubes of side $2/k$;
by radial projection (with origin as center) this yields a decomposition of $\US$ into
$2d\cdot k^{d-1}$ closed subsets, each of which is a diffeomorphic image of $[0,1]^{d-1}$.
We can take $\scrS$ to be the family of subsets of $\US$ obtained when the previous construction is carried out for
all $k\in\Z^+$.}
For $S\in\scrS$ we set $\lambda_S:=\omega(S)^{-1}\omega_{|S}\in P(\US)$.
Let $\scrL$ be the set of all these probability measures $\lambda_S$. %

Let us now fix a constant $t>0$ and set, for each $0<\rho<1$,
\begin{align}
G[\rho]:=\rho^{1+{t}}\Z^d\cap\scrB_{\rho^{1-d-\alpha}}^d.
\end{align}
We also set $\rho_n=n^{-{t}}$ for $n\in\Z_{\geq2}$.
Finally, we fix a constant $\gamma$ subject to $\alpha<\gamma<1$.
Given a fixed $\scrP\in N_s(\R^d)$, we
define %
\begin{align}\label{tXIdef}
&\tQ_\rho(\vecq,\vecv)=\bigl((\scrP\setminus\scrB^d(\vecq,\rho^{\gamma}))-\vecq\bigr)R(\vecv)D_\rho
\qquad (\vecq\in\R^d,\: \vecv\in\US),
\end{align}
and for each $\lambda\in P(\US)$,
we let $\tmu_{\vecq,\rho}^{(\lambda)}\in P(N_s(\R^d))$ be the distribution of $\tQ_\rho(\vecq,\vecv)$
for $\vecv$ random in $(\US,\lambda)$.
\begin{lem}\label{POISSONlem4}
Let $\scrP\in N_s(\R^d)$ be given.
Assume that for each fixed $A\in\scrA$ and $\lambda\in\scrL$,
\begin{align}\label{POISSONlem4ass}
\tmu_{\vecq,\rho_n}^{(\lambda)}(A)-\psi(A)\to0\:\text{ as }\:n\to\infty,
\:\text{ uniformly over all }\:
\vecq\in{G}[\rho_n].
\end{align}
Then for every $\lambda\in\Pac(\US)$,
we have $\mu_{\vecq,\rho}^{(\lambda)}\xrightarrow[]{\textup{ w }}\psi$
as $\rho\to0$, uniformly over all 
$\vecq\in\scrB_{\rho^{1-d-\alpha}}^d$ subject to $d_\scrP(\vecq)>2\rho^{\gamma}$.
In particular [P2] holds,
i.e.\ for every $T\geq1$ we have $\mu_{\vecq,\rho}^{(\lambda)}\xrightarrow[]{\textup{ w }}\psi$
as $\rho\to0$, uniformly over all $\vecq\in\scrP_T(\rho)$.
\end{lem}
(In the last statement, $\scrP_T(\rho)=\scrP\cap\scrB^d_{T\rho^{1-d}}\setminus\scrE$, %
with $\scrE$ defined as in Proposition \ref{POISSONprop}.)

\begin{proof}
Note that the second statement of the lemma is a trivial consequence of the first,  %
since $\vecq\in\scrP_T(\rho)$ for $\rho$ sufficiently small implies
$\|\vecq\|<T\rho^{1-d}<\rho^{1-d-\alpha}$
and $d_\scrP(\vecq)>\|\vecq\|^{-\alpha/(d-1)}>T^{-\alpha/(d-1)}\rho^\alpha>2\rho^{\gamma}$.
In order to prove the first statement, by Lemma \ref{PPUNICCONVCONCRETElem3}
it suffices to prove that for any $\lambda\in\Pac(\US)$,
any bounded Borel set $B\subset\R^d$ with $\vol(\partial B)=0$, and any $r\in\Z^+$,
\begin{align}\label{POISSONlem4newpf1}
\mu_{\vecq,\rho}^{(\lambda)}(\{Y\in N_s(\R^d)\col \#(Y\cap B)\geq r\})-\psi(\{Y\in N_s(\R^d)\col \#(Y\cap B)\geq r\})
\to0
\end{align}
as $\rho\to0$, uniformly over all 
$\vecq\in\scrB_{\rho^{1-d-\alpha}}^d$ subject to $d_\scrP(\vecq)>2\rho^{\gamma}$.
It follows from our choice of $\scrL$ that the set of densities of finite linear combinations of measures in
$\scrL$ is dense in $\L^1(\US,\omega)$;
thus it suffices to prove \eqref{POISSONlem4newpf1} under the restriction that $\lambda\in\scrL$.

Hence we now fix some $\lambda\in\scrL$ and some $B\subset\R^d$ as above,
and $r\in\Z^+$,
and seek to prove the uniform convergence in \eqref{POISSONlem4newpf1}.
Let $\ve>0$ be given.
Note that %
$B$ is Jordan measurable; hence there exist some $\eta>0$ 
and $B',B''\in\tF$ such that
\begin{align}\label{POISSONlem4newpf8}
B'\subset B\setminus\partial_\eta B,
\qquad
B\cup\partial_\eta B\subset B'',
\qquad
\vol(B''\setminus B')<\ve/c,
\end{align}
where $\partial_\eta B$ denotes the 
$\eta$-neighborhood of the boundary of $B$, that is,
$\partial_\eta B=$\linebreak$\cup_{\vecp\in\partial B}\,\scrB^d(\vecp,\eta)$.
We set
\begin{align*}
A':=\{Y\col \#(Y\cap B')\geq r\};
\qquad
A:=\{Y\col \#(Y\cap B)\geq r\};
\hspace{60pt}
\\
A'':=\{Y\col \#(Y\cap B'')\geq r\}.
\end{align*}
Then $A'\subset A\subset A''$ and $A',A''\in\scrA$,
and by our assumption, \eqref{POISSONlem4ass},
there is some integer $N\geq2$ such that
\begin{align}\label{POISSONlem4newpf17}
\bigl|\tmu_{\vecq',\rho_n}^{(\lambda)}(A')-\psi(A')\bigr|<\ve
\quad\text{and}\quad
\bigl|\tmu_{\vecq',\rho_n}^{(\lambda)}(A'')-\psi(A'')\bigr|<\ve,
\qquad\forall n\geq N,\:\vecq'\in{G}[\rho_n].
\end{align}
Take $R>0$ so that $B''\subset\scrB_R^d$.
After possibly enlarging $N$, we may also assume that
for every $n\geq N$,
\begin{align}\label{POISSONlem4newpf3}
&\sqrt d\rho_n^{1+{t}}<\rho_n^{\gamma}
\qquad\text{and}\qquad
R\bigl((\rho_{n-1}/\rho_n)^{d-1}-1\bigr)+\sqrt d\rho_{n}^{t}<\eta.
\end{align}
Having thus fixed $N$, we claim that 
\begin{align}\label{POISSONlem4newpf7}
\bigl|\mu_{\vecq,\rho}^{(\lambda)}(A)-\psi(A)\bigr|<2\ve\qquad
\text{for all }\:\rho\in(0,\rho_N)\:\text{ and }\:
\vecq\in\scrB_{\rho^{1-d-\alpha}}^d\text{ with }d_\scrP(\vecq)>2\rho^{\gamma}.
\end{align}
Since $\ve>0$ was arbitrary, 
this will prove \eqref{POISSONlem4newpf1}, and thus complete the proof of the lemma.

Let $\rho\in(0,\rho_N)$ and $\vecq\in\scrB_{\rho^{1-d-\alpha}}^d$ be given, subject to $d_\scrP(\vecq)>2\rho^{\gamma}$.
Take $n>N$ so that $\rho_{n}\leq\rho<\rho_{n-1}$.
Let $\vecq'$ be the point in $G[\rho_n]$
lying nearest to $\vecq$
(if there are several options then just pick one).
Then $\|\vecq'-\vecq\|<\sqrt d\rho_n^{1+{t}}<\rho_n^{\gamma}$
(cf.\ \eqref{POISSONlem4newpf3}),
i.e.\ $\vecq$ lies in the ball 
$\scrB^d(\vecq',\rho_n^{\gamma})$;
also $d_\scrP(\vecq)>2\rho^{\gamma}$ implies that 
$\scrP\cap \scrB^d(\vecq',\rho_n^{\gamma})$ equals $\{\vecq\}$ or $\emptyset$,
and so $\scrP\setminus\scrB^d(\vecq',\rho_n^{\gamma})=\scrP\setminus\{\vecq\}$.
This implies that for each $\vecv\in\US$,
\begin{align}\label{POISSONlem4pf4}
\tQ_{\rho_n}(\vecq',\vecv)
=\scrQ_\rho(\vecq,\vecv)D_{\rho_n/\rho}+(\vecq-\vecq')R(\vecv)D_{\rho_n}.
\end{align}
Using this we will prove
\begin{align}\label{POISSONlem4pf11}
\bigl[\scrQ_\rho(\vecq,\vecv)\in A
\:\Rightarrow\:
\tQ_{\rho_n}(\vecq',\vecv)\in A''\bigr],\qquad
\forall \vecv\in\US.
\end{align}
Indeed, assume $\scrQ_\rho(\vecq,\vecv)\in A$,
i.e.\ $\#(\scrQ_\rho(\vecq,\vecv)\cap B)\geq r$.
Given a point $\vecp\in\scrQ_\rho(\vecq,\vecv)\cap B$,
we set
\begin{align}\label{POISSONlem4pf5}
\vecp':=\vecp D_{\rho_n/\rho}+(\vecq-\vecq')R(\vecv)D_{\rho_n}.
\end{align}
Note that $\|(\vecq-\vecq')R(\vecv)D_{\rho_n}\|<\sqrt d\rho_n^{{t}}$,
since $\|\vecq-\vecq'\|<\sqrt d\rho_n^{1+{t}}$;
furthermore $\|\vecp\|<R$
since $B\subset B''\subset\scrB_R^d$.
Hence
\begin{align}\label{POISSONlem4pf13}
\|\vecp-\vecp'\|
<\|\vecp D_{\rho_n/\rho}-\vecp\|+\sqrt d\rho_n^{t}
<R((\rho/\rho_n)^{d-1}-1)+\sqrt d\rho_n^{t}<\eta,
\end{align}
where we used $\rho<\rho_{n-1}$ and \eqref{POISSONlem4newpf3}.
If $\vecp'\notin B''$ then 
using $B\cup\partial_\eta B\subset B''$ (cf.\ \eqref{POISSONlem4newpf8})
it follows that $\vecp'$ has distance $\geq\eta$ from $\partial B$,
and also $\vecp'\notin B$, so that the line segment between $\vecp$ and $\vecp'$ intersects $\partial B$;
these together yield a contradiction against \eqref{POISSONlem4pf13}.
Hence $\vecp'\in B''$ must hold.
Also $\vecp'\in\tQ_{\rho_n}(\vecq',\vecv)$, by \eqref{POISSONlem4pf4} and \eqref{POISSONlem4pf5}.
Hence each $\vecp\in\scrQ_\rho(\vecq,\vecv)\cap B$
gives rise to a corresponding point
$\vecp'\in\tQ_{\rho_n}(\vecq',\vecv)\cap B''$;
therefore $\#(\tQ_{\rho_n}(\vecq',\vecv)\cap B'')\geq \#(\scrQ_\rho(\vecq,\vecv)\cap B)\geq r$, 
so that $\tQ_{\rho_n}(\vecq',\vecv)\in A''$,
and we have proved \eqref{POISSONlem4pf11}.
By a %
similar argument we also have
\begin{align}\label{POISSONlem4pf14}
\bigl[
\tQ_{\rho_n}(\vecq',\vecv)\in A'
\:\Rightarrow\:
\scrQ_\rho(\vecq,\vecv)\in A
\bigr]\qquad
(\forall \vecv\in\US).
\end{align}
Together, \eqref{POISSONlem4pf11} and \eqref{POISSONlem4pf14} imply
\begin{align}\label{POISSONlem4pf15}
\tmu_{\vecq',\rho_n}^{(\lambda)}(A')\leq\mu_{\vecq,\rho}^{(\lambda)}(A)\leq\tmu_{\vecq',\rho_n}^{(\lambda)}(A'').
\end{align}
Note also that
\begin{align}\label{POISSONlem4pf16}
\psi(A''\setminus A')\leq\psi(\{Y\col Y\cap B''\setminus B'\neq\emptyset\})\leq c\vol(B''\setminus B')<\ve
\end{align}
(cf.\ \eqref{POISSONlem4newpf8}).
Using \eqref{POISSONlem4newpf17}, \eqref{POISSONlem4pf15}, \eqref{POISSONlem4pf16},
and $A'\subset A\subset A''$,
we obtain $\bigl|\mu_{\vecq,\rho}^{(\lambda)}(A)-\psi(A)\bigr|<2\ve$,
and we have thus proved \eqref{POISSONlem4newpf7} and with it the lemma.
\end{proof}

\begin{lem}\label{POISSONlem5}
Under the assumption of Lemma \ref{POISSONlem4},
for each $\Lambda\in\Pac(\T^1(\R^d))$
we have $\mu_\rho^{(\Lambda)}\xrightarrow[]{\textup{ w }}\psi$ as $\rho\to0$.
\end{lem}
\begin{proof}
Fix $f\in\C_b(N_s(\R^d))$;
then our task is to prove $\mu_\rho^{(\Lambda)}(f)\to\psi(f)$, %
i.e.\
\begin{align}\label{POISSONlem5pf1}
\rho^{d(d-1)}\int_{\R^d}\int_{\US}f(\scrQ_\rho(\vecq,\vecv))\,\Lambda'(\rho^{d-1}\vecq,\vecv)\,d\vecv\,d\vecq\to\psi(f),
\qquad\text{as }\:\rho\to0,
\end{align}
where $\Lambda'\in\L^1(\T^1(\R^d))$ is the density of $\Lambda$.
Without loss of generality we may assume $\Lambda'\in\C_c(\T^1(\R^d))$.
Take $R>0$ so that $\supp\Lambda'\subset\scrB_R^d\times\US$.

It follows from Lemma \ref{POISSONlem4} and Lemma \ref{UNIFPORTMlemB} that, for each fixed $\vecx\in\R^d$,
\begin{align}\label{POISSONlem5pf2}
\int_{\US}f(\scrQ_\rho(\vecq,\vecv))\,\Lambda'(\vecx,\vecv)\,d\vecv\to\biggl(\int_{\US}\Lambda'(\vecx,\vecv)\,d\vecv\biggr)\psi(f)
\end{align}
as $\rho\to0$, uniformly over all
$\vecq\in\scrB_{\rho^{1-d-\alpha}}^d$ subject to $d_\scrP(\vecq)>2\rho^{\gamma}$.
By a standard subsequence argument, using the fact that $\Lambda'\in\C_c$,
\eqref{POISSONlem5pf2} is upgraded to also hold uniformly over all $\vecx\in\R^d$;
in particular we may take $\vecx=\rho^{d-1}\vecq$ in the statement.
Note also that $R\rho^{1-d}<\rho^{1-d-\alpha}$ for all small $\rho$,
and the total volume of all $\vecq\in\scrB_{R\rho^{1-d}}^d$
satisfying $d_\scrP(\vecq)\leq2\rho^{\gamma}$ is $\ll\rho^{-d(d-1)+d\gamma}$,
which gives a negligible contribution to the left hand side \eqref{POISSONlem5pf1} as $\rho\to0$.
Using these facts, we conclude that \eqref{POISSONlem5pf1} holds.
\end{proof}

Note that the conclusion of Lemma \ref{POISSONlem5} implies in particular that 
the condition [P3] %
holds; cf.\ Remark \ref{ASS:bdfreepathrem}.
Hence in order to complete the proof of Proposition \ref{POISSONprop} it now suffices to prove that
$\psi$-almost every $\scrP\in N_s(\R^d)$ satisfies the assumption 
\eqref{POISSONlem4ass} in Lemma~\ref{POISSONlem4}.
In fact, since $\scrA$ and $\scrL$ are countable, it suffices to prove that for any 
\textit{fixed} $A\in\scrA$ and $\lambda\in\scrL$,
the condition \eqref{POISSONlem4ass} holds for $\psi$-almost every $\scrP$.
Thus let $A\in\scrA$ and $\lambda\in\scrL$ be given;
take $B\in\tF$ and $r\in\Z^+$ so that $A=\{Y\in N_s(\R^d)\col \#(Y\cap B)\geq r\}$,
and take $S\in\scrS$ so that $\lambda=\lambda_S$.
The following proof is modelled on the proof of Prop.\ 2.3 in \cite{Boldrighini83}.

We fix constants $\beta_1$ and $\beta_2$ satisfying 
\begin{align}\label{BETA1BETA2restr}
0<\beta_1<\tfrac12(1-\gamma)
\qquad\text{and}\qquad
\beta_1<\beta_2<1-\gamma.
\end{align}
For each sufficiently small $\rho$,
we fix a choice of subsets $S_1,\ldots,S_k\subset S$ satisfying
$\diam(S_\ell)<\rho^{\beta_1}$
and $\omega(S_\ell)\asymp\rho^{\beta_1(d-1)}$,
which are separated so that $\varphi(S_{\ell},S_{\ell'})>\rho^{\beta_2}$
for any $\ell\neq \ell'$,
and which fill up most of $S$
in the sense that $\omega(S\setminus\cup_{\ell=1}^kS_{\ell})\ll\rho^{\beta_2-\beta_1}$.
\footnote{An explicit choice of such subsets $S_1,\ldots,S_k$ is as follows:
By assumption there is a fixed diffeomorphism $\Phi$ from an open set $U\subset\R^{d-1}$ onto an open subset of $\US$
such that $[0,1]^{d-1}\subset U$ and $S=\Phi([0,1]^{d-1})$.
Now fix $C>1$ large,
and for each small $\rho$ set $n=\lceil C\rho^{-\beta_1}\rceil$, $k=n^{d-1}$,
and let $S_1,\ldots,S_k$ be the sets
$\Phi(n^{-1}\vecm+[0,n^{-1}-C\rho^{\beta_2}]^{d-1})$ with $\vecm$ running through $\{0,1,\ldots,n-1\}^{d-1}$.
If $C$ is larger than a certain constant which only depends on $\Phi$,
then for all sufficiently small $\rho$, the sets $S_1,\ldots,S_k$ satisfy all the conditions.}
It follows that $k\asymp\rho^{\beta_1(1-d)}$. 
(Here and in the following, the implied constant in any ``$\ll$'', ``$\asymp$'' or ``$O(\cdots)$''
depends only on $d$, $S$ and $B$.)

Also for $\rho$ small we set
\begin{align*}
\tOmega^{(\rho)}=\{\scrP\in N_s(\R^d)\col\#(\scrP\cap\scrB_{\rho^{\gamma}}^dD_\rho)=0\}.
\end{align*}
This is a Borel subset of $N_s(\R^d)$ and letting $\mu=c\vol(\scrB_{\rho^{\gamma}}^d)$ we have
\begin{align}\label{POISSONpropPF1}
\psi(\tOmega^{(\rho)})=e^{-\mu}=1-O(\rho^{d\gamma}).
\end{align}
We write $\tPP^{(\rho)}=\psi(\cdot\mid\tOmega^{(\rho)})$
for the corresponding conditional probability,
i.e.\ $\tPP^{(\rho)}(A')=\psi(A'\cap\tOmega^{(\rho)})/\psi(\tOmega^{(\rho)})$ for any Borel set $A'\subset N_s(\R^d)$.
For any $\scrP\in N_s(\R^d)$ let
\begin{align}\label{RrhojAqdef}
R_{\rho,\vecq}^{\scrP}
:=\tmu_{\vecq,\rho}^{(\lambda)}(A)-\psi(A)
=\lambda\big(\big\{\vecv\in S\col\tQ_\rho(\vecq,\vecv)\in A\bigr\}\bigr)-\psi(A).
\end{align}
and  %
\begin{align}\label{RnjlAqDEF}
R_{\rho,\ell,\vecq}^{\scrP}
:=\lambda\big(\big\{\vecv\in S_\ell\col\tQ_\rho(\vecq,\vecv)\in A\bigr\}\bigr)-\lambda(S_\ell)\tPP^{(\rho)}(A).
\end{align}
Using $\lambda(S\setminus\cup_{\ell=1}^{k}S_{\ell})\ll \rho^{\beta_2-\beta_1}$
and $|\psi(A)-\tPP^{(\rho)}(A)|\ll\rho^{d\gamma}$
(cf.\ \eqref{POISSONpropPF1}) we then have
\begin{align}\label{POISSONpropPF6}
R_{\rho,\vecq}^{\scrP}
=\sum_{\ell=1}^{k}R_{\rho,\ell,\vecq}^{\scrP}
+O\bigl(\rho^{\beta_2-\beta_1}+\rho^{d\gamma}\bigr).
\end{align}
The point of using $\tPP^{(\rho)}(A)$ in \eqref{RnjlAqDEF} is that 
we have the identity
\begin{align}\label{Epqzero}
\EE_\psi R_{\rho,\ell,\vecq}^{\scrP}=0.
\end{align}
Indeed, by Fubini,
\begin{align}\label{POISSONpropPF3}
\EE_\psi R_{\rho,\ell,\vecq}^{\scrP}
=\int_{S_\ell}\Bigl(
\psi\bigl(\bigl\{\scrP\col\tQ_\rho(\vecq,\vecv)\in A\bigr\}\bigr)
- \tPP^{(\rho)}(A)\Bigr)
\,d\lambda(\vecv).
\end{align}
Here for each $\vecv$,
\begin{align}\label{POISSONpropPF2}
\psi\bigl(\tQ_\rho(\vecq,\vecv)\in A\bigr)
=\psi\bigl((\scrP\setminus\scrB_{\rho^{\gamma}}^d)R(\vecv)D_{\rho}\in A\bigr)
=\psi\bigl(\scrP R(\vecv)D_{\rho}\in A\mid\scrP \cap\scrB_{\rho^{\gamma}}^d=\emptyset\bigr)
\\\notag
=\tPP^{(\rho)}(A),
\end{align}
where we first used \eqref{tXIdef} and the fact that $\psi$ is translation invariant,
then the fact that for $\scrP$ random in $(N_s(\R^d),\psi)$,
the two random elements $\scrP \setminus\scrB_{\rho^{\gamma}}^d$
and $\scrP \cap\scrB_{\rho^{\gamma}}^d$ in $N_s(\R^d)$ are independent;
and finally, for the last equality,
we noted that 
$\scrP \cap\scrB_{\rho^{\gamma}}^d=\emptyset$
is equivalent with $\scrP R(\vecv)D_{\rho} \cap\scrB_{\rho^{\gamma}}^dD_{\rho}=\emptyset$,
and then used the fact that $\psi$ is $\SL(d,\R)$-invariant.
Hence the integrand in \eqref{POISSONpropPF3} vanishes identically,
and we have proved \eqref{Epqzero}.

Next we claim that for $\rho$ sufficiently small and for every $\vecq\in\R^d$,
if $\scrP$ is random in $(N_s(\R^d),\psi)$
then the random variables $R_{\rho,\ell,\vecq}^{\scrP}$ for $\ell=1,\ldots,k$
are mutually independent.
In view of the definition \eqref{RnjlAqDEF},
and noticing that 
$\tQ_\rho(\vecq,\vecv_\ell)\in A$ is equivalent with
$\scrP-\vecq$ having at least $r$ points in the region
\begin{align}\label{INDEPPF1}
BD_\rho^{-1} R(\vecv_\ell)^{-1}\setminus\scrB_{\rho^{\gamma}}^d,
\end{align}
it follows that it suffices 
to prove that for any $\vecv_1\in S_1,\ldots,\vecv_k\in S_k$,
the regions in \eqref{INDEPPF1} for $\ell=1,\ldots,k$ are pairwise disjoint.
Fix $R>0$ so that $B\subset\scrB_R^d$.
\begin{lem}\label{INDEPLEM1}
If $\vecv,\vecv'\in\US$ and $4R\rho^{1-\gamma}\leq\varphi(\vecv,\vecv')\leq\pi/2$
then 
\begin{align*}
BD_\rho^{-1} R(\vecv)^{-1}\cap BD_\rho^{-1} R(\vecv')^{-1}\setminus\scrB_{\rho^\gamma}^d=\emptyset.
\end{align*}
\end{lem}
\begin{proof}
Suppose $\vecx\in BD_\rho^{-1} R(\vecv)^{-1}\cap BD_\rho^{-1} R(\vecv')^{-1}$
and $\vecx\neq\bn$. %
Set $\varphi=\min(\varphi(\vecv,\vecx),\varphi(\vecv,-\vecx))$ and 
$\varphi'=\min(\varphi(\vecv',\vecx),\varphi(\vecv',-\vecx))$.
The point $\vecx$ lies in $BD_\rho^{-1} R(\vecv)^{-1}\subset\scrB_R^d D_\rho^{-1} R(\vecv)^{-1}
\subset (\R\times\scrB_{R\rho}^{d-1})R(\vecv)^{-1}$;
this implies that $\vecx$ has distance $<R\rho$ from the line $\R\vecv$,
and thus $\varphi\leq2\sin\varphi<2R\rho/\|\vecx\|$.
Similarly $\varphi'<2R\rho/\|\vecx\|$.
Now by the triangle inequality in $\US/\{\pm\}$ we have
$4R\rho^{1-\gamma}\leq\varphi(\vecv,\vecv')\leq\varphi+\varphi'<4R\rho/\|\vecx\|$,
implying $\|\vecx\|<\rho^\gamma$,
i.e.\ $\vecx\in\scrB_{\rho^\gamma}^d$.
\end{proof}
If $\rho$ is sufficiently small then the lemma applies to
any pair of points $\vecv_\ell,\vecv_{\ell'}$,
$\ell\neq\ell'$,
since then $\varphi(\vecv_\ell,\vecv_{\ell'})\geq\varphi(S_\ell,S_{\ell'})>\rho^{\beta_2}>4R\rho^{1-\gamma}$
(cf.\ \eqref{BETA1BETA2restr}),
and also $\varphi(\vecv_\ell,\vecv_{\ell'})<\pi/2$ since we have assumed that each $S\in\scrS$ has diameter $<\pi/2$.
This completes the proof that $R_{\rho,\ell,\vecq}^{\scrP}$ for $\ell=1,\ldots,k$ are indeed independent.

Set
\begin{align}\label{HRHODEF}
V_{\rho,\vecq}:=\sum_{\ell=1}^{k}\EE_\psi\bigl((R_{\rho,\ell,\vecq}^{\scrP})^2\bigr)
\quad\text{and}\quad
H_\rho:=\max\bigl\{\lambda(S_\ell)\col\ell\in\{1,\ldots,k\}\bigr\}\asymp\rho^{\beta_1(d-1)}.
\end{align}
Then $|R_{\rho,\ell,\vecq}^{\scrP}|\leq H_\rho$ everywhere.
In view of our observations, in particular \eqref{Epqzero} and the independence just proved,
we have the following inequality of Bernstein type,
for any $X>0$ and $0<h<3/H_\rho$
(cf., e.g., \cite[eq.\ (2a)]{Bennett}):
\begin{align}\label{BERNSTEINineq}
\psi\biggl(\biggl\{\scrP\in N_s(\R^d)\col
\biggl|\sum_{\ell=1}^{k}R_{\rho,\ell,\vecq}^{\scrP}\biggr|\geq X\biggr\}\biggr)
\leq 2e^{-hX}\exp\Bigl(\frac{h^2}2V_{\rho,\vecq}\Bigl(1-\frac{hH_\rho}3\Bigr)^{-1}\Bigr).
\end{align}
In order to bound $V_{\rho,\vecq}$, note that 
\begin{align}\notag
\EE_\psi\Bigl(\lambda\big(\big\{\vecv\in S_\ell\col\tQ_\rho(\vecq,\vecv)\in A\bigr\}\bigr)^2\Bigr)
=\int_{S_{\ell}\times S_{\ell}}
\psi\bigl( %
\tQ_{\rho}(\vecq,\vecv)\in A,\:\: %
\tQ_{\rho}(\vecq,\vecw)\in A\bigr)
\hspace{20pt}
\\\label{POISSONpropPF4}
\times d\lambda(\vecw)\,d\lambda(\vecv).
\end{align}
Here for any $\vecv,\vecw$ with $\varphi(\vecv,\vecw)\geq4R\rho^{1-\gamma}$,
it follows from Lemma \ref{INDEPLEM1} that the events
$\tQ_{\rho}(\vecq,\vecv)\in A$ and $\tQ_{\rho}(\vecq,\vecw)\in A$
are independent on $(N_s(\R^d),\psi)$,
and so the integrand equals $\tPP^{(\rho)}(A)^2$
(cf.\ \eqref{POISSONpropPF2}).
Furthermore for any $\vecv\in S_{\ell}$,
the set of $\vecw\in S_{\ell}$ for which $\varphi(\vecv,\vecw)<4R\rho^{1-\gamma}$
has measure $\ll\rho^{(1-\gamma)(d-1)}$ with respect to $\lambda$.
Using these facts, and \eqref{RnjlAqDEF} and \eqref{Epqzero},
we obtain
\begin{align}\label{POISSONpropPF5}
\EE_\psi\Bigl((R_{\rho,\ell,\vecq}^{\scrP})^2\Bigr)
\ll\lambda(S_{\ell})\rho^{(1-\gamma)(d-1)}
\ll\rho^{(1-\gamma+\beta_1)(d-1)}.
\end{align}
Adding this over $\ell=1,\ldots,k$ it follows that
\begin{align*}
V_{\rho,\vecq}\ll\rho^{(1-\gamma)(d-1)}.
\end{align*}
We now fix a constant $\delta$ satisfying
\begin{align*}
0<\delta<\beta_1(d-1),
\end{align*}
and apply \eqref{BERNSTEINineq} with $h=1/H_\rho$
and $X=H_\rho\rho^{-\delta}$.
Then $h^2V_{\rho,\vecq}<1$ provided that $\rho$ is sufficiently small
(cf.\ \eqref{HRHODEF} and \eqref{BETA1BETA2restr}), and we obtain
\begin{align*}
\psi\biggl(%
\biggl|\sum_{\ell=1}^{k}R_{\rho,\ell,\vecq}^{\scrP}\biggr|\geq H_\rho\rho^{-\delta}\biggr)
\ll e^{-\rho^{-\delta}}.
\end{align*}
Setting $\ve(\rho)=H_\rho\rho^{-\delta}+C_1\bigl(\rho^{\beta_2-\beta_1}+\rho^{d\gamma}\bigr)$,
where $C_1>0$ is the implied constant in the big-$O$ expression in  \eqref{POISSONpropPF6},
it follows that
\begin{align}\label{POISSONpropPF7}
\psi\bigl(%
\bigl|R_{\rho,\vecq}^{\scrP}\bigr|\geq \ve(\rho)\bigr)
\ll e^{-\rho^{-\delta}},
\end{align}
This holds for every sufficiently small $\rho$, and all $\vecq\in\R^d$.
Using also $\# G[\rho]\ll\rho^{-d(d+\alpha+{t})}$
it follows that for $\rho$ small,
\begin{align*}
&\psi\bigl(%
\exists\vecq\in{G}[\rho]\text{ s.t. }
\bigl|R_{\rho,\vecq}^{\scrP}\bigr|\geq \ve(\rho)\bigr)
\leq\sum_{\vecq\in G[\rho]}\psi\bigl(\bigl|R_{\rho,\vecq}^{\scrP}\bigr|\geq \ve(\rho)\bigr)
\ll\rho^{-d(d+\alpha+t)}e^{-\rho^{-\delta}}.
\end{align*}
This implies that the sum
\begin{align*}
\sum_{n=2}^\infty\psi\bigl(%
\exists\vecq\in{G}[\rho_n]\text{ s.t. }
\bigl|R_{\rho_n,\vecq}^{\scrP}\bigr|\geq \ve(\rho_n)\bigr)
\end{align*}
converges (recall $\rho_n=n^{-t}$),
and hence by the Borel-Cantelli Lemma,
for $\psi$-almost every $\scrP\in N_s(\R^d)$
there is some $N=N(\scrP)$ such that
$|R_{\rho_n,\vecq}^\scrP|<\ve(\rho_n)$ holds for all $n\geq N$ and all $\vecq\in{G}[\rho_n]$.
Therefore we have uniform convergence as in \eqref{POISSONlem4ass} in Lemma \ref{POISSONlem4},
and the proof of Proposition \ref{POISSONprop} is complete.
\hfill $\square$

\vspace{10pt}

Let us conclude this section by computing the collision kernels.
Recall that $\Sigma$ is a singleton,
which we remove from our notation.
Thus %
$\scrX_\perp=\R^{d-1}$, $\Omega=\UB$ and $\mu_\Omega=v_{d-1}^{-1}\vol_{\R^{d-1}}$, 
and we find that for every $\vecx'\in\scrX_\perp$,
the distribution of the random point $(w_1,\vecw'):=\iota(\vecz(\Xi_\vs-\vecx'))$ equals
$\oxi^{-1}e^{-w_1/\oxi}\,dw_1\,d\mu_\Omega(\vecw')$
with $\oxi=(v_{d-1}c)^{-1}$;
cf.\ \eqref{OXIFORMULAG}.
Hence
\begin{align*}
k(\vecx',\xi,\vecx)=\oxi^{-1}e^{-\xi/\oxi},
\qquad\text{for }\: \vecx',\vecx\in\UB,\:\xi>0.
\end{align*}
Similarly
\begin{align*}
k^{\g}(\xi,\vecx)=\oxi^{-1}e^{-\xi/\oxi}.
\end{align*}
From \eqref{pbndefG} and \eqref{pgendef} we now get
\begin{align*}
p_\bn\bigl(\vecv_0,\vecv;\xi,\vecvp\bigr)
=p\bigl(\vecv;\xi,\vecvp\bigr)
=c\,\sigma(\vecv,\vecv_+)\,e^{-\xi/\oxi}.
\end{align*}
Hence the generalized Boltzmann equation reads
\begin{align*}
\bigl(\partial_t+\vecv\cdot\nabla_\vecq-\partial_{\xi}\bigr)f(t,\vecq,\vecv,\xi,\vecv_+) 
=\int_{\US} f\bigl(t,\vecq,\vecv_0,0,\vecv\bigr)
\, p(\vecv;\xi,\vecv_+)\,d\vecv_0.
\end{align*}
As we discussed in the introduction,
upon making the ansatz 
\begin{align*}
f(t,\vecq,\vecv,\xi,\vecv_+)=f(t,\vecq,\vecv)\,\sigma(\vecv,\vecv_+)\, e^{-\xi/\oxi},
\end{align*}
this equation reduces to the standard linear Boltzmann equation, \eqref{KINETICEQ008}.

\section{Periodic point sets}
\label{periodicpointsetsSEC}

In this section we let $\scrP$ be a locally finite \textit{periodic} point set in $\R^d$
($d\geq2$).
This means that there exists a lattice $\scrL$ of full rank in $\R^d$,
such that $\scrP+\vecell=\scrP$ for all $\vecell\in\scrL$.
We fix, once and for all, $g\in\SL(d,\R)$ and $\delta>0$ such that
$\scrL=\delta^{1/d}\Z^dg$.
We can then choose a finite number of vectors $\vecb_1,\ldots,\vecb_m\in\R^d$
such that
\begin{align}\label{PERIODICpointsetexpl}
\scrP=\bigcup_{j=1}^m \delta^{1/d}(\vecb_j+\Z^d)g,
\end{align}
and $\vecb_i-\vecb_j\notin\Z^d$ for all $i\neq j$. %
Without loss of generality, 
since we may from the start replace $\scrP$ by a translate of $\scrP$,
we may also require that $\vecb_1=\bn$.

We will prove that such a set $\scrP$ satisfies the assumptions in Section \ref{ASSUMPTLISTsec},
with 
\begin{align*}
\Sigma:=\{1,\ldots,m\},
\end{align*}
with the map $\vs:\scrP\to\Sigma$ given by $\vs(\vecq)=j$ 
for all $\vecq\in\delta^{1/d}(\vecb_j+\Z^d)g$ ($j\in\Sigma$),
and with $\mm$ being the uniform probability measure on $\Sigma$.
We will start by giving an explicit description of the map
$j\mapsto\mu_j$ from $\Sigma$ to $P(N_s(\scrX))$.
This requires some preparation.

Let $B$ be the matrix in $M_{m,d}(\R)$ whose row vectors are $\vecb_1,\ldots,\vecb_m$ 
(in this order).
Let $B_1,\ldots,B_d\in\R^m$ be the \textit{column} vectors of $B$,
and %
let $\scrJ$ be the smallest closed subgroup of $\R^m$
containing $\Z^m$ and $B_1,\ldots,B_d$.
In other words, $\scrJ$ equals the closure of the integer span of $\Z^m$ and $B_1,\ldots,B_d$:
\begin{align*}
\scrJ=\overline{\Z^m+\Z B_1+\cdots+\Z B_d}.
\end{align*}
This is a closed Lie subgroup of $\R^m$.
Let $\scrJ^\circ$ be the connected subgroup of $\scrJ$ containing $\bn$;
this is a linear subspace of $\R^m$ which intersects $\Z^m$ in a lattice
(that is, there exists an $\R$-linear basis of $\scrJ^\circ$
consisting of vectors in $\scrJ^\circ\cap\Z^m$).
Furthermore,
$\scrJ\subset\Q^m+\scrJ^\circ$,
and either $\scrJ=\scrJ^\circ=\R^m$ or
$\scrJ$ is a union of a countable number of translates of $\scrJ^\circ$.
Note also that if we assume $\vecb_1=\bn$ then $\scrJ^\circ\subset\vece_1^\perp$.
\begin{remark}\label{scrJcircREM}
Equivalently, $\scrJ^\circ$ can be defined as the orthogonal complement in $\R^m$ of
the set of integer vectors $\vech\in\Z^m$ which satisfy $\vech\cdot B_j\in\Z$ for all
$j\in\{1,\ldots,d\}$. %
\end{remark}

We identify the product space
${\scrJ^{\circ}}^d=\scrJ^{\circ}\times\cdots\times\scrJ^{\circ}$
with the subspace of matrices in $M_{m,d}(\R)$
all of whose column vectors belong to $\scrJ^\circ$.
Recall that $\scrJ^\circ\cap\Z^m$ is a lattice in $\scrJ^\circ$;
hence ${\scrJ^\circ_{\Z}}^d:={\scrJ^\circ}^d\cap M_{m,d}(\Z)$ is a lattice in ${\scrJ^\circ}^d$;
we let $\TT_{{\scrJ^\circ}^d}={\scrJ^\circ}^d/{\scrJ^\circ_{\Z}}^d$ be the quotient torus,
and let $\eta_{\TT}$ be the translational invariant probability measure on $\TT_{{\scrJ^\circ}^d}$.

For each $j$ we have $B_j\in\scrJ\subset\Q^m+\scrJ^\circ$;
hence there exists a positive integer $q$ such that
$B_j\in q^{-1}\Z^m+\scrJ^\circ$ for each $j\in\{1,\ldots,d\}$,
or in other words,
\begin{align}\label{BinqinvMZpJd}
B\in q^{-1}M_{m,d}(\Z)+{\scrJ^\circ}^d.
\end{align}
We fix $q$ once and for all, and set
\begin{align*}
\Gamma(q)=\{\gamma\in\SL(d,\Z)\col \gamma\equiv I\mod q\}.
\end{align*}
Also let $F_q\subset\SL(d,\R)$ be a fixed (Borel measurable) fundamental domain
for $\Gamma(q)\bs\SL(d,\R)$,
and let $\eta$ be the (left and right) Haar measure on $\SL(d,\R)$,
normalized so that $\eta(F_q)=1$.

Recall that $\scrX=\R^d\times\Sigma$.
Now for each $\ell\in\Sigma$ we define the map
\begin{align*}
J_\ell: F_q%
\times\TT_{{\scrJ^\circ}^d}\to N_s(\scrX)
\end{align*}
by
\begin{align}\label{Jelldef}
J_\ell\bigl(A,U+{\scrJ^\circ_{\Z}}^d\bigr)=
\biggl(\bigcup_{j=1}^m \delta^{1/d}\bigl(\vecb_j-\vecb_\ell+\vecu_j-\vecu_\ell+\Z^d\bigr)A\times\{j\}\biggr)
\setminus\{(\bn,\ell)\},
\end{align}
where $\vecu_1,\ldots,\vecu_m\in\R^d$ are the \textit{row} vectors of the matrix $U$
(in order).
Of course, the right hand side of \eqref{Jelldef} is independent of the choice of 
the representative $U\in{\scrJ^\circ}^d$ for the point
$U+{\scrJ^\circ_{\Z}}^d$ in $\TT_{{\scrJ^\circ}^d}$,
since any other representative $U'$ for the same point has
$\vecu_j'\in\vecu_j+\Z^d$ for each $j$.
These maps $J_1,\ldots,J_m$ are continuous.
Finally, we define $\mu_{\ell}$ to be the pushforward by $J_\ell$ of the probability measure
$\eta\times\eta_{\TT}$ on $F_q\times\TT_{{\scrJ^\circ}^d}$:
\begin{align}\label{periodicmuwdef}
\mu_{\ell}:=(\eta\times\eta_{\TT})\circ J_\ell^{-1}\in P(N_s(\scrX)).
\end{align}
It will be clear from the proof of Proposition \ref{PERIODICSETprop}
that %
$\mu_\ell$ is independent of the choice of $q$ 
and the choice of the fundamental domain $F_q$.

\begin{prop}\label{PERIODICSETprop}
For $\scrP, \Sigma,\mm$ as above, 
all the assumptions in Section \ref{ASSUMPTLISTsec} are satisfied,
with $\scrE=\emptyset$ {\blu in [P2]},
and with the map $\ell\mapsto\mu_\ell$ from $\Sigma$ to $P(N(\scrX))$ given by \eqref{periodicmuwdef}.
\end{prop}

We give the proof of Proposition \ref{PERIODICSETprop}
in Section \ref{PERIODICsetproofsec},
by deriving it as a special case of our main result for quasicrystals of cut-and-project type,
Proposition \ref{QCprop}, but with the limit measures $\mu_\ell$ given in more explicit form.

\vspace{5pt}

\begin{remark}
It is immediate from \eqref{Jelldef} that
\begin{align*}
J_\ell\bigl(A,U+{\scrJ^\circ_{\Z}}^d\bigr)\cap (\R^d\times\{\ell\})
=(\delta^{1/d}\Z^dA\setminus\{\bn\})\times\{\ell\}.
\end{align*}
It follows that a point process $\Xi$ in $\scrX$ with distribution $\mu_\ell$
has the property that the projection of $\Xi\cap(\R^d\times\{\ell\})$ in $\R^d$
is a random lattice of covolume $\delta$ in $\R^d$
minus the origin,
distributed according to the standard invariant measure on such lattices.
\end{remark}

\begin{ex}\label{PERIODICgenericEX}
Assume w.l.o.g.\ $\vecb_1=\bn$.
It follows from Remark \ref{scrJcircREM} that $\scrJ^\circ=\vece_1^\perp$ holds if and only if
the vectors $\vece_1,\ldots,\vece_d,\vecb_2,\ldots,\vecb_m$ are linearly independent over $\Q$.
In this case, ${\scrJ^\circ}^d$ consists of all matrices in $M_{m,d}(\R)$ with
vanishing top row; we identify this space with $M_{m-1,d}(\R)$ in the obvious way,
and then get $\TT_{{\scrJ^\circ}^d}=M_{m-1,d}(\R)/M_{m-1,d}(\Z)$.
Note also that we can take $q=1$.
It follows that in this case, 
a point process $\Xi$ in $\scrX$ with distribution $\mu_\ell$ can
be constructed as follows:
Pick a random lattice $L$ of covolume $\delta$ in $\R^d$
distributed according to the $\SL(d,\R)$ invariant probability measure on such lattices;
then pick $m-1$ random points $\{\vecp_j\}_{j\in\{1,\ldots,m\}\setminus\{\ell\}}$
in the torus $\R^d/L$, independently and uniformly distributed;
and finally set:
\begin{align*}
\Xi=\bigl((L\setminus\{\bn\})\times\{\ell\}\bigr)\:\:\bigcup\:\:\bigcup_{j\in\{1,\ldots,m\}\setminus\{\ell\}} \bigl((L+\vecp_j)\times\{j\}\bigr).
\end{align*}
Note that if $p_1$ is the projection map $\scrX\to\R^d$ then
\begin{align*}
p_1(\Xi)=(L\setminus\{\bn\})\:\:\bigcup\:\:\bigcup_{j\in\{1,\ldots,m\}\setminus\{\ell\}} (L+\vecp_j),
\end{align*}
and the distribution of this point process in $\R^d$ is independent of $\ell$.
Hence in the present situation it is possible to
discard the set of marks, i.e.\
the assumptions in Section \ref{ASSUMPTLISTsec} 
can be satisfied with $\Sigma$ being a singleton set and $\scrX=\R^d$ (up to obvious identification).
\end{ex}

\begin{ex}\label{PERIODICm2ex}
Now assume $m=2$, and again $\vecb_1=\bn$.
Then $\scrJ^\circ=\vece_1^\perp$ if and only if $\vecb_2\notin\Q^d$,
and in this case the description in Ex.\ \ref{PERIODICgenericEX} applies.
On the other hand if $\vecb_2\in\Q^d$ then $\scrJ^\circ=\{\bn\}$,
thus $\TT_{{\scrJ^\circ}^d}=\{\bn\}$,
and $q$ is any positive integer such that $\vecb_2\in q^{-1}\Z^d$.
In this case the formulas for $J_1,J_2$ become
\begin{align*}
&J_1(A)=\Bigl(\delta^{1/d}(\Z^d\setminus\{\bn\}) A\times\{1\}\Bigr) \:\:\bigcup\:\:\Bigl(\delta^{1/d}(\vecb_2+\Z^d)A\times\{2\}\Bigr);
\\
&J_2(A)=\Bigl(\delta^{1/d}(\Z^d\setminus\{\bn\}) A\times\{2\}\Bigr) \:\:\bigcup\:\:\Bigl(\delta^{1/d}(-\vecb_2+\Z^d)A\times\{1\}\Bigr).
\end{align*}
Here we remark that for $A$ random in $(F_q,\eta)$,
the two random point sets $p_1(J_1(A))$ and $p_1(J_2(A))$
have the same distribution.
(The proof of this fact uses the observation that there exists an element $\gamma\in\SL(d,\Z)$ such that
$(-\vecb_2+\Z^d)\gamma=\vecb_2+\Z^d$; 
we then get $p_1(J_2(\gamma A))=p_1(J_1(A))$ for all $A\in\SL(d,\R)$;
finally note that $\gamma F_q$ is a fundamental domain for $\Gamma(q)\bs\SL(d,\R)$
since $\Gamma(q)$ is normal in $\SL(d,\Z)$.)
Hence as in Ex.\ \ref{PERIODICgenericEX} it is again 
possible to satisfy the assumptions in Section \ref{ASSUMPTLISTsec}
without using markings, i.e.\ with $\Sigma$ being a singleton set.

It should also be noted that the %
two cases for $m=2$ just discussed
(i.e., $\scrJ^\circ=\vece_1^\perp$
and $\scrJ^\circ=\{\bn\}$) %
are closely related to %
\cite[Cor.\ 5.4]{partI}
and \cite[Cor.\ 5.9]{partI}, respectively.
\end{ex}

\begin{ex}\label{HONEYCOMBex}
A special case of the situation
in Ex.\ \ref{PERIODICm2ex}
is the honeycomb point set,
for which the limit distribution of the free path length 
in the low density limit was considered in 
Boca and Gologan \cite{Boca09}
and Boca \cite{Boca10}; 
cf.\ also \cite[Remark 2.2]{qc}.
The honeycomb point set
can be represented as in
\eqref{PERIODICpointsetexpl} with $d=2$, $m=2$, $\delta=\sqrt3/2$, $g=\matr10{1/2}{\sqrt3/2}$ and
$\vecb_1=\bn$, $\vecb_2=\frac13(1,1)$;
thus $\scrJ^\circ=\{\bn\}$ and we can take $q=3$ in Ex.\ \ref{PERIODICm2ex}.
\end{ex}

\begin{ex}\label{MARKINGSessEX}
In the previous examples we noted that it was possible to discard the marking space.
A simple example where the marking space \textit{cannot} be discarded 
is obtained by taking $m=3$,
$\vecb_1=\bn$, $\vecb_2\in\Q^d\setminus\Z^d$ and $\vecb_3\notin\Q^d$
in \eqref{PERIODICpointsetexpl}.
(However for this example 
it is still possible to reduce from $\#\Sigma=3$ to $\#\Sigma=2$.)
\end{ex}

\section{Quasicrystals of cut-and-project type}
\label{QCexsec}

In this section we let $\scrP$ be a \textit{regular cut-and-project set};
such $\scrP$ are also referred to as (Euclidean) model sets. 
{\blu These are point sets which are typically non-periodic, 
yet strongly correlated.
Famous examples of point sets $\scrP$ which are covered by the theory in the present section
are the vertex sets of a Penrose tiling and of an Ammann-Beenker tiling.}
Previous results on the Lorentz gas in a quasicrystal have been limited to numerical simulations \cite{Kraemer12} and the distribution of free path lengths \cite{Wennberg12,qc}. 

\subsection{Preliminaries}
\label{QCexsec1}

We will use almost the same notation as in \cite{qc}:
Let $d\geq2$, $m\geq0$, $n=d+m$, and denote by $\pi$ and $\pi_\intl$ the projection of 
$\R^n=\R^d\times\R^m$ onto the first $d$ and last $m$ coordinates.
We refer to $\RR^d$ and $\RR^m$ as the {\em physical space} and {\em internal space}, respectively. 
{\blu Let $\scrL$ be a grid (also called affine lattice) in $\R^n$, i.e.\ a translate of a lattice in $\R^n$ of full rank.\label{afflatt}
A \textit{cut-and-project set} $\scrP=\scrP(\scrW,\scrL)$ 
is defined as the set of all projections to $\R^d$
of points in $\scrL$ which lie above a certain \textit{window} set $\scrW\subset\R^m$,
that is:
\begin{equation}\label{CUTPROJDEF}
	\scrP=\scrP(\scrW,\scrL) := \{ \pi(\vecy) \col \vecy\in\scrL, \; \pi_\intl(\vecy)\in\scrW \} \subset \RR^d .
\end{equation}

Conditions which are often imposed in the quasicrystal literature is that 
$\pi|_\scrL$ is injective and $\pi_\intl(\scrL)$ is dense in $\R^m$;
however we will not require any of these here.\footnote{This will allow us to also include periodic sets as part of the same
setting; see Section \ref{PERIODICsetproofsec} for details.}
Allowing this generality makes it necessary to introduce some notation of a more technical nature\footnote{In the special case when
$\pi_\intl(\scrL)$ is dense in $\R^m$, the notation which we introduce here 
could be dispensed with,
since in this case we simply have:
$\scrA=\scrA^\circ=\R^m$; $m_1=m$; 
$\mu_{\scrA}=$ Lebesgue measure on $\R^m$; $\scrV^\circ=\scrV=\R^n$;
$\scrL_{\scrV^\circ}=\scrL-\scrL$ (this 
is the lattice which $\scrL$ is a translate of),
and $\mu_{\scrV}=$ Lebesgue measure on $\R^n$.}:} 
Let $\scrA$ be the closure of $\pi_\intl(\scrL)$ in $\R^m$;
then $\scrA$ is a translate of the set $\scrA-\scrA=\{a-a'\col a,a'\in\scrA\}$,
which is a closed %
subgroup of $\R^m$.
We denote by $\scrA^\circ$ the connected component of $\scrA-\scrA$ containing $\bn$;
this is a linear subspace of $\R^m$,
and both $\scrA-\scrA$ and $\scrA$ are countable disjoint unions of translates of $\scrA^\circ$.
Set $m_1=\dim\scrA^\circ$.
We define $\mu_\scrA$ to be the natural volume measure on $\scrA$,
i.e.\ the measure which restricts to the standard $m_1$ dimensional Lebesgue measure on each translate of $\scrA^\circ$
contained in $\scrA$. 
Set\footnote{Our usage of the symbols ``$\scrV$'' and ``$\scrV^\circ$'' differs slightly from that in \cite{qc}.}
\begin{align*}
\scrV^\circ=\R^d\times\scrA^\circ,\qquad
\scrV=\R^d\times\scrA=\scrV^\circ+\scrL,\quad
\text{and}\qquad
\scrL_{\scrV^\circ}:=(\scrL-\scrL)\cap\scrV^\circ.
\end{align*}
Then $\scrL_{\scrV^\circ}$ is a lattice of full rank in $\scrV^\circ$.
We let $\mu_\scrV=\vol\times\mu_{\scrA}$, the natural volume measure on $\scrV$.
By abuse of notation, we will write $\mu_\scrA$ \textit{also} for the $m_1$ dimensional Lebesgue measure on $\scrA^\circ$,
and $\mu_\scrV$ also for the natural volume measure on $\scrV^\circ$.

{\blu We will always assume that the window set $\scrW$ is a subset of $\scrA$;
note that this is no loss of generality, since the cut-and-project set $\scrP=\scrP(\scrW,\scrL)$
remains the same upon replacing $\scrW$ by $\scrW\cap\scrA$.
We will furthermore assume that $\scrW$ is bounded, 
and that $\scrW$ has non-empty interior with respect to the topology of $\scrA$.}
If $\scrW$ has boundary of measure zero with respect to $\mu_\scrA$, we will say that $\scrP$
is \textit{regular.}
It follows from Weyl equidistribution (see \cite{Hof98} or \cite[Prop.\ 3.2]{qc})
that for any regular cut-and-project set $\scrP$ and 
any bounded $B\subset\RR^d$ with boundary of measure zero with respect to Lebesgue measure, 
\begin{equation}\label{WEYLCOUNTING1}
\lim_{T\to\infty} \frac{\#\{ \vecb \in \scrL\col \pi(\vecb)\in\scrP\cap T B\}}{T^d} 
= \delta_{d,m}(\scrL)\mu_\scrA(\scrW)\vol(B),
\end{equation}
{\blu where
\begin{align}\label{DELTADMLDEF}
\delta_{d,m}(\scrL):=\frac1{\mu_{\scrV}(\scrV^\circ/\scrL_{\scrV^\circ})}.
\end{align}}

{\blu Our final assumption is that the window $\scrW$ is 
appropriately chosen so that 
$\pi|_{\scrL\cap\pi_\intl^{-1}(\scrW)}$ is injective, and thus %
a bijection onto $\scrP$.}
Then \eqref{WEYLCOUNTING1} implies that $\scrP$ has asymptotic density
\begin{align}\label{qccPdef}
c_\scrP=\delta_{d,m}(\scrL)\mu_\scrA(\scrW),
\end{align}
i.e., \eqref{ASYMPTDENSITY0} holds with this $c_\scrP$.
Under the above assumptions $\scrP$ is a Delone set, i.e., uniformly discrete and relatively dense in $\RR^d$
\cite[Prop.\ 3.1]{qc};
in particular $\scrP$ is locally finite. %

Let $\ASL(n,\R)=\SL(n,\R)\ltimes\R^n$, with multiplication law\label{ASLMULTLAW}
\begin{align*}
(M,\vecx)(M',\vecx')=(MM',\vecx M'+\vecx').
\end{align*}
We let $\ASL(n,\R)$ act {\blu from the right} on $\R^n$ %
by affine linear maps, through
\begin{align*}
\vecy\mapsto \vecy(M,\vecx):=\vecy M+\vecx.
\end{align*}
Set $G=\ASL(n,\R)$ and $\Gamma=\ASL(n,\Z)$.
We fix, once and for all, $g\in G$ and $\delta>0$ so that %
$\scrL=\delta^{1/n}(\Z^ng)$. %
We define an embedding of $\ASL(d,\RR)$ in $G$ by
\begin{equation*}
\varphi_g: \ASL(d,\RR) \to G,\quad (A,\vecx) \mapsto g \left( \begin{pmatrix} A  &  0 \\ 0 & 1_m \end{pmatrix},(\vecx,\vecnull) \right) g^{-1} .
\end{equation*}
We also set $G^1=\SL(n,\R)$ and $\Gamma^1=\SL(n,\Z)$, and 
identify $G^1$ with a subgroup of $G$ through
$M\mapsto(M,\bn)$;
similarly we identify $\SL(d,\R)$ with a subgroup of $\ASL(d,\R)$.
It follows from {\blu the celebrated results of Ratner} \cite{Ratner91a}, \cite{Ratner91b}
that there exists a unique closed connected subgroup $H=H_g$ of $G$ such that $\Gamma\cap H$ is a lattice in $H$, $\varphi_g(\SL(d,\R))\subset H$,
and the closure of $\Gamma\backslash\Gamma\varphi_g(\SL(d,\RR))$ in $\GamG$ is given by 
\begin{align}\label{qcXdef}
X:=\Gamma\backslash\Gamma H.
\end{align}
We set $\Gamma_H:=\Gamma\cap H$,
and note that $X$ can be naturally identified with the homogeneous space $\Gamma_H\backslash H$.
We denote the unique right-$H$ invariant probability measure on $X$ %
by $\mu_g$; sometimes we will also let $\mu_g$ denote the corresponding Haar measure on $H$.

Similarly, there exists a unique closed connected subgroup $\widetilde H=\widetilde H_g$ of $G$ such that $\Gamma\cap \widetilde H$ is a lattice in $\widetilde H$, $\varphi_g(\ASL(d,\R))\subset\widetilde H$,
and the closure of $\Gamma\backslash\Gamma\varphi_g(\ASL(d,\RR))$ in $\GamG$ is given by 
\begin{align*}
\tX=\Gamma\backslash\Gamma\widetilde H.
\end{align*}
Note that $\tX$ %
can be naturally identified with the homogeneous space $(\Gamma\cap\widetilde H)\backslash\widetilde H$. 
We denote the unique right-$\widetilde H$ invariant probability measure on either of these spaces by $\tmu_g$;
sometimes we will also use $\tmu_g$ to denote the corresponding Haar measure on $\widetilde H$.
Of course, $H\subset \widetilde H$. %
It holds that $\pi_\intl(\delta^{1/n}(\Z^nhg))\subset\scrA$ for all $h\in\tH$, and
$\scrA$ equals the closure of $\pi_\intl(\delta^{1/n}(\Z^nhg))$ for $\tmu_g$-almost all $h\in\tH$
and also for $\mu_g$-almost all $h\in H$; cf.\ \cite[Props.\ 3.5 and 4.5]{qc}.

The following is a corrected and slightly generalized version of the
Siegel-Veech formula \cite[Theorem 5.1]{qc}.
\begin{thm}\label{SIEGELVEECHTHM}
For any $f\in\L^1(\scrV,\mu_\scrV)$,
\begin{align}\label{SIEGELVEECHTHMres}
\int_X\: %
\sum_{\substack{\vecm\in\delta^{1/n}(\Z^n hg)\\\pi(\vecm)\neq\bn}}f(\vecm)\,d\mu_g(h)
=\delta_{d,m}(\scrL)\int_{\scrV}f\,d\mu_\scrV.
\end{align}
\end{thm}
(Note that $\delta^{1/n}(\Z^n hg)$ is invariant under $h\mapsto\gamma h$ for all $\gamma\in\Gamma$, %
and $\delta^{1/n}(\Z^n hg)\subset\scrV$ for all $h\in H$. %
Hence the left hand side of \eqref{SIEGELVEECHTHMres} is well defined.)
\begin{proof}
If $\delta=1$, then Theorem \ref{SIEGELVEECHTHM} is exactly
\cite[Theorem 5.1]{qc} (as explained in the erratum to that paper,
the summation condition in \cite[(5.1)]{qc} should be
corrected to ``$\vecm\in\Z^nhg\setminus\pi^{-1}(\{\bn\})$'').
The proof of that theorem is easily generalized to the case of an arbitrary $\delta>0$.
Alternatively, the extension to general $\delta$ can be done by a simple scaling argument.
\end{proof}

The following lemma gives information regarding the summation condition in Theorem~\ref{SIEGELVEECHTHM}.
\begin{lem}\label{excpCONSTlem}
If $\vecm\in\Z^n$ and $\pi(\vecm g)=\bn$, then $\vecm h=\vecm$ for all $h\in H$.
\end{lem}
\begin{proof}
Let $h\in H$ be given. It follows from the defining properties of $H$ that 
there exist $\gamma_1,\gamma_2,\ldots\in\Gamma$ and $A_1,A_2,\ldots\in\SL(d,\R)$
such that $\gamma_j \varphi_g(A_j)\to h$   %
in $G$ as $j\to\infty$.
But $\pi(\vecm g)=\bn$ implies that $\vecm \varphi_g(A)=\vecm$ for all $A\in\SL(d,\R)$,
and thus $\vecm(\gamma_j\varphi_g(A_j))^{-1}=\vecm\gamma_j^{-1}\in\Z^n$ for all $j$.
However $\vecm(\gamma_j\varphi_g(A_j))^{-1}\to\vecm h^{-1}$ in $\R^n$ as $j\to\infty$,
and since $\Z^n$ is discrete this forces $\vecm h^{-1}\in\Z^n$.
But $H$ is connected; %
hence the fact that $\vecm h^{-1}\in\Z^n$ for all $h\in H$
implies that $\vecm h^{-1}$ is independent of $h$.
\end{proof}

Now %
we may reformulate Theorem \ref{SIEGELVEECHTHM} as follows.
Let us set
\begin{align}\label{tZndef}
\hZ:=\{\vecm\in\Z^n\col \pi(\vecm g)\neq\bn\}.   %
\end{align}
\noindent\hypertarget{SVTHMbislink}{\textbf{Theorem \ref*{SIEGELVEECHTHM}'.}}
\textit{For any $f\in\L^1(\scrV,\mu_\scrV)$,}
\begin{align}\label{SIEGELVEECHTHMres2}
\int_X \: %
\sum_{\vecm\in\delta^{1/n}(\hZ hg)}f(\vecm)\,d\mu_g(h)
=\delta_{d,m}(\scrL)\int_{\scrV}f\,d\mu_\scrV.
\end{align}
(Note that $\hZ\gamma=\hZ$ for every $\gamma\in\Gamma_H$, by Lemma \ref{excpCONSTlem}.
Hence the point set $\hZ hg$ is a well-defined function of $\Gamma h\in X$,
and the left hand side of \eqref{SIEGELVEECHTHMres2} is well-defined.)
\begin{proof}
Note that for any $\veck\in\Z^n\setminus\hZ$,
the condition ``$\pi(\vecm)\neq\bn$'' implies that the vector $\vecm=\delta^{1/n}(\veck hg)$ is excluded
from the sum in \eqref{SIEGELVEECHTHMres}, for all $\Gamma h\in X$.
On the other hand, if $\veck\in\hZ$,
then by a simple argument using real-analyticity, $\pi(\veck hg)\neq\bn$ for $\mu$-almost all $h\in H$
(cf.\ \cite[Lemma 8]{qcvisible}),
and so the vector $\vecm=\delta^{1/n}(\veck hg)$ is included in the sum 
in \eqref{SIEGELVEECHTHMres}, for almost all $\Gamma h\in X$.
Hence the left hand side of \eqref{SIEGELVEECHTHMres} equals the left hand side of \eqref{SIEGELVEECHTHMres2}.
\end{proof}

The following is a strengthening of \cite[Prop.\ 3.7]{qc} (which dealt with the case of $\scrW$ open):
\begin{lem}\label{INJlem}
Let %
$\scrW\subset\scrA$, $\mu_{\scrA}(\partial\scrW)=0$, and assume that the projection map from
$\{\vecy\in\scrL\col\pi_\intl(\vecy)\in\scrW\}$ to $\scrP(\scrW,\scrL)$ is bijective.
Then for $\mu_g$-almost all $h\in H$ the projection map from 
$\{\vecy\in\delta^{1/n}(\Z^nhg)\col\pi_\intl(\vecy)\in\scrW\}$ to $\scrP(\scrW,\delta^{1/n}(\Z^nhg))$
is bijective.
\end{lem}
\begin{proof}
The projections are surjective by construction.
To prove the injectivity, set
\begin{align*}
D=(\R^d\times\partial\scrW)\:\cup\:(\{\bn\}\times\scrA)\subset\scrV,
\end{align*}
and consider the following two subsets of $H$:
\begin{align}\label{S1def}
&S_1=\bigl\{h\in H\col\delta^{1/n}(\hZ hg)\cap D\neq\emptyset\bigr\};
\\
&S_2=\bigl\{h\in H\col\exists\vecy_1\neq\vecy_2\in\delta^{1/n}(\Z^nhg)\cap\pi_\intl^{-1}(\scrW^\circ)
\:\text{ satisfying }\: \pi(\vecy_1)=\pi(\vecy_2)\bigr\},
\end{align}
where $\scrW^\circ$ is the interior of $\scrW$.
Then $\mu_g(S_1)=0$, by Theorem \hyperlink{SVTHMbislink}{\ref*{SIEGELVEECHTHM}'}
applied with $f$ as the characteristic function of $D$.
Also $\mu_g(S_2)=0$, by \cite[Prop.\ 3.7]{qc} 
(after scaling by $\delta^{1/n}$). %
We will prove %
that every $h\in H\setminus(S_1\cup S_2)$ has the desired injectivity property.

Thus let $h\in H\setminus(S_1\cup S_2)$, and 
let $\vecy_1\neq\vecy_2$ be two arbitrary points in $\delta^{1/n}(\Z^nhg)\cap\pi_\intl^{-1}(\scrW)$.
Take $\vecm_j\in\Z^n$ so that $\vecy_j=\delta^{1/n}(\vecm_jhg)$.
If $\vecm_j\notin\hZ$ then $\pi(\vecy_j)=\bn$ and $\vecy_j=\delta^{1/n}(\vecm_j g)\in\scrL$ by 
Lemma \ref{excpCONSTlem};
hence our assumption that the projection map from
$\{\vecy\in\scrL\col\pi_\intl(\vecy)\in\scrW\}$ to $\scrP(\scrW,\scrL)$
is injective implies that at least one of $\vecm_1$ and $\vecm_2$ must lie in $\hZ$; say $\vecm_1\in\hZ$.
Then $\pi(\vecy_1)\neq\bn$, since $h\notin S_1$.
If $\vecm_2\notin\hZ$ then $\pi(\vecy_1)\neq\bn=\pi(\vecy_2)$ and we are done;
hence from now on we may assume that both $\vecm_1,\vecm_2\in\hZ$.
Again using $h\notin S_1$ we then have $\pi_\intl(\vecy_j)\notin\partial\scrW$,
i.e.\ $\pi_\intl(\vecy_j)\in\scrW^\circ$, for both $j=1,2$.
Now it follows from $h\notin S_2$ that $\pi(\vecy_1)\neq\pi(\vecy_2)$,
and the injectivity is proved.
\end{proof}

In fact a similar injectivity property holds also for \textit{shifts} of $\scrW$,
at least away from $\bn\in\R^d$:
\begin{lem}\label{INJlem2}
Let $\scrW$ be as in Lemma \ref{INJlem}, and fix $\tw\in\scrA$.
Then for $\mu_g$-almost all $h\in H$, the restriction of $\pi$ to
$\delta^{1/n}(\Z^nhg)\cap\pi_\intl^{-1}(\scrW+\tw)\setminus\pi^{-1}(\{\bn\})$ is injective.
\end{lem}
\begin{proof}
Set $\tW=\scrW+\tw$. Define $D,S_1,S_2$ be as in the proof of Lemma \ref{INJlem},
but with $\scrW$ replaced by $\tW$.
Then $\mu_g(S_1)=0$ as before.
Furthermore, as in the proof of \cite[Prop.\ 3.7]{qc},
for any $h\in H$ with $\overline{\pi_\intl(\delta^{1/n}(\Z^nhg))}=\scrA$,
we have $h\notin S_2$ if and only if 
\begin{align}\label{INJlem2pf1}
\scrW_0\cap\pi_\intl\bigl(\delta^{1/n}(\Z^nhg)_0\cap(\{\bn\}\times\R^m)\bigr)=\{\bn\},
\end{align}
where $\scrW_0:=\tW^\circ-\tW^\circ=\scrW^\circ-\scrW^\circ\subset\R^m$ and 
$(\Z^nhg)_0:=(\Z^nhg)-(\Z^nhg)\subset\R^n$.
Our injectivity assumption 
implies that \eqref{INJlem2pf1} holds for $h=1_n$.
Hence by \cite[Props.\ 3.5 and 3.6]{qc}, $\mu_g(S_2)=0$.

Now let $h\in H\setminus(S_1\cup S_2)$ and let $\vecy_1\neq\vecy_2$ be two arbitrary points in
$\delta^{1/n}(\Z^nhg)\cap\pi_\intl^{-1}(\tW)\setminus\pi^{-1}(\{\bn\})$.
Take $\vecm_j\in\Z^n$ so that $\vecy_j=\delta^{1/n}(\vecm_jhg)$.
It follows from $\vecy_1,\vecy_2\notin\pi^{-1}(\{\bn\})$ and Lemma \ref{excpCONSTlem} that both $\vecm_1,\vecm_2\in\hZ$.
Now $h\notin S_1$ implies $\pi_\intl(\vecy_1),\pi_\intl(\vecy_2)\in\tW^\circ$,
and thus, using also $h\notin S_2$, we have  $\pi(\vecy_1)\neq\pi(\vecy_2)$,
and the injectivity is proved.
\end{proof}

Finally, we also recall the Siegel-Veech formula for $\tH$.
The following is \cite[Cor.\ 5.2]{qc} but for a general density $\delta$.
\begin{cor}\label{SIEGELVEECHTILDEcor}
For any $f\in\L^1(\scrV,\mu_\scrV)$,
\begin{align*}
\int_{\Gamma\bs\Gamma\tH}\sum_{\vecm\in\delta^{1/n}(\Z^n hg)}f(\vecm)\,d\tmu_g(h)
=\delta_{d,m}(\scrL)\int_{\scrV}f\,d\mu_\scrV.
\end{align*}
\end{cor}

\subsection{Verifying that the assumptions of Section \ref*{ASSUMPTLISTsec} hold}
\label{QCexsec2}

From now on we assume that $\scrP=\scrP(\scrW,\scrL)$ is a regular cut-and-project set with
$\scrL$ a \textit{genuine} lattice, viz., $\bn\in\scrL$.
We will prove that $\scrP$ satisfies the assumptions in Section \ref{ASSUMPTLISTsec},
with 
\begin{align}\label{qcSIGMA}
\Sigma:=\overline{\scrW},
\end{align}
with the map $\vs:\scrP\to\Sigma$ defined by letting $\vs(\vecq)$ be the unique point $\vecw\in\scrW$ for which 
$(\vecq,\vecw)\in\scrL$,
and with 
\begin{align}\label{qcmmdef}
\mm:=\mu_{\scrA}(\scrW)^{-1}\mu_{\scrA\mid\overline{\scrW}}.
\end{align}
The map $\vs\mapsto\mu_\vs$, or as we will call it here,
$\vecw\mapsto\mu_\vecw$, from $\Sigma=\oW$ to $P(N(\scrX))$,
is defined as follows.
We assume that the fixed element $g$ lies in $G^1$;
this is permitted since $\bn\in\scrL$.
Then also $H\subset G^1$ and $\delta^{1/n}(\Z^nhg)=\delta^{1/n}\Z^{n}hg$ for all $h\in H$.
For each $\vecw\in\oW$, define the map $J_\vecw:X\to N_s(\scrX)$ through
\begin{align}\label{Jwdef}
J_\vecw(\Gamma h):=\big(\delta^{1/n}\hZ hg+(\bn,\vecw)\big)\cap\big(\R^d\times\scrW\big).
\end{align}
Noticing that the map $X\to N_s(\R^n)$,
$\Gamma h\mapsto\delta^{1/n}\hZ hg+(\bn,\vecw)$ is continuous
(and thus Borel measurable),
and using \cite[Thm.\ A2.3(iv)]{kallenberg02},
one verifies that $J_\vecw$ is Borel measurable.
Now define 
\begin{align}\label{qcmuwdef}
\mu_\vecw:=\mu_g\circ J_{\vecw}^{-1}\in P(N_s(\scrX)).
\end{align}

\begin{prop}\label{QCprop}
For $\scrP, \Sigma,\mm$ as above, 
all the assumptions in Section \ref{ASSUMPTLISTsec} are satisfied,
with $\scrE=\emptyset$
{\blu in [P2]},
and with the map $\vs\mapsto\mu_\vs$ given by \eqref{qcmuwdef}.
\end{prop}
We split the proof in a series of lemmas.
\begin{lem}\label{ASYMPTDENSITY1qclem}
The %
density assumption [P1] holds for our $\tP$,
with $c_\scrP=\delta_{d,m}(\scrL)\cdot\mu_\scrA(\scrW)$.
\end{lem}
\begin{proof}
We have $\scrX=\R^d\times\oW$ and $\tP=\scrL\cap(\R^d\times\scrW)$. %
By \cite[Prop.\ 3.2]{qc},
[P1] holds for any set $B\subset\scrX$ of the form
$B={D}\times{U}$ where 
${D}$ is a bounded subset of $\R^d$ with $\vol(\partial{D})=0$,
and ${U}$ is a bounded subset of $\overline{\scrW}$ with $\mu_\scrA(\partial{U})=0$.
Now for an arbitrary bounded subset $B\subset\scrX$ with $\mu_\scrX(\partial B)=0$,
note that $B$ is a Jordan measurable %
subset of the space $\scrV$;
hence for any $\ve>0$ there exist subsets $B',B''\subset\scrV$ which are both finite unions of disjoint boxes in $\scrV$,
and which satisfy $B'\subset B\subset B''$ and $\mu_\scrV(B''\setminus B)<\ve$.
By what we have already noted, 
[P1] holds for the two sets $B'\cap(\R^d\times\overline\scrW)$ and $B''\cap(\R^d\times\overline\scrW)$,
and by letting $\ve\to0$ we conclude that [P1] also holds for $B$.
\end{proof}

We next show that $\vecw\mapsto\mu_\vecw$ is a continuous map; %
cf.\ Lemma \ref{qcmuwcontlem}.

\begin{lem}\label{qcJcontlem}
Let $\vecw,\vecw_1,\vecw_2,\ldots\in\oW$ and
$h,h_1,h_2,\ldots \in H$,
subject to $\vecw_j\to\vecw$ and $\Gamma h_j\to\Gamma h$ as $j\to\infty$.
Furthermore assume $\bigl(\delta^{1/n}\hZ hg+(\bn,\vecw)\bigr)\cap(\R^d\times\partial\scrW)=\emptyset$.
Then $J_{\vecw_j}(\Gamma h_j)\to J_{\vecw}(\Gamma h)$ in $N_s(\scrX)$ as $j\to\infty$.
\end{lem}
\begin{proof}
Immediate.
\end{proof}

\begin{lem}\label{qcmuwcontlem}
The map $\vecw\mapsto\mu_\vecw$ from $\oW$ to $P(N_s(\scrX))$ is continuous.
\end{lem}
\begin{proof}
Let $\vecw,\vecw_1,\vecw_2,\ldots\in\oW$ and assume $\vecw_n\to\vecw$;
then our task is to prove %
$\mu_g\circ J_{\vecw_n}^{-1}\xrightarrow[]{\textup{ w }}\mu_g\circ J_\vecw^{-1}$.
However this follows from the continuous mapping theorem
\cite[Thm.\ 4.27]{kallenberg02}
together with Lemma \ref{qcJcontlem},
once we note that
\begin{align}\label{qcmuwcontlempf1}
\mu_g\big(\bigl\{\Gamma h\in X\col\bigl(\delta^{1/n}\hZ hg+(\bn,\vecw)\bigr)
\cap(\R^d\times\partial\scrW)\neq\emptyset\bigr\}\big)=0,
\end{align}
by Theorem \hyperlink{SVTHMbislink}{\ref*{SIEGELVEECHTHM}'}.
\end{proof}

Next we prove that the key assumption, [P2], holds,
{\blu with $\scrE=\emptyset$ and} with stronger uniformity:
\begin{lem}\label{qcasskeyprop}
For any $\lambda\in\Pac(\US)$,
$\mu_{\vecq,\rho}^{(\lambda)}\xrightarrow[]{\textup{ w }}\mu_{{\vs}(\vecq)}$
as $\rho\to0$, uniformly over all $\vecq\in\scrP$.
\end{lem}
\begin{proof}
Recall that $\tP=\scrL\cap(\R^d\times\scrW)$.
For any fixed $\vecq\in\scrP$,
letting $\vecw=\vs(\vecq)$ and using $\scrL=\scrL+(\vecq,\vecw)$, we have
\begin{align*}%
\tP_\vecq-\vecq=\big(\delta^{1/n}\hZ g+(\bn,\vecw)\big)\cap(\R^d\times\scrW),
\end{align*}
and so
\begin{align}\label{qcasskeyproppf1}
\scrQ_\rho(\vecq,\vecv) &=\big((\delta^{1/n}\hZ g) R(\vecv)D_\rho+(\bn,\vecw)\big)\cap\big(\R^d\times\scrW\big)
=J_\vecw\big(F_\rho(\vecv)\big),
\end{align}
where $F_\rho:\US\mapsto X$ is the map $F_\rho(\vecv)=\Gamma\varphi_g(R(\vecv)D_{\rho})$.
Let $\rho_n\in(0,1)$ and $\vecq_n\in\scrP$ for $n=1,2,\ldots$,
and assume that $\rho_n\to0$ and $\vs(\vecq_n)\to\vecw$ as $n\to\infty$, for some $\vecw\in\oW$;
by Lemma~\ref{UNIFPORTMlemC} it then suffices to prove that 
$\mu_{\vecq_n,\rho_n}^{(\lambda)}\xrightarrow[]{\textup{ w }}\mu_{\vecw}$ as $n\to\infty$.
Set $\vecw_n=\vs(\vecq_n)$ %
and $\nu_n=\lambda\circ F_{\rho_n}^{-1}$;
then by \eqref{qcasskeyproppf1} and \eqref{qcmuwdef},
our task is to prove $\nu_n\circ J_{\vecw_n}^{-1}\xrightarrow[]{\textup{ w }}\mu_g\circ J_\vecw^{-1}$. %
By \cite[Thm.\ 4.1]{qc} we have $\nu_n\xrightarrow[]{\textup{ w }}\mu_g$.
Now the desired result follows from 
\cite[Thm.\ 4.27]{kallenberg02},
using Lemma \ref{qcJcontlem}
and \eqref{qcmuwcontlempf1}.
\end{proof}

The following lemma shows that the assumption [Q1] %
holds, in a much stronger form.
\begin{lem}\label{qcsldinvlem}
For each $\vecw\in\oW$, $\mu_\vecw$ is invariant under the action of $\SL(d,\R)$.
\end{lem}
\begin{proof}
It follows from $\varphi_g(\SL(d,\R))\subset H$ 
that for every $A\in\SL(d,\R)$, right multiplication by $\varphi_g(A)$ on $X$ preserves the measure $\mu_g$.
The lemma follows from this fact,
together with \eqref{qcmuwdef} and the fact that 
$J_\vecw(\Gamma h\varphi_g(A))=J_\vecw(\Gamma h)A$
for all $\Gamma h\in X$ and $A\in\SL(d,\R)$.
\end{proof}

\begin{lem}\label{qce1simplelem}
The assumption [Q2] %
holds,
i.e.\ for every $\vecw\in\oW$ and $\mu_\vecw$-almost every $Y\in N_s(\scrX)$
we have $\vecy\cdot\vece_1\neq\vecy'\cdot\vece_1$ for all $\vecy\neq\vecy'\in Y$.
\end{lem}
\begin{proof}
Let $\vecw\in\oW$. 
Our task is to prove that
\begin{align*}
\mu_g\big(\big\{\Gamma h\in X\col \exists\vecy,\vecy'\in J_\vecw(\Gamma h)\text{ s.t. }
\vecy\neq\vecy'\text{ and }
\vecy\cdot\vece_1=\vecy'\cdot\vece_1\big\}\big)=0,
\end{align*}
and for this it suffices to prove that for any two fixed $\vecm\neq\vecm'\in\hZ$,
\begin{align}\label{qce1simplelempf1}
\mu_g\big(\big\{h\in H\col 
\pi_\intl(\delta^{1/n}\vecm hg)\in\scrW-\vecw,\:
\pi_\intl(\delta^{1/n}\vecm' hg)\in\scrW-\vecw,\:
\hspace{50pt}
\\\notag
(\vecm-\vecm')hg\cdot\vece_1=0\big\}\big)=0.
\end{align}
This is obviously true if the larger set
$S:=\big\{h\in H\col (\vecm-\vecm')hg\cdot\vece_1=0\big\}$ satisfies $\mu_g(S)=0$;
hence from now on we may assume $\mu_g(S)>0$.
Then a real-analyticity argument implies that
$(\vecm-\vecm')hg\cdot\vece_1=0$ for \textit{all} $h\in H$.
Using $\varphi_g(\SL(d,\R))\subset H$ this forces $\pi((\vecm-\vecm')g)=\bn$.
This means that $\vecm-\vecm'\notin\hZ$,
and thus by Lemma \ref{excpCONSTlem}
we have $\pi((\vecm-\vecm')hg)=\bn$ for all $h\in H$.
Now for every $h$ appearing in the set in \eqref{qce1simplelempf1},
we either have $\pi(\vecm hg)=\pi(\vecm'hg)=\bn$
or else the restriction of $\pi$ to
$\delta^{1/n}\Z^nhg\cap\pi_\intl^{-1}(\scrW-\vecw)\setminus\pi^{-1}(\{\bn\})$ is non-injective.
Hence \eqref{qce1simplelempf1} follows from Lemma \ref{INJlem2}
and the fact that the set $S_1$ in \eqref{S1def} satisfies $\mu_g(S_1)=0$.
\end{proof}

Next, the following lemma shows that the assumption [Q3] %
holds,
in a stronger form.
\begin{lem}\label{qcassnonemptylem}
For every $\ve>0$ there is some $R>0$ and an open set $X_\ve\subset X$ such that
$\mu_g(X_\ve)>1-\ve$ 
and $J_\vecw(\Gamma h)\cap(\scrB^d(\vecx,R)\times\oW)\neq\emptyset$
for all $\Gamma h\in X_\ve$, $\vecw\in\oW$, $\vecx\in\R^d$.
\end{lem}
\begin{proof}
Since $\scrW$ has non-empty interior, there exist $\veca\in\scrA$ and an open ball $B\subset\scrA^\circ$
centered at $\bn$, such that $\veca+B\subset\scrW$.
For any $\Gamma h\in X$ we write $\scrL_h:=\delta^{1/n}\Z^nhg$.
Recall that %
$\scrL_h\subset\scrV$ for all $\Gamma h\in X$. %
To start with, we note that if $\Gamma h\in X$ and $R>0$ satisfy
\begin{align}\label{qcassnonemptylempf1}
\scrL_h + (\scrB_{R/2}^d\times B)=\scrV,
\end{align}
then $J_\vecw(\Gamma h)\cap (\scrB^d(\vecx,R)\times\oW)\neq\emptyset$
for all $\vecw\in\oW$, $\vecx\in\R^d$.
Indeed, given $\vecw\in\oW$ and $\vecx\in\R^d$,
take $\vecx'\in\R^d$ so that $\scrB^d(\vecx',R/2)\subset\scrB^d(\vecx,R)\setminus\{\bn\}$;
then $(\vecx',\veca-\vecw)\in\scrV$, 
and so by \eqref{qcassnonemptylempf1} 
there exists a point $\vecy\in\scrL_h$ such that
\begin{align*}
(\vecx',\veca-\vecw)\in\vecy+ (\scrB_{R/2}^d\times B).
\end{align*}
Then $\pi(\vecy)\in\scrB^d(\vecx',R/2)\subset\scrB^d(\vecx,R)\setminus\{\bn\}$
and $\vecy\in\delta^{1/n}\hZ hg$ since $\pi(\vecy)\neq\bn$;
also $\pi_\intl(\vecy)+\vecw\in \veca+B\subset\scrW$.
It follows that $\vecy+(\bn,\vecw)\in J_\vecw(\Gamma h)\cap (\scrB^d(\vecx,R)\times\oW)$,
proving our assertion.

Now let $X'$ be the set of $\Gamma h\in X$ satisfying
$\overline{\pi_\intl(\scrL_h)}=\scrA$;
recall that $\mu_g(X')=1$.
For every $\Gamma h\in X'$,
the subspace $\R^d\times\{\bn\}$ maps to a dense subset in the torus 
$\scrV/\scrL_h=\scrV^\circ/(\scrL_h\cap\scrV^\circ)$
and hence there exists some $R>0$ for which \eqref{qcassnonemptylempf1} holds.
It follows that if we let
\begin{align*}
X(R)=\{\Gamma h\in X\col\scrL_h + (\scrB_{R/2}^d\times B)=\scrV\}
\end{align*}
then $X'$ is contained in the union $\cup_{R>0}X(R)$. %
Also each $X(R)$ is an open subset of %
$X$, and $X(R)$ is increasing with respect to $R$.
It follows that $\lim_{R\to\infty}\mu_g(X(R))=1$,
and so there is some $R>0$ such that $\mu_g(X(R))>1-\ve$.
Then $X_\ve:=X(R)$ has the desired properties.
\end{proof}

Finally, we will prove [P3]. %
As in the case of the Poisson process, we will do so by explicitly identifying the macroscopic limit process.
Recall that $\tX=\Gamma\bs\Gamma\tH$,
and define the map $\tJ:\tX\to N_s(\scrX)$ by
\begin{align*}
\tJ(\Gamma h):=\delta^{1/n}(\Z^n hg) \cap(\R^d\times\scrW).
\end{align*}
This map is Borel measurable. We set
\begin{align}\label{qcmudef}
\mu:=\tmu_g\circ \tJ^{-1}\in P(N_s(\scrX)).
\end{align}
\begin{lem}\label{qcgenlimitprop}
Let $\Lambda\in\Pac(\T^1(\R^d))$.
Then $\mu_\rho^{(\Lambda)}\xrightarrow[]{\textup{ w }}\mu$ as $\rho\to0$.
\end{lem}
\begin{proof}
Recall that $\tP=\scrL\cap(\R^d\times\scrW)$;
hence for any $\vecq\in\R^d\setminus\scrP$ we have
\begin{align*}
\scrQ_\rho(\vecq,\vecv)=(\delta^{1/n}\Z^ng-(\vecq,\bn))R(\vecv)D_\rho\cap(\R^d\times\scrW),
\end{align*}
and so for any $(\vecq,\vecv)\in\T^1(\R^d)$ and $\rho>0$, if $\rho^{1-d}\vecq\notin\scrP$ then
\begin{align*}
\scrQ_\rho(\rho^{1-d}\vecq,\vecv)=\tJ(\TF_\rho(\vecq,\vecv)),
\end{align*}
where $\TF_\rho:\T^1(\R^d)\to\tX$ is the map 
$\TF_\rho(\vecq,\vecv)=\Gamma\varphi_g((1_d,-(\delta^{-1/n}\rho^{1-d}\vecq,\bn))R(\vecv)D_\rho)$.
Hence $\mu_\rho^{(\Lambda)}=\Lambda\circ\TF_\rho^{-1}\circ\tJ^{-1}$.
Now by \cite[Thm.\ 4.7]{qc},
$\Lambda\circ\TF_\rho^{-1}\xrightarrow[]{\textup{ w }}\tmu_g$ as $\rho\to0$.
Also the map $\tJ$ has the property that if $h,h_1,h_2,\ldots\in H$ satisfy
$\Gamma h_n\to\Gamma h$ as $n\to\infty$
and $\delta^{1/n}(\Z^n hg)\cap(\R^d\times\partial\scrW)=\emptyset$,
then $\tJ(\Gamma h_n)\to\tJ(\Gamma h)$.
Furthermore, by Corollary \ref{SIEGELVEECHTILDEcor},
\begin{align*}
\tmu_g\big(\big\{\Gamma h\in\tX\col\delta^{1/n}(\Z^n hg)\cap(\R^d\times\partial\scrW)\neq\emptyset\big\}\big)=0.
\end{align*}
Hence by \cite[Thm.\ 4.27]{kallenberg02},
$\Lambda\circ\TF_\rho^{-1}\circ\tJ^{-1}\xrightarrow[]{\textup{ w }}\tmu_g\circ\tJ^{-1}$ as $\rho\to0$,
as desired.    %
\end{proof}

\begin{lem}\label{qcassgenbdfreepathlem}
The assumption [P3] %
holds.
\end{lem}
\begin{proof}
In view of Lemma \ref{qcgenlimitprop} and Remark \ref{ASS:bdfreepathrem},
it suffices to prove that $\mu(\{\emptyset\})=0$.
Let us write $\scrL_h=\delta^{1/n}(\Z^nhg)$, and let $\tX'$ be the set of all $\Gamma h\in\tX$ for which
$\overline{\pi_\intl(\scrL_h)}=\scrA$;
recall that $\tmu_g(\tX')=1$.
For every $\Gamma h\in\tX'$, $\tJ(\Gamma h)$ is non-empty, since $\scrW$ has non-empty interior.
Hence $\tJ^{-1}(\{\emptyset\})\subset X\setminus X'$ and thus $\mu(\{\emptyset\})\leq\tmu_g(X\setminus X')=0$.
\end{proof}

We have now proved that all the assumptions in Section \ref{ASSUMPTLISTsec} are satisfied,
i.e.\ the proof of Proposition \ref{QCprop} is complete.
\hfill$\square$

\begin{remark}
Of course, by combining Theorem \ref{GENLIMITthm} and Lemma \ref{qcgenlimitprop},
it now also follows that the measure $\mu$ in \eqref{qcmudef} agrees with the measure defined in \eqref{GENMUdef}.
\end{remark}

\begin{remark}\label{QCPALMrem}
It follows from Theorem \hyperlink{SVTHMbislink}{\ref*{SIEGELVEECHTHM}'}
that for each $\vecw\in\overline{\scrW}$,
a point process $\Xi_\vecw$ with distribution $\mu_{\vecw}$
has intensity measure $c_\scrP\mu_\scrX$.
Hence Section \ref{TKPALMsec} applies,
leading to an expression for the transition kernel
in terms of the Palm distributions of $\Xi_\vecw$.
However it is possible to give more explicit formulas for the transition kernels
in terms of Haar measures on certain homogeneous spaces. %
For the special case of $\scrP$ a lattice (i.e.\ $m=0$)
this was done in \cite[Sections 4 and 8]{partI};
and precise asymptotic formulas for the transition kernels were given in \cite{partIV}. 
\end{remark}

\subsection{The case of periodic point sets}\label{PERIODICsetproofsec}
We now specialize to the case of a periodic point set $\scrP$
as considered in Section \ref{periodicpointsetsSEC}.
Thus let $\scrP,\scrL,\delta,g$ and $\vecb_1,\ldots,\vecb_m$ be as in \eqref{PERIODICpointsetexpl}.
We will prove Proposition \ref{PERIODICSETprop}
by realizing it as a special case of Proposition \ref{QCprop},
following \cite[Sec.\ 2.3]{qc}.

Set $n=d+m$ and let 
\begin{align*}
\scrL'=\scrL\times\{\bn\}+\delta^{1/d}(\vecb_1g,\vece_1)+\cdots+\delta^{1/d}(\vecb_mg,\vece_m).
\end{align*}
This is a lattice of full rank in $\R^n$
which can be expressed as 
\begin{align*}
\scrL'={\delta'}^{1/n}\Z^n g',
\end{align*}
where $\delta':=\delta^{n/d}$,
and where $g'\in G^1=\SL(n,\R)$ is given by
\begin{align}\label{gpdef}
g'=\matr g0{Bg}I=g_B\matr g00I;
\qquad g_B:=\matr I0BI.
\end{align}
(Recall from Sec.\ \ref{periodicpointsetsSEC} that $B$ is the matrix in $M_{m,d}(\R)$ 
whose row vectors are $\vecb_1,\ldots,\vecb_m$.
Also in \eqref{gpdef}, ``$I$'' stands for the identity matrix of order $d$ or $m$, depending on the position.)
We will apply the set-up of Sec.\ \ref{QCexsec1}--\ref{QCexsec2}
with $\scrL'$ in place of $\scrL$.
Note that for this lattice we have $\scrA=\delta^{1/d}\Z^m$,
and thus $\mu_{\scrA}$ is counting measure.
We fix the following (regular) window set:
\begin{align*}%
\scrW=\{\delta^{1/d}\vece_1,\ldots,\delta^{1/d}\vece_m\}.
\end{align*}
The point of these choices of $\scrL'$ and $\scrW$ is that now our periodic set $\scrP$ in \eqref{PERIODICpointsetexpl}
equals the cut-and-project set $\scrP(\scrW,\scrL')$.
Hence Proposition \ref{QCprop} applies to the set $\scrP$,
and we will see that in this case,
the statement of Proposition \ref{QCprop} is equivalent with the statement
of Proposition~\ref{PERIODICSETprop}.
We note that \eqref{qcSIGMA} gives $\Sigma=\scrW$,
and \eqref{qcmmdef} means that $\mm$ is the uniform probability measure on $\Sigma$,
assigning mass $m^{-1}$ to each point.
This agrees with $\Sigma$ and $\mm$ in Proposition \ref{PERIODICSETprop},
if we identify each $j\in\{1,\ldots,m\}$ with the vector $\delta^{1/d}\vece_j\in\scrW$.
It remains to prove that the map given by \eqref{qcmuwdef}
agrees with the map in \eqref{periodicmuwdef}.
The key step in doing so is the following lemma,
which gives an explicit formula for the subgroup
$H=H_{g'}$ of $G$.

\begin{lem}\label{HgkformulaLEM}
\begin{align*}
H=\left\{g_B\matr I0UI\matr A00I g_B^{-1}\col A\in\SL(d,\R),\: U\in{\scrJ^\circ}^d\right\}.
\end{align*}
\end{lem}
\begin{proof}
Recall that, by definition,
$H=H_{g'}$ is the unique closed connected subgroup of $G$ such that
$\Gamma\cap H$ is a lattice in $H$,
$\varphi_{g'}(\SL(d,\R))\subset H$,
and the closure of $\Gamma\backslash\Gamma\varphi_{g'}(\SL(d,\RR))$ in $\GamG$ equals
$\Gamma\backslash\Gamma H$.
It follows from \eqref{gpdef} that $\varphi_{g'}(A)=\varphi_{g_B}(gAg^{-1})$ for all $A\in\SL(d,\R)$;
hence $H=H_{g_B}$.
Let $G'$ be the following closed Lie subgroup of $G$:
\begin{align}\label{PERIODICGdef}
G'=\left\{\matr A0U{I}\col A\in\SL(d,\R),\: U\in M_{m,d}(\R)\right\}.
\end{align}
Note that $\varphi_{g_B}(\SL(d,\R))\subset G'$ and $\Gamma\cap G'$ is a lattice in $G'$;
hence
\begin{align*}
H=H_{g_B}\subset G'.
\end{align*}

For any linear subspace $\scrU\subset\R^m$,
let us write
$\scrU^d$ for the space of matrices in $M_{m,d}(\R)$ all of whose column vectors belong to $\scrU$.
Note that $\scrU^d\cdot A=\scrU^d$ for all $A\in\SL(d,\R)$.
Set
\begin{align*}
H_{\scrU}
&=\left\{g_B\matr A0UI g_B^{-1}\col A\in\SL(d,\R),\: U\in{\scrU}^{\,d}\right\}
\\
&=\left\{g_B\matr I0UI\matr A00I g_B^{-1}\col A\in\SL(d,\R),\: U\in\scrU^{\,d}\right\};
\end{align*}
this is a closed connected subgroup of $G'$.

Let $\Omega$ be the set of matrices $U\in M_{m,d}(\R)$ such that
$\matr I0UI\in H$;
this is a closed subgroup of $\langle M_{m,d}(\R),+\rangle$.
By mimicking part of the proof of
\cite[Lemma 7]{DMS2016},
we find that there exists a linear subspace $\scrV$ of $\R^m$ such that
$\Omega=\scrV^d$. %
Using $\varphi_{g_B}(\SL(d,\R))\subset H$ it also follows that
$H=H_{\scrV}$.
Hence it now remains to prove that $\scrJ^\circ=\scrV$.

Note that for every linear subspace $\scrU\subset\R^m$,
$H_{\scrU}$ is a closed connected Lie subgroup of $G'$ which contains $\varphi_{g_B}(\SL(d,\R))$.
Hence, by the definition of $H=H_{g_B}$,
$\scrV$ can be characterized as the unique smallest linear subspace $\scrU\subset\R^m$
with the property that $H_{\scrU}$ intersects $\Gamma$ in a lattice.
Let $\pi:G'\to\SL(d,\R)$ be the projection homomorphism 
$\matr A0UI\mapsto A$.
Using \cite[Cor.\ 8.28]{Raghunathan},
it follows that
$H_{\scrU}$ intersects $\Gamma$ in a lattice
if and only if 
\begin{align}\label{HgkformulaLEMpf10}
\scrU\cap\Z^m\text{ is a lattice in $\scrU$}
\end{align}
and $\pi(\Gamma\cap H_{\scrU})$ is a finite index subgroup of $\SL(d,\Z)$;
and as in the proof of
\cite[Lemma 8]{DMS2016},
one sees that the latter condition holds if and only if
\begin{align}\label{HgkformulaLEMpf11}
B_j\in \Q^m+\scrU,\quad\forall j,
\end{align}
where $B_1,\ldots,B_d$ are the column vectors of $B$.
Hence:
$\scrV$ is the smallest linear subspace $\scrU\subset\R^m$
which satisfies both \eqref{HgkformulaLEMpf10} and \eqref{HgkformulaLEMpf11}.

Let $s=\dim\scrV$.
Since $\scrV$ satisfies \eqref{HgkformulaLEMpf10}, there exists a $\Z$-basis
$V_1,\ldots,V_m$ of $\Z^m$ such that
$V_1,\ldots,V_s$ is a $\Z$-basis of $\scrV\cap\Z^m$
and an $\R$-linear basis of $\scrV$.
Since $\scrV$ satisfies \eqref{HgkformulaLEMpf11},
there exists some $q\in\Z^+$ such that
$B_j\in q^{-1}(\Z V_{s+1}+\cdots+\Z V_m)+\scrV$, $\forall j$.
Now $q^{-1}(\Z V_{s+1}+\cdots+\Z V_m)+\scrV$ is a closed subgroup of $\R^m$ containing
$\Z^m$ and $B_1,\ldots,B_d$;
hence $\scrJ\subset q^{-1}(\Z V_{s+1}+\cdots+\Z V_m)+\scrV$,
and thus $\scrJ^\circ\subset\scrV$.
On the other hand, recall that $\scrJ^\circ\cap\Z^m$ is a lattice in $\scrJ^\circ$,
i.e.\ $\scrJ^\circ$ satisfies \eqref{HgkformulaLEMpf10},
and furthermore we have
$B_j\in\scrJ\subset\Q^m+\scrJ^\circ$ for each $j$,
i.e.\ $\scrJ^\circ$ satisfies \eqref{HgkformulaLEMpf11}.
Hence $\scrV\subset\scrJ^\circ$,
i.e.\ we have proved $\scrJ^\circ=\scrV$, and thereby the lemma.
\end{proof}

Using Lemma \ref{HgkformulaLEM} we now conclude the proof of Proposition \ref{PERIODICSETprop}.
Writing 
\begin{align*}
h=g_B\matr I0UI\matr A00I g_B^{-1}
\end{align*}
with $A\in\SL(d,\R)$ and $U\in{\scrJ^\circ}^d$,
one verifies that in the present situation, the formula
\eqref{Jwdef} can be expressed in the following more explicit way,
for any $\vecw=\delta^{1/d}\vece_\ell$ in $\scrW$:
\begin{align}\label{PERIODICSETproppf1}
J_{\delta^{1/d}\vece_\ell}(\Gamma h)=
\biggl(\bigcup_{j=1}^m\delta^{1/d}\bigl(\Z^d+(\vece_j-\vece_\ell)(B+U)\bigr)Ag\times\{\delta^{1/d}\vece_j\}\biggr)
\setminus\{(\bn,\delta^{1/d}\vece_\ell)\}.
\end{align}
It follows from \eqref{BinqinvMZpJd}
that $B\gamma-B\in{\scrJ^\circ}^d+M_{m,d}(\Z)$
for each $\gamma\in\Gamma(q)$;
hence there exist $U_\gamma\in{\scrJ^\circ}^d$ and $\alpha_\gamma\in M_{m,d}(\Z)$
such that $B\gamma-B=U_\gamma+\alpha_\gamma$.
Using Lemma \ref{HgkformulaLEM} 
one now verifies that
the lattice $\Gamma\cap H$ contains
\begin{align*}
\Gamma_q':=\left\{\matr{\gamma}0{\alpha_\gamma+\alpha}I \col \gamma\in\Gamma(q),\:\alpha\in{\scrJ_{\Z}^\circ}^d\right\}
\end{align*}
as a subgroup of finite index.
Hence in the definition of $\mu_{\vecw}$ in \eqref{qcmuwdef}
we may just as well view $J_{\vecw}$ as a map from $\Gamma_q'\bs H$ to $N_s(\scrX)$,
with $\mu_g$ (in the present situation: $\mu_{g'}$) being the invariant probability measure
on $\Gamma_q'\bs H$.
Taking $F_q\subset\SL(d,\R)$ to be a fundamental domain for $\Gamma(q)\bs\SL(d,\R)$,
as in Section \ref{periodicpointsetsSEC},
and $F_{\scrJ}\subset{\scrJ^\circ}^d$ a fundamental domain for
$\TT_{{\scrJ^\circ}^d}={\scrJ^\circ}^d/{\scrJ^\circ_{\Z}}^d$,
one verifies that the following set is a fundamental
domain for $\Gamma_q'\bs H$:
\begin{align}\label{PERIODICSETproppf2}
\left\{g_B\matr I0UI\matr{Ag^{-1}}00I g_B^{-1}\col A\in F_q,\: U\in F_{\scrJ}\right\}.
\end{align}
Note also that when parametrizing the last set %
by $\langle A,U\rangle\in F_q\times F_{\scrJ}$, the probability measure $\mu_{g'}$
corresponds to the probability measure $\eta\times\eta_{\TT}$
which we considered in Section \ref{periodicpointsetsSEC};
furthermore, after renaming the markings ``$1,\ldots,m$'' instead of
``$\delta^{1/d}\vece_1,\ldots,\delta^{1/d}\vece_m$'',
the formula \eqref{PERIODICSETproppf1} turns into \eqref{Jelldef}.
(We used ``$Ag^{-1}$'' instead of ``$A$'' in \eqref{PERIODICSETproppf2}
so as to get rid of the ``$g$'' in \eqref{PERIODICSETproppf1}.)
Hence the formula for $\mu_{\vecw}$, \eqref{qcmuwdef},
turns into the formula for $\mu_{\ell}$ in \eqref{periodicmuwdef}.
This completes the proof of Proposition \ref{PERIODICSETprop}.
\hfill$\square$

\section{Scattering potentials satisfying the conditions in Section \ref*{SCATTERINGMAPS}}
\label{AppA1}

We consider scattering described by a Hamiltonian flow with a spherically symmetric potential $W$
having compact support in the unit ball.
Thus, by a slight abuse of notation,
the potential $W:\R^d\setminus\{\bn\}\to\R$
is given by $W(\vecq)=W(r)$ with $r=\|\vecq\|$,
where we assume $W\in\C(\R_{>0})$
and $W(r)=0$ for $r\geq1$. %
Furthermore we assume
\begin{align}\label{Woksingularity}
\liminf_{r\to0}r^2W(r)\geq0.
\end{align}
For scattering at the single-site potential $W,$
considering a particle hitting the unit ball with unit speed
and with an impact parameter of length $w\in(0,1)$,
the deflection angle
and the total time which the particle spends inside the scatterer 
are given by the formulas
\cite[Sect.~5.1]{Newton82}
\begin{align}\label{deflangle}
\theta(w)=\pi-2w\int_{r_0}^{\infty}
\frac{r^{-2}\,dr}{\sqrt{1-2W(r)-w^2r^{-2}}},
\end{align}
and
\begin{align}\label{scatteringtime}
T(w)=2\int_{r_0}^1\frac{dr}{\sqrt{1-2W(r)-w^2r^{-2}}},
\end{align}
respectively.
Here $r_0=r_0(w)\in(0,1)$ is the largest solution to the equation
$1-2W(r)-w^2r^{-2}=0$
(this number $r_0$ is guaranteed to exist because of \eqref{Woksingularity}).
The deflection angle $\theta(w)$ in \eqref{deflangle} can take any value %
in $[-\infty,\pi]$, 
which for $\theta(w)<0$ represents spiralling motion around the center;
however, in the case of an everywhere repulsive potential
(i.e., $W$ monotonically decreasing) %
we have $0\leq\theta(w)\leq\pi$ \cite[Sect.~5.4]{Newton82}.

Comparing \eqref{scatteringtime} and \eqref{deflangle}, we note that
\begin{align}\label{Twbound}
T(w)<\frac{\pi-\theta(w)}w,\qquad\forall w\in(0,1).
\end{align}
In particular, if the function $\theta$ is bounded,
then $T(w)$ is uniformly bounded on any interval $w\in[\ve,1)$, $\ve>0$.

When replacing the potential $W$ by the rescaled version
$\vecq\mapsto W(\rho^{-1}\vecq)$,
as in \eqref{Vrhodef},
the formula for the deflection angle remains the same,
with $w\in(0,1)$ now denoting the \textit{normalized} impact parameter;
furthermore the function $T$ is replaced by $w\mapsto \rho\, T(w)$.

\vspace{5pt}

We next discuss conditions on the scattering potential ensuring that the scattering map $\Psi$
satisfies the conditions (i)--(iii) in Section \ref{SCATTERINGMAPS}.

\begin{definition}\label{DISPERSINGDEF}
{\blu For the purposes of the present paper,}
we say that a potential $W\in\C(\R_{>0})$
with $\supp(W)\subset(0,1]$
is \textit{dispersing}
if $W|_{(0,1]}$ is $\C^2$,
$\liminf_{r\to0}r^2W(r)\geq0$
and $\limsup_{r\to0}W(r)\neq\frac12$ \footnote{These are the conditions imposed in Section \ref{sec:soft}
in order to make the flow $\Phi_t^{(\rho)}$ everywhere well-defined.},
and furthermore the function $\theta:(0,1)\to\R$ given by 
\eqref{deflangle} extends to a $\C^1$ function on $[0,1)$ satisfying
$\theta(0)=k\pi$ for some $k\in\Z$,
and $\theta'(w)\neq0$ and $|\theta(w)-k\pi|<\pi$ 
for all $w\in[0,1)$.
\end{definition}

When $W$ is dispersing,
it follows from Lemma \ref{DEFLANGLElem} and Remark \ref{COND3completecondrem}
that the scattering map $\Psi$
(which is given by \eqref{PSI1formula} and \eqref{PSI2formula}, using the extended function $\theta:[0,1)\to\R$)
satisfies the conditions (i)--(iii) in Section \ref{SCATTERINGMAPS}.

One example of a dispersing potential is 
the truncated (``Muffin-tin'')
Coulomb potential, 
\begin{align}\label{MUFFINTINW}
W(r)=\alpha\cdot I(r\leq1)\cdot (r^{-1}-1),
\end{align}
for any constant $\alpha\notin\{0,-1\}$.
Indeed, by a straightforward modification of the classical treatment of the non-truncated Coulomb potential
(cf., e.g., \cite[Sec.\ 8.E]{Arnold}),
one verifies that in this case,
\begin{align}\label{MTCOULOMBtheta}
\theta(w)=2\arctan\Bigl(\frac{\alpha}{1+\alpha}\cdot\frac{(1-w^2)^{1/2}}w\Bigr)-I(\alpha<-1)\cdot2\pi.
\end{align}
On the other hand, if $\alpha=-1$, then $\theta(w)\equiv-\pi$,
and the scatterer is a so called \textit{Eaton lens:}
Each particle is reflected a perfect $180^\circ$ angle
independently of the impact parameter,
and the potential is not dispersing. %

In Lemma~\ref{genokscatterercondLEM} below we give a simple criterion which ensures
that every $W$ in a certain general class of repulsive potentials is dispersing.
As a preparation we first give an explicit formula for the first derivative of $\theta(w)$.
\begin{lem}\label{thetapLEM}
Assume that $W|_{(0,1)}$ is $\C^2$.
Then $\scrU:=\{w\in(0,1)\col r_0^3\, W'(r_0)\neq w^2\}$
is an open subset of $(0,1)$,
and the function $\theta(w)$ is $\C^1$ on $\scrU$,
satisfying
\begin{align}\label{genokscatterercondLEMpf1}
\theta'(w)=-2\int_{r_0(w)}^\infty\frac{\bigl(w^2r^{-2}+4W(r)+rW'(r)-2\bigr)\frac{r_0^4 W'(r_0)}{w^2-r_0^3W'(r_0)}
+r^2W'(r)}{r^3\bigl(1-2W(r)-w^2r^{-2}\bigr)^{3/2}}\,dr
\end{align}
for all $w\in\scrU$.
\end{lem}
(Recall that we always assume $W\in\C(\R_{>0})$,
$W(r)=0$ for $r\geq1$, and that \eqref{Woksingularity} holds.)
\begin{proof}
Recall that $r_0=r_0(w)\in(0,1)$ is the largest solution to the equation
$1-2W(r)-w^2r^{-2}=0$.
Fix a point $w_0\in\scrU$.
Then by the implicit function theorem,
the function $r_0(w)$ is $\C^2$
in some neighbourhood of $w_0$,
with 
\begin{align}\label{r0pformula}
r_0'(w)=\frac{wr_0}{w^2-r_0^3W'(r_0)}.
\end{align}
In particular, since $r_0$ is continuous near $w_0$,
$\scrU$ contains a neighbourhood of $w_0$.
This proves that $\scrU$ is open.
Also for any $w\in\scrU$ we have 
$1-2W(r)-w^2r^{-2}>0$ for all $r>r_0(w)$;
hence $\frac d{dr}\bigl(1-2W(r)-w^2r^{-2}\bigr)\big|_{r=r_0(w)}\geq0$,
viz., $w^2r_0^{-3}-W'(r_0)\geq0$.
Hence in view of the definition of $\scrU$, we have:
\begin{align}\label{thetapLEMpf1}
w^2>r_0^3\,W'(r_0)
\:\text{ and thus }\:
r_0'(w)>0,
\qquad\forall w\in\scrU.
\end{align}

Set $\xi_0(w)=r_0(w)/w$.
Substituting $r=w(\xi_0(w)+h)$ in \eqref{deflangle},
differentiating formally under the integration sign,
and then substituting $h=\xi-\xi_0(w)$,
we obtain 
\begin{align}\label{thetapLEMres}
\theta'(w)=-2\int_{\xi_0(w)}^{\infty}
\frac{\bigl(\xi^{-2}+4W(w\xi)+w\xi W'(w\xi)-2\bigr)\xi_0'(w)+\xi^2W'(w\xi)}
{\xi^3\,(1-2W(w\xi)-\xi^{-2})^{3/2}}\,d\xi
%
%
%
\end{align}
for all $w\in\scrU$.
This formula is easily verified to be equivalent with \eqref{genokscatterercondLEMpf1}.
In order to justify %
the preceding 
manipulations,
set $A(w,x)=2x\cdot B(w,\xi_0(w)+x^2)$, where $B(w,\xi)$ is the integrand
in \eqref{thetapLEMres};
then the right hand side of \eqref{thetapLEMres}
equals 
\begin{align*}
-2\int_0^\infty A(w,x)\,dx.
\end{align*}
Using \eqref{thetapLEMpf1} and $W|_{(0,1)}\in\C^2$ one verifies that
$A(w,x)$ extends to a continuous function on all $\scrU\times[0,\infty)$.
Furthermore $B(w,\xi)=\xi^{-3}(1-\xi^{-2})^{-3/2}(\xi^{-2}-2)\xi_0'(w)$ for $\xi>w^{-1}$,
implying that $A(w,x)\ll x^{-5}$ for $x$ large, uniformly over $w$ in any compact subset of $\scrU$.
It follows from these observations that the right hand side of \eqref{thetapLEMres}
is a continuous function of $w\in\scrU$.
Now to complete the proof it suffices to verify that 
\begin{align}\label{thetapLEMpf2}
-2\int_{w_1}^{w_2}\int_0^\infty A(w,x)\,dx\,dw=\theta(w_2)-\theta(w_1)
\end{align}
whenever $w_1<w_2$ and $[w_1,w_2]\subset\scrU$.
However, this follows immediately using
Fubini's Theorem and \eqref{deflangle},
together with the fact that %
\begin{align*}
\int_{w_1}^{w_2}A(w,x)\,dw=
\biggl[2x\cdot\frac{(\xi_0(w)+x^2)^{-2}}{\sqrt{1-2W(w(\xi_0(w)+x^2))-(\xi_0(w)+x^2)^{-2}}}
\biggr]_{w=w_1}^{w=w_2}.
\end{align*}
\end{proof}

\begin{lem}\label{genokscatterercondLEM}
Let $\beta$ be the numerical constant
$\beta=\frac{(1+\alpha)^2(1-\alpha)}{2\alpha^4-\alpha+2}=0.7124\ldots$,
where $\alpha=0.4093\ldots$ is the unique zero of $2x^5+2x^4-8x^3+2x^2-7x+3$ in $[0,1]$.
Assume that $W|_{(0,1]}$ is $\C^2$, $W$ is convex and $W'(r)\leq-\beta$, $\forall r\in(0,1)$.
Then $W$ is dispersing.
\end{lem}

\begin{remark}
It seems likely that the assumptions in Lemma \ref{genokscatterercondLEM} 
can be significantly relaxed:
We believe that $W$ is dispersing %
whenever $W|_{(0,1]}$ is $\C^2$, strictly decreasing, convex,
and satisfies $\lim_{r\to0}W(r)>\frac12$;
however at present we have no proof of this claim.

The constant $\beta$ appearing in the statement of Lemma \ref{genokscatterercondLEM} 
is in fact the smallest possible %
for the requirement that the integrand in
\eqref{genokscatterercondLEMpf1} be nonnegative
for all $w\in(0,1)$ and $r>r_0(w)$.
Indeed, this nonnegativity fails for any linear potential %
$W(r)=c\cdot I(r\leq1)\cdot (1-r)$
with slope $0<c<\beta$,
as one verifies by computations similar to those appearing in the proof below.
\end{remark}

\begin{proof}[Proof of Lemma \ref{genokscatterercondLEM}]
The assumptions on $W$ imply that $W(r)\geq \beta(1-r)$ $\forall r\in(0,1]$,
and so $r_0(w)$ extends to a strictly increasing $\C^1$ function on $[0,1)$ with $r_0(0)\geq1-(2\beta)^{-1}=0.298\ldots$,
satisfying \eqref{r0pformula} for all $w\in[0,1)$.
(This of course means that we only need the assumptions in
Lemma \ref{genokscatterercondLEM} to hold for $1-(2\beta)^{-1}<r<1$;
the behavior of the potential $W(r)$ for $r<1-(2\beta)^{-1}$ is completely irrelevant for our discussion.)
Also Lemma \ref{thetapLEM} applies with $\scrU=(0,1)$,
and so the formula \eqref{genokscatterercondLEMpf1} holds for all $0<w<1$.
The integrand in \eqref{genokscatterercondLEMpf1} 
decays like $r^{-3}$ as $r\to\infty$, uniformly with respect to $w\in[0,1)$;
furthermore the numerator of the integrand
is a $\C^1$ function of $(w,r)$ in $[0,1)\times(0,\infty)$,
vanishing for $r=r_0(w)$,
and regarding the denominator we note that
$\frac d{dr}(1-2W(r)-w^2r^{-2})\geq2\beta$ for all $w\in[0,1)$, $r\in(0,1)$.
From these facts it follows that the right hand side of 
\eqref{genokscatterercondLEMpf1} is well-defined for all $w\in[0,1)$,
and depends continuously on $w$ in this interval.
Hence $\theta$ extends to a $\C^1$ function on $[0,1)$.
Letting $w\to0$ in \eqref{deflangle} we see that this extended function
satisfies $\theta(0)=\pi$.
Recall also that $0\leq\theta(w)\leq\pi$ for all $w\in[0,1)$,
as we noted below \eqref{deflangle}.

Now to complete the proof of the lemma,
it suffices to prove that $\theta'(w)<0$ for all $w\in[0,1)$.
From now on we keep $w\in[0,1)$ fixed.
Note that the integrand in \eqref{genokscatterercondLEMpf1} is positive for all $r>1$;
hence it suffices to prove that the (numerator of the) integrand is nonnegative for $r_0(w)<r<1$,
i.e.\ that
\begin{align*}
\bigl(w^2r^{-2}+4W(r)+rW'(r)-2\bigr)r_0^4 W'(r_0)+r^2W'(r)\bigl(w^2-r_0^3W'(r_0)\bigr)
\geq0.
\end{align*}
Using $w^2=r_0^2(1-2W(r_0))$ and $W(r)\leq W(r_0)+W'(r)(r-r_0)$ we see that it suffices to prove:
\begin{align}\notag
r_0(4r_0-r)(r-r_0) W'(r)W'(r_0)
+\bigl(r_0^2(2-r_0^2r^{-2})W'(r_0)-r^2W'(r)\bigr)\bigl(2W(r_0)-1\bigr)
\\\label{convexpf1wtp3rep}
\geq0.
\end{align}
Next, noticing that 
\begin{align}\label{genokscatterercondLEMpf3}
\beta(1-r)-W'(r)(r-r_0)\leq W(r_0)\leq\tfrac12
\end{align}
and using the general fact that $c_1\leq x\leq c_2\Rightarrow a+bx\geq a+\min(bc_1,bc_2)$, %
we see that it suffices to prove that 
\eqref{convexpf1wtp3rep} holds when replacing $W(r_0)$ by $\frac12$ and when replacing 
$W(r_0)$ by $\beta(1-r)-W'(r)(r-r_0)$.
When replacing $W(r_0)$ by $\frac12$, 
\eqref{convexpf1wtp3rep} simplifies into
\begin{align*}
r_0(4r_0-r)(r-r_0) W'(r)W'(r_0)\geq0,
\end{align*}
which holds since $r_0<r<1$ and $r_0\geq r_0(0)>0.29$.
Next, replacing $W(r_0)$ by $\beta(1-r)-W'(r)(r-r_0)$,
\eqref{convexpf1wtp3rep} turns into
\begin{align}\label{convexpf1wtp4repNEW}
C\,W'(r_0)
+r^2\bigl(2(r-r_0)W'(r)+1-2\beta(1-r)\bigr)W'(r)\geq0.
\end{align}
where
\begin{align}\label{genokscatterercondLEMpf2}
C:=r_0r(2r_0^3r^{-3}-1)(r-r_0) W'(r)+r_0^2(2-r_0^2r^{-2})\bigl(2\beta(1-r)-1\bigr).
\end{align}
In order to prove \eqref{convexpf1wtp4repNEW} we first show that $C\leq0$.
If $2r_0^3r^{-3}-1\leq0$ then using
$(r-r_0)W'(r)\geq\beta(1-r)-\frac12$
(cf.\ \eqref{genokscatterercondLEMpf3})
we see that
$C\leq0$ holds provided that
$r_0r(2r_0^3r^{-3}-1)\frac12+r_0^2(2-r_0^2r^{-2})\geq0$,
viz., $\frac12r_0(4r_0-r)\geq0$,
and this holds since $r_0<r<1$ and $r_0\geq r_0(0)>0.29$.
Now assume $2r_0^3r^{-3}-1>0$.
Using $W'(r)\leq-\beta$ we then find that $C\leq0$ holds provided that
\begin{align}\label{genokscatterercondLEMpf4}
(r^4-5r^3r_0+2r_0^4+4r^2r_0-2r_0^3)\beta\leq 2r_0r^2-r_0^3.
\end{align}
If $r^4-5r^3r_0+2r_0^4+4r^2r_0-2r_0^3\leq0$ then \eqref{genokscatterercondLEMpf4} is automatic,
and if $r^4-5r^3r_0+2r_0^4+4r^2r_0-2r_0^3>0$ then 
\eqref{genokscatterercondLEMpf4}
is a consequence of
$\beta\leq(2(1-r_0))^{-1}$ (which follows from $\frac12\geq W(r_0)\geq\beta(1-r_0)$)
together with 
$r^2(r-r_0)(4r_0-r)\geq0$.

Having thus proved $C\leq0$,
we see using $W'(r_0)\leq W'(r)<0$ that \eqref{convexpf1wtp4repNEW} holds provided that
$C+r^2\bigl(2(r-r_0)W'(r)+1-2\beta(1-r)\bigr)\leq0$,
or equivalently,
\begin{align}\label{genokscatterercondLEMpf5}
(2r_0^4-r_0r^3+2r^4) W'(r)+(r+r_0)^2(r-r_0)\bigl(1-2(1-r)\beta\bigr)\leq0.
\end{align}
Using $W'(r)\leq-\beta$
one finds that \eqref{genokscatterercondLEMpf5} holds provided that
\begin{align}\label{genokscatterercondLEMpf6}
\beta\geq f(r_0,r):=\frac{(r+r_0)^2(r-r_0)}{2r_0^4-r_0r^3+2r^4+2(r+r_0)^2(r-r_0)(1-r)}.
\end{align}
(Note that the denominator is obviously positive.) %
It is our task to prove that \eqref{genokscatterercondLEMpf6} holds for all $r$ with $r_0<r<1$.
Expanding and simplifying $\frac{\partial}{\partial r}f(r_0,r)$ and using $(r+r_0)(8r_0^2-7rr_0+5r^2)>0$
one finds that $\frac{\partial}{\partial r}f(r_0,r)>0$ whenever $r\geq r_0>0$.
Hence it suffices to prove that $f(r_0,1)\leq\beta$ for all $r_0\in[r_0(0),1]$.
However we have
\begin{align*}
\frac{\partial}{\partial x}f(x,1)=\frac{(1+x)g(x)}{(2x^4-x+2)^2}
\qquad\text{with }\:
g(x)=2x^5+2x^4-8x^3+2x^2-7x+3.
\end{align*}
We have
$g(0)=3$, $g(1)=-6$
and $g'(x)<0$ for all $x\in[0,1]$;
hence $g$ has a unique zero in $[0,1]$;
by definition this is the number $\alpha=0.409\ldots$,
and $g(x)$ is positive for $x\in[0,\alpha)$ and negative for $x\in(\alpha,1]$.
Hence $f(x,1)\leq f(\alpha,1)=\beta$ for all $x\in[0,1]$,
and the proof is complete.
\end{proof}

\vspace{4pt}

While Lemma \ref{genokscatterercondLEM} gives an example of a simple general criterion
which ensures that the potential $W$ is dispersing,
let us note that there certainly exist other general classes of 'nice' potentials which are 
\textit{not} dispersing.
For example, recall that for $W$ dispersing, 
the range of $\theta(w)$ %
is an interval of length at most $\pi$,
while in fact there exist potentials $W$ for which $\theta(w)$ varies over an arbitrarily large portion
of the negative real axis,
meaning that the particle goes around the center of the scatterer many times
\cite[Sect.~5.4]{Newton82}.
Also the condition $\theta'(w)\neq0$ in Definition \ref{DISPERSINGDEF}
need not hold for general potentials $W$.

\section{More general scattering potentials}
\label{AppA2}

In this section we give an outline of how the main results of the present paper
may be extended %
to a more general class of spherically symmetric potentials.
Let $\theta(w)$ be the deflection angle, as in \eqref{deflangle}.
Our precise assumption will be the following:
\begin{align}\label{GENSCATASS}
&\textit{There exists an open subset $\scrU$ of $(0,1)$ of full Lebesgue measure}
\\\notag
&\textit{such that $\theta|_{\scrU}\in\C^1$ and $\theta'(w)\neq0$ for all $w\in\scrU$.}
\end{align}
It seems likely that this condition is fulfilled for \textit{generic} potentials
$W$ within several natural spaces of functions.
However note that there also exist non-trivial, 'nice', potentials $W$
for which \eqref{GENSCATASS} fails;
indeed this happens for the truncated Coulomb potential in \eqref{MUFFINTINW}
with $\alpha=-1$.

From now on we assume that \eqref{GENSCATASS} holds.
Note that then also the set $\{w\in\scrU\col \theta(w)\notin\pi\Z\}$
has full Lebesgue measure in $(0,1)$.
As this set is open, it can be expressed as a union
of a finite or countable family $\{I_\alpha\}_{\alpha\in\scrA}$ of pairwise disjoint open intervals.
By construction, for each $\alpha\in\scrA$,
$I_\alpha$ is an open sub-interval of $(0,1)$,
we have $\theta|_{I_\alpha}\in\C^1$,
$\theta'(w)$ has constant sign in $I_\alpha$,
and there is some $k_\alpha\in\Z$ such that
$\theta(w)\in(k_\alpha\pi,(k_\alpha+1)\pi)$ for all $w\in I_\alpha$.
Furthermore, %
$\sum_{\alpha\in\scrA}|I_\alpha|=1$,
where $|I_\alpha|$ denotes the length of $I_\alpha$.
Let us note that in the special case when $W$ is dispersing,
these conditions are fulfilled with $\scrA$ %
\textit{singleton}:
$\scrA=\{\alpha_0\}$ and $I_{\alpha_0}=(0,1)$.

For every $\alpha\in\scrA$ we set
\begin{align*}
\scrS_{\alpha,-}=\{(\vecv,\vecb)\in\US\times\US\col\vecv\cdot\vecb<0,\:\sin\varphi(\vecv,\vecb)\in I_\alpha\};
\end{align*}
this is an open subset of $\scrS_-$, and 
the family $\{\scrS_{\alpha,-}\}_{\alpha\in\scrA}$
is pairwise disjoint.
We let $\scrS'_-$ be the union of all the sets $\scrS_{\alpha,-}$;
this is an open set of full measure (wrt.\ $\omega\times\omega$) in $\scrS_-$.
The formulas
\eqref{PSI1formula} and \eqref{PSI2formula}
define a $\C^1$ map $\Psi=(\Psi_1,\Psi_2):\scrS'_-\to\scrS_+$
satisfying conditions (i) and (ii) in Section \ref{SCATTERINGMAPS}.
For each $\alpha\in\scrA$ and $\vecv\in\US$ we also set:
\begin{align*}
\scrS_{\vecv,\alpha,-}:=\{\vecb\in\US\col(\vecv,\vecb)\in\scrS_{\alpha,-}\}
\end{align*}
and
\begin{align*}
\scrV_{\vecv,\alpha}:=\{\Psi_1(\vecv,\vecb)\col\vecb\in\scrS_{\vecv,\alpha,-}\}.
\end{align*}
Both these are open subsets of $\US$.
By a simple modification of the proof of Lemma \ref{DEFLANGLElem},
using the fact that $\theta'$ has constant sign on $I_\alpha$,
we have:
\begin{align}\label{CONDiiiREPLrep}
\left\{\text{\parbox{250pt}{For each $\alpha\in\scrA$ and $\vecv\in\US$,
the map $\Psi_1(\vecv,\cdot)$ is a $\C^1$ diffeomorphism from
$\scrS_{\vecv,\alpha,-}$ onto 
$\scrV_{\vecv,\alpha}$.}}\right.
\end{align}
Let us write
\begin{align*}
\vecbeta_{\vecv,\alpha}^-:\scrV_{\vecv,\alpha}\to\scrS_{\vecv,\alpha,-}
\end{align*}
for the inverse diffeomorphism. %
We also set
\begin{align*}
\vecbeta_{\vecv,\alpha}^+(\vecu):=\Psi_2(\vecv,\vecbeta_{\vecv,\alpha}^-(\vecu))
\qquad(\vecv\in\US,\:\vecu\in\scrV_{\vecv,\alpha}).
\end{align*}
Both $\vecbeta_{\cdot,\alpha}^-$
and $\vecbeta_{\cdot,\alpha}^+$ are spherically symmetric
in the sense that $\vecbeta_{\vecv K,\alpha}^{\pm}(\vecu K)=\vecbeta_{\vecv,\alpha}^{\pm}(\vecu)K$
for all $K\in\SO(d)$.
This implies in particular that both functions $\vecbeta_{\vecv,\alpha}^{\pm}(\vecu)$
are jointly $C^1$ in $\vecv,\vecu$.

In the present general setting, 
the differential cross section is given by
\begin{align*}
\sigma(\vecv,\vecv_+)=\sum_{\alpha\in\scrA}\sigma_\alpha(\vecv,\vecv_+)
\qquad(\vecv,\vecv_+\in\US),
\end{align*}
where
\begin{align*}
\sigma_\alpha(\vecv,\vecv_+)=\begin{cases}
{\displaystyle|\theta'(w)|^{-1}\Bigl|\frac w{\sin\theta(w)}\Bigr|^{d-2}
\quad\text{with }\:
w=\bigl\|\bigl(\vecbeta_{\vecv,\alpha}^-(\vecv_+)R(\vecv)\bigr)_\perp\bigr\|}
\\
\rule{0pt}{0pt}\hspace{225pt}
\text{if }\:\vecv_+\in\scrV_{\vecv,\alpha};
\\[3pt]
0\hspace{220pt}\text{if }\:\vecv_+\notin\scrV_{\vecv,\alpha}.
\end{cases}
\end{align*}
Thus $\sigma$ is a function on $\US\times\US$ taking values in $\R_{\geq0}\cup\{+\infty\}$.
As before we have
$\int_{\US}\sigma(\vecv,\vecv_+)\,d\vecv_+=v_{d-1}$,
implying that $\sigma$ is almost everywhere finite.
More generally,
for any $\vecv\in\US$ and any bounded, Borel measurable function $f:\US\to\R$,
we have
\begin{align}\label{gensigmakeyformula}
\int_{\UB}f\bigl(\Psi_1(\vece_1,s_-(\vecw))R(\vecv)^{-1}\bigr)\,d\vecw
=\int_{\US}f(\vecv_+)\,\sigma(\vecv,\vecv_+)\,d\vecv_+.
\end{align}

In our present setting, since the incoming and outgoing velocities $\vecv_{\pm}$ in a scatterer collision
do not in general determine the impact parameter uniquely,
in order for the limiting 
joint distribution of the first $n$ flight segments and scatterer marks
to be a a finite-memory Markov process,
we will also keep track of the index $\alpha$ such that the
impact parameter belongs to $I_\alpha$.
It turns out to be natural to lump this index %
together with the marking of the scatterer,
thus forming an element $\chi=(\vs,\alpha)$ in the space
$\Sigma_\scrA:=\Sigma\times\scrA$.
We equip $\Sigma_\scrA$ with the measure $\mm_\scrA:=\mm\times c_\scrA$,
where $c_\scrA$ is the counting measure on $\scrA$.
We use the letters $\vs$ and $\alpha$ also to denote
the projection maps from $\Sigma_\scrA$ to $\Sigma$ and $\scrA$,
respectively;
thus $\chi=(\vs(\chi),\alpha(\chi))$ for all $\chi\in\Sigma_\scrA$.

We modify the definitions appearing at the 
end of Sec.\ \ref{SCATTERINGMAPS} very slightly,
by letting $\fw(j;\rho)$ be the subset of points
$(\vecq_0,\vecv_0)\in\fw(j-1;\rho)$
for which $\tau_j<\infty$, $\vecq_{j-1}+\tau_j\vecv_{j-1}$
lies on the boundary of a separated scatterer,
\textit{and $\|\vecw_j\|\in\cup_\scrA I_\alpha$.}
We then let $\alpha_j=\alpha_j(\vecq_0,\vecv_0;\rho)$ be the unique index $\alpha$ for which
$\|\vecw_j\|\in I_\alpha$,
and set $\chi_j=\chi_j(\vecq_0,\vecv_0;\rho)=(\vs_j,\alpha_j)\in\Sigma_\scrA$.
The sets
$\fw_{\vecq,\rho,j}^{\vecbeta}$ and
$\fW(j;\rho)$ are still defined by \eqref{fwqrhonDEF} and \eqref{fWndef}, 
but using the new definition of $\fw(j;\rho)$.

The definitions of the collision kernels in Sec.\ \ref{COLLKERsec}
(cf.\ \eqref{pbndefG} and \eqref{pgendef})
are generalized as follows:
For any
$\xi>0$, $\chi,\chi_+\in\Sigma_\scrA$,
and $\vecv_0,\vecv,\vecv_+\in\US$, 
we set
\begin{align}\notag
p_\bn\bigl(&\vecv_0, \chi,\vecv;\xi,\chi_+,\vecvp\bigr)
\\\notag %
&=
\frac{\sigma_{\alpha(\chi_+)}(\vecv,\vecvp)}{v_{d-1}}
k\Bigl(\bigl(\vecbeta_{\vecv_0R(\vecv),\alpha(\chi)}^+(\vece_1)_\perp,\vs(\chi)\bigr),
\xi,\bigl(\vecbeta_{\vece_1,\alpha(\chi_+)}^-(\vecvp R(\vecv))_\perp,\vs(\chi_{+})\bigr)\Bigr).
\end{align}
if $\vecv\in\scrV_{\vecv_0,\alpha(\chi)}$,
$\vecv_+\in\scrV_{\vecv,\alpha(\chi_+)}$,
and otherwise
$p_\bn\bigl(\vecv_0,\chi,\vecv;\xi,\chi_+,\vecvp\bigr)=0$.
For $U$ an open subset of $\US$, 
$\vecbeta\in \C_b(U,\R^d)$,
$\xi>0$, $\vs\in\Sigma$, $\chi_+\in\Sigma_\scrA$, we set
\begin{align}\notag
p_{\bn,\vecbeta}\bigl({\vs},\vecv; & \xi,\chi_+,\vecvp\bigr)=
\\\notag %
&=\frac{\sigma_{\alpha(\chi_+)}(\vecv,\vecvp)}{v_{d-1}}
k\Bigl(\bigl((\vecbeta(\vecv)R(\vecv))_\perp,{\vs}\bigr),
\xi,\bigl(\vecbeta_{\vece_1,\alpha(\chi_+)}^-(\vecvp R(\vecv))_\perp,\vs(\chi_+)\bigr)\Bigr)
\end{align}
if $\vecv\in{U}$ and $\vecvp\in\scrV_{\vecv,\alpha(\chi_+)}$, and otherwise
$p_{\bn,\vecbeta}\bigl({\vs},\vecv;\xi,\chi_+,\vecvp\bigr)=0$.
We then have
\begin{align}\notag %
p_\bn\bigl(\vecv_0,\chi,\vecv;\xi,\chi_+,\vecvp\bigr)\equiv
p_{\bn,\vecbeta_{\vecv_0,\alpha(\chi)}^+}\bigl({\vs(\chi)},\vecv;\xi,\chi_+,\vecvp\bigr).
\end{align}
Similarly we set %
\begin{align}\notag
p\bigl(\vecv;\xi,\chi_+,\vecvp\bigr)=\frac{\sigma_{\alpha(\chi_+)}(\vecv,\vecvp)}{v_{d-1}}
\,k^{\g}\Bigl(\xi,\bigl(\vecbeta_{\vece_1,\alpha(\chi_+)}^-(\vecvp R(\vecv))_\perp,\vs(\chi_+)\bigr)\Bigr)
\end{align}
if $\vecv_+\in\scrV_{\vecv,\alpha(\chi_+)}$,
and otherwise $p\bigl(\vecv;\xi,\chi_+,\vecvp\bigr)=0$.

We now describe the generalizations of the main theorems in Section \ref{MAINRESsec}.
We replace the definition of $X^{(n)}_U$ (cf.\ \eqref{XSPACEdef}) by
\begin{align}\notag
X_U^{(n)}:=\Bigl\{\langle\vecv_0;\langle\xi_j,\chi_j,\vecv_j\big\rangle_{j=1}^{n}\rangle\in
U\times(\R_{>0}\times\Sigma_{\scrA}\times\US)^n
\col \hspace{70pt}
\\\notag
\vecv_j\in\scrV_{\vecv_{j-1},\alpha(\chi_j)}\: %
(j=1,\ldots,n)\Bigr\}.
\end{align}
Using our slightly modified notation, %
Theorem \ref{MAINTECHNTHM2A}
carries over almost verbatim 
to the present %
situation:
\begin{thm}\label{MAINTECHNthm2agen}
Let $\scrP$ satisfy all the conditions in Section \ref{ASSUMPTLISTsec} and \eqref{SIGMAeqp},
and let $\Psi$ be a scattering process arising as described above.
Let $n\in\Z_{\geq1}$ and $T\in\R_{\geq1}$;
let $U$ be an open subset of $\US$; %
let $F_1$ be an equismooth family of probability measures on $\S_1^{d-1}$ 
such that $\lambda(U)=1$ for each $\lambda\in F_1$;
let $F_2$ be a uniformly bounded and pointwise equicontinuous family
of functions $f:{X_U^{(n)}}\to\R$; %
and let $F_3$ be an admissible subset of $\C^1_b(U,\R^d)$.
Then
\begin{align}\notag
\int_{\fw_{\vecq,\rho,n}^{\vecbeta}}f\Bigl(\vecv,\Big\langle\rho^{d-1}\tau_j(\vecq_{\rho,\vecbeta}(\vecv),\vecv;\rho)
,\chi_j(\vecq_{\rho,\vecbeta}(\vecv),\vecv;\rho),
\vecv_j(\vecq_{\rho,\vecbeta}(\vecv),\vecv;\rho)\Big\rangle_{j=1}^{n}\Bigr)\,d\lambda(\vecv)
\hspace{1pt}
\\\label{unifmodThm2agenres}
-\int_{{X_U^{(n)}}} %
f\bigl(\vecv_0;\big\langle\xi_j,{\chi}_j,\vecv_j\big\rangle_{j=1}^{n}\bigr)
p_{\bn,\vecbeta}\bigl(\vs(\vecq),\vecv_0;\xi_1,{\chi}_1,\vecv_1\bigr)
\hspace{85pt}
\\\notag
\times\prod_{j=2}^n p_\bn(\vecv_{j-2},{\chi}_{j-1},\vecv_{j-1};\xi_j,{\chi}_{j},\vecv_{j})
\, d\lambda(\vecv_0)\,\prod_{j=1}^n\bigl( d\xi_j\,d\mm_{\scrA}({\chi}_j)\,d\vecv_j\bigr)\to0
\end{align}
as $\rho\to0$,
uniformly with respect to all $\vecq\in\scrP_T(\rho)$, $\lambda\in F_1$, $f\in F_2$, $\vecbeta\in F_3$.
\end{thm}
Theorem \ref{MAINTECHNthm2agen} can be proved by following the %
arguments in Sections \ref{COLLKERsec}--\ref{MAINRESMACRO} fairly closely.
We here give a brief description of the most important modifications required:
First, instead of
\eqref{WVPDEF}, we now set
\begin{align*}
\scrV_{\vecv,\alpha}^\eta:=\begin{cases}
\scrV_{\vecv,\alpha}\setminus\overline{\partial_\eta(\scrV_{\vecv,\alpha})}
&\text{if }\: |I_\alpha|>\eta
\\
\emptyset &\text{if }\: |I_\alpha|\leq\eta
\end{cases}
\qquad
(\vecv\in\US,\:\alpha\in\scrA,\:\eta>0).
\end{align*}
Note that for given $\eta$ there are only finitely many $\alpha$ with 
$\scrV_{\vecv,\alpha}^\eta\neq\emptyset$;
this is a crucial point for several steps in the proof.
The definition 
\eqref{FUETA} is replaced by
\begin{align*}
\fU_\eta:=\UB\setminus\bigcup_{\alpha\in\scrA}\vecbeta_{\vece_1,\alpha}^-(\scrV_{\vece_1,\alpha}^{10\eta})_\perp.
\end{align*}
Thus each $\fU_\eta$ is a union of a finite number of annuli centered at the origin,
and $\vol(\fU_\eta)\to0$ as $\eta\to0$.
We define $\fg_{\vecq,\rho,\eta}^{\vecbeta}$ exactly as in Definition \ref{fgDEF}
(using our new $\fU_\eta$);
then %
Proposition~\ref{ETAGRAZINGprop} remains true; %
similarly Definition \ref{fGrhoetaDEF} and Proposition \ref{ETAGRAZINGgenprop} extend to the present situation.
In Section \ref{scatmapsmore}, the definition of $\nu_{\vecs}$
in \eqref{LAMBDA0DEFnew} is replaced by the following:
For any $\vecs\in\R^d\setminus\{\bn\}$
and $\alpha\in\scrA$, we let $\nu_{\vecs,\alpha}$
be the probability measure on $\US$ given by
\begin{align*}
d\nu_{\vecs,\alpha}(\vecv)=\frac1{\vol((\scrS_{\hs,\alpha,-})_\perp)}\,\sigma_{\alpha}(\hs,\vecv)\,d\vecv.
\end{align*}
Note that $\nu_{\vecs,\alpha}$ is supported on $\scrV_{\hs,\alpha}$.
Also for $\eta>0$ so small that $\scrV_{\hs,\alpha}^\eta\neq\emptyset$,
we define
$\nu_{\vecs,\alpha}^\eta:=\nu_{\vecs,\alpha}(\scrV_{\hs,\alpha}^\eta)^{-1}\cdot\nu_{\vecs,\alpha}\big|_{\scrV_{\hs,\alpha}^\eta}$
(cf.\ \eqref{nuseta}).
Now Lemma \ref{SHINELEM1} extends to the present situation;
the difference is that each of (i), (ii), (iii) in Lemma \ref{SHINELEM1} is now a statement which
holds for every $\alpha\in\scrA$
such that $\scrV_{\hs,\alpha}^\eta\neq\emptyset$;
for example (i) now says that for every such $\alpha$,
if $\overline{\vecV_{\!\!\alpha}}=\overline{\vecV}_{\!\!\rho,\vecs,\vecbeta,\alpha}$ is the 
restriction of $\vecV=\vecV_{\!\!\rho,\vecs,\vecbeta}$ to 
$\vecB_{\rho,\vecs,\vecbeta}^{-1}(\scrS_{\hs,\alpha,-})\cap\vecV_{\!\!\rho,\vecs,\vecbeta}^{-1}(\scrV_{\hs,\alpha}^\eta)$,
then $\overline{\vecV_{\!\!\alpha}}$ is a $\C^1$ diffeomorphism onto $\scrV_{\hs,\alpha}^\eta$.

Turning to the proof in Section \ref{MAINTECHNthm2apfSEC},
the definition of $\nu_{\ell,\vecs}$,
\eqref{LAMBDALSdef},
is now replaced by:
\begin{align}\notag %
\nu_{\ell,\vecs,\alpha}=\nu_{\vecs,\alpha}(D_\ell)^{-1}\cdot\nu_{\vecs,\alpha}\big|_{D_\ell},
\end{align}
for any $\ell\in\{1,\ldots,N\}$ and $\alpha\in\scrA$ with $\nu_{\vecs,\alpha}(D_\ell)>0$;
and in place of \eqref{Aelldef} and \eqref{F1elldef} we now set
$A_{\ell,\alpha}=\{\vecs\in\US\col D_{\ell}\subset\scrV_{\vecs,\alpha}^{5\eta}\}$
and $F_{1,\ell}=\{\nu_{\ell,\vecs,\alpha}\col\alpha\in\scrA,\:\vecs\in A_{\ell,\alpha}\}$.
Note that $F_{1,\ell}$ is still an equismooth family of probability measures,
for each $\ell$.
Furthermore, \eqref{F2elldef} is replaced by
\begin{align}\notag
F_{2,\ell}:=\bigl\{f_{[\vecv_0,\xi_0,\chi_0]}\col\xi_0>0,\:\chi_0\in\Sigma_{\scrA},\:
\vecv_0\in U\cap A_{\ell,\alpha(\chi_0)}
\bigr\};
\end{align}
this is again a uniformly bounded and equicontinuous family of functions on $X_{D_\ell}^{(n-1)}$.
A bit further down, the definition of $U_2$, \eqref{U2def},
now takes the form:
\begin{align}\notag
U_2:=\Bigl\{\vecv\in U_1\cap\fw_{\vecq,\rho,1}^{\vecbeta}
\col C_1^{-1}<\rho^{d-1}\tau_1(\vecv)<C_1,
\:\:\:
\vecq^{(1)}(\vecv)\in\scrP\setminus\scrE,
\hspace{70pt}
\\\notag
[\vecv_1(\vecv)]\subset\scrV_{\widehat{\vecs_1(\vecv)},\alpha_1(\vecv)}^{5\eta},
\text{  and }
\bigl[\forall\vecu\in[\vecv_1(\vecv)]\col\exists\vecv'\in U_1\cap\fw_{\vecq,\rho,1}^{\vecbeta}\text{ such that }
\hspace{30pt}
\\\notag
\vecq^{(1)}(\vecv')=\vecq^{(1)}(\vecv),\:
\alpha_1(\vecv')=\alpha_1(\vecv)
\text{ and }\vecv_1(\vecv')=\vecu\bigr]\Bigr\}.
\end{align}
(Here $\alpha_1(\vecv):=\alpha_1(\vecq_{\rho,\vecbeta}(\vecv),\vecv;\rho)$,
just as $\vecv_1(\vecv):=\vecv_1(\vecq_{\rho,\vecbeta}(\vecv),\vecv;\rho)$
and $\tau_1(\vecv):=\tau_1(\vecq_{\rho,\vecbeta}(\vecv),\vecv;\rho)$.)
With this, Lemma \ref{GOODCOLLlem} now carries over to our situation.
Finally, \eqref{MAINTECHNthm2aPF3a} now reads
\begin{align*}
\sum_{\vecq'\in\scrP_{T_1}(\rho)}\:\sum_{\langle\ell,\alpha\rangle\in M(\vecq')}
\:\int_{U_{\vecq',\ell,\alpha}\cap\fw_{\vecq,\rho,n}^{\vecbeta}}
f\bigl(\vecv,\big\langle\rho^{d-1}\tau_j(\vecv),{\chi}_j(\vecv),
\vecv_j(\vecv)\big\rangle_{j=1}^{n}\bigr)\,d\lambda(\vecv),
\end{align*}
where now
\begin{align*}
M(\vecq')=\bigl\{\langle\ell,\alpha\rangle\col
\tD_\ell\subset\scrV_{\vecs_1,\alpha}^{5\eta}
\:\:\text{ and }\:\:
[\forall\vecu\in\tD_\ell\col\exists\vecv'\in U_1\cap\fw_{\vecq,\rho,1}^{\vecbeta}\text{ such that }
\hspace{30pt}
\\\notag %
\vecq^{(1)}(\vecv')=\vecq',\:\alpha_1(\vecv')=\alpha\text{ and }\vecv_1(\vecv')=\vecu]\bigr\}.
\end{align*}
and
\begin{align*}
U_{\vecq',\ell,\alpha}:=\{\vecv\in U_1\cap\fw_{\vecq,\rho,1}^{\vecbeta}\col\vecq^{(1)}(\vecv)=\vecq'
,\:\alpha_1(\vecv)=\alpha
,\:\vecv_1(\vecv)\in \tD_\ell\}.
\end{align*}
With this setup in place, 
the remaining part of the proof of Theorem \ref{MAINTECHNTHM2A} carries over in
a fairly straightforward manner.

Next, the generalization of Theorem \ref{MAINTECHNthm2G} is as follows.
We define $X^{(n)}$
(cf.\ \eqref{Xmacrdef}) by:
\begin{align}\notag
X^{(n)}:=\Bigl\{\langle\vecq,\vecv_0,\langle\xi_j,\chi_j,\vecv_j\big\rangle_{j=1}^{n}\rangle\in
\T^1(\R^d)\times(\R_{>0}\times\Sigma_\scrA\times\US)^n
\col
\hspace{50pt}
\\\label{Xmacrdefgen}
\vecv_j\in\scrV_{\vecv_{j-1},\alpha(\chi_j)}\: %
(j=1,\ldots,n)\Bigr\}.
\end{align}
Hence in particular, we now have,
in place of \eqref{XXdef2}:
\begin{align}\label{XXdef2gen}
X=X^{(1)}=\big\{ \big\langle\vecq,\vecv,\xi,\chi,\vecv_+\big\rangle \in \T^1(\R^d)\times\R_{>0}\times\Sigma_\scrA\times\US \col 
\vecv_+\in\scrV_{\vecv,\alpha(\chi)} \big\}.
\end{align}
\begin{thm}\label{MAINTECHNthm2Ggen}
Let $\scrP$ {\blu and $\scrE$} satisfy all the conditions in Section \ref{ASSUMPTLISTsec} and \eqref{SIGMAeqp},
and let $\Psi$ be a scattering process arising as described above.  %
Then for any $n\geq1$, $\Lambda\in\Pac(\T^1(\R^d))$ and $f\in\C_b(X^{(n)})$, we have
\begin{align}\notag
\lim_{\rho\to0}\int_{\fW(n;\rho)}f\Bigl(\vecq,\vecv,
\Big\langle\rho^{d-1}\tau_j(\rho^{1-d}\vecq,\vecv;\rho),\chi_j(\rho^{1-d}\vecq,\vecv;\rho),
\vecv_j(\rho^{1-d}\vecq,\vecv;\rho)
\Big\rangle_{j=1}^n\Bigr)
\hspace{15pt}
\\\label{MAINTECHNthm2Ggenres}
\times d\Lambda(\vecq,\vecv)
\\\notag
=\int_{X^{(n)}} %
f\Bigl(\vecq,\vecv_0,\big\langle \xi_j,\chi_j,\vecv_j\big\rangle_{j=1}^n\Bigr)
\,p\bigl(\vecv_0;\xi_1,\chi_1,\vecv_1\bigr)
\hspace{120pt}
\\\notag
\times \prod_{j=2}^n p_\bn(\vecv_{j-2},\chi_{j-1},\vecv_{j-1};\xi_j,\chi_{j},\vecv_{j})
\, d\Lambda(\vecq,\vecv_0)\prod_{j=1}^n\bigl( d\xi_j\,d\mm_\scrA(\chi_j)\,d\vecv_j\bigr).
\end{align}
\end{thm}

\vspace{5pt}

Using Theorem \ref{MAINTECHNthm2Ggen} and mimicking the discussion in Section \ref{LIMITFLIGHTsec},
one proves that Theorem~\ref{thm:M1V} extends verbatim to the present setting.
The explicit description of the limiting random flight process $\Theta$ 
remains the same as given in the beginning of Section \ref{LIMITFLIGHTsec},
with the only difference that the definition of 
$X^{(\infty)}$ in \eqref{Xinftydef}
is replaced by
\begin{multline*}
X^{(\infty)}:=\Bigl\{\langle\vecq_0,\vecv_0,\langle\xi_j,\chi_j,\vecv_j\big\rangle_{j=1}^{\infty}\rangle\in
\T^1(\R^d)\times\prod_{j=1}^\infty(\R_{>0}\times\Sigma_\scrA\times\US)
\col\:\: \\
\vecv_j\in\scrV_{\vecv_{j-1},\alpha(\chi_j)},\:\forall j\geq1\Bigr\},
\end{multline*}
and in \eqref{nuLambdadef} we use $\mm_\scrA$ in place of $\mm$.

Also the definition of $\hTheta$ from the beginning of
Section \ref{KINETICEQsec} carries over immediately,
recalling that the extended phase space $X$ is now given by \eqref{XXdef2gen},
and equipped with the measure $d\vecq\,d\vecv\, d\xi\,d\mm_\scrA(\chi)\,d\vecv_+$,
The rest of Section \ref{KINETICEQsec} carries over in the obvious way.
The forward Kolmogorov equation for $\hTheta$ now reads:
\begin{multline}
\bigl(\partial_t+\vecv\cdot\nabla_\vecq-\partial_\xi\bigr)f(t,\vecq,\vecv,\xi,\chi,\vecvp) \\
= \int_{\Sigma_\scrA\times\US} f(t,\vecq,\vecv_0,0,\chi',\vecv) \, p_\bn(\vecv_0,\chi',\vecv;\xi,\chi,\vecv_+)\,d\mm_\scrA(\chi')\,d\vecv_0.
\end{multline}

\section{Open questions}

\label{sec:open}

In addition to our main hypotheses [P1-3] and [Q1-3] on the scatterer configuration $\scrP$, key assumptions in the present study are that all scatterers are identical, spherically symmetric and finite range, and that there are no external force fields. Furthermore very little is known, except in special examples, on the limiting Markov processes we have derived. This section provides a brief survey of some of the remaining challenges.

\subsection{Admissible scatterer configurations}\label{sec:openScat}

The discussion in Section \ref{PoissonSEC} is restricted to realisations $\scrP$ of Poisson processes with constant intensity. 
It would be interesting to extend the discussion to more general randomly generated sets, for example Gibbs point processes, determinantal point processes and cluster processes. A particularly simple example is the process studied in \cite{Baddeley84}, where $\RR^d$ is partitioned into unit cubes, and  with a random number of points distributed uniformly and independently in each cube. In all of these examples, we expect the spherical average \eqref{ASS:KEY} in assumption [P2] to converge to a Poisson process, and hence the Boltzmann-Grad limit to be given by the linear Boltzmann equation. 
 
An immediate challenge is to extend the discussion in Sections \ref{PoissonSEC}--\ref{QCexsec} to unions of the point sets considered there. For example, take $\scrP=\scrP_1\cup\scrP_2$, where $\scrP_1$ is a fixed realisation of a Poisson point process with intensity $c$ (as in Section \ref{PoissonSEC}), and $\scrP_2$ is a fixed full-rank lattice in $\RR^d$ of covolume $\delta$. In this case we expect all hypotheses to be satisfied, with $\Sigma=\{ 1,2 \}$ as the space of marks, labelling the points from  $\scrP_1$ and $\scrP_2$, respectively. The limiting process is the union of two independent marked point processes, a Poisson point process with intensity $c$ (whose points are marked 1), and a random lattice of covolume $\delta$ (whose points are marked 2). The independence of the two processes will imply a rather simple formula for the transition and collision kernels in terms of the corresponding kernels for the limiting processes for the Lorentz gases with configurations $\scrP_1$ and $\scrP_2$, respectively (cf.\ \cite{powerlaw}). The same should go through if $\scrP_2$ is taken to be a periodic point set (as in Section \ref{periodicpointsetsSEC}) or a quasicrystal (as in Section \ref{QCexsec}). A slightly different challenge is to understand the case when $\scrP_1$ and $\scrP_2$ are both full-rank Euclidean lattices. If the two lattices are incommensurate, the paper \cite{powerlaw} shows that the limit process is the union of two independent random lattices, thus establishing condition [P2] -- however without the required uniformity in $\vecq$. 

Another class of examples $\scrP$  is obtained by ``thinning'' an existing scatterer configuration $\scrP_0$. That is, for $0<p<1$, remove each point in $\scrP_0$ independently with probability $p$, and consider $\scrP$ as a realisation of the resulting random point set. In this case, the paper \cite{MarklofVinogradov17} establishes condition [P2] almost surely, again without the required uniformity in $\vecq$. It should be an interesting exercise to prove all necessary assumptions in Section \ref{ASSUMPTLISTsec} in this setting. 

In a similar vein, one might ask whether the Boltzmann-Grad limit exists for a Lorentz process where $\scrP$ is the set of primitive lattice points. Using the fact that this set can be realised as an adelic cut-and-project set, El-Baz \cite{El-Baz17} proved that assumption [P2] holds, again with no uniformity in $\vecq$. The subtlety in the problem is that the window set required in the cut-and-project construction has empty interior. It would be interesting to establish the analogous results of Section \ref{QCexsec} in this more singular setting.

\subsection{Necessity of hypotheses and $\SL(d,\RR)$-invariance}

One might ask whether any of our assumptions [P1-3] and [Q1-3] on admissible scatterer configurations can be weakened, or even dropped completely. If the Boltzmann-Grad limit does not exist for a given $\scrP$, can one at least establish convergence along subsequences under appropriate hypotheses? It is natural to also consider sequences of point sets $\scrP=\scrP_\rho$, for example modelling the case of polycrystals \cite{polycrystals}, and it would be interesting to extend our theory to this case.
Assumption [Q1] stipulates that the limit measure $\mu_{\vs}$ is $\SO(d-1)$-invariant, and we have noted that $\mu_{\vs}$ is necessarily invariant under the diagonal group $\{D_r\}_{r>0}$ (Lemma \ref{DIAGINVlem}). All examples discussed in this paper however enjoy the significantly stronger property that $\mu_{\vs}$ is invariant under the action of $\SL(d,\RR)$. So --- are there any $\scrP$ for which $\mu_{\vs}$ is {\em not} $\SL(d,\RR)$-invariant? 
Note that there are simple examples of point sets $\scrP$ with constant density for which the spherical average $(\scrP-\vecq)\,R(\vecv)\,D_\rho$, for some fixed $\vecq\in\scrP$, does not converge to a $\SL(d,\RR)$-invariant limit process. %
The challenge here is to find examples for which we have convergence, but no $\SL(d,\RR)$-invariance, for a positive density of $\vecq\in\scrP$.

\subsection{Non-spherically symmetric scatterers} 

There is no principal obstruction for our approach to be generalized to non-spherically symmetric scatterers, 
{\blu under suitable conditions on smoothness and invertibility 
(excluding, for instance, polytopal scatterers).}
The extension to a Lorentz gas with identical scatterers given by hard ellipses and more general strictly convex bodies should be relatively straightforward. The transition kernel $k$ defining the limit process will now depend on the direction of travel $\vecv$, as the size of the cross section is given by the projection of the elliptical scatterer onto the hyperplane perpendicular to $\vecv$. Identify the hyperplane with $\RR^{d-1}$ as before, and denote the projection of the scatterer by $\scrE_\vecv^{d-1}$. The only modification required in the definition of the limit process is to replace the cylinder $\fZ_\xi=(0,\xi)\times\scrB_1^{d-1}$ in the definition of the transition and collision kernels by $\fZ_\xi^\vecv=(0,\xi)\times\scrE_\vecv^{d-1}$. The strict convexity of the scatterer is critical for our theory to work. Polyhedral scatterers (as in the Ehrenfest wind-tree model) would not lead, in the Boltzmann-Grad limit, to a finite-memory Markov process. We refer the reader to \cite{Bachurin11} for a simple model of this phenomenon. An important outstanding task is to generalise the present work to non-radial potentials, still assuming compact support (cf.~the extension to long-range potentials below).

\subsection{Non-identical scatterers}

In Section \ref{sec:openScat} we mentioned scatterer configurations $\scrP$ that are finite unions of point sets $\scrP_1,\scrP_2,\ldots$. To model crystals such as NaCl (where $\scrP_1=\ZZ^3$ and $\scrP_2=\ZZ^3+(\tfrac12,\tfrac12,\tfrac12)$), or lattices with impurities (where $\scrP_1$ is a lattice and $\scrP_2$ a realisation of a Poisson point process, say), it is natural to assume that the scatterers centered at $\scrP_i$ are not the same as those in $\scrP_j$. In this case the necessary modifications in the limit process will include of course different cross sections for the scatterers from different families, and transition kernels that will take into account the varying scattering radii. This requires replacing the single cylinder $\fZ_\xi=(0,\xi)\times\scrB_1^{d-1}$ by cylinders $\fZ_\xi^i=(0,\xi)\times\scrB_{r_i}^{d-1}$, each corresponding to scatterers from $\scrP_i$. 

\subsection{Long-range potentials}

The assumption that the scattering potentials have compact support is central to our approach. Even an extension of our results to exponentially decaying potentials will require some non-trivial estimates. The case of potentials with power-law decay 
{\blu has currently only been investigated}
in the case of random scatterer configurations \cite{Desvillettes99,Nota18}.
A related problem is to consider different scaling limits for compact potentials, where the strength of the potential is reduced, and at the same time the scatterer density rescaled suitably to achieve a non-trivial limit. In this case grazing collisions become important, and one expects a different kinetic equation for the macroscopic dynamics. 
See \cite{kp80,DGL87,Desvillettes01} for the corresponding result for a random scatterer configuration---here the limiting kinetic equation is the classical Fokker-Planck equation.

\subsection{External force fields}

A key feature of our approach is the {\em linear rescaling} by the subgroup $\{D_r\}_{r>0}$ of particle trajectories between collisions. This clearly breaks down when the trajectories are curved due to the presence of an external force field. Progress in this non-linear setting has so far been limited to random scatterer configurations \cite{Desvillettes04,Bobylev01,Marcozzi16}, and any extension of these results to lattices or quasicrystals would be a significant achievement.

\subsection{Transition kernels and distribution of free path lengths}

We have seen above that in the case of a Poisson scatterer configuration $\scrP$, the limiting transition kernels are explicit, with an exponential path length distribution. The only other case where we currently have explicit formulas is when $\scrP$ is a Euclidean lattice in dimension $d=2$ \cite{Caglioti08,partIII}.
{\blu For higher dimensional Euclidean lattices there are no explicit formulas, but we have precise asymptotics for the transition kernels \cite{partIV}.
These asymptotics in particular imply precise power-law asymptotics for the free path length distributions;
thus for example, for $\scrP$ a covolume one Euclidean lattice in arbitrary dimension $d\geq2$,
the limiting distribution of the free path length between consecutive collisions
satisfies $\overline{\Phi}_0(\xi)=\frac{2^{2-d}}{d(d+1)\zeta(d)}\xi^{-3}\bigl(1+O(\xi^{-\frac2d}\log\xi)\bigr)$ as $\xi\to\infty$
\cite{partIV}. 
We remark that the asymptotics for the limiting transition kernels 
for $\scrP$ a Euclidean lattice also play an important role
in the derivation of the 
long-time asymptotics of the limiting random flight processes $t\mapsto\Theta(t)$;
cf.\ \cite{MarklofToth} and the next paragraph.}
It would be extremely interesting to extend these results to other scatterer configurations.

\subsection{Diffusion vs.\ superdiffusion, entropy estimates}

Can we characterise the long-time asymptotics of the limiting random flight processes $t\mapsto\Theta(t)$? Do they converge to Brownian motion under an appropriate rescaling? The only known affirmative answers to these questions are in the case of random $\scrP$ \cite{Spohn78,Basile14,Basile15,LutskoToth}, where the mean-square displacement is linear in $t$ (diffusion), and lattice configurations $\scrP$ \cite{MarklofToth}, where we have a $t\log t$ scaling (superdiffusion). It is remarkable that for fixed scatterer size, convergence to Brownian motion is only known for periodic configurations $\scrP$ \cite{BS80,Dolgopyat09,Szasz07} (with $t\log t$ scaling in the case of infinite horizon); the case of random $\scrP$ is completely open \cite{Lenci11}. Finally, it would be instructive to generalise the entropy estimates for the limiting process $\Theta$ in \cite{Caglioti10} for the two-dimensional lattice setting to general scatterer configurations.

\backmatter

\bibliographystyle{amsalpha}

\chapter*{Index of notation}

\begin{center}
\begin{longtable}{llr}
$\perp$ & $\vecx_\perp=\vecx-(\vecx\cdot\vece_1)\vece_1$ 
and $(\vecx,\vs)_\perp=(\vecx_\perp,\vs)$
& \pageref{Xperpdef1}, \pageref{perpdef2}
\\
$\widehat{\:\:}$ & $\hs=\|\vecs\|^{-1}\vecs$ 
& \hyperlink{hsdeflink}{\pageref*{hsdef}}
\\
$\partial_\ve$ & 
$\partial_\ve{U}:=\{\vecv\in\S_1^{d-1} \col
\exists \vecw\in\partial{U}\col \varphi(\vecv,\vecw)<\ve\}$
(for ${U}\subset\US$)
&\pageref{PARTIALVEdef}
\\
$\ASL(n,\R)$ & $\SL(n,\R)\ltimes \RR^n$ & \pageref{ASLMULTLAW}
\\
$\scrB^d(\vecq,r)$ & open ball in $\R^d$ with center $\vecq$ and radius $r$ & \pageref{scrBdef}
\\
$\scrB_r^d$ & $\scrB^d(\bn,r)$ & \pageref{BRd} 
\\
$B_\Psi$ & constant in \eqref{VPDEF} & \pageref{VPDEF}
\\
$\vecB_{\rho,\vecs,\vecbeta}(\vecv)$ & impact location on $\S_1^{d-1}$ & \pageref{SHINELEM2VDEF}
\\
$\C_b(U,\R^d)$ & space of bounded continuous functions $\vecbeta:U\to\RR^{d}$ %
&\pageref{CbURddef}
\\
$\C^1_b(U,\R^d)$ & $\{\vecbeta\in\C_b(U,\R^d)\col \vecbeta\text{ is $\C^1$ and }
\sup_{\vech\in \T^1(U)} \|D_\vech\vecbeta\|<\infty\}$
&\pageref{Cb1URddef}
\\
$\fC_\tau$ & open cylinder $(c,c+\tau)\times\fD$ &\pageref{CtauDEF}
\\
$c_\scrP$ & density of $\scrP$ & \pageref{ASYMPTDENSITY0}
\\
$C_\eta$ & constant defined in \eqref{C1THETAETADEF} & \pageref{C1THETAETADEF}
\\
$D_\rho$ & diagonal matrix $\diag[\rho^{d-1},\rho^{-1},\cdots,\rho^{-1}]$ & \pageref{Drhodef}
\\
$\scrD_\vecs^\eta$ & 
ball of radius $\eta$ with center $\hs$ in $\US$
&\pageref{scrDsetaDEF}
\\
$d_\scrP(\vecq)$
&$\inf\{\|\vecp-\vecq\|\col\vecp\in\scrP\setminus\{\vecq\}\}$ & \pageref{dscrPdef}
\\
$d\vecq$
& shorthand for $d\!\vol(\vecq)$ & \pageref{dvdo} 
\\
$d\vecv$
& shorthand for $d\omega(\vecv)$ & \pageref{dvdo} 
\\
$\EE\xi$ & intensity measure of $\xi$ & \pageref{intensitydef}
\\
$\scrE$ & fixed subset of $\scrP$ (``exceptional points'') & \pageref{Edef}
\\
$\vece_1$ & $(1,0,\ldots,0)$ & \pageref{unitvecdef}
\\
$F_\rho$ & map on $\Delta$ defined on p.\ \pageref{Frhodef} & \pageref{Frhodef}
\\
$F_{f,g}(\mu)$ & $\int_{S'}g(p,\eta(f))\,d\mu(p,\eta)$ & \pageref{Ffgdef}
\\
$\fg_{\vecq,\rho,\eta}^\vecbeta$ & set of those $\vecv\in\fw_{\vecq,\rho}^\vecbeta$
which give rise to $\eta$-grazing paths
& \pageref{ETAGRAZINGprop}
\\
$\fG_{\rho,\eta}$ &
set of all $(\vecq,\vecv)\in\fW(1;\rho)$ which give rise to $\eta$-grazing paths,
&\pageref{fGrhoetaDEF}
\\
$g'$, $g_B$ & matrices defined in \eqref{gpdef} & \pageref{gpdef}
\\
$\iota$ & reflection map on $\scrX$ & \pageref{iotaDEF}
\\
$\scrK_\rho$ & $\R^d\setminus\cup_{\vecq\in\scrP}\scrB^d(\vecq,\rho)$, billiard domain & \pageref{Krhodef}
\\
$k$ & transition kernel defined in \eqref{kdef} & \pageref{kdef}
\\
$k^g$ & transition kernel defined in \eqref{kgdef} & \pageref{kgdef}
\\
$L_t^{(\rho)}$ & evolution operator for flow $\widetilde\Phi_t^{(\rho)}$ & \pageref{Ltrhodef} 
\\
$\scrL$ & lattice of full rank in $\R^d$ or a grid in $\RR^n$ & \pageref{periodicpointsetsSEC}, \pageref{afflatt}
\\
$M(\scrX)$ & set of locally finite Borel measures on $\scrX$ & \pageref{MXdef}
\\
$\scrM$ & Borel $\sigma$-algebra of $M(\scrX)$ & \pageref{scrMdef}
\\
$\mm$ & fixed Borel probability measure on $\Sigma$ & \pageref{mmdef}
\\
$N(\scrX)$ & set of counting measures in $M(\scrX)$ & \pageref{NXdef}
\\
$N_s(\scrX)$ & set of simple counting measures in $M(\scrX)$ 
& \pageref{NsXdef}
\\
$\scrN$ & Borel $\sigma$-algebra of $N(\scrX)$ & \pageref{scrNdef}
\\
$n_t$ & number of collisions until time $t$ & \pageref{eq:nT}
\\
$\scrP$ & a fixed locally finite subset of $\R^d$ & \pageref{scrPdef}
\\
$\scrP(\scrW,\scrL)$ & cut-and-project set in $\R^d$ & \pageref{CUTPROJDEF}
\\
$\tP$ & $\{(\vecp,{\vs}(\vecp))\col\vecp\in\scrP\}$ %
(a subset of $\scrX$)
& \pageref{tPdef}
\\
$\scrP^\circ(\rho)$ & maximal $2\rho$-separated subset of $\scrP$ & \pageref{Pcirc}
\\
$\scrP_T(\rho)$ & $\scrP\cap\scrB^d_{T\rho^{1-d}}\setminus\scrE$ & \pageref{ASS:KEY}
\\
$\tP_\vecq$ & $\tP\setminus\{(\vecq,\vs(\vecq))\}$ if $\vecq\in\scrP$, otherwise $\tP$ & \pageref{Pqdef}
\\
$P(S)$ & space of Borel probability measures on $S$ & \pageref{PSdef}
\\
$\Pac(S)$ & subspace of absolutely continuous measures in $P(S)$ & \pageref{acBpm}, \pageref{acBpm2}
\\
$p_\bn$ %
&collision kernel defined in \eqref{pbndefG} & \pageref{pbndefG}
\\
$p_{\bn,\vecbeta}$ %
&collision kernel defined in \eqref{p0bdef} & \pageref{p0bdef}
\\
$p$ %
&collision kernel defined in \eqref{pgendef} & \pageref{pgendef}
\\
$p_1$ & projection map $\scrX\to\R^d$  & \pageref{p1def}
\\
$\scrQ_\rho(\vecq,\vecv)$ & $(\tP_\vecq-\vecq)\,R(\vecv)\,D_\rho$ & \pageref{XIRHOqv}
\\
$\scrQ_\rho(\vecq,\vecbeta,\vecv)$ & $(\tP_\vecq-\vecq-\rho\vecbeta(\vecv))\,R(\vecv)\,D_\rho$ & \pageref{XIrhoqvdef}
\\
$\oQ_\rho(\vecq,\vecv)$ & $(\tP-\vecq)\,R(\vecv)\,D_\rho$ & \pageref{oQrhoqvDEF}
\\
$\vecq_j(\vecq,\vecv;\rho)$ 
& particle position at $j$th collision & \pageref{taujdef}
\\
$\vecq^{(j)}(\vecq,\vecv;\rho)$
& center of scatterer causing $j$th collision 
& \pageref{q1scattererdef}, \pageref{qjscattererdef}
\\
$\vecq_{\rho,\vecbeta}(\vecv)$ & initial position $\vecq+\rho\vecbeta(\vecv)$ 
&\pageref{qrhobetavDEF}
\\
$R$ & a fixed map $\US\to\SO(d)$ such that $\vecv R(\vecv)=\vece_1$, $\forall \vecv\in\US$ & \pageref{Rdef}
\\
$\US$ & unit sphere in $\R^d$ centered at the origin & \pageref{unitsphere}
\\
$\scrS_-$ & $\{(\vecv,\vecb)\in\US\times\US\col\vecv\cdot\vecb<0\}$, set of incoming data & \pageref{scrSmdef}
\\
$\scrS_+$ & $\{(\vecv,\vecb)\in\US\times\US\col\vecv\cdot\vecb>0\}$, set of outgoing data & \pageref{scrSmdef}
\\
$s_-$ & diffeomorphism from $\UB$ onto $\{\vecx\in\US\col x_1<0\}$; \eqref{smDEF}
&\pageref{smDEF}
\\
$s_\Psi$ & $\pm 1$ defined by \eqref{Psi1vmv} & \pageref{Psi1vmv}
\\
$T_n$ & time until $n$th collision & \pageref{eq:nT}
\\
$\SO(d-1)$ & $\{k\in\SO(d)\col\vece_1k=\vece_1\}$ & \pageref{SOd-1}
\\
$\T^1(\scrK_\rho)$ & $\scrK_\rho\times\US$ & \pageref{T1Krhodef}
\\
$\T^1(\partial\scrK_\rho)_{\out}$ & set of outgoing positions  & \pageref{T1Krhooutdef}
\\
$\T^1(U)$ & Unit tangent bundle of an open set $U\subset\US$ & \pageref{UNITTBDLEofU}
\\
$\fU_\eta$ & $\vecbeta^-_{\vece_1}\bigl(\scrV_{\vece_1}\setminus\scrV_{\vece_1}^{10\eta}\bigr)_\perp$ & \pageref{FUETA}
\\
$\tfU_\eta$ & $\fU_\eta\cup(\scrB_1^{d-1}\setminus\scrB_{1-\eta}^{d-1})$ & \pageref{tfUdef}
\\
$v_{d-1}$ & $\vol(\scrB_1^{d-1})$  &  \pageref{DscatCS} %
\\
$\vecv_j(\vecq,\vecv;\rho)$ & velocity immediately after $j$th collision & \pageref{taujdef}
\\
$V_\rho$ & potential defining the Lorentz gas & \pageref{Vrhodef}
\\
$\scrV_\vecv$ & $\scrV^0_\vecv$ & \pageref{VPDEF}
\\
$\scrV^\eta_\vecv$ & $\{\vecu\in\S_1^{d-1}\col{s_\Psi}\cdot(B_\Psi-\varphi(\vecu,\vecv))>\eta\}$ & \pageref{WVPDEF}
\\
$\vecV_{\!\!\!\rho,\vecs,\vecbeta}(\vecv)$ & $\Psi_1(\vecv,\vecB_{\rho,\vecs,\vecbeta}(\vecv))$ & \pageref{SHINELEM2VDEF}
\\
$\fw(\rho)$ & $\T^1(\scrK_\rho^\circ) \cup \T^1(\partial\scrK_\rho)_{\out}$ & \pageref{wini}
\\
$\tfw(\rho)$ &  accessible phase space for flow in potential & \pageref{tfwdef}
\\
$\fw(j;\rho)$ & set of initial conditions leading to at least $j$ collisions & \pageref{winij}
\\
$\fw_{\vecq,\rho}^{\vecbeta}$ & $\{\vecv\in U\col(\vecq_{\rho,\vecbeta}(\vecv),\vecv)\in\fw(1;\rho)\}$
&\pageref{wqrbdef}
\\
$\fw_{\vecq,\rho,n}^{\vecbeta}$ & $\{\vecv\in U\col(\vecq_{\rho,\vecbeta}(\vecv),\vecv)\in\fw(n;\rho)\}$
&\pageref{fwqrhonDEF}
\\
$\vecw_1(\vecq,\vecv;\rho)$ & normalized impact parameter for the first collision & \pageref{w1def}
\\
$\fW(j;\rho)$ & $\{(\vecq,\vecv)\in\T^1(\R^d)\col\langle\rho^{1-d}\vecq,\vecv\rangle\in\fw(j;\rho)\}$
&\pageref{fWjrhoDEF}
\\
$W$ & single-site scattering potential & \pageref{Vrhodef}
\\
$\scrX$ & arbitrary lcscH space (Section \ref{scrXdef1} only) & \pageref{scrXdef1}
\\
$\scrX$ & $\R^d\times\Sigma$ 
& \pageref{scrXdef2}
\\
$\scrX_\perp$ & $\R^{d-1}\times\Sigma$ & \pageref{Xperp}
\\
$X$ & extended phase space $\T^1(\RR^d)\times\R_{>0}\times\Sigma\times\US$ with $\vecv_+\in\scrV_\vecv$  & \pageref{XXdef}
\\
${X_U}=X_U^{(n)}$
& set %
defined in \eqref{XSPACEdef}
(for $n=1$: also \eqref{XSPACEdef1})
& \pageref{XSPACEdef1}, \pageref{XSPACEdef}
\\
$X^{(n)}$ & %
set defined in \eqref{Xmacrdef}
& \pageref{Xmacrdef}
\\
$\fZ_\xi$ & open cylinder $(0,\xi)\times\scrB_1^{d-1}$ & \pageref{fZxidef}, \pageref{Zinfty}
\\
$\fZ_\infty$ & infinite cylinder $\R_{>0}\times\scrB_1^{d-1}$ & \pageref{Zinfty}
\\
$\vecz$ & map from $N_s(\scrX)$ to $\Delta$ defined in \eqref{DeltaDEF} & \pageref{DeltaDEF}
\\
$\vecz_\rho$ & map from $N_s(\scrX)$ to $\Delta$ defined in \eqref{zrhoDEF} & \pageref{zrhoDEF}
\\
$\vecbeta_\vecv^-$ & map from $\scrV_\vecv$ to $\{\vecb\in\S_1^{d-1}\col\vecv\cdot\vecb<0\}$ & \pageref{betavm}
\\
$\vecbeta_\vecv^+$ & map from $\scrV_\vecv$ to $\{\vecb\in\S_1^{d-1}\col\vecv\cdot\vecb>0\}$ & \pageref{betavm}
\\
$\vecbeta_\vecu$ & function $\vecbeta_\vecu(\vecv)=\vecu R(\vecv)^{-1}+(1-\|\vecu\|^2)^{1/2}\vecv$
& \pageref{BETAUdef}
\\
$\Delta$ & $(\R_{>0}\times\Omega)\sqcup\{\undef\}$ (with the disjoint union topology) & \pageref{DeltaDEF}
\\
$\eta_{\vecq,\rho}^{(\vecbeta,\lambda)}$ & 
\parbox[t]{300pt}{distribution
of $(\vecv,[F_\rho\circ\vecz_\rho](\scrQ_\rho(\vecq,\vecbeta,\vecv)))$
in $U\times\Delta$
for $\vecv$ random in $(U,\lambda|_{U})$}
&\pageref{etaqrhobetalambdaDEF}
\\
$\eta_{\vs}^{(\vecbeta,\lambda)}$ & 
\parbox[t]{300pt}{distribution
of $(\vecv,[\iota\circ\vecz](\Xi_{\vs}-(\vecbeta(\vecv)R(\vecv))_\perp))$
in $U\times\Delta$
for $\vecv$ random in $(U,\lambda|_{U})$
and independent from $\Xi_{\vs}$}
&\pageref{etaqrhobetalambdaDEF}
\\
$\eta_\rho^{(\Lambda)}$ & 
\parbox[t]{300pt}{distribution of 
$(\vecq,\vecv,[F_\rho\circ\vecz_\rho](\scrQ_\rho(\rho^{1-d}\vecq,\vecv)))$
in $\T^1(\R^d)\times\Delta$
for $(\vecq,\vecv)$ random in $(\T^1(\R^d),\Lambda)$}
&\pageref{etarhoLambdaDEF}
\\
$\eta^{(\Lambda)}$ 
& \parbox[t]{300pt}{distribution of
$(\vecq,\vecv,[\iota\circ\vecz](\Xi))$
in $\T^1(\R^d)\times\Delta$
for $(\vecq,\vecv)$ random in $(\T^1(\R^d),\Lambda)$ and independent from $\Xi$}
&\pageref{etarhoLambdaDEF}
\\
$\Theta$, $\Theta^{(\rho)}$ & random flight process & \pageref{eq:RFP} 
\\
$\hTheta$ & Markovian random flight process & \pageref{eq:RFP2}
\\
$\theta(w)$ & deflection angle & \pageref{PSI1formula}, \pageref{deflangle}
\\
$\kappa((\vecx,\vs);\:\cdot\,)$ & distribution of the 
random point $\iota(\vecz(\Xi_{{\vs}}-\vecx))$ in $\R_{>0}\times\Omega$
&\pageref{kappadef}
\\
$\kappa^{\g}$ & distribution of the random point $\iota(\vecz(\Xi))$ in $\R_{>0}\times\Omega$.
& \pageref{kappagDEF}
\\
$\mu_\scrX$ & measure $\vol\times\mm$ on $\scrX$ & \pageref{mmdef}
\\
$\mu_\Omega$ & measure $v_{d-1}^{-1}\vol_{\R^{d-1}}\times\mm$ on $\scrX_\perp$ & \pageref{ppdefG}
\\
$\mu_\vs$ & we have fixed a continuous map ${\vs}\mapsto\mu_{\vs}$ from $\Sigma$ to $P(N(\scrX))$
& \pageref{muvsdef}
\\
$\mu_{{\vs}}^{(\vecbeta,\lambda)}$ 
& \parbox[t]{300pt}{distribution of 
$\Xi_{\vs}-(\vecbeta(\vecv)R(\vecv))_\perp$, with $\vecv$ random in $(\US,\lambda)$ and independent from $\Xi_{\vs}$}
&\pageref{muvsbetalambdaDEF}
\\
$\tmu_{{\vs}}^{(\vecbeta,\lambda)}$
& distribution of $(\vecv,\Xi_{\vs}-(\vecbeta(\vecv)R(\vecv))_\perp)$
& \pageref{muvsbetalambdaDEF}
\\
$\mu_{\vecq,\rho}^{(\lambda)}$ & distribution of $\scrQ_\rho(\vecq,\vecv)$ for  $\vecv$ random in $(\US,\lambda)$
& \pageref{muqrholambdaDEF}
\\
$\tmu_{\vecq,\rho}^{(\lambda)}$ & 
distribution of $(\vecv,\scrQ_\rho(\vecq,\vecv))$ %
for $\vecv$ random in $(\US,\lambda)$
& \pageref{tmuqrholambdaDEF}
\\
$\tmu_{\vecq,\rho}^{(\vecbeta,\lambda)}$ & distribution of 
$(\vecv,\scrQ_\rho(\vecq,\vecbeta,\vecv))$ for $\vecv$ random in $(\US,\lambda)$
& \pageref{tmuqrhobetalambdaDEF}
\\
$\omu_{\vs}$ & distribution of $\oXi_{\vs}$ & \pageref{oXivsDEF}
\\
$\omu_{\vs}^{(\vecx)}$ & distribution of $\oXi_{\vs}+\vecx$ & \pageref{oXivsDEF}
\\
$\mu$ & a probability measure on $N_s(\scrX)$ defined by \eqref{GENMUdef} & \pageref{GENMUdef}
\\
$\mu_\rho^{(\Lambda)}$ & distribution of
$\scrQ_\rho(\rho^{1-d}\vecq,\vecv)$ for $(\vecq,\vecv)$ random in $(\T^1(\R^d),\Lambda)$
& \pageref{murhoLambda}
\\
$\omu_{\vecq,\rho}^{(\vecx,\lambda)}$ & distribution of $\oQ_\rho(\vecq,\vecv)+\vecx$
for $\vecv$ random in $(\US,\lambda)$
& \pageref{ASS:KEYbblem2}
\\
$\nu_\vecs$ & probability measure on $\US$ & \pageref{nusDEF}
\\
$\nu_\vecs^\eta$ & normalized restriction of $\nu_\vecs$ to $\scrV_\vecs^\eta$ & \pageref{nuseta}
\\
$\Xi_\vs$ & a point process in $\scrX$ with distribution $\mu_{\vs}$ & \pageref{Xivsdef}
\\
$\oXi_{\vs}$ & $\Xi_{\vs}\cup\{(\bn,{\vs})\}$ & \pageref{oXivsDEF}
\\
$\Xi$ & a point process in $\scrX$ with distribution $\mu$ & \pageref{GENMUdef}
\\
$\oxi$ & $(v_{d-1}c_\scrP)^{-1}$, mean free path length & \pageref{OXIFORMULAG}
\\
$\xi_\rho(\vecx)$ & $\inf\{\xi\in\R_{>0}\col\vecx\in\xi\vece_1+\scrB_\rho^dD_\rho\}$ & \pageref{xirhoDEF}
\\
$\pi$, $\pi_\intl$ & projection of  $\R^n=\R^d\times\R^m$ onto $\RR^d$, $\RR^m$ & \pageref{QCexsec1}
\\
$\Sigma$ & fixed compact metric space (the space of marks) & \pageref{Sigmadef}
\\
$\Sigma'$ & $\overline{\{{\vs}(\vecq)\col\vecq\in\scrP\setminus\scrE\}}\:\:$
& \pageref{SIGMAp}
\\
$\vs$ & a map from $\scrP$ to $\Sigma$ (the marking) & \pageref{Sigmadef}
\\
$\vs_j(\vecq,\vecv;\rho)$ & $\vs(\vecq^{(j)})$ & \pageref{qjscattererdef}
\\
$\sigma(\vecv,\vecv_+)$ & differential cross section & \pageref{crosssecdef}
\\
$\tau_j(\vecq,\vecv;\rho)$
& free path length between $(j-1)$st and $j$th collisions & \pageref{taujdef}
\\
$\tau_{\rho,\vecs,\vecbeta}(\vecv)$
& $\inf\{t>0\col\rho\vecbeta(\vecv)+t\vecv\in\vecs+\scrB_\rho^d\}$ & \pageref{TAUrhovecsvecbetaDEF}
\\
$\Phi_t=\Phi_t^{(\rho)}$ & billiard flow on $\T^1(\scrK_\rho)$ & \pageref{T1Krhodef}
\\
$\widetilde\Phi_t=\widetilde\Phi_t^{(\rho)}$ & rescaled billiard flow  & \pageref{tPhidef}
\\
$\varphi(\vecu,\vecv)$ & angle between $\vecu$ and $\vecv$ & \pageref{angledef}
\\
$\Psi$ & scattering process; a map from $\scrS_-$ to $\scrS_+$ & \pageref{PSIdef}
\\
$\psi$ & distribution of a Poisson process in $\R^d$ & \pageref{POISSONprop}
\\
$\Omega$ & $\scrB_1^{d-1}\times\Sigma$ & \pageref{OmegaDEF}
\\
$\omega$ & $\vol_{\S_1^{d-1}}$, Lebesgue measure on $\US$ & \pageref{omegadef}
\\
$\omega_1$ & $\omega(\US)^{-1}\,\omega$ & \pageref{omegadef}
\\
$\vecomega_1(\vecq,\vecv;\rho)$ & $\bigl(\vecw_1(\vecq,\vecv;\rho),{\vs}(\vecq^{(1)}(\vecq,\vecv;\rho))\bigr)$
&\pageref{vecomega1def}
\end{longtable}
\end{center}

\end{document}